\documentclass[11pt]{amsart}
\usepackage{hyperref,amsmath,amsthm,amsfonts,amscd,flafter,epsf,amssymb,
ifsym,wasysym,color}
\usepackage{epsfig}
\usepackage{amsfonts}
\usepackage{amssymb}
\usepackage{amsmath}
\usepackage{amsthm}
\usepackage{latexsym}
\usepackage{amscd}
\usepackage{epsf}
\usepackage[lmargin=1in,rmargin=1in,tmargin=0.8in,bmargin=0.8in]{geometry}
\usepackage{enumitem}
\usepackage{todonotes}

\input{diagram}
\newcommand{\spincT}{\mathfrak{T}}
\newcommand{\DD}{\mathbb{D}}

\newcommand{\lav}{\mathfrak{v}}
\newcommand{\law}{\mathfrak{w}}

\newcommand{\F}{\mathbb F}

\newcommand{\Rr}{\mathfrak{R}}

\newcommand{\fmap}{\mathfrak{f}}

\newcommand{\hmove}{e}
\newcommand{\hpath}{\mathfrak{h}}
\newcommand{\tfrak}{\mathfrak{t}}
\newcommand{\Sqmap}{{\mathfrak{S}}}
\newcommand{\AModCat}{\Ring\text{-}\mathsf{Mod}}
\newcommand{\AModuleCat}{\Ring\text{-}\mathsf{Modules}}

\newcommand{\Tangle}{\mathcal{T}}
\newcommand{\Lin}{\mathcal{L}}

\newcommand{\ATangleCat}{\Ring\text{-}\mathsf{Tang}}
\newcommand{\ACobCat}{\Ring\text{-}\mathsf{Tangles}}

\newcommand{\HFTmap}{\HFT}
\newcommand{\Rin}{\mathbb{M}}

\newcommand{\Cerf}{{\mathfrak{C}}}
\newcommand{\Modi}{{\mathfrak{M}}}
\newcommand{\HD}{\mathcal{H}}

\newcommand{\nd}{\mathrm{nd}}

\newcommand{\la}{\mathfrak{u}}
\newcommand{\polygon}{\Psi}
\newcommand{\Poly}{\mathrm{Polygons}}
\newcommand{\Ht}{\mathrm{H}}
\newcommand{\Hbb}{\mathbb{H}}

\newcommand{\ov}{\widehat}
\newcommand{\ovl}{\overline}
\newcommand{\ti}{\widetilde}

\newcommand{\Fsphere}{\mathbb{S}}

\newcommand{\Ncal}{\mathcal{N}}

\newcommand{\Tai}{{\mathbb{T}}_{i}}

\newcommand{\PP}{\mathbb{P}}

\newcommand{\Gcal}{\mathcal{G}}

\newcommand{\Dcal}{\mathcal{D}}

\newcommand{\el}{\kappa}
\newcommand{\Pcal}{\mathcal{P}}
\newcommand{\sig}{\sigma}
\newcommand\Field{\mathbb F}

\newcommand\ub{\underline{b}}

\newcommand\Dual{\mathcal D}
\newcommand\Duality\Dual

\newcommand\interior{\mathrm{int}}

\newcommand{\relspinc}{{\underline{\spinc}}}

\newcommand\x{\mathbf x}
\newcommand\w{\mathbf w}
\newcommand\z{\mathbf z}
\newcommand\p{\mathbf p}

\newcommand\y{\mathbf y}

\newcommand\sX{\mathcal X}
\newcommand\sY{\mathcal Y}
\newcommand\sZ{\mathcal Z}

\newcommand\ModSphere{\ModFlow\left({\mathbb S}\longrightarrow
\Sym^{g-1}(\Sigma_{1})\times \Sym^2(\Sigma_{2})\right)}
\newcommand\ModSpheres\ModSphere

\newcommand\gr{\mathrm{gr}}

\newcommand\UnparModSp{\widehat \ModSp}
\newcommand\UnparModFlow\UnparModSp

\newcommand\PD{\mathrm{PD}}

\newcommand{\spinc}{\mathfrak s}
\newcommand{\spinct}{\mathfrak t}

\newcommand\sur{F}

\newcommand\ModMaps{\mathcal M}
\newcommand\ModSp\ModMaps

\newcommand\Ta{{\mathbb T}_{\alpha}}
\newcommand\Tb{{\mathbb T}_{\beta}}
\newcommand\Tc{{\mathbb T}_{\gamma}}

\newcommand\del{\partial}

\newcommand\alphas{\mbox{\boldmath$\alpha$}}

\newcommand\betas{\mbox{\boldmath$\beta$}}
\newcommand\gammas{\mbox{\boldmath$\gamma$}}

\newcommand\deltas{\mbox{\boldmath$\delta$}}

\newcommand\Ring{\mathbb A}

\newcommand\spincrel\relspinc

\newtheorem{thm}{Theorem}[section]
\newtheorem{prop}[thm]{Proposition}
\newtheorem{cor}[thm]{Corollary}

\newtheorem{lem}[thm]{Lemma}

\newtheorem{remark}[thm]{Remark}

\theoremstyle{definition}
\newtheorem{defn}{Definition}[section]
\newtheorem{example}{Example}[section]

\def\endproof{\relax\ifmmode\expandafter\endproofmath\else
  \unskip\nobreak\hfil\penalty50\hskip.75em\hbox{}\nobreak\hfil\bull
  {\parfillskip=0pt \finalhyphendemerits=0 \bigbreak}\fi}
\def\endproofmath$${\eqno\bull$$\bigbreak}
\def\bull{\vbox{\hrule\hbox{\vrule\kern3pt\vbox{\kern6pt}\kern3pt\vrule}\hrule}}

\newcommand{\R}{\mathbb{R}}
\newcommand{\T}{\mathbb{T}}

\newcommand{\Z}{\mathbb{Z}}

\newcommand{\ind}{\mathrm{ind}}
\newcommand{\Image}{\mathrm{Im}}
\newcommand{\Span}{\mathrm{Span}}

\newcommand{\Nbd}[1]{{\mathrm{nd}}(#1)}
\newcommand{\pii}[1]{\pi_0(#1)}
\newcommand{\nbd}[1]{\Nbd{#1}}

\newcommand{\ModSWfour}{\mathcal{M}}
\newcommand{\ModFlow}{\ModSWfour}

\newcommand{\SpinC}{{\mathrm{Spin}}^c}

\newcommand\abuts\Rightarrow
\newcommand\Sym{\mathrm{Sym}}

\newcommand{\Sarc}{J}
\newcommand{\arcj}{J}

\newcommand{\Farc}{{\mathbb{I}}}
\newcommand{\Cob}{\mathcal{C}}
\newcommand{\Cobb}{\mathcal{W}}

\newcommand{\HFKT}{\text{HFK}}
\newcommand{\CFT}{\mathrm{CF}}
\newcommand{\HFT}{\mathrm{HF}}

\newcommand{\D}{\mathbb{D}}

\newcommand{\M}{\mathbb{M}}
\newcommand{\m}{\mathfrak{m}}

\newcommand{\lra}{\longrightarrow}
\newcommand{\ra}{\rightarrow}
\newcommand{\Sig}{\Sigma}

\newcommand{\Mod}{\mathcal{M}}

\newcommand{\Ci}{{\Cerf}}

\newcommand{\Ker}{\text{Ker}}

\newcommand{\zet}{\mathfrak{v}}

\makeatletter
\providecommand\@dotsep{5}
\def\listtodoname{List of Todos}
\def\listoftodos{\@starttoc{tdo}\listtodoname}
\makeatother

\begin{document}

\title{Tangle Floer homology and cobordisms between tangles}%
\author{Akram Alishahi}
\address{Department of Mathematics, Columbia University, New York, NY 10027}
\email{alishahi@math.columbia.edu}
\author{Eaman Eftekhary}%
\address{School of Mathematics, Institute for Research in Fundamental Sciences (IPM),
P. O. Box 19395-5746, Tehran, Iran}%
\email{eaman@ipm.ir}

%\todo{Change the abstract}
\begin{abstract}

We introduce a generalization of oriented tangles, which are still 
called \emph{tangles}, so that they are in one-to-one correspondence 
with the sutured manifolds. We define cobordisms between sutured 
manifolds (tangles) by generalizing cobordisms between oriented 
tangles. For every commutative algebra $\Ring$ over $\Z/2\Z$, we 
define $\ACobCat$ to be the category consisting of $\Ring$-tangles, 
which are balanced tangles with $\Ring$-colorings of the tangle 
strands and fixed $\SpinC$ structures, and $\Ring$-cobordisms as 
morphisms. An $\Ring$-cobordism is a cobordism with a compatible 
$\Ring$-coloring and an \emph{affine} set of $\SpinC$ structures.  
Associated with every $\Ring$-module $\Rin$ we construct a functor
\begin{displaymath}
\HFT^\Rin:\ACobCat\lra \AModuleCat,
\end{displaymath}
called the {\emph{tangle Floer homology functor}}, where $\AModuleCat$ denotes the the category of $\Ring$-modules and  
$\Ring$-homomorphisms between them. Moreover, for any $\Ring$-tangle $\Tangle$ the $\Ring$-module $\HFT^{\Rin}(\Tangle)$ is the extension of sutured Floer homology defined in an earlier work of the authors.

%
%
%
%which associates the $\Ring$-module $\HFT^\Rin(\Tangle)$ to the
%$\Ring$-tangle $\Tangle$ and assigns an $\Ring$-homomorphism
%\[\fmap_{\Cob}^\Rin:\HFT^\Rin(\Tangle)\ra \HFT^\Rin(\Tangle')\] 
%to every $\Ring$-cobordism $\Cob$ from $\Tangle$ to $\Tangle'$. Here, $\AModuleCat$, is the the category of $\Ring$-modules and  
%$\Ring$-homomorphisms between them.

In particular, this construction generalizes the $4$-manifold 
invariants of Ozsv\'ath and Szab\'o. Moreover, applying the above 
machinery to decorated cobordisms between links, we get functorial 
maps on link Floer homology.
\end{abstract}
\maketitle

\tableofcontents

%\listoftodos

\section{Introduction}\label{sec:intro}
% ----------------------------------------------------------------

\subsection{Introduction and background}
Ozsv\'ath and Szab\'o introduced  Heegaard Floer homology 
for closed three dimensional manifolds \cite{OS-3m1,OS-3m2} which
resulted in powerful tools for the study of various structures in low 
dimensional topology, including invariants for knots 
\cite{OS-knot, Ras, Ef-LFH}, for links \cite{OS-link}, for contact 
structures \cite{OS-contact} and for sutured manifolds 
\cite{Juh,AE-1}. Juh\'asz and Thurston \cite{JT} showed that the 
Heegaard Floer groups associated with three-dimensional objects 
(closed manifolds, links and sutured manifolds)are in fact functors 
which associate a concrete module to any of the aforementioned 
topological objects, rather than just the isomorphism class of it.
Typically, Heegaard Floer homology groups come in different flavours, 
which are denoted by $\ov\HFT,\HFT^+,\HFT^-$ and $\HFT^\infty$, 
besides many other flavours which appear in knot and link Floer 
homology theories. The simplest version of these invariants, 
$\widehat{\HFT}$ and $\widehat{\HFKT}$, has been generalized to 
compact, non-closed $3$-manifolds, with a specific on the boundary, 
called sutured manifolds by Juh\'{a}sz. The authors gave a framework 
that  generalizes sutured Floer homology and brings all flavours of 
Heegaard Floer homology under the same roof in \cite{AE-1}.\\

In this paper, to define a natural notion of cobordism between 
sutured manifolds, that generalizes cobordisms between $3$-manifolds, 
knots and links, we introduce a generalization of classical oriented 
tangles such that they are in one-to-one correspondence with sutured 
manifolds without toroidal sutures. So, we call our Heegaard Floer 
invariants, \emph{tangle Floer homology}, and denote it by $\HFT$.

\begin{defn}\label{def:tangle}
A {\emph{tangle}} $(M,T)$ is an oriented $3$-manifold 
$M$ with boundary and a properly embedded, oriented $1$-manifold 
$T$. Both $M$ and $T$ have no closed components, $\del M$ is equipped 
with a fixed decomposition $\del M=\del^+M\amalg\del^-M$ such that 
$\del^-T\subset \del^-M$ and $\del^+T\subset\del^+M$.

The tangle $(M,T)$ is called {\emph{balanced}} if every component of 
$\del M$ intersects $T$ and for every connected  component $M_\circ$ 
of $M$, $\chi(\del^+ M_\circ)= \chi(\del^-M_\circ)$.
\end{defn}

Fix  an algebra $\Ring$ over $\F=\Z/2\Z$ (which will always be 
commutative through this paper). For every connected component 
$s\in\pi_0(\del ^\circ M)$ and any map $\la:\pi_0(T)\to\Ring$ define 
\[\la(s):=\prod_{\substack{t\in \pi_0(T)\\ 
\imath_T^\circ(t)=s}}\la(t)\] 
where $\circ=+,-$ and  $\imath_T^\circ:
\pi_0(T)=\pi_0(\partial^\circ T)\ra \pi_0(\partial^\circ M)$
is the map induced by the 
inclusion $\del^{\circ} T\subset \del^\circ M$.

\begin{defn}
An \emph{$\Ring$-coloring} for a balanced tangle $(M,T)$ is a map 
$\la:\pi_0(T)\to \Ring$ satisfying the following two conditions.

\begin{enumerate}
\item If $s\in\pi_0(\del M)$ corresponds to a 
connected component  with positive genus, then $\la(s)=0$,
\item For every connected component $M_\circ$ of $M$
$$\sum_{s\in\pii{\del^+M_\circ}}\la(s)=
\sum_{s\in\pii{\del^-M_\circ}}\la(s),$$ 
\end{enumerate}
\end{defn}
In particular, $\Ring=\F$ and $\la=0$  give an $\Ring$-coloring for 
any balanced tangle. 

The set $\SpinC(M)$ of $\SpinC$ structures over a tangle $(M,T)$ is 
the set of homology classes of nonzero vector fields 
on $M$ which restrict to the outward normal of $\partial^+M$ and 
the inward normal of $\partial^-M$.

\begin{defn}\label{def:A-tangle}
An $\Ring$\emph{-tangle} is a $4$-tuple 
$\Tangle=[M,T,\spinc,\la]$ where  $(M,T)$ is a balanced tangle, 
$\spinc\in\SpinC(M)$ is a $\SpinC$ structure over $M$
and  $\la:\pii{T}\ra \Ring$ is an $\Ring$-coloring of $(M,T)$.
\end{defn}

If $\Tangle$ is an $\Ring$-tangle, we use 
$[M_\Tangle,T_\Tangle,\spinc_\Tangle,\la_\Tangle]$
to denote the corresponding $4$-tuple. 
%Moreover, for every $\Ring$-tangle $\Tangle$ 
%we define $\SpinC(\Tangle):=\SpinC(M_\Tangle)$.\\
Given an $\Ring$-module $\Rin$ and associated with an $\Ring$-tangle  
$\Tangle$ the construction of the authors in \cite{AE-1} defines an 
$\Ring$-module
\begin{displaymath}
\HFT^\Rin(\Tangle)=H_*(\CFT(\Tangle)\otimes_\Ring \Rin)
\end{displaymath}
so  that its isomorphism type is an invariant of $\Tangle$. 
In light of Juh\'asz and Thurston's naturality discussions in 
\cite{JT}, one can strengthen this result as follows. 
\begin{defn}\label{def:ATangleCat} The category $\ATangleCat$ is 
defined such that:
\begin{enumerate}
\item Objects are $\Ring$-tangles.
\item Morphisms from an $\Ring$-tangle $\Tangle=[M,T,\spinc,\la]$ 
to another $\Ring$-tangle $\Tangle'=[M',T',\spinc',\la']$ are the 
diffeomorphisms $d:(M,T)\ra (M',T')$ such that $d^*\spinc'=\spinc$ 
and the following diagram is commutative. 
\begin{diagram}
\pii{T}&&\rTo{d_*}&&\pii{T'}\\
&\rdTo{\la}&&\ldTo{\la'}&\\
&&\Ring &&
\end{diagram}
\end{enumerate}
Here $d^*:\SpinC(M')\ra \SpinC(M)$ and $d_*:\pi_0(T)\ra \pi_0(T')$ 
are the maps induced by the diffeomorphism $d$.
\end{defn}
Let $\AModCat$ denote the category of $\Ring$-modules together with 
the isomorphisms between them. 
Using  \cite[Theorem 2.39]{JT} one can prove
the following theorem (see Section~\ref{sec:naturality}).

\begin{thm}\label{thm:intro-category}
For every algebra $\Ring$ over $\F$ and every 
$\Ring$-module $\Rin$,
assigning the $\Ring$-module 
$\HFT^\Rin(\Tangle)$ to the $\Ring$-tangle $\Tangle$
in  $\ATangleCat$ gives a functor
\begin{displaymath}
\HFT^\Rin: \ATangleCat\lra \AModCat.
\end{displaymath}
\end{thm}

\begin{example}\label{ex:HFcomp}
Tangle Floer homology extends sutured Floer homology, 
as well as different 
versions of knot and link Floer homology in the following sense:

\begin{itemize}
\item[(a)] Let $\Ring=\F$. Equipping any balanced tangle $(M,T)$ 
with the trivial $\Ring$-coloring $\la=0$ and an arbitrary $\SpinC$ 
class $\spinc\in\SpinC(M)$, we have 
$\HFT^{\F}(M,T,\spinc,\la)=\mathrm{SFH}(M,T,\spinc)$. 
\item[(b)] Corresponding to any closed, oriented $3$-manifold $Y$ 
with a based point $p\in Y$, there is a tangle $(Y_p,T_p)$ where $Y_p$ 
is obtained from $Y$ by removing two disjoint $3$-balls, and 
$T_{p}\subset Y_{p}$ is a properly embedded, oriented arc connecting 
the two sphere boundary component. Furthermore, $T_{p}$ passes 
through the point $p$. For $\Ring=\F[\la]$, there is a natural 
$\Ring$-coloring $\la_p$ of $T_{p}$ which labels $T_p$ by $\la$. 
Then, for any $\SpinC$ structure $\spinc\in\SpinC(Y)$, setting the 
$\Ring$-module $\Rin$ equal to $\F, \F[\la], \F[\la,\la^{-1}]$
and $ \F[\la^{-1}]$ the tangle Floer homology  
$\HFT^{\Rin}(Y_p,T_p,\spinc,\la)$ is equal to $\widehat{\HFT}(Y,\spinc)$, 
$\HFT^-(Y,\spinc)$, $\HFT^{\infty}(Y,\spinc)$ and 
$\HFT^+(Y,\spinc)$, respectively.

\item[(c)] Suppose that  $L=\amalg_{i=1}^mL_i$ is an oriented link 
in a connected, oriented, closed $3$-manifold $Y$ and 
$\p=\{p_1,\ldots,p_n\}$ is a collection of markings on $L$, so that 
each component of $L$ contains at least one marked point. We may 
consider a collection of $n$ small arcs on $L$ containing $\p$, and 
remove small balls from a neighborhood of the endpoints of these 
arcs. This gives a $3$-manifold $Y_{\p}$ with $2n$ sphere boundary 
components. The orientation on $L$ may be used to decompose this 
boundary into $n$ spheres in $\partial^+Y_{\p}$ and $n$ spheres in 
$\partial^-Y_{\p}$. After changing the orientation of the small arcs, 
we also obtain a tangle $L_{\p}$ with $2n$ connected components which 
connect the negative boundary to the positive boundary. Let 
$\Ring=\F[\la_1,\ldots,\la_n,\lav_1,..,\lav_m]$, and label the small 
arc containing $p_i$ by $\la_i$ and the remaining arcs on $L_i$ by 
$\lav_i$. This gives an $\Ring$-coloring of $(Y_{\p},L_{\p})$, denote 
by $\la_{\p}$, which assigns the variables $\la_1,\ldots,\la_n$
to the marked points in $\p$ and the variables $\lav_1,\ldots,\lav_m$
to the connected components of the link $L$.  For any $\SpinC$ 
structure $\spinc\in\SpinC(Y)$, the tangle Floer homology groups 
$\HFT^{\Rin}(Y_{\p},L_{\p},\la_{\p},\spinc)$ give the usual link 
Floer homology groups associated with $L$ using different 
$\Ring$-modules $\M$. In particular, for any knot $K\subset S^3$ with 
one based point $p$, 
$\HFT^{\frac{\F[\la,\lav]}{\langle \lav\rangle}}
(S^3_p,K_p,\la_p,\spinc)=\HFKT^-(K,\spinc)$. 
\end{itemize}
\end{example}

\subsection{Main results}\label{subsec:intro-main-results}
Associated with a cobordism $W$ from a closed, oriented $3$-manifold 
$M$ to another closed, oriented $3$-manifold $M'$, and a $\SpinC$ 
class $\spinct\in\SpinC(W)$, Ozsv\'ath and Szab\'o construct  the  
homomorphisms 
\begin{displaymath}
\fmap^\circ_{W,\spinct}:\HFT^\circ(M,\spinct|_{M})\lra
\HFT^\circ(M',\spinct|_{M'}).
\end{displaymath}
With Theorem~\ref{thm:intro-category} in place, it is natural to ask 
if the construction of  Ozsv\'ath and Szab\'o may be extended to an 
invariant for cobordisms between tangles. First, we define the notion 
of cobordisms between tangles, by generalizing the notion of 
cobordisms between classical tangles.
\begin{defn}
A {\emph{cobordism}} $(W,\sur)$ from  $(M,T)$ to $(M',T')$  consists 
of a smooth oriented  four-manifold $W$, with boundary and corners 
and without closed components,  and a properly embedded smooth 
oriented surface $\sur$ in $W$, with boundary and corners  and
without closed components, such that:
\begin{enumerate}
\item The boundary $(\partial W,\partial F)$ of  $(W,F)$ 
consists of a horizontal part 
\begin{displaymath}
\begin{split}
(\del_hW,\del_hF)&=(\del_h^+W,\del_h^+F)\amalg 
(\del_h^- W,\del_h^- F)\\
&=\left(\del^+M\times I,\del^+T\times I\right)
\amalg \left(\del^-M\times I,\del^-T\times I\right)\\
&= \left(\del^+M'\times I,\del^+T'\times I\right)
\amalg \left(\del^-M'\times I,\del^-T'\times I\right)
\end{split}
\end{displaymath}
and a vertical part $(\del_vW,\del_vF)= -(M,T)\amalg (M',T')$,
with corners  \[(\del _v W,\del_v F)\cap(\del_h W,\del_h F)
=(\del M,\del T)\amalg(\del M',\del T').\]
\item  For every component $\sur_\circ$ 
of $\sur$  the orientation induced 
on $\partial \sur_\circ$ by the orientation of $\sur_\circ$ 
agrees with the orientation inherited from $-T\amalg T'$.
\end{enumerate} 
The cobordism $(W,\sur)$  is called {\emph{stable}} if  $(M,T)$ and 
$(M',T')$ are balanced and for every connected component  
$\sur_\circ$ of $\sur$ which is not homeomorphic to a disk, 
$T\cap \sur_\circ$ and $T\cap\sur_\circ$ have 
more than one connected component. 
\end{defn}

Assume that $(W,\sur)$ is a stable cobordism from a balanced tangle 
$(M,T)$ to a balanced tangle $(M',T')$, as above. The inclusions of 
$T$ and $T'$ in $F$ induce maps
\[\jmath_T:\pii{T}\ra \pii{F}\quad \text{and}\quad
\jmath_{T'}:\pii{T'}\ra \pii{F}.\]

\begin{defn}
An $\Ring$-coloring for $(W,F)$ is a map $\la:\pi_0(\sur)\to\Ring$ 
such that $\la\circ\jmath_{T}$ and $\la\circ\jmath_{T'}$ are 
$\Ring$-colorings for $(M,T)$ and $(M',T')$, respectively.
\end{defn}

Let $\SpinC(W)$ denote the set of $\SpinC$ structures on $W$ 
(Definition \ref{def:spinc}). By an {\emph{affine set}}  of $\SpinC$ 
structures over $W$ we mean a subset 
$\spincT=\spinct+H_\spincT\subset\SpinC(W)$ which is determined by 
a $\SpinC$ structure $\spinct\in\SpinC(W)$ and a 
submodule of $\Ker(\pi)\subset H^2(W,\Z)$. Here, $\pi$ is the map 
from the cohomology long exact sequence 
\begin{diagram}
\cdots  &\rTo & H^2(W, M\amalg M',\Z) & \rTo{\imath} 
&H^2(W,\Z)&\rTo{\pi}& H^2(M\amalg M',\Z)
&\rTo&\cdots  
\end{diagram}
for $(W, M\amalg M')$. In particular, if this fixed submodule is 
trivial, then $\spincT$ consists of a single $\SpinC$ structure 
$\spinct$. If $\spinct-\spinct'\in\ker(\pi)$, then 
$\spinct|_{M}=\spinct'|_{M}$ and $\spinct|_{M'}=\spinct'|_{M'}$, 
so  $\spincT|_{M}$ and $\spincT|_{M'}$ are well-defined.

\begin{defn}\label{defn:cobordism}
An $\Ring$-{\emph{cobordism}} $\Cob=\left[W,\sur,\spincT,\la\right]$ 
from $\Tangle$ to  $\Tangle'$ consists of a stable cobordism 
$(W,\sur)$ from $(M,T)$ to $(M',T')$, an affine set 
$\spincT\subset \SpinC(W)$ of $\SpinC$ structure over $W$, and an 
$\Ring$-coloring  $\la:\pi_0(\sur)\ra \Ring$ so that 
\[\Tangle=[M,T,\spinc=\spincT|_{M},\la\circ \jmath_T]\quad\text{and}
\quad\Tangle'=[M',T',\spinc'=\spincT|_{M'},\la\circ \jmath_{T'}].\] 
If $\Cob$ is an $\Ring$-cobordism from $\Tangle$ to $\Tangle'$,
we write $\Cob:\Tangle\leadsto \Tangle'$.
\end{defn}

The $4$-tuple associated with an $\Ring$-cobordism $\Cob$ is denoted
by $[W_\Cob,\sur_\Cob,\spincT_\Cob,\la_\Cob]$ and we set  
$\SpinC(\Cob)=\SpinC(W_\Cob)$. 
When $\spincT$ consists of a single $\SpinC$ structure $\spinct$, 
we abuse the notation and denote it by $\spinct$ (or use similar 
notation for it).

Using affine sets of $\SpinC$ structures instead of single $\SpinC$ 
structures over cobordisms allows us  to compose $\Ring$-cobordisms. 
In fact, if 
\[\Cob_1=[W_1,\sur_1,\spincT_1,\la_1]:\Tangle\leadsto\Tangle'
\quad \text{and}\quad 
\Cob_2=[W_2,\sur_2,\spincT_2,\la_2]:\Tangle'\leadsto\Tangle''\]
are $\Ring$-cobordisms, then the Mayer-Vietoris exact sequence 
\begin{diagram}
\cdots  &\rTo&H^1(M_{\Tangle'},\Z)&\rTo{\delta}&H^2(W,\Z)&\rTo{\pi}&
H^2(W_1,\Z)\oplus H^2(W_2,\Z)&\rTo&\cdots  
\end{diagram}
for $W=W_1\cup W_2$, and the submodules corresponding to $\spincT_1$ 
and $\spincT_2$ determine a submodule of $H^2(W,\Z)$ as their 
pre-image under $\pi$. Since $\spincT_1$ restricts to 
$\spinc=\spinc_{\Tangle}$ on $M=M_{\Tangle}$ and $\spincT_2$ 
restricts to $\spinc''=\spinc_{\Tangle''}$ on $M''=M_{\Tangle''}$, 
this submodule determines an affine set $\spincT$ of $\SpinC$ 
structures on $W$, which is bigger than $\spincT_1\times\spincT_2$ 
unless the map $\delta$ in the above sequence is trivial. We may 
then compose the $\Ring$-cobordisms $\Cob_1$ and $\Cob_2$ to obtain 
\[\Cob=\Cob_1\cup_{\Tangle'}\Cob_2:\Tangle\leadsto\Tangle''.\]

Definition~\ref{defn:cobordism} gives a category $\ACobCat$.
The objects of $\ACobCat$ are $\Ring$-tangles and the  morphisms 
are $\Ring$-cobordisms. The category $\ATangleCat$ is a 
subcategory of $\ACobCat$: given a diffeomorphism
\[d:[M,T,\spinc,\la]\ra [M',T',\spinc',\la']\] 
define an $\Ring$-cobordism $\Cob=[W,\sur,\spinct,\la_\sur]$, where
$(W,\sur)=(M,T)\times[0,1]$, and $(M,T)\times \{1\}$ is identified 
with $(M',T')$ by the diffeomorphism $d$.
Further, the $\SpinC$ structure $\spinct$ and the map $\la_\sur$ are 
trivially determined by $\spinc$ and $\la$, respectively.  

For every $\Ring$-cobordism $\Cob$ from an $\Ring$-tangle $\Tangle$ 
to an $\Ring$-tangle $\Tangle'$,  and any $\Ring$-module $\Rin$, we 
define an $\Ring$-homomorphism 
\[\fmap_{\Cob}^{\Rin}:\HFT^{\Rin}(\Tangle)\to\HFT(\Tangle').\]

\begin{thm}\label{thm:main-intro}
For every $\F$-algebra $\Ring$, the functor $\HFT^{\Rin}$ from 
Theorem ~\ref{thm:intro-category}, extends to a functor
\begin{displaymath}
\begin{split}
&\HFT^\Rin:\ACobCat\lra\AModuleCat, \\
\end{split}
\end{displaymath}
by setting $\HFT^{\Rin}(\Cob)=\fmap_{\Cob}^{\Rin}$, for any 
$\Ring$-cobordism $\Cob$. Here, $\AModuleCat$ 
denotes the category of $\Ring$-modules with $\Ring$-homomorphisms 
between them.
\end{thm}

If $\phi:\Rin\ra\Rin'$ is a homomorphism of $\Ring$-modules,
we obtain a corresponding homomorphism 
of $\Ring$-modules
\begin{displaymath}
\fmap^\phi:\HFT^\Rin(\Tangle)\lra \HFT^{\Rin'}(\Tangle).
\end{displaymath}
Given an $\Ring$-cobordism $\Cob:\Tangle\leadsto\Tangle'$ ,
the following diagram is commutative:
\begin{diagram}
\HFT^\Rin(\Tangle)&\rTo{\ \ \fmap^\Rin_{\Cob}\ \ }&
\HFT^\Rin(\Tangle')\\
\dTo{\fmap^\phi}&&\dTo{\fmap^\phi}\\
\HFT^{\Rin'}(\Tangle)&\rTo{\ \ \fmap^{\Rin'}_{\Cob}\ \ }&
\HFT^{\Rin'}(\Tangle').
\end{diagram}
Moreover, given a short exact sequence
\begin{diagram}
0&\rTo &\Rin &\rTo{\imath} &\Rin' &\rTo{\pi} &\Rin'' &\rTo &0
\end{diagram}
of $\Ring$-modules, we obtain a corresponding exact triangle 
of tangle Floer homology $\Ring$-modules:
\begin{diagram}
\HFT^{\Rin}(\Tangle) &&\lTo{\fmap^\delta}&& \HFT^{\Rin''}(\Tangle)\\
&\rdTo{\fmap^{\imath}}&&\ruTo{\fmap^{\pi}}&\\
&&\HFT^{\Rin'}(\Tangle)&&
\end{diagram}
where $\fmap^\delta$ is a connecting homomorphism.

\subsection{Examples and applications}
\label{subsec:intro-aplications}
Let us first review some of the familiar cases of the above 
construction. \\

\begin{example}
Let $\sY=(Y,p)$ and $\sY'=(Y',p')$ be pointed, 
oriented, connected and closed $3$-manifolds. If $X$ is a smooth 
$4$-dimensional cobordism from $Y$ to $Y'$ and if $\sig$ is a simple 
path in $X$ from $p$ to $p'$, then $\sX=(X,\sig)$ gives a cobordism 
$(X_{\sig},\sur_{\sig})$ from $(Y_p,T_p)$ to $(Y_{p'},T_{p'})$, 
where $\sur_{\sig}$ is a disk. Let $\la_{\sig}$  be the 
$\F[\la]$-coloring that labels $\sur_{\sig}$ by $\la$. For every 
$\SpinC$ structure $\spinct\in\SpinC(X)$, 
$\Cob_{\sX,\spinct}=(X_{\sig},\sur_{\sig},\spinct,\la_{\sig})$ is 
an $\Ring$-cobordism from $\Tangle_{\sY,\spinc}$ to 
$\Tangle_{\sY',\spinc'}$, where $\spinc=\spinct|_{Y_p}$ and 
$\spinc'=\spinct|_{Y'_{p'}}$.
In view of Example \ref{ex:HFcomp} (part (a)),  
when we choose $\M$ equal to $\F,\F[\la],\F[\la^{-1}]$ or 
$\F[\la,\la^{-1}]$, the $\Ring$-homomorphism 
$\HFT^{\Rin}(\Cob_{\sX,\spinct})=\fmap_{\Cob_{\sX,\spinct}}^\M$ 
is the cobordism map of Ozsv\'ath and Szab\'o in the corresponding 
cases.
\end{example}

\begin{example}
(\emph{Functoriality of knot and link Floer homologies}) 
Suppose that $(L,\p)\subset Y$ and $(L',\p')\subset Y'$ are marked 
links, as in Example \ref{ex:HFcomp}, with $|\p|=|\p'|=n$. 
For a decorated cobordism $\sZ=(Z,\sur,\sig)$ consisting of a smooth, 
oriented, $4$-dimensional cobordism $Z$ from $Y$ to $Y'$,  a properly 
embedded, smooth, oriented surface $\sur\subset Z$ which connects $L$ 
to $L'$ and a properly embedded, oriented $1$-manifold 
$\sig\subset \sur$ which connects $\p$ to $\p'$ i.e. 
$\del^-\sig=\p$ and $\del^+\sig=\p'$, we construct a cobordism 
$(X_{\sig},\sur_{\sig})$ connecting $(Y_{\p},L_{\p})$ to 
$(Y'_{\p'},L'_{\p'})$. We require that $\sig$ does not have any 
closed components and that any connected component of 
$\sur-\sig$ with positive genus intersects $L$ and $L'$ in more 
than one connected component, to achieve stability. Then the 
connected components $\sig_1,\ldots,\sig_n$ give a matching between
$\p$ and $\p'$. We may thus assume that the endpints of $\sig_i$ are
$p_i\in\p$ and $p_i'\in\p'$, for $i=1,\ldots,n$.

Let us assume that $\sur=\coprod_{j=1}^m\sur_j$ and set 
$\Ring=\F[\la_1,\ldots,\la_n,\lav_1,\ldots,\lav_m]$. The 
$\Ring$-coloring of $(X_\sig,\sur_\sig)$ is given by labeling the 
disk associated with each $\sig_i$ by $\la_i$, and labeling 
each connected component of $\sur_j\setminus\sigma$ by 
$\lav_j$. If $\Ring_{L,\p}$ denotes the ring associated 
with the marked link $(L,\p)$, there is a quotient map from 
$\Ring_{L,\p}$ to $\Ring$. If two link components in $L$ are on the 
boundary of the same connected component of $\sur$, the variables 
associated to these link components by $\la_\p$ are identified 
in the quotient. In particular, if each connected component of $\sur$
intersects precisely one connected component of each one $L$ and 
$L'$, $\Ring_{L,\p}$ and $\Ring_{L',\p'}$ are both identified with 
$\Ring$. In this case, we call $(Z,\sur,\sig)$ a {\emph{decorated
link cobordism}} from $(Y,L,\p)$ to $(Y',L',\p')$.

Let us assume that $(Z,\sur,\sig)$ is a decorated link 
cobordism from $(Y,L,\p)$ to $(Y',L',\p')$. Associated with 
every $\SpinC$ structure $\spinct\in\SpinC(Z)$ we obtain an 
$\Ring$-cobordism 
\[\Cob_{\sZ,\spinct}=[W_{\sZ},\sur_{\sZ},\spinct,\la_{\sZ}]:
\Tangle_{L,\p,\spinct|_Y}\leadsto\Tangle_{L',\p',\spinct|_{Y'}}.\]
Correspondingly, we obtain the cobordism maps
\[\fmap_{\sZ,\spinct}^\M:\HFT^\M(Y,L,\p,\spinct|_{Y})\ra 
\HFT^\M(L',\p',\spinct|_{Y'}).\]
The functoriality of link Floer homology then follows from our 
main theorem.
\end{example}

\begin{remark} A similar construction, in the context of pointed 
links and cobordisms between them, is independently given by 
Ian Zemke in \cite{Ian}.
\end{remark}

%\todo{AA I changed the following paragraph. EE very good}
Suppose that the oriented knots $K$ and $K'$ differ by changing one crossing. Corresponding to this crossing change there is a cobordism obtained by a band attachment from $K$ to $L=K'\#H$, where $H$ is the right- or left-handed Hopf link. In \cite{AE-unknotting}, we use the corresponding cobordism maps for appropriate choices of based points on $K$ and $L$ and decoration on the cobordism, and we define a lower bound $\mathfrak{l}(K)$ on the unknotting number of $K$. This bound is greater 
than or equal to $\nu^-(K)$, $\nu^-(-K)$ and the order of $U$-torsions in $\mathrm{HFK}^-(K)$. Additionally, it only vanishes for the unknot, and we present 
examples of slice knots $K$ such that $\mathfrak{l}(K)$ is 
arbitrarily large. A parallel construction is used by the first
author in \cite{A-BN} to construct lower bounds on the unknotting 
number from Khovanov homology, also see \cite{AD-KM}.

\subsection{Outline of the paper}\label{subsec:intro-outline}
The paper is organized as follows. In Section~\ref{sec:back} we 
review Heegaard poly-tuples and $\SpinC$ structures over them. 
Then in Section~\ref{sec:naturality} we follow the footsteps of 
Juh\'asz and Thurston \cite{JT} to show that  tangle Floer homology 
for $\Ring$-tangles gives functors from $\ATangleCat$ to $\AModCat$.
In Section~\ref{sec:Cerf} we study parametrized Cerf decompositions 
of cobordisms between tangles. We show that any two parametrized 
Cerf decompositions for a stable cobordism can be connected by a 
sequence of Cerf moves.

In Section~\ref{sec:connected-sum} we define cobordism maps for 
parametrized cobordisms associated with attaching one or three 
handles and show that the map is invariant. In Section~\ref{sec:map} 
we introduce a special $\Ring$-tangle $\Tangle_{\sur}$ associated to 
the positive boundary of any $\Ring$-cobordism together with a 
distinguished generator $\Theta_{\sur}\in\HFT(\Tangle_{\sur})$. 
The distinguished generator makes it possible to define invariant 
cobordism maps for cobordisms parametrized with framed links and 
framed arcs.
In Section~\ref{sec:inv} we define cobordism maps for arbitrary 
$\Ring$-cobordisms by composing cobordism maps constructed in 
Sections~\ref{sec:connected-sum} and~\ref{sec:map} for cobordisms 
parametrized by framed $0$-spheres, framed knots and arcs and 
framed $2$-spheres. We prove that this map is in fact an invariant 
and does not depend on the parametrized Cerf decomposition. 
Moreover, we show that this construction gives functors from 
$\ACobCat$ to $\AModuleCat$.  
Finally in Section~\ref{sec:applications} we discuss some special 
cases and applications. In particular, decorated cobordisms between 
pointed links induce functorial maps on link Floer homology. \\

{\bf{Acknowledgements.}} The authors would like to thank Andr\'as 
Juh\'asz and Robert Lipshitz for helpful discussions and suggestions. 
Most of this work was done when the first author was a postdoc at 
Max Planck Institute for Mathematics (MPIM) in Bonn. She gratefully 
acknowledges the support and the hospitality of MPIM through this 
period.

\newpage
\section{Tangles, $\SpinC$ structures and Heegaard poly-tuples}
\label{sec:back}
In this section, we describe the correspondence between sutured manifolds and tangles. Then, we review some definitions and results from \cite{AE-1} about Heegaard Floer homology for tangles (sutured manifolds), to fix our notation.  We will also reformulate some definitions and results about Heegaard diagrams, Heegaard poly-tuples and $\SpinC$ structures from \cite{AE-1,GW,OS-3m1,OS-4mfld} to work in our setup.

\subsection{Sutured manifolds and tangles}\label{sec:alg}
Sutured manifolds were introduced by Gabai in 
\cite{Gabai-foliations1,Gabai-foliations2,Gabai-foliations3}. 
Throughout this paper, we use 
a less general family of sutured manifolds by excluding toroidal sutures.
\begin{defn}
A \emph{sutured manifold} $(X,\tau)$ is an oriented 
$3$-manifold $X$ with boundary, together 
with a set of pairwise disjoint, oriented, simple closed curves 
$\tau=\{\tau_1,\ldots,\tau_{\el}\}$ on $\del X$. 
We will denote by $A(\tau_i)$ a tubular neighborhood of $\tau_i$ 
in $\del X$, which is an annulus. 
We let $A(\tau)=A(\tau_1)\amalg\cdots\amalg A(\tau_{\el}).$ 
Every connected component of $\Rr(\tau)=\del X\setminus A(\tau)^{\circ}$
is oriented with the orientation 
induced from $X$, where $A(\tau)^\circ$ denotes the interior of 
$A(\tau)$. Furthermore, we require 
that $\Rr(\tau)=\Rr^+(\tau)\amalg\Rr^-(\tau)$ where $\Rr^+(\tau)$ 
(respectively, $\Rr^-(\tau)$) denotes 
the union of components of $\Rr(\tau)$ such that the orientation induced on
$\tau$ as the boundary of 
$\Rr^+(\tau)$ (respectively, $\Rr^-(\tau)$) agrees with
(respectively, is the opposite of) 
the orientation of $\tau$.
\end{defn}

Every sutured manifold $(X,\tau)$ determines a tangle $(M,T)$, 
where $M=\ovl{X}$ is obtained from 
$X$ by filling the sutures (i.e. attaching $2$-handles along the sutures) 
and $T$ is the set of cocores of 
these $2$-handles. The orientation on $\tau$ induces an 
orientation on $T$ and the decomposition 
$\del \Rr(\tau)=\Rr^+(\tau)\amalg\Rr^-(\tau)$ induces a decomposition $\del 
M=\del^+M\amalg\del^-M$ of the boundary of $M$.  
Conversely, every tangle $(M,T)$ determines a sutured manifold $(X,\tau)$ where $X=M\setminus \nbd{T}$ and $\tau$ is the set of meridians of $T$ along with the induced orientation. 

Note that a tangle $(M,T)$ is balanced if and only 
if the corresponding sutured manifold is balanced in 
the sense of \cite[Definition 2.2]{Juh}.

%Assume $\Ring$ be a commutative algebra over $\F=\Z/2\Z$. Given a map $\la:\pii{T}\ra \Ring$,
%for every $r\in\pi_0(\del M)$ we define 
%$$\la(r):=\prod_{\substack{t\in\pi_0(T)\\ \imath_T(t)=r}}\la(t)$$
%where $\imath_T:\pi_0(T)\ra \pi_0(\del M)$ is the map induced by inclusion. 
%Following Definition~\ref{def:A-tangle}, an $\Ring$-tangle is a $4$-tuple
%$[M,T,\spinc,\la:\pii{T}\ra \Ring]$ where  $(M,T)$ is a  balanced tangle
%and $\spinc\in\SpinC(M)$ is a $\SpinC$ structure on $M$. 
%Moreover, if $r\in\pii{\del M}$ corresponds to a 
%connected component of $\del M$ with positive genus then $\la(r)=0$, and 
%for every connected component $M_\circ$ of $M$
%\[\sum_{r\in\pii{\del^+M_\circ}}\la(r)=
%\sum_{r\in\pii{\del^-M_\circ}}\la(r).\] 
%For instance, we may take $\Ring=\F$. 
%Then any balanced tangle which is equipped with a 
%$\SpinC$ structure, together with the map $\la$ which 
%maps everything to zero, is an $\Ring$-tangle. \\

In \cite{AE-1}, we introduced a $\Z$-algebra $\Ring_{\tau}$ 
associated to the boundary of any 
balanced sutured manifold $(X,\tau)$. Assume 
$\tau=\amalg_{i=1}^{\el}\tau_i$, 
$$\Rr^-(\tau)=\coprod_{i=1}^k R_i^-\ \ \ \ \ 
\text{and}\ \ \ \ \ \Rr^+(\tau)=\coprod_{j=1}^{l}R_j^+.$$
Associated with the connected components of $\Rr(\tau)$, consider the elements 
$$\la_i^-:=\prod_{\tau_j\subset\del R_i^-}\la_j,\ \ \ i=1,\ldots,k\ \ \ \ \text{and}\ \ \ \ 
\la_i^+:=\prod_{\tau_j\subset\del R_i^+}\la_j,\ \ \ i=1,\ldots,l,$$
in the free $\Z$-algebra $\Z[\la_1,\ldots,\la_{\el}]$ generated by $\la_1,\ldots,\la_{\el}$. Then 
$$\Ring_{\tau}=\frac{\Z[\la_1,\ldots,\la_{\el}]}{\left\langle\la^+(\tau)-\la^-(\tau)\right\rangle+
\left\langle\la_i^+\ |\ g_i^+>0\right\rangle+\left\langle\la_i^-\ |\ g_i^->0\right\rangle }$$
where $\la^-(\tau)=\sum_{i=1}^{k}\la_i^-$, 
$\la^+(\tau)=\sum_{i=1}^l\la_i^+$ and $g_i^\bullet$ denotes 
the genus of $R_i^\bullet$ for $\bullet=+,-$. For any balanced tangle $(M,T)$, denote the algebra associated to its corresponding balanced sutured manifold by $\Ring_T$.

Suppose $(M,T)$ be a connected balanced tangle (i.e. a balanced tangle 
with $M$ connected). The map $\la_T:\pi_0(T)\ra \Ring_T$ which sends each component $T_i$ to the variable $\la_i$ corresponding to its meridian gives an $\Ring_T$-coloring for $(M,T)$. In fact, any 
map $\la:\pi_0(T)\ra\Ring$ is an 
$\Ring$-coloring for $(M,T)$ if and only if $\la=\phi\circ\la_T$ for  
a homomorphism $\phi$ from $\Ring_T$ to $\Ring$. 

%For a connected balanced tangle $(M,T)$ (i.e. a balanced tangle 
%with $M$ connected) let $\Ring_T$ denote the 
%algebra associated to the corresponding 
%balanced sutured manifold and $\la_T:\pi_0(T)\ra \Ring_T$ denote 
%the map which sends the component $T_i$ of $T$ to the 
%corresponding variable $\la_i\in \Ring_T$. Then, for every $\SpinC$ structure 
%$\spinc\in\SpinC(M)$, the balanced tangle $(M,T)$ together with $\spinc$ and a 
%map $\la:\pi_0(T)\ra\Ring$ becomes an 
%$\Ring$-tangle if and only if $\la=\phi\circ\la_T$ for  
%a homomorphism $\phi$ from $\Ring_T$ to $\Ring$.  

\subsection{Heegaard Floer homology for tangles} Recall that a balanced Heegaard diagram is a $4$-tuple 
$H=(\Sig,\alphas,\betas,\z)$ where $\Sig$ is a closed oriented surface, $\alphas$ and $\betas$ are sets of 
pairwise disjoint circles on $\Sig$, and $\z\subset\Sig\setminus (\alphas\cup\betas)$ is a set of points. 
Further, $|\alphas|=|\betas|$ and $\z$ intersects every connected component of $\Sig\setminus \alphas$ 
and $\Sig\setminus\betas$. Every balanced Heegaard diagram $H$, specifies a balanced tangle $(M,T)$, 
where $M$ is obtained from $\Sig\times [0,1]$ by attaching $2$-handles along the circles $\alphas\times \{0\}$ 
and $\betas\times \{1\}$, and $T=\z\times [0,1]$. 

A Heegaard diagram for a balanced tangle $(M,T)$, is a balanced Heegaard diagram $H=(\Sig,\alphas,\betas,\z)$, 
so that $\Sig$ is an embedded separating surface in $M$ which cuts $T$ transversely in $\z$, $\alphas$ and 
$\betas$ bound disjoint disks on the two sides of $\Sig$. Moreover, let $\Sig[\alphas]$ and $\Sig[\betas]$ be 
embedded surfaces obtained from compressing $\Sig$ along the $\alphas$ and $\betas$ curves, 
respectively. Then, $(\Sig[\alphas],\z)$ and $(\Sig[\betas],\z)$ are isotopic relative 
to $T$ to $(\del^-M,\del^- T)$ and $(\del^+M,\del^+T)$, respectively.

Let $H=(\Sig,\alphas,\betas,\z=\{z_1,\ldots,z_{\el}\})$ be a Heegaard diagram which corresponds 
to a balanced tangle $(M,T)$. We set $\SpinC(H)$ equal to $\SpinC(M)$. 
For an $\Field$-algebra $\Ring$, 
consider a map $\la:\z\to\Ring$. Then, corresponding to any $2$-chain $\mathcal{D}$ on $\Sig$ with 
boundary on $\alphas\cup\betas$, let $n_i(\Dcal)$ denote 
the coefficient of $\Dcal$ at $z_i$ and set 
\begin{equation}\label{eq:2-chainlab}
\la(\Dcal):=\prod_{i=1}^{\el}\la_i^{n_i(\Dcal)}\in\Ring,
\end{equation}
where $\la_i=\la(z_i)$. Let $\Sig\setminus\alphas=\amalg_{i=1}^kA_i$ and 
$\Sig\setminus\betas=\amalg_{j=1}^lB_j$.

\begin{defn} With the above notation fixed, $\la$ is called an \emph{$\Ring$-coloring} for $H$ if it 
satisfies the following two conditions.
\begin{itemize}
\item $\sum_{i=1}^k\la(A_i)=\sum_{j=1}^l\la(B_j)$,
\item if $A_i$ or $B_j$ is not a punctured sphere, then $\la(A_i)=0$ or $\la(B_j)=0$, respectively. 
\end{itemize}
\end{defn}
\begin{defn}
A Heegaard diagram $H$, together with an $\Ring$-coloring $\la$ and a $\SpinC$ class $\spinc\in\SpinC(H)$ 
is called an \emph{$\Ring$-diagram}, if $(H,\la)$ is 
\emph {$\spinc$-admissible}.  Here, $(H,\la)$ is called $\spinc$-admissible  if 
for any periodic domain $\Pcal$ with $\langle c_1(\spinc),H(\Pcal)\rangle=0$, 
either $\la(\mathcal{P})=0$ or the coefficient of $\Pcal$ at some point is negative.

\end{defn}

Any $\Ring$-diagram specifies an $\Ring$-tangle. Given an $\Ring$-tangle $\Tangle=[M,T,\spinc,\la]$, an 
$\Ring$-diagram for $\Tangle$ consists of a Heegaard diagram for $(M,T)$, together with the $\Ring$-coloring 
induced by $\la$ and the $\SpinC$ class induced by $\spinc$.   

Let $\HD=(\Sig,\alphas,\betas,\la:\z\to\Ring,\spinc)$ be an $\Ring$-diagram for the $\Ring$-tangle 
$\Tangle$. Choose a generic path $J_s$ of almost complex structures 
%$(\mathfrak{j},J_s)$ of a complex structure $\mathfrak{j}$ 
on $\Sym^{\ell}(\Sig)$, where $\ell=|\alphas|=|\betas|$. In \cite{AE-1}, we construct a chain complex 
$$\CFT_{J_s}(\Sig,\alphas,\betas,\la,\spinc),$$ 
which is generated by the intersection points $\x\in\Ta\cap\Tb$ with 
$\spinc(\x)=\spinc$, and its chain homotopy type is an invariant of $\Tangle$. 
We usually drop $J_s$ from the notation for simplicity. Furthermore, for any $\Ring$-module 
$\Rin$ the isomorphism type of the Floer homology group
$$\HFT^{\Rin}(\Tangle):=H_{\star}(\CFT(\Sig,\alphas,\betas,\la,\spinc)\otimes_{\Ring}\Rin)$$
is an invariant of $\Tangle$. For simplicity, we set $\HFT(\Tangle):=\HFT^{\Ring}(\Tangle)$. 
Moreover, for any balanced tangle $(M,T)$, let $$\HFT(M,T,\spinc):=\HFT(M,T,\la_T:\pi_0(T)\to\Ring_T,\spinc)$$
where $\Ring_T$ is the algebra associated to $(M,T)$ and $\la_T$ is its corresponding coloring map. 

%Following the notation set in Definition~\ref{def:A-tangle} and
% Definition~\ref{defn:cobordism} we may 
%define a pair of categories associated with every commutative
% algebra $\Ring$ over $\F=\Z/2\Z$.
%
%\begin{defn}\label{def:categories}
%The category of $\Ring$-tangles, denoted by $\ATangleCat$, is given as follows.
% Its objects are $\Ring$-tangles 
%and the set of morphisms from an $\Ring$-tangle $\Tangle=[M,T,\spinc,\la_T]$ 
%to another $\Ring$-tangle 
%$\Tangle'=[M',T',\spinc',\la_{T'}]$ are diffemorphisms $d:(M,T)\ra (M',T')$ such that 
%$d^*\spinc'=\spinc$ and the diagram
% \begin{diagram}
% \pii{T}&&\rTo{d}&&\pii{T'}\\
% &\rdTo{\la_T}&&\ldTo{\la_{T'}}&\\
% &&\Ring &&
% \end{diagram}
% is commutative. Every such diffeomorphism is denoted by 
% $d:\Tangle\ra \Tangle'$.
% The cobordism category of $\Ring$-tangles, 
% $\ACobCat$, is the category whose objects are $\Ring$-tangles,
%  while its set of morphisms consists of $\Ring$-cobordisms.  
%\end{defn}
%
%The category $\ATangleCat$ is a subcategory of $\ACobCat$: given a diffeomorphism
%$d:\Tangle\ra \Tangle'$, we have an $\Ring$-cobordism 
%$[W,\sur,\spinct,\la]$ where \[(W,\sur)=(M_\Tangle,T_\Tangle)\times [0,1],\]
%$\spinct$ is the trivial extension of $\spinc_\Tangle$ to $W$ and 
%$(M_\Tangle,T_\Tangle)\times \{1\}$ is identified with 
%$(M_{\Tangle'},T_{\Tangle'})$ via $d$. Moreover, $\la:\pi_0(\sur)\ra \Ring$ is the 
%map induced by $\la_\Tangle$. When $d$ is the identity, we denote the above 
%$\Ring$-cobordism by $\Tangle\times[0,1]$.
%\\
%

\subsection{$\SpinC$ structures over cobordisms}\label{subsec:SpinC}
Given a stable cobordism $(W,\sur)$ from $(M,T)$ to $(M',T')$, let $\xi$ denote the 
oriented $2$-plane field in $\del_h W=\del M\times [0,1]$ consisting of the tangent 
planes of the surfaces $\del M\times\{t\}$ for any $t\in [0,1]$. Fix an almost 
complex structure $J_0$ on 
$TW|_{\del_h W}$ such that $\xi$ consists of complex lines i.e. $2$-planes in $\xi$ 
are invariant under $J_0$. One may further extend $J_0$ to a an almost complex 
structure on $\sur$ by requiring that $J_0$ preserves the tangent space of $\sur$.
Note that the set of almost complex structures $J_0$ with the above property 
is contractible.

\begin{defn}\label{def:spinc}
A  $\SpinC$ structure on $W$ is the homology class of a pair $(J,P)$, where  
\begin{enumerate}
\item $P\subset W\setminus \del_h W$ is a finite collection of points,
\item $J$ is an almost complex structure  on $W\setminus P$ with $J|_{\del_hW}=J_0$.
\end{enumerate}
We call the pairs $(J_1,P_1)$ and $(J_2,P_2)$  {\emph{homologous}} if there exists 
a compact $1$-manifold $C\subset W\setminus\del_hW$ without closed components so that
$\del C=P_1\cup P_2$, and $J_1|_{W\setminus C}$ is isotopic relative to $\del_hW$ to $J_2|_{W\setminus C}$. Denote the set of  $\SpinC$ structures on $W$ by 
$\SpinC(W)$. Similarly, a relative $\SpinC$ structure on the pair $(W,\sur)$ 
is the homology class of a pair $(J,P)$, where $P\subset W\setminus(\del_h W\cup \sur)$
is a finite collection of points and $J$ is an almost complex structure on $W\setminus P$ 
which agrees with $J_0$ over $\del_hW\cup \sur$. The notion of homologous 
pairs is defined similarly, the only difference is that we need to consider isotopies relative $\del_hW\cup F$. The set of relative $\SpinC$ structures over 
$(W,\sur)$ is denoted by $\SpinC(W,\sur)$.
\end{defn}
Since $J_0$ is chosen from a contractible family, the above definition does not depend 
on the particular choice of $J_0$.\\

After fixing a metric over the $4$-manifold $W$, any oriented  
$2$-plane field which extends $\xi$ and is 
defined in the complement $W-P$ determines a corresponding  almost complex
structure. 
It is not hard to show that $\SpinC(W)$ is an affine space over $\Ht^2(W,\del_hW;\Z)$.\\

For every  $\SpinC$ structure $\spinc$ over $W$, the induced $\SpinC$ structures $\spinc|_{M}$ and $\spinc|_{M'}$ are defined as follows. Consider an almost complex structure $J$ (defined on $W\setminus P$) representing $\spinc$. At any point $p$ in $M$, $J$ specifies a subspaces $V_p=T_pM\cap J(T_pM)$ of $T_pM$. Similarly, at any $p\in M'$, it specifies a subspace $V'_p\subset T_{p}M'$. Let $V$ and $V'$ denote the corresponding plane fields in $M$ and $M'$, respectively. Then, $\spinc|_{M}$ and $\spinc|_{M'}$ are defined to be the $\SpinC$ classes represented by $V$ and $V'$, respectively.

\subsection{Heegaard poly-tuples}
Let $H=(\Sig,\alphas^1,\ldots,\alphas^m,\z)$ be a balanced Heegaard diagram. This 
means that $\Sig$ is a closed oriented surface and for any 
$1\le i\le m$, $\alphas^i=\{\alpha^i_1,\ldots,\alpha^i_{\ell}\}$
is a set of $\ell$ disjoint simple closed curves on $\Sig$ for some $\ell>0$.
Moreover, $\z=\{z_1,\ldots,z_\el\}$ is a set of marked points in 
$\Sig-\alphas^1-\alphas^2-\cdots-\alphas^m$ such that for any 
$1\le i\le m$ every connected component of $\Sig-\alphas^i$ intersects $\z$.\\

Associated with the balanced Heegaard diagram $H$, we define a pair 
$(W_H,\sur_H)$ as follows.  For every $1\le i\le m$, the Heegaard diagram 
$(\Sig,\alphas^i,\emptyset,\z)$ determines a tangle $(U_i,T_i)$, where $U_i=C[\alphas^i]$ 
is the compression body determined by $\alphas^i$ i.e. 
it is obtained from $\Sig\times [0,1]$ by attaching 2-handles along the curves
$\alphas^i\times\{1\}$, and $T_i=\z\times [0,1]$. 
Thus $\del U_i=\del^-U_i\amalg\del^+U_i$, where
$$\del^- U_i=\Sig\ \ \ \ \text{and}\ \ \ \ \del^+ U_i=\Sig[\alphas_i]$$
and $\Sig[\alphas_i]$ is obtained by cutting $\Sig$ along $\alphas_i$ and attaching 
disks to the boundary components of the resulting surface.  
We denote the finite set 
$T_i\cap\del^-U_i=\z\times\{0\}$ by $\z_i$.

Let $\DD_{m}$ be a $m$-gon with the vertices $v_1,\ldots,v_m$, labelled in clockwise order, 
and edges $e_1,\ldots,e_m$, where $e_i$ connects $v_{i-1}$ to $v_{i}$ for $i=2,\ldots,m$ and 
$e_1$ connects $v_m$ to $v_1$. Define
\begin{displaymath}
\begin{split}
&W_H:=\frac{\left(\Sig\times \DD_m\right)%\big
\amalg\left(\coprod_{i=1}^{m}
U_i\times e_i\right)}{\Sig\times e_i\sim\del^-U_i\times e_i}\ \ \text{and}\ \
\sur_H:=\frac{\left(\z\times \DD_m\right)%\big
\amalg\left(\coprod_{i=1}^{m}T_i\times e_i\right)}{\z\times e_i\sim\z_i\times e_i}.
\end{split}
\end{displaymath}
We  smooth the corners of $W_H$ and $\sur_H$ along $v_i\times\Sig$ for  
$1\le i\le m$. Corresponding to any vertex $v_i$, we obtain a balanced tangle in 
$(\del W_H,\del\sur_H)$ determined by the Heegaard diagram 
$(\Sig,\alphas^i,\alphas^{i+1},\z)$  denoted by 
$(M_{i,i+1},T_{i,i+1})$. Note that $\alphas^{m+1}=\alphas^{1}$ and $(M_{m,m+1},T_{m,m+1})=(M_{m,1},T_{m,1})$. 
Let 
$$M':=M_{1,2}\amalg M_{2,3}\amalg\cdots\amalg M_{m,1}\subset \del W_H\ \ \ \ 
\text{and}\ \ \ \ Z:=\del W_H-\interior(M')$$
Thus, $Z$ is a product and
$$Z=\left(\del^-M_{1,2}\amalg...\amalg\del^-M_{m,1}\right)
\times [0,1].$$

Fix an almost complex structure $J_0$ on $TW_H|_{Z}$ such that for any 
$t\in [0,1]$ the tangent planes to the surface 
$\left(\del^-M_{1,2}\amalg\cdots \amalg\del^-M_{m,1}\right)\times\{t\}$ are 
complex lines. 
The almost complex structure $J_0$ may further be extended to $\sur_H$ so 
that it preserves 
the tangent space of $\sur_H$, i.e. the tangent planes of $\sur_H$ are all complex
lines. We may then talk about the $\SpinC$ structures on $W_H$ 
and the relative $\SpinC$ structures on $(W_H,\sur_H)$.

\begin{defn} The set of $\SpinC$ structures on $W_H$, denoted by 
$\SpinC(W_H)$, is defined as the set of homology classes of the pairs $(J,P)$ 
consisting of a finite set of
points $P$ in the interior of $W_H$ and an almost complex structure
$J$ on $W_H-P$ 
such that $J|_{Z}=J_0$.  
Similarly, the set of \emph{relative $\SpinC$ structures} on $(W_H,\sur_{H})$, 
denoted by $\SpinC(W_H,\sur_H)$, is defined as the set of homology classes of 
the pairs $(J,P)$ 
consisting of a finite set $P$ of
points in the interior of $W_H-\sur_{H}$ and an almost complex structure $J$ 
on $W_H-P$ 
such that $J|_{Z\cup \sur_H}=J_0$.    
\end{defn}
%As before, 
$\SpinC(W_H)$ is an affine space over $\Ht^2(W_H,Z;\Z)$, and 
$\SpinC(W_H,\sur_{H})$ is an affine space over 
$\Ht^2(W_H,Z\cup\sur_H;\Z)$.

Let $\Tai\subset\Sym^\ell(\Sig)$ denote the torus $\alpha_1
^i\times\cdots \times\alpha_\ell^i$. Given intersection points  $\x_i\in\Tai\cap\T_{{i+1}}$, the homotopy classes of  maps 
\[\polygon:\DD_m\ra \Sym^\ell(\Sig)\ \  \ \text{s.t.}\ \ \  
\polygon(v_i)=\x_i, \ \polygon(e_i)\subset \Tai,\ \ i=1,\ldots,m\] 
is denoted by $\pi_2(\x_1,\ldots, \x_m)$. If $\pi_2(\x_1,\ldots, \x_m)$ is non-empty, 
\cite[Proposition 3.3]{GW} implies that there is an affine correspondence
\begin{displaymath}
\pi_2(\x_1,\ldots, \x_m)\simeq \Ker\left(\bigoplus_{i=1}^m\Span(\alphas^i)
\ra \Ht_1(\Sig;\Z)\right)\cong \Ht^2(W_H,\del W_H;\Z)
\end{displaymath}
where $\Span (\alphas^i)$ denotes the submodule of $\Ht_1(\Sig;\Z)$ spanned by 
the elements of $\alphas^i$. 

%The following proposition is a re-statement of Proposition 3.7 from \cite{GW}:
%\begin{prop}\label{Spin-map}
%With the above notation fixed, there is a well-defined map 
%$$\spinc_{H}:\pi_2(\x_1,\ldots,\x_m)\ra\SpinC(W_H)$$
%such that for every $\polygon\in \pi_2(\x_1,\ldots,\x_m)$
%$$\spinc_{H}(\polygon)|_{M_{i,i+1}}=\spinc(\x_{i})\ \ \text{for}\ \ i=1,\ldots,m-1\ \ \text{and}\ \ 
%\spinc_{H}(\polygon)|_{M_{m,1}}=\spinc(\x_m).$$
%\end{prop}

%Proposition 3.9 from \cite{GW}  implies that for $\polygon_1,\polygon_2
%\in\pi_2(\x_1,\ldots, \x_m)$, 
%we have $\spinc_H(\polygon_1)=\spinc_H(\polygon_2)$ 
%if and only if the difference $\Dcal(\polygon_1)-\Dcal(\polygon_2)$ between the formal 
%domains associated with $\polygon_1$ and $\polygon_2$
%is a $\Z$-linear combination of doubly periodic domains, i.e. 
%periodic domains for the Heegaard diagrams
%$(\Sig,\alphas^i,\alphas^{i+1},\z)$ for $i=1,\ldots,  m$, where $\alphas^{m+1}=\alphas^1$.\\

Fix intersection points $\x_i\ ,\x_i'\in\Tai\cap\T_{{i+1}}$ for $i=1,\ldots,m$. 
Two homotopy classes $\polygon\in\pi_2(\x_1,\ldots,\x_m)$ and 
$\polygon'\in\pi_2(\x_1',\ldots,\x_m')$ are called \emph{equivalent} if there exist Whitney 
disks $\psi_i\in \pi_2(\x_i,\x_i')$ for $i=1,\ldots,m$
so that $\polygon$ is obtained from $\polygon'$ by 
juxtaposition of each disk $\psi_i$ at the vertex 
$\x_i'$. Let $\Poly_{\{1,\ldots,m\}}$ 
denote the set of equivalence classes of such $m$-gons. We may thus re-state 
Proposition 3.9 of \cite{GW} as follows.

\begin{prop}
There is a one to one map 
$$\spinc_H:\Poly_{\{1,\ldots,m\}}\longrightarrow \SpinC(W_H)$$
so that for any $\Psi\in\pi_2(\x_1,\ldots,\x_{m})$, we have $\spinc_{H}([\Psi])|_{M_{i,i+1}}
=\spinc(\x_i)$ for any $i=1,\ldots,m$.
\end{prop}

For every index set 
$$I=\{i_1<\cdots <i_p\}\subset\{1,\ldots,m\}$$ with $|I|\ge 3$ we 
may consider the cobordism $W_I$ which corresponds to the 
compression of $W_H-\coprod_{i\notin I}(U_i\times e_i)$
along the edges $\Sig\times e_i\subset \Sig\times\DD_m$ with 
$i\notin I$. Thus, $W_I$ is represented by the Heegaard diagram
$H_I=(\Sig,\alphas^{i_1},\ldots, \alphas^{i_p},\z)$.
Let 
\[r_{I}:\SpinC(W_H)\to \SpinC(W_I).
\]
be the corresponding restriction map. Denote $\spinc_{I}=r_I(\spinc)$ for any $\spinc\in\SpinC(W_H)$.

\subsection{$\Ring$-diagrams and holomorphic polygon maps}
Let $\Ring$ be a commutative algebra over $\F$ and  $H=(\Sig,\alphas^1,\ldots, \alphas^m,\z)$ be a balanced Heegaard 
diagram as before. Consider a map $\la:\z\ra \Ring$. We will denote $\la(z_i)$ by $\la_i$. For any $2$-chain $\mathcal{D}$ on $\Sig$ with $\del \mathcal{D}$ on $\alphas^1\cup \cdots \cup\alphas^{m}$ we define $\la(\mathcal{D})$ as in Equation (\ref{eq:2-chainlab}). Suppose $\Sig\setminus\alphas^i=\amalg _{j=1}^{m^i}A^{i}_j$. Then, as before, $\la$ is called an \emph{$\Ring$-coloring} for $H$ if the following are satisfied:
\begin{enumerate}
\item For any $i,j\in\{1,\ldots, m\}$ we have $\sum_{k=1}^{m^i}\la (A_k^i)=\sum_{l=1}^{m^j}\la(A_l^j)$.
\item If $A_j^i$ is not a punctured sphere, then $\la (A^i_j)=0$.
\end{enumerate}
% If $\la$ is a representation,
%associated with every $2$-chain $\Dcal$ on $\Sig$ with boundary on 
%$\alphas^1\cup\cdots  \cup\alphas^m$ let $n_i(\Dcal)$ denote 
%the coefficient of $\Dcal$ at $z_i$ and set 
%$$\la(\Dcal):=\prod_{i=1}^{\el}\la_i^{n_i(\Dcal)}\in\Ring.$$
%

\begin{defn} \label{def:A-diagram}
Let the balanced Heegaard diagram $H=(\Sig,\alphas^1,\ldots,\alphas^m,\z)$, 
the $\Ring$-coloring $\la:\z\ra \Ring$ and $\SpinC$ structure $\spinc\in\SpinC(W_H)$  be as above. We call
\[\HD=(\Sig,\alphas^1,\ldots,\alphas^m,\la:\z\ra\Ring,\spinc)\]
an $\Ring$-{\emph{diagram}} if it is $\spinc$-{\emph{admissible}}, i.e. if 
the following admissibility condition is satisfied.
For every index set $I=\{i_1<\cdots <i_p\}$ and every doubly periodic domain 
$$\Pcal=\Pcal_{i_1i_2}+\Pcal_{i_2i_3}+\cdots +\Pcal_{i_{p}i_1}$$
with $\Pcal_{ij}$ a periodic domain on $(\Sig,\alphas^i,\alphas^j,\z)$, 
the following is true. If 
\begin{displaymath}
\begin{split}
&\sum_{j=1}^p{\Big\langle c_1(\spinc_{i_ji_{j+1}}),
H(\Pcal_{i_ji_{j+1}})\Big\rangle}=0
\end{split}
\end{displaymath}
then either $\la(\Pcal)=0$ in $\Ring$ or the coefficient of the domain
$\Pcal$ at some point $w$ is negative. Here, $\spinc_{ij}$ is the
$\SpinC$ structure on the $3$-manifold $M_{ij}$ (corresponding to 
$(\Sig,\alphas^i,\alphas^j,\z)$) which is determined by $\spinc$.
\end{defn}

Note that the admissibility condition only depends on the restrictions 
$\spinc_{ij}$ of $\spinc$. In particular, for every family $\spincT$ of 
$\SpinC$ structures with the same restrictions $\spinc_{ij}$, we can 
speak of $\Ring$-diagrams 
\[(\Sig,\alphas^1,\ldots,\alphas^m,\la:\z\ra\Ring,\spincT).\]

Let $\HD=(\Sig,\alphas^1,\ldots,\alphas^m,\la,\spinc)$ 
be an  $\Ring$-diagram  and  choose an appropriate translation 
invariant family of generic almost complex structure $J=\{J_z\}_{z\in\D_m}$.
For every pair of indices $i<j$ we may thus define the chain complex
$\CFT_{J}(\Sig,\alphas^i,\alphas^j,\la,\spinc_{ij})$.
Associated with any subset 
$$I=\{i_1<\cdots <i_p\}\subset \{1,\ldots,m\}$$ of indices, we may then define a holomorphic 
polygon map 
\begin{displaymath}
\begin{split}
&\fmap_{I}:\bigotimes_{j=1}^{p-1}
\CFT_{J}
\left(\Sig,\alphas^{i_j},\alphas^{i_{j+1}},\la,\spinc_{i_ji_{j+1}}\right)
\lra\CFT_{J}
\left(\Sig,\alphas^{i_1},\alphas^{i_{p}},\la,\spinc_{i_1i_{p}}\right).\\
&\fmap_{I}\big(\x_1\otimes\x_2\otimes \cdots \otimes\x_{p-1}\big):=
\sum_{\substack{\x_p\in\mathbb{T}_{{i_1}}\cap
\mathbb{T}_{{i_p}}\\
\spinc(\x_p)=\spinc_{i_1i_p}}}\sum_{\substack{\polygon\in
\pi_2(\x_1,\x_2,\ldots,\x_{p})\\ \mu(\polygon)=3-p\\ [\Psi]=\spinc_I}}
\big(\m(\polygon)\la(\polygon)\big).\x_{p}
\end{split}
\end{displaymath}
where $\mu(\polygon)$ denotes the Maslov index of the polygon class $\polygon$ and
$\m(\polygon)$ is the count of points in 
$\Mod(\polygon)$ modulo $2$.

\newpage

\section{Tangle complex and the issue of naturality}\label{sec:naturality}
In this section we address the issue of naturality in the sense of \cite{JT} for the
construction of \cite{AE-1}.
More precisely, we strengthen our result in \cite{AE-1} 
and show that for any $\Ring$-module $\Rin$, $\HFT^{\Rin}$ defines a functor 
from $\ATangleCat$ (See Definition \ref{def:ATangleCat}) to  $\AModCat$.
For the most part,
the argument of \cite{JT} for the naturality of sutured Floer homology may be 
copied here without significant modifications. In fact, we need to go 
through the  argument presented in Section 9 of \cite{JT} and check that all 
statements remain valid, a task that is outlined in the present section. 
%This would also help us fix  our notation for the rest of the paper.

\subsection{Oriented graph of isotopy Heegaard diagrams}
Let us fix a commutative algebra $\Ring$ over $\F$ and an $\Ring$-module 
$\Rin$ as before. 
Consider an $\Ring$-tangle $\Tangle=[M,T,\spinc,\la]$ and let $\HD=(\Sig,\alphas,\betas,\la,\spinc)$ be a Heegaard diagram for $\Tangle$.  %In this section, by an $\Ring$-diagram we mean a $5$-tuple 
%$(\Sig,\alphas,\betas,\la:\z\ra\Ring,\spinc)$, which is an $\Ring$-diagram
%in the sense of Definition~\ref{def:A-diagram}. \
Any collection of pairwise disjoint circles on $\Sig$ (like $\alphas$ and $\betas$)  is called an
{\emph{attaching set}}. 
Recall that each attaching set $\gammas$ on the pointed surface $(\Sig,\z)$
determines a tangle $(C[\gamma],T[{\gamma}])$ where $C[\gamma]$ is the compression body obtained from $\Sig\times [0,1]$
by attaching $2$-handles along $\gamma\times\{1\}$ and $T[\gamma]=\z\times [0,1]$. Then, $\del^-C[\gammas]=\Sig\times\{0\}\cong\Sig$ and 
$\del^+C[\gammas]\cong\Sig[\gammas]$, where $\Sig[\gammas]$ denotes the surface 
obtained by performing surgery on $\Sig$ along the curves in $\gammas$. 
Attaching sets $\gammas$ and $\gammas'$ on $\Sig$ are called 
\emph{compression equivalent}, denoted by $\gammas\sim\gammas'$, if there is a 
diffeomorphism $d:(C[\gammas],T[\gamma])\ra (C[\gammas'],T[\gamma'])$ such that it is the identity on 
$(\Sig,\z)=(\del^-C[\gammas],\del^-T[\gamma])=(\del^-C[\gammas'],\del^-T[\gamma'])$. This gives an equivalence relation
which descends to the isotopy classes, denoted by $[\gammas]$, of attaching sets 
$\gammas$ on
$\Sig$. If $\gammas\sim \gammas'$ then $[\gammas]$ and $[\gammas']$ are 
related by a sequence of handle slides, \cite[Lemma 2.11]{JT}. 
%\begin{defn}
%We say that an $\Ring$-diagram $(\Sig,\alphas,\betas,\la:\z\ra\Ring,\spinc)$
%is a {\emph{Heegaard diagram}} for an $\Ring$-tangle $\Tangle$ if $\Sig$ is a 
%separating surface in $M_\Tangle$ which cuts $T_\Tangle$ 
%transversely in $\z$, 
%$\la$ is induced by $\la_\Tangle$ 
%and $\alphas$ and $\betas$ are 
%attaching sets of curves such that elements of $\alphas$ and $\betas$ bound 
%disjoint disks on the 
%two sides of $\Sig$. Moreover, let $\Sig[\alphas]$ and $\Sig[\betas]$ be surfaces 
%obtained from compressing $\Sig$ along the $\alphas$ and $\betas$ curves, 
%respectively. Then, $(\Sig[\alphas],\z)$ and $(\Sig[\betas],\z)$ are isotopic relative 
%to $T_\Tangle$ to 
%\[(\del^-M_\Tangle,\del^- T_\Tangle)\ \ \ \ \text{and}\ \ \ \ 
%(\del^+M_\Tangle,\del^+T_\Tangle),\] respectively. Finally, under the aforementioned 
%isotopy, the $\SpinC$ class $\spinc$ corresponds to $\spinc_\Tangle$. 
%\end{defn}
%Let $(\Sig,\alphas,\betas,\la,\spinc)$ be a Heegaard diagram for $\Tangle$.

If for the 
attaching sets $\alphas'$ and $\betas'$ we have $[\alphas]=[\alphas']$
and $[\betas]=[\betas']$, and the $\spinc$-admissibility 
condition is satisfied then $(\Sig,\alphas',\betas',\la,\spinc)$ is also 
a Heegaard diagram for $\Tangle$. We may thus refer to the set of all such 
Heegaard diagrams as an {\emph{isotopy diagram}}
$(\Sig,A,B,\la,\spinc)$ for $\Tangle$, where $A=[\alphas]$ and $B=[\betas]$.

The isotopy diagrams 
\[\HD_1=(\Sig_1,A_1,B_1,\la_1:\z_1\ra\Ring,\spinc_1)\ \  \text{and}\ \   
\HD_2=(\Sig_2,A_2,B_2,\la_2:\z_2\ra\Ring,\spinc_2)\] 
are called $\alpha$-{\emph{equivalent}} if
$\Sig_1=\Sig_2,\ \z_1=\z_2,\ \la_1=\la_2$ and 
$B_1=B_2$ while $A_1\sim A_2$ and under the natural correspondence between 
$\SpinC$ structures, $\spinc_1=\spinc_2$. 
Similarly, we may define
$\beta$-equivalence. Stabilization and destabilization also induce operations on 
isotopy diagrams. Furthermore, a \emph{diffeomorphism} from $\HD_1$ to $\HD_2$ is an orientation preserving diffeomorphism 
$d:\Sig_1\ra \Sig_2$ such that 
$d(A_1)=A_2, d(B_1)=B_2,d(\z_1)=\z_2, d^*\spinc_2=\spinc_1$.

Finally, we may define the oriented {\emph{graph}}  $\Gcal=\Gcal_{\Tangle}$ 
(in the sense of 
\cite[Definition 2.21]{JT}) as follows. The vertices of $\Gcal$ are isotopy diagrams for 
$\Tangle$ and for any two vertices 
$\HD_1$ and $\HD_2$, the set of edges connecting $\HD_1$ to $\HD_2$, 
denoted by $\Gcal(\HD_1,\HD_2)$, is a union 
of four sets
\begin{displaymath}
\Gcal(\HD_1,\HD_2)=\Gcal_\alpha(\HD_1,\HD_2)\amalg
\Gcal_\beta(\HD_1,\HD_2)\amalg\Gcal_{\mathrm{stab}}(\HD_1,\HD_2)\amalg
\Gcal_{\mathrm{diff}}(\HD_1,\HD_2),
\end{displaymath} 
defined as follows. The set $\Gcal_\alpha(\HD_1,\HD_2)$ (or 
$\Gcal_{\beta}(\HD_1,\HD_2)$) consists of a single arrow if $\HD_1$ and $\HD_2$
are $\alpha$-equivalent (or $\beta$-equivalent), otherwise it is empty. Similarly, 
$\Gcal_{\mathrm{stab}}(\HD_1,\HD_2)$
consists of a single arrow if $\HD_2$ is obtained from $\HD_1$ by a stabilization or 
destabilization and $\Gcal_{\mathrm{diff}}(\HD_1,\HD_2)$ contains an arrow 
corresponding to any diffeomorphism
from $\HD_1$ to $\HD_2$. Let $\Gcal_\alpha,\Gcal_\beta,
\Gcal_{\mathrm{stab}}$ and $\Gcal_{\mathrm{diff}}$  denote the corresponding 
sub-graphs of $\Gcal$. The sub-graphs $\Gcal_\alpha,\Gcal_\beta$ 
and $\Gcal_{\mathrm{diff}}$ are in fact categories when endowed with the 
obvious compositions.
The graph $\Gcal$ is connected as an oriented graph.
\subsection{Special Heegaard diagrams} A Heegaard diagram $H=(\Sig,\alphas,\betas,\z)$ is called 
\emph{special} if $\alphas\sim \betas$, i.e. $\alphas$ is compression equivalent to $\betas$. Any special 
Heegaard diagram determines a balanced tangle $(M,T)$ that is obtained from a 
product tangle by applying sugeries on $\ell$ $0$-dimensional spheres, where $\ell=|\alphas|=|\betas|$. 
Furthermore, if a balanced tangle $(M,T)$ has a special Heegaard diagram then $\HFT(M,T,\spinc)=0$ 
for any \emph{non-torsin} class $\spinc\in\SpinC(M)$. Recall that a $\SpinC$ class $\spinc\in\SpinC(M)$ 
is called \emph{torsion} if $\spinc=[\relspinc]$ for a torsion relative $\SpinC$ class 
$\relspinc\in\SpinC(M,T)$ i.e. $c_1(\relspinc)=0$ as an element of $\Ht^2(X,\Z)$ for $X=M-\nd(T)$.

Associated with every $z_i\in\z=\{z_1,\hdots,z_\el\}$, let $\mu_i$ denote a small simple closed 
curve which is the boundary of a small disk around $z_i\in \Sig$. 
The torsion relative $\SpinC$ structures corresponding to a special Heegaard diagram are related 
to each other by adding integer multiples of the cohomology classes $\PD[\mu_i]\in\Ht^2(M,T;\Z)$, 
for $i=1,\hdots,\el$. In fact, the torsion relative $\SpinC$ structures $\relspinc\in\SpinC(M,T)$ with the property
that $\HFT(M,T,\relspinc)$ is non-zero form a cone, in the sense that every such relative $\SpinC$ structure
$\relspinc$ is of the form 
\begin{align*}
\relspinc=\relspinc_0+\sum_{i=1}^\el a_i\PD[\mu_i],\ \ \ \ \text{where}\ a_1,\hdots,a_\el\in\Z^{\geq 0}.
\end{align*}
The relative $\SpinC$ structure $\relspinc_0\in\SpinC(M,T)$ is thus uniquely associated with our special 
Heegaard diagram, and is called its {\emph{distinguished}} (torsion) relative $\SpinC$ structure for 
the special Heegaard diagram $H$.

Fix a special Heegaard diagram $H=(\Sig,\alphas,\betas,\z)$ as above, and let $(M,T)$ be the corresponding 
balanced tangle. Further, assume
that for the torsion $\SpinC$ class $\spinc_0\in\SpinC(M,T)$ the 
diagram $H$ is $\spinc_0$-admissible. If $\z=\{z_1,\ldots,z_{\el}\}$, then 
$\Ring_T=\frac{\Z[\la_1,\ldots,\la_{\el}]}{\mathcal{I}}$ where $\la_i$ is the variable
corresponding to $z_i$ and $\mathcal{I}$ is the ideal of relations. Note that $\mathcal{I}$ is 
generated by monomials corresponding to connected components of $\Sig\setminus \alphas$ 
(or $\Sig\setminus\betas$) that are not punctured spheres. 

Consider 
$a\cdot\x,b\cdot\y\in\CFT(M,T,\relspinc)$ such that $\x,\y\in\Ta\cap\Tb$, $\relspinc\in\SpinC(M,T)$ is torsion and 
$$a=\la_1^{a_1}\cdots \la_{\el}^{a_{\el}}\ \ \ \ \text{and}\ \ \ \ b=\la_1^{b_1}\cdots \la_{\el}^{b_{\el}}$$
are non-zero monomials in $\Ring_T$. Then, we say a homotopy disk $\phi\in\pi_2(\x,\y)$
connects $a\cdot\x$ to $b\cdot\y$ if $a_i+n_{z_i}(\phi)=b_i$ for any $i=1,\ldots,\el$.  We define 
$\mathrm{gr}(a\cdot\x,b\cdot\y)=\mu(\phi)$ for a homotopy disk 
$\phi\in\pi_2(\x,\y)$ connecting $a\cdot\x$ to $b\cdot\y$.

\begin{lem}
The map $\mathrm{gr}$ is well-defined and induces a relative $\Z$-grading on 
$$\CFT(M,T,\relspinc)=\CFT(\Sig,\alphas,\betas,\z,\relspinc),$$
for every torsion relative $\SpinC$ structure $\relspinc\in\SpinC(M,T)$.  
Furthermore, the subgroup of $\HFT(M,T,\relspinc_0)$ in top 
homological grading  is isomorphic to $\F$ for the distinguished 
relative $\SpinC$ structure $\relspinc_0\in\SpinC(M,T)$.
\end{lem}

\begin{proof}
Let $\phi_1,\phi_2\in\pi_2(\x,\y)$ be homotopy disks connecting $a.\x$ to $b.\y$. 
Then, $\Pcal=\phi_1-\phi_2$ would be a periodic domain such that $n_{z_i}
(\Pcal)=0$ for any $i=1,\ldots,\el$. Hence, 
$$\mu(\phi_1)-\mu(\phi_2)=\mu(\Pcal)=\langle c_1(\relspinc),H(\Pcal)\rangle=0.$$ 
In particular, $\mathrm{gr}$ is well-defined. For the second part,
it is not hard to show that $\HFT(M,T,\relspinc_0)$ is  invariant up to 
isomorphism in each relative grading. So the discussions in Section 6.2 of 
\cite{AE-1} implies that the subgroup with the top grading is isomorphic to 
$\F$.
\end{proof}

Hence, the homology group in the top grading has a well-defined 
generator up to sign, called {\emph{top generator}}, 
in the summand $\HFT(M,T,\relspinc_0)$ of $\HFT(M,T,\spinc_0)$, where 
$\relspinc_0\in \SpinC(M,T)$ is the distinguished torsion $\SpinC$ structure and 
$\spinc_0=[\relspinc_0]$ is its image in $\SpinC(M)$. This generator is denoted by $\Theta_{\alpha\beta}$.

\subsection{Weak Heegaard invariance}\label{sec:HFisom}
In this section, we check that $\HFT^{\Rin}$ is a \emph{weak Heegaard invariant} i.e. 
for any $\Ring$-tangle $\Tangle$, we construct a well-defined $\Ring$-module $\HFT^{\Rin}(\HD)$ 
for every vertex $\HD$ of $\Gcal_{\Tangle}$ and associate an isomorphism to any edge of $\Gcal_{\Tangle}$.

Consider an $\Ring$-diagram $\HD=(\Sig,\alphas,\betas,\la,\spinc)$ and a complex structure 
$\mathfrak{j}$ on $\Sig$. Recall that any generic path $J_s$ of perturbations of $\Sym^{\ell}(\mathfrak{j})$ 
gives a chain complex
$$\CFT_{J_s}(\Sig,\alphas,\betas,\la,\spinc),$$ 
generated by the intersection points $\x\in\Ta\cap\Tb$ with 
$\spinc(\x)=\spinc$, where $\ell=|\alphas|=|\betas|$. We will usually drop the complex structure $\mathfrak{j}$
on $\Sig$ from the notation.
Given two different choices $(\mathfrak{j},J_s)$ and $(\mathfrak{j}',J_s')$, the proof of Lemma~2.11 from 
\cite{OS-4mfld} gives an isomorphism of $\Ring$-modules
\begin{displaymath}
\Phi_{J_s\ra J_s'}:
\HFT^\M_{J_s}(\Sig,\alphas,\betas,\la,\spinc)\lra 
\HFT^\M_{J_s'}(\Sig,\alphas,\betas,\la,\spinc).
\end{displaymath} 
for any $\Ring$-module $\Rin$. 
% one obtains the chain map
% \begin{displaymath}
%\Phi_{J_s\ra J_s'}:
% \CFT_{J_s}(\Sig,\alphas,\betas,\la,\spinc)\lra 
% \CFT_{J_s'}(\Sig,\alphas,\betas,\la,\spinc).
% \end{displaymath}
%The proof of Lemma~2.11 from \cite{OS-4mfld} 
%implies that for any $\Ring$-module $\Rin$, the above chain map gives a corresponding 
%isomorphism of $\Ring$-modules
%\begin{displaymath}
%\Phi_{J_s\ra J_s'}:
% \HFT^\M_{J_s}(\Sig,\alphas,\betas,\la,\spinc)\lra 
% \HFT^\M_{J_s'}(\Sig,\alphas,\betas,\la,\spinc).
% \end{displaymath} 
% where $\HFT^\M_{J_s}(\Sig,\alphas,\betas,\la,\spinc)$ and 
%$\HFT^\M_{J_s'}(\Sig,\alphas,\betas,\la,\spinc)$ are the homology groups of the 
%chain complexes 
%\begin{displaymath}
%\CFT_{J_s}(\Sig,\alphas,\betas,\la,\spinc)\otimes_\Ring\Rin\ \ \text{and}\ \  
%\CFT_{J_s'}(\Sig,\alphas,\betas,\la,\spinc)\otimes_\Ring\Rin,
%\end{displaymath}
%respectively. 
Moreover,
$$\Phi_{J_s'\ra J_s''}\circ \Phi_{J_s\ra J_s'}=\Phi_{J_s\ra J_s''}.$$
%We drop the $\SpinC$ class $\spinc$ from the notation for simplicity. 
One may thus define
\begin{displaymath}
\HFT^\M(\Sig,\alphas,\betas,\la,\spinc)=\frac{\coprod_{J_s}
\HFT^\M_{J_s}(\Sig,\alphas,\betas,\la,\spinc)}{\sim}
\end{displaymath}
where $x\sim y$ if $y=\Phi_{J_s\ra J_s'}(x)$ for some $J_s$ and $J_s'$. 
Since we will face several equivalence relations in the definition of 
$\HFT^\M(\Tangle)$
we will abuse the notation and denote all of them by $\sim$, leaving it to the 
reader to make the distinctions.
%%%%%%%%%%%%%

Let us ssume that 
\[\HD=(\Sig,\alphas,\betas,\betas',\la,\spinct)\] is an 
$\Ring$-diagram and $\betas\sim\betas'$. Furthermore,
suppose that the restriction of $\spinct$ to $M_{\beta\beta'}$ is the torsion 
$\SpinC$ structure $\spinc_0$, which was discussed in the previous subsection.
With this restriction in place, $\spinct$ is determined by its restriction 
$\spinc$ to $M_{\alpha\beta}$, and the induced $\SpinC$ structure 
on $M_{\alpha\beta'}$ is in correspondence with $\spinc$ under the 
natural identification of $M_{\alpha\beta}$ with $M_{\alpha\beta'}$. We will thus 
abuse the notation and denote the above $\Ring$-diagram by 
\[(\Sig,\alphas,\betas,\betas',\la,\spinc)\] 
implying that the above restrictions on $\spinct$ are imposed.
Then, using the top generator 
$\Theta_{\beta\beta'}$ we may define an isomorphism
$$\Phi_{\beta\ra\beta'}^{\alpha}:\HFT^\Rin(\Sig,\alphas,\betas,\la,\spinc)
\lra \HFT^\Rin(\Sig,\alphas,\betas',\la,\spinc).$$

The arguments in the sequence of lemmas preceding Proposition 9.9 in \cite{JT}
may then be copied  in our setup to show that if 
$(\Sig,\alphas,\betas,\la,\spinc)$ and $(\Sig,\alphas,\betas',\la,\spinc)$
are Heegaard diagrams for $\Tangle$, one may define 
a chain map and a well-defined induced isomorphism
$$\Phi_{\beta\ra\beta'}^{\alpha}:\HFT^\Rin(\Sig,\alphas,\betas,\la,\spinc)
\lra \HFT^\Rin(\Sig,\alphas,\betas',\la,\spinc),$$
using the top generators and a detour to admissible triple diagrams (note that
the above assumptions do not imply that the $(\Sig,\alphas,\betas,\betas',\la,\spinc)$
is an $\Ring$-diagram).
Furthermore, if $(\Sig,\alphas,\betas'',\la,\spinc)$ is a third 
Heegaard diagram for $\Tangle$ so that $\betas\sim\betas'\sim\betas''$, we get
\begin{equation}\label{eq:functoriality-1}
\Phi_{\beta'\ra\beta''}^{\alpha}\circ \Phi_{\beta\ra\beta'}^{\alpha}=
\Phi_{\beta\ra\beta''}^{\alpha}.
\end{equation}
Similarly, one may define $\Phi^{\alpha\ra \alpha'}_{\beta}$.  If 
$(\Sig,\alphas,\betas,\la,\spinc)$, $(\Sig,\alphas',\betas,\la,\spinc)$, 
$(\Sig,\alphas,\betas',\la,\spinc)$
and $(\Sig,\alphas',\betas',\la,\spinc)$ are Heegaard diagrams for $\Tangle$ so 
that $\alphas\sim\alphas'$ and $\betas\sim\betas'$, 
one may define 
$$\Phi_{\beta\ra\beta'}^{\alpha\ra \alpha'}=
\Phi_{\beta\ra\beta'}^{\alpha'}\circ \Phi_{\beta}^{\alpha\ra \alpha'}=
\Phi_{\beta'}^{\alpha\ra \alpha'}\circ \Phi_{\beta\ra\beta'}^{\alpha}.$$ 
It follows from \cite[Lemma 9.11]{JT} that these isomorphisms satisfy
\begin{equation}\label{eq:functoriality-2}
\begin{split}
&\Phi_{\beta'\ra\beta''}^{\alpha'\ra \alpha''}
\circ\Phi_{\beta\ra\beta'}^{\alpha\ra \alpha'}=
\Phi_{\beta\ra\beta''}^{\alpha\ra \alpha''}\ \ \ \text{and}\ \ \ 
\Phi_{\beta\ra\beta}^{\alpha\ra \alpha}=\Phi_{\beta\ra\beta}^{\alpha}
=\Phi_{\beta}^{\alpha\ra \alpha}=
\mathrm{Id}_{\HFT^\Rin(\Sig,\alpha,\beta,\la,\spinc)}.
\end{split}
\end{equation}

These isomorphisms may be used to construct a well-defined $\Ring$-module $\HFT^{\Rin}(\HD)$ for every 
vertex $\HD$ of $\Gcal_{\Tangle}$ and associate isomorphisms to edges of $\Gcal_{\alpha}$ and 
$\Gcal_{\beta}$, as follows.

Given an isotopy diagram $\HD=(\Sig,A,B,\la,\spinc)$ for $\Tangle$, we denote by 
$\Mod_{\HD}$ the set of all
Heegaard diagrams corresponding to $\HD$, and we let
\begin{displaymath}
\HFT^\Rin(\HD)=
\frac{\coprod_{(\Sig,\alpha,\beta,\la,\spinc)\in \Mod_{\HD}} 
\HFT^\Rin(\Sig,\alphas,\betas,\la,\spinc)}{\sim}
\end{displaymath} 
where the generators $x\in \HFT^\Rin(\Sig,\alphas,\betas,\la,\spinc)$ and 
$x'\in\HFT^\Rin(\Sig,\alphas',\betas',\la,\spinc)$ are 
equivalent if and only if $x'=\Phi_{\beta\ra\beta'}^{\alpha\ra \alpha'}(x)$.
%This construction associates a well-defined $\Ring$-module $\HFT^\Rin(\HD)$
%to every vertex $\HD$ of $\Gcal_{\Tangle}$.

Suppose $\HD=(\Sig,A,B,\la,\spinc)$ and 
$\HD'=(\Sig,A,B',\la,\spinc)$ are $\beta$-equivalent.  Pick  
representatives $(\Sig,\alphas,\betas,\la,\spinc)$ and 
$(\Sig,\alphas,\betas',\la,\spinc)$ of 
$\HD$ and $\HD'$, respectively. The formal proof of  
\cite[Lemma 9.17]{JT} then implies 
that $\Phi^{\alpha}_{\beta\ra\beta'}$ descends to a well-defined isomorphism
\begin{displaymath}
\Phi^{A}_{B\ra B'} :\HFT^\Rin(\HD)\lra \HFT^\Rin(\HD').
\end{displaymath}
Similarly, an $\alpha$-equivalence from $\HD=(\Sig,A,B,\la,\spinc)$ 
to $\HD'=(\Sig,A',B,\la,\spinc)$ gives a
well-defined isomorphism 
\begin{displaymath}
\Phi^{A\ra A'}_{B} :\HFT^\Rin(\HD)\lra \HFT^\Rin(\HD').
\end{displaymath}

Consider an edge of $\Gcal_{\mathrm{diff}}$ i.e. a diffeomorphism $d:\HD\ra \HD'$. It gives a correspondence 
between complex structures, which in turn gives the isomorphism
\begin{displaymath}
d_*:\HFT^\Rin(\HD)\lra \HFT^\Rin(\HD'),
\end{displaymath}
c.f.  \cite[Definition 9.19 and Lemma 9.20]{JT}. 

Finally, 
if the Heegaard diagram $(\Sig',\alphas',\betas',\la',\spinc')$
is obtained from $(\Sig,\alphas,\betas,\la,\spinc)$ by stabilization, there is a correspondence 
between the homotopy classes of disks on the two sides. Furthermore,  
for  suitable almost complex
structures on the two diagrams,  there is an isomorphism of 
chain complexes associated with the two diagrams.
If $\HD$ and $\HD'$ denote the isotopy diagrams corresponding to the 
above two Heegaard diagrams for $\Tangle$,
\cite[Lemma 9.21]{JT} (which is basically  \cite[Lemma 2.15]{OS-4mfld})
may be used to construct the isomorphism
\begin{displaymath}
\sig_{\HD\ra \HD'}:\HFT^\Rin(\HD)\lra \HFT^\Rin(\HD').
\end{displaymath}

Hence, $\HFT^\Rin$ is a weak Heegaard invariant, and the above considerations reprove the 
invariance of the quasi-isomorphism type of the chain complex $\CFT^\Rin(\Tangle)$. 

\subsection{Strong Heegaard invariance} The weak Heegaard invariant $\HFT^{\Rin}$ is a 
\emph{strong Heegaard invariant} if it satisfies the following axioms, \cite{JT}:
\begin{enumerate}
\item {\bf{Fuctoriality:}} The restriction of $\HFT^\Rin$ to the categories 
$\Gcal_\alpha,\Gcal_\beta$ and $\Gcal_{\mathrm{diff}}$ is a functor to the category
$\AModCat$ of $\Ring$-modules. Moreover, if $\sig:\HD\ra \HD'$ is a stabilization and 
$\sig':\HD'\ra \HD$ is the corresponding destabilization then 
$\HFT^\Rin(\sig')=\HFT^\Rin(\sig)^{-1}$.
\item{\bf{Commutativity:}} For every distinguished rectangle
\begin{displaymath}
\begin{diagram}
\HD_1&\rTo{\hmove_1}&\HD_2\\
\dTo{\hmove_2}&&\dTo{\hmove_3}\\
\HD_3&\rTo{\hmove_4}&\HD_4
\end{diagram}
\end{displaymath}
in $\Gcal$ we have $\HFT^\Rin(\hmove_3)\circ\HFT^\Rin(\hmove_1)=
\HFT^\Rin(\hmove_4)\circ \HFT^\Rin(\hmove_2)$, where $\HFT^{\Rin}(e_i)$ 
denotes the associated isomorphism to $e_i$.
Here, distinguished rectangles are the rectangles of the above type 
such that  one of the following is the case:
\begin{itemize}
\item The horizontal arrows are $\alpha$-equivalences while the vertical 
arrows are $\beta$-equivalences.
\item Both horizontal arrows are $\alpha$- or $\beta$-equivalences, while the 
vertical arrows are stabilizations.
\item  Both horizontal arrows are $\alpha$- or $\beta$- equivalences, while the 
vertical arrows are diffeomorphisms with the same restriction on the surface.
\item The square corresponds to the two possible ways of performing disjoint 
stabilizations.
\item The horizontal arrows are diffeomorphisms and the vertical arrows are 
stabilizations which correspond to one another via the diffeomorphisms.
\end{itemize}
\item{\bf{Continuity:}} If $\HD$ is a vertex of $\Gcal$ and 
$d\in\Gcal_{\mathrm{diff}}(\HD,\HD)$
is a diffeomorphism isotopic to the identity, then 
$\HFT^\Rin(d)=\mathrm{Id}_{\HFT^\Rin(\HD)}$.
\item{\bf{Handleswap invariance:}} For every simple handle swap 
\begin{displaymath}
\begin{diagram}
\HD_1&&\\
\uTo{\hmove_3}&\rdTo{\hmove_1}&\\
\HD_3&\lTo{\hmove_2}&\HD_2
\end{diagram}
\end{displaymath}
in the sense of \cite[Definition 2.32]{JT} in $\Gcal$ we have 
\[\HFT^\Rin(\hmove_3)\circ\HFT^\Rin(\hmove_2)\circ\HFT^\Rin(\hmove_1)
=\mathrm{Id}_{\HFT^\Rin(\HD_1)}.\]
\end{enumerate} 

Functoriality follows from equations (\ref{eq:functoriality-1}) and 
(\ref{eq:functoriality-2}) for 
$\alpha$- and $\beta$-equivalences, and is a consequence of the definition for 
the diffeomorphisms. Moreover, the equality 
$\HFT_\Rin(\sig')=\HFT_\Rin(\sig)^{-1}$ is a consequence 
of the definition. 

Our earlier considerations proves the axiom of commutativity for the first 
three types of distinguished rectangles. The commutativity of a diagram of the 
fourth type is satisfied in the level of chain complexes if the 
$\Ring$-diagrams (representing $\HD_i$),  and the complex structures on the 
surfaces are chosen correctly.  Similarly and following the argument of 
\cite[Subsection 9.2]{JT}, for suitable $\Ring$-diagrams 
representing the isotopy diagrams $\HD_i$
and the correct choice of almost complex structures, the commutativity
of the diagrams of the fifth type is satisfied in the level of chain complexes.

One may then copy the proof of   \cite[Proposition 9.23]{JT} and show that if for the 
$\Ring$-diagram $(\Sig,\alphas,\betas,\la,\spinc)$ the map
$d:\Sig\ra \Sig$
is isotopic to identity, then $d_*=\Phi_{\beta\ra\beta'}^{\alpha\ra \alpha'}$, where 
$\alpha'=d(\alpha)$ and $\beta'=d(\beta)$. Thus, it follows from our definition of $\HFT_\Rin(\HD)$ that 
the induced isomorphism $\HFT^{\Rin}(d)$ is identity.
This completes the proof of continuity.

Finally, let $\HD_1,\HD_2$ and $\HD_3$ denote the isotopy diagrams corresponding to a 
handleswap. 
Choose $\Ring$-diagrams $(\Sig,\alphas_i,\betas_i,\la,\spinc_i)$ representing
$\HD_i,\ i=1,2,3$. 
The argument given for handleswap invariance in Subsection 9.3 of \cite{JT}
regards the region corresponding to the handleswap in the aforementioned Heegaard
diagrams as a genus $2$ subsurface $\Sig^0$ of the  
connected sum of $\Sig^0$  with another surface 
$\Sig^1$, such that the Heegaard diagrams are identical on $\Sig^1$, all markings 
in $\z$ lie on $\Sig^1$, and two curves from each one of the attaching sets 
$\alphas_i$ and $\betas_i$ are on $\Sig^0$. It is implied by  
\cite[Lemma 9.25 and Lemma 9.28]{JT}
(which stay correct in our more general framework)  that
every triangle class which contributes
to $\HFT^\Rin(\hmove_1)$ or $\HFT^\Rin(\hmove_2)$ may be decomposed as a disjoint 
union of a triangle class on $\Sig^1$ which does not pass through the connected sum 
region and a small traingle class on $\Sig^0$. The proof of Proposition 9.24 from 
\cite{JT} thus goes through without difficulty, completing the proof of handleswap 
invariance.

Let $\hpath$ denote an oriented path in $\Gcal_{\Tangle}$ from the isotopy diagram 
$\HD$ to the isotopy diagram $\HD'$. Composing the isomorphisms corresponding to the 
edges in $\hpath$ we obtain an isomorphism
\begin{displaymath}
\HFT^\Rin(\hpath):\HFT^\Rin(\HD)\lra \HFT^\Rin(\HD').
\end{displaymath} 
If $\hpath$ and $\hpath'$ are different oriented paths from $\HD$ to $\HD'$,
\cite[Theorem 2.39]{JT}, implies that $\HFT^\Rin(\hpath)=\HFT^\Rin(\hpath')$, since 
$\HFT^\Rin$ is a strong Heegaard invariant. In other words, associated with the vertices 
$\HD$ and $\HD'$ of $\Gcal_{\Tangle}$ there is a well-defined isomorphism
\begin{displaymath}
\HFTmap^\Rin_{\HD\ra \HD'}:\HFT^\Rin(\HD)\lra \HFT^\Rin(\HD'). 
\end{displaymath}
It is clear that 
\begin{displaymath}
\HFTmap_{\HD'\ra \HD''}^\Rin\circ \HFTmap_{\HD\ra \HD'}^\Rin=\HFTmap_{\HD\ra \HD''}^\Rin.
\end{displaymath}
One may thus define
\begin{displaymath}
\HFT^\Rin(\Tangle):=\frac{\coprod_{\HD\in |\Gcal_{\Tangle}|}\HFT^\Rin(\HD)}{\sim},
\end{displaymath}
where $x$ is equivalent to $y$ if $y=\HFTmap_{\HD\ra \HD'}^\Rin(x)$. 
Thus, associated with every $\Ring$-tangle $\Tangle$ and $\Ring$-module 
$\Rin$ we obtain a well-defined  $\Ring$-module $\HFT^\Rin(\Tangle)$. 

Let $d:\Tangle=[M,T,\la,\spinc]\ra \Tangle'=[M',T',\la',\spinc']$ be a diffeomorphism of $\Ring$-tangles. Pick an isotopy diagram $\HD=(\Sig,A,B,\la,\spinc)$ for $\Tangle$. The diffeomorphism $d$ takes $\Sig$ to a surface $\Sig'$ in $M'$ and the markings 
$\z$ to a set $\z'$ of markings such that $\z'=\Sig'\cap T'$.  Let $A'=d(A)$ and $B'=d(B)$. Furthermore, we obtain a 
$\SpinC$ structure $\spinc'$ for  $(\Sig',A',B',\z')$ which corresponds to 
$\spinc$. We may define $\la':\z'\ra\Ring$
as $\la\circ d^{-1}$.
We thus obtain the isotopy diagram
$\HD'=d(\HD)$ for $\Tangle'$. The isomorphism 
$$\HFTmap^\Rin(d):\HFT^\Rin(\HD)\lra \HFT^\Rin(\HD')$$
associated with the diffeomorphism 
$d$ from $H$ to $H'$, induces a well-defined isomorphism 
\begin{displaymath}
\HFTmap^\Rin(d):\HFT^\Rin(\Tangle)\lra \HFT^\Rin(\Tangle').
\end{displaymath}  
The above considerations imply the following theorem.

\begin{thm}\label{thm:functor-for-tangle}
For every algebra $\Ring$ over $\F$ and every $\Ring$-module $\Rin$, 
assigning $\HFT^\Rin(\Tangle)$ to the $\Ring$-tangle $\Tangle$
in  $\ATangleCat$ gives a functor
\begin{displaymath}
\HFT^\Rin: \ATangleCat\lra \AModCat.
\end{displaymath}
\end{thm}

\subsection{The action of $\Lambda^*(H_1(M,\Z)/\mathrm{Tors})$}
As in the usual setup of the Heegaard Floer homology, there is a natural action of 
$\Lambda^*(H_1(M_\Tangle,\Z)/\mathrm{Tors})$ on $\HFT^\M(\Tangle)$ as follows.
Let us assume that $H=(\Sig,\alphas,\betas,\la:\z\ra\Ring,\spinc)$ is an $\Ring$-diagram
for the $\Ring$-tangle $\Tangle=[M,T,\la,\spinc]$.
First of all, as discussed in Subsection 2.4 of \cite{OS-3m1}, there is a 
homotopy long exact sequence
\begin{diagram}
0&\rTo&\Z&\rTo&\pi_1(\Omega(\Ta,\Tb))&\rTo&\pi_1(\Ta\times \Tb)&\rTo&
\pi_1(\Sym^g(\Sig)).
\end{diagram}
Here $\Omega(\Ta,\Tb)$ denotes the space of paths in $\Sym^g(\Sig)$ joining 
$\Ta$ to $\Tb$.
Under the identification of $\pi_1(\Sym^g(\Sig))$ with $H^1(\Sig,\Z)$, $\pi_1(\Ta)$ and 
$\pi_1(\Tb)$ correspond to $H^1(C(\alphas),\Z)$ and $H^1(C(\betas),\Z)$, respectively.
After comparing the above exact sequence with the cohomology long exact sequence 
for the decomposition 
$M=C(\alphas)\cup_\Sig C(\betas)$, we obtain the short exact sequence
\begin{diagram}
0&\rTo&\Z&\rTo& \pi_1(\Omega(\Ta,\Tb))&\rTo & H^1(M,\Z)&\rTo &0.
\end{diagram}
Applying $\mathrm{Hom}(-,\Z)$ to the above short exact sequence we obtain
\begin{diagram}
0&\rTo&H_1(M,\Z)/\mathrm{Tors}%\simeq\mathrm{Hom}(H^1(M,\Z),\Z)
&\rTo{}& H^1(\Omega(\Ta,\Tb),\Z)&\rTo &\Z.
\end{diagram}
Every element $\zeta\in H_1(M,\Z)/\mathrm{Tors}$ may thus be realized 
as an element of  $H^1(\Omega(\Ta,\Tb),\Z)$, which is represented by a $1$-cocycle
$Z(\zeta)\in Z^1(\Omega(\Ta,\Tb),\Z)$ in the space of paths connecting $\Ta\cap\Tb$. 
If $\phi\in\pi_2(\x,\y)$ is the homotopy class of a Whitney disk, it may be viewed as an arc 
in $\Omega(\Ta,\Tb)$ which connects the constant path at $\x$ to the constant path 
at $\y$. The evaluation of  $Z(\zeta)$ over this path gives a value $n_{Z(\zeta)}(\phi)$. 
Correspondingly, we may define a map
\begin{align*}
&A_{Z(\zeta)}:\CFT(\Sig,\alphas,\betas,\la,\spinc)\ra \CFT(\Sig,\alphas,\betas,\la,\spinc)\\
&A_{Z(\zeta)}(\x):=\sum_{\substack{\y\in\Ta\cap\Tb\\ \spinc(\y)=\spinc}}
\sum_{\phi\in\pi_2^1(\x,\y)}n_{Z(\zeta)}(\phi)\m(\phi)\la(\phi)\y.
\end{align*}
The proof of  \cite[Lemma 4.18]{OS-3m1} implies that $A_{Z(\zeta)}$ is a 
chain map and the proof of  \cite[Lemma 4.19]{OS-3m1} implies that if 
$Z(\zeta)$ is a coboundary then $A_{Z(\zeta)}$ is chain homotopic 
to zero. The map induced by $A_{Z(\zeta)}$ on homology is thus independent of the choice 
of the cocycle $Z(\zeta)$, and may thus be represented by 
\[A_\zeta:\HFT^\M(\Sig,\alphas,\betas,\la,\spinc)\ra \HFT^\M(\Sig,\alphas,\betas,\la,\spinc).\] 
The proof of   \cite[Proposition 4.17]{OS-3m1} may then be copied to show that 
$A_\zeta\circ A_\zeta=0$. Consequently, we obtain an action of the exterior 
algebra $\Lambda^*(H_1(M,\Z)/\mathrm{Tors})$ on 
$\HFT^\M(\Sig,\alphas,\betas,\la,\spinc)$.

We may then follow the steps toward weak and strong Heegaard invariance of the functor
$\HFT^\M$, and observe that all the isomorphisms which correspond to the edges of the 
graph $\Gcal_{\Tangle}$ preserve the action of  $\Lambda^*(H_1(M,\Z)/\mathrm{Tors})$
constructed above. This observation implies the following proposition.

\begin{prop}\label{prop:action}
For every $\Ring$-tangle $\Tangle$ and every $\Ring$-module $\M$, there
is a natural action of  $\Lambda^*(H_1(M_\Tangle,\Z)/\mathrm{Tors})$
on $\HFT^\M(\Tangle)$.
\end{prop}

\newpage

% !TEX TS-program = pdflatexmk
\section{Parametrized Cerf decomposition}\label{sec:Cerf}
\subsection{Parametrized elementary cobordism}\label{subsec:elem-cob}
Let $(M,T)$ be a  balanced tangle. For $k=0,1,2$, a framed $k$-sphere $\Fsphere$ in 
$(M,T)$ is an embedding of $S^k\times D^{3-k}$ in $M-T$. In other word, it is an embedded 
$k$-sphere $a(\Fsphere)=S^{k}\times\{0\}$, called attaching sphere, 
%\todo{AA added this notation, similar to Andras's paper} 
in $M-T$ together with a trivialization $\nu$ of its normal bundle. We will denote by 
$$\Cobb(\Fsphere )=(W(\Fsphere),\sur(\Fsphere))$$ 
the cobordism obtained by attaching a 
$k$-handle to $M\times [0,1]$ along $\Fsphere\times\{1\}$ to construct $W(\Fsphere)$ and 
setting $\sur(\Fsphere)=T\times [0,1]\subset W(\Fsphere)$. 
Thus $\Cobb(\Fsphere)$ is a cobordism from $(M,T)$ 
to $(M(\Fsphere),T)$, where $M(\Fsphere)$ is obtained by surgery on $M$ along $\Fsphere$. \\

Similarly, a framed arc $\Farc$ in $(M,T)$ is an embedding of 
$D^2\times D^1$ in $M$ such that 
\begin{displaymath}
\Farc^{-1}(T)=\left(\{x=0\}\times\{-1\}\right)\coprod  
\left(\{x=0\}\times\{1\}\right),
\end{displaymath}
where $(x,y)$ denotes the standard coordinate system on $D^2$ (as a subset of $\R^2$).  Moreover, $\Farc$ is called \emph{orientation preserving} if the restriction of $\Farc$ to 
\[\{x=0\}\times\{-1\}\coprod\{x=0\}\times\{1\}\subset \partial\left(\{x=0\}
\times [-1,1]\right),\]
which is equipped with the boundary orientation, is orientation preserving as a map 
to $T$. We think 
of $\Farc$ as an embedded arc $a(\Farc)=\{0\}\times D^1$, called attaching arc, connecting the 
two points $\del a(\Farc)\cap T$, 
together with a trivialization $\nu$ of its normal bundle in $M$ which is  
compatible at the end points, with the trivialization induced from the orientation of $T$. 
Abusing the notation, we denote $\Farc(D^2\times D^1)$ by $\Farc$. Associated with a framed 
arc $\Farc$, let $T({\Farc})\subset M$ be the properly embedded 1-manifold obtained by doing 
band surgery on $T$ along $\Farc$ i.e. 
$$T({\Farc})=\left(T-\Farc\cap T)\right)\bigcup\Farc \left( (\{x=-1\}\cup\{x=1\})
\times [-1,1]\right).$$
Note that if $\Farc$ is orientation preserving, then the orientation on $T$ induces an orientation 
on $T({\Farc})$. Moreover, 
for every set $\Farc=\{\Farc_1,\ldots,\Farc_n\}$ of framed arcs in $(M,T)$, denote 
the properly embedded 1-manifold constructed by doing surgery on $T$ along the framed arcs 
$\Farc_1,\ldots,\Farc_n$ by $T({\Farc})$. Thus, if $\Farc$ is a single framed arc, $T(\Farc)$ and 
$T(\{\Farc\})$ are the same object.
\begin{defn}
A set $\Farc=\{\Farc_1,\ldots,\Farc_n\}$ of framed arcs in $(M,T)$ is called 
\emph{acceptable} if for any 
$i$, $\Farc_i$ is orientation preserving and $(M,T({\Farc}))$ is a tangle. 
\end{defn}

Let $\Farc=\{\Farc_1,\ldots,\Farc_n\}$ be an acceptable set of framed arcs in $(M,T)$. 
Corresponding to 
$\Farc$, we construct a cobordism $\Cobb({\Farc})=(W({\Farc}),\sur({\Farc}))$ from $(M,T)$ to 
$(M,T({\Farc}))$ by attaching a {\emph{standard pair}} to 
$(M\times[0,1],T\times [0,1])$ along 
$\Farc\subset M\times\{1\}$ as follows. The {\emph{standard saddle}} is the pair 
$(H,B)$, which is identified in $\R^4$ via
\begin{displaymath}
\begin{split}
&H:=D^2\times D^1\times D^1=\left\{(x,y,z,t)\in\R^4\ \Big|\ 
\begin{array}{c}
x^2+y^2\le1\\
z,t\in [-1,1]
\end{array}
\right\}\ \ \text{and}
\\
& B=\left\{(x,y,z,t)\in D^2\times D^1\times D^1\ \Big|\ 
\begin{array}{c}
(t+1)y^2+(t-1)z^2=2t\\
x=0
\end{array}\right\}
\end{split}
\end{displaymath}
For $i\in\{-1,1\}$,  denote $\del_{i}(H,B):=(D^2\times D^1\times\{i\}, B\cap (D^2\times D^1\times\{i\}))$. 
See Figure~\ref{fig:Farc} for a picture of the projection of the standard pair over 
the $3$-dimensional box
$\{x=0\}\times D^1\times D^1$.\\

\begin{figure}[ht]
\def\svgwidth{5cm}
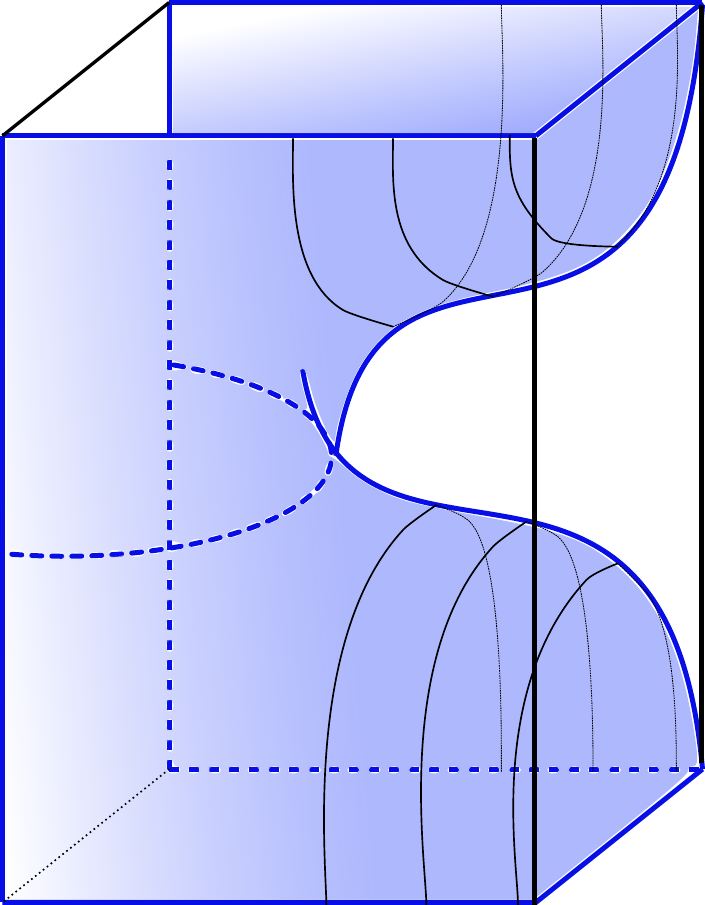
\caption{The standard pair, which is pictured in the cube $\{x=0\}=[-1,1]^3$.
}\label{fig:Farc}
\end{figure}

Let $\Cobb(\Farc)=(W(\Farc),\sur(\Farc))$ denote the cobordism 
\begin{displaymath}
\frac{\left(M\times [0,1],T\times [0,1]\right)\coprod 
\left(\coprod_{i=1}^n(H_i,B_i)\right) }
{\left\{(\Farc_i\times\{1\}, (T\cap\Farc_i)\times\{1\})\sim\del_{-1}(H_i,B_i)\ \Big|\ \text{for any}~~1\le i\le n\right\}}
\end{displaymath}
after smoothing the corresponding corners.  Here $(H_i,B_i)$ are copies of the standard pair $(H,B)$ for 
$i=1,\ldots,n$. Note that $W(\Farc)\simeq M\times [0,1]$.
\begin{defn}
A stable cobordism $\Cobb=(W,\sur)$ from the tangle $(M,T)$ to the tangle $(M',T')$ is called 
{\emph{elementary}}, if for a framed sphere $\Fsphere$ or an acceptable set of framed arcs 
$\Farc$ 
in $(M,T)$, $\Cobb$ is diffeomorphic to the corresponding cobordism $\Cobb(\Fsphere)$ or 
$\Cobb(\Farc)$.

Moreover, a {\emph{parametrized elementary cobordism}} is an elementary cobordism $\Cobb$ 
as above, accompanied with one of the following:
\begin{enumerate}
\item A framed sphere $\Fsphere\subset M$, together with the isotopy class of 
a diffeomorphism 
$$d:(M(\Fsphere),T)\ra (M',T')$$
so that for a diffeomorphism $D: \Cobb(\Fsphere)\ra \Cobb$  with $D|_{(M,T)}=\mathrm{Id}$, 
we have $d=D|_{(M(\Fsphere),T)}$.
\item A framed arc $\Farc\subset M$, together with the isotopy class of 
a diffeomorphism 
$$d:(M,T(\Farc))\ra (M',T')$$
so that for a diffeomorphism $D: \Cobb(\Farc)\ra \Cobb$ with $D|_{(M,T)}=\mathrm{Id}$, we 
have $d=D|_{(M,T(\Farc))}$.
\end{enumerate} 
\end{defn}
Note that in the above definition we might also have $\Fsphere=\emptyset$.

\begin{defn}
A {\emph{parametrized Cerf decomposition}} of a stable cobordism $\Cobb=(W,\sur)$ 
from a tangle $(M,T)$ 
to a tangle $(M',T')$ is a decomposition
$$\Cobb=\Cobb_1\cup_{(M_1,T_1)}\cdots \cup_{(M_{m-1},T_{m-1})}\Cobb_m,$$
where $\Cobb_i=(W_i,\sur_i)$ is a parametrized elementary cobordism from the tangle
$(M_{i-1},T_{i-1})$  to the tangle $(M_{i},T_{i})$, $(M_0,T_0)=(M,T)$ and 
$(M_{m},T_{m})=(M',T')$. For every $i=1,\ldots,m$, 
depending on the type of $\Cobb_i$, we let $(\Fsphere_i,d_i)$ or $(\Farc_i,d_i)$ denote its 
parametrization.
\end{defn}

\subsection{Parametrized Morse data} 
Let $\Cobb=(W,\sur)$ be a cobordism from  $(M,T)$ to  $(M',T')$. 
A function $G:W\ra [a,b]$ is called a Morse function on $\Cobb$, if $G$ has no critical points in 
a neighborhood of $\del W\cup \sur$, both $G$ and $g=G|_{\sur}$ are Morse  
(on $W$ and $F$, respectively) and 
$$G^{-1}(a)=M,\ \ \ G^{-1}(b)=M'\ \ \ \text{and}\ \ \ G|_{\del_hW}=\pi_2,$$
where $\pi_2$ is the projection over the second factor under an identification 
\begin{displaymath}
(\del_h W,\del_h \sur)=(\del M,\del T)\times [a,b].
\end{displaymath}
The set of critical points of $G$ on $\Cobb$, $\mathrm{Crit}_{\Cobb}(G)$, is defined as 
$$\mathrm{Crit}_{\Cobb}(G):=\mathrm{Crit}_{W}(G)\cup\mathrm{Crit}(g)$$
where $\mathrm{Crit}_W(G)$ and $\mathrm{Crit}(g)$ are the sets of critical points of $G$ 
on $W$ and $g=G|_{\sur}$ on $\sur$, respectively. A Morse function $G$ on $\Cobb$ is called 
{\emph{proper}} if it has distinct values at its
critical points on $\Cobb$. $G$ is called {\emph{indefinite}} if both $G$ and $g$, as  
Morse functions on $W$ and $\sur$
respectively, are indefinite i.e. have no critical points with the minimal and maximal index. 
Therefore, if $G$ is indefinite on $\Cobb$,  $g$ has only critical points of index $1$, while 
the critical points of $G$ are of indices $1,2$ or $3$.\\

\begin{defn}
Let $G:W\ra [a,b]$ be a Morse function on the cobordism $\Cobb=(W,\sur)$. 
A vector field $\xi$ on $W$ is called an 
{\emph{embedded gradient-like}} vector field for $G$ if it satisfies the following conditions:
\begin{enumerate}
\item For every $p\in W$ which is not in $\mathrm{Crit}_{\Cobb}(G)$, $dG_p(\xi_p)>0$.
\item The vector field $\xi$ is tangent to both $\sur$ and $\del_hW$.
\item For any critical point $p\in\mathrm{Crit}_W(G)$ there is an open neighborhood 
$U\subset W$ of $p$ together with a positively oriented local coordinate 
system $(x_1,x_2,x_3,x_4)$ centered at $p$ such that
\begin{displaymath}
\begin{split}
&G(x_1,x_2,x_3,x_4)=G(p)\pm x_1^2\pm x_2^2\pm x_3^2\pm x_4^2\quad \text{and}\quad
\xi(x_1,x_2,x_3,x_4)=(\pm x_1,\pm x_2,\pm x_3,\pm x_4)
\end{split}
\end{displaymath}
\item There is an open neighborhood 
$U\subset W$ of every critical point $p$ of $g$  and a positively oriented local coordinate system 
$(x_1,x_2,y_1,y_2)$ centered at $p$, such that
\begin{displaymath}
\begin{split}
& U\cap \sur=\{(x_1,x_2,y_1,y_2)\in U\ |\ y_1=y_2=0\},\\
& G(x_1,x_2,y_1,y_2)=G(p)\pm x_1^2\pm x_2^2+y_1\quad\text{and}\quad 
\xi=(\pm x_1,\pm x_2,y_1^2+y_2^2,0).
\end{split}
\end{displaymath}
\end{enumerate}
\end{defn}
\begin{defn}
A {\emph{Morse datum}} for $\Cobb$ is a triple $\Modi=(G,\ub,\xi)$, where 
\[\ub=(a=b_0<\cdots <b_{m}=b)\in\R^{m+1},\]
is an ordered  $(m+1)$-tuple of regular values for the proper Morse function 
$G:W\ra [a,b]$ on $\Cobb$,
and  $\xi$ is an embedded gradient like vector field for $G$. 
Over each interval $(b_i,b_{i+1})$, $G$ has at most one critical point in $W$, and if 
it has a critical point in $G^{-1}(b_i,b_{i+1})$ then $g$ has no critical point over this 
interval.\\
The Morse datum
$(G,\ub,\xi)$ is called {\emph{good}} if $g$ is indefinite and $g^{-1}(b_i)$ is 
a union of arcs connecting $\del_h^-W$ to $\del_h^+W$ for $i=0,\ldots,m$.   
\end{defn}

Any good Morse datum $\Modi=(G,\ub,\xi)$ induces a parametrized 
Cerf decomposition $\Ci(G,\ub,\xi)$ of $\Cobb$ by 
taking $W_i=G^{-1}[b_{i-1},b_{i}]$, $\sur_i=g^{-1}[b_{i-1},b_{i}]$ 
and $(M_i,T_i)=(G^{-1}(b_i),g^{-1}(b_i))$.  
If $G$ has a critical 
point $p\in G^{-1}[b_{i-1},b_{i}]$ of index $k_i$, the descending flow 
of $\xi$ maps $p$ to a $k_i-1$ dimensional embedded sphere $S_{i}\subset G^{-1}(b_{i-1})$, 
which is disjoint from $T_{i-1}$. Moreover, using a positive local coordinate 
system in the neighborhood of $p$ 
such that $G$ and $\xi$ have the standard structure, we obtain a framing for 
$S_i$. As in Remark 2.12 of  \cite{Juh-TQFT}, if $k_i=0,4$, the framed sphere 
$\Fsphere_i$ does not  depend on the choice of the local coordinates and is 
uniquely determined up to isotopy. Otherwise, for $k\neq 0,4$, depending on 
the positive local coordinate system, we 
obtain two non-isotopic embedded framed sphere $\Fsphere_i$ and $\ovl{\Fsphere}_i$. 
Note that for any framed 
sphere $\Fsphere:S^{k-1}\times D^{4-k}\ra M-T$ in a tangle $(M,T)$, $\ovl{\Fsphere}$ is 
defined by
$$\ovl{\Fsphere}(x,y)=\Fsphere(r_{k}(x),r_{4-k}(y)),$$
for $x\in S^{k-1}\subset \R^{k}$, $y\in D^{4-k}\subset \R^{4-k}$, where
$$r_k(x_1,x_2,\ldots,x_k)=(-x_1,x_2,\ldots,x_k).$$
Along with the framed sphere $\Fsphere_i$, we  obtain a diffeomorphism
\begin{displaymath}
d_i:(M_{i-1}(\Fsphere_i),T_{i-1})\lra (M_{i},T_{i})
\end{displaymath}
which is given using the flow of the vector field $\xi$ over 
$M_{i-1}\setminus \nd(\Fsphere_i)$, the complement
of the framed sphere $\Fsphere_i$. 
On the other hand, if $g$ has index one critical points 
$p_1,\ldots,p_{n_i}\in G^{-1}[b_{i-1},b_{i}]\cap \sur$, the descending manifolds of 
$p_1,\ldots,p_{n_i}$ under the 
flow of $\xi$ determines a set of orientation preserving framed arcs 
$\Farc_i=\{\Farc_1^i,\ldots,\Farc_{n_i}^i\}$ with end points on $T_{i-1}=g^{-1}(b_{i-1})$. 
Since $g^{-1}(b_{i})$ is a union of arcs connecting $\del^- M_{i}$ to $\del^+M_{i}$, 
the set $\Farc_i$ is acceptable.
The flow of $\xi$ gives a 
diffeomorphism 
\begin{displaymath}
d_i:(M_{i-1},T_{i-1}(\Farc_i))\lra (M_{i},T_{i}).
\end{displaymath}   
Furthermore, if $G$ has no critical points in $\Cobb_i$, we set $\Fsphere_i=\emptyset$ and the 
flow of $\xi$ defines a diffeomorphism $d_i$ from $(M_{i-1},T_{i-1})$ to $(M_{i},T_{i})$.\\

Therefore, any good Morse datum $\Modi=(G,\ub,\xi)$ defines a  parametrized Cerf 
decomposition for $\Cobb$ denoted by $\Ci(\Modi)$, well-defined up to replacing some of the 
framed spheres $\Fsphere$ by the corresponding framed spheres $\ovl{\Fsphere}$. Because of this ambiguity, we say two good Morse data $\Modi$ and $\Modi'$ induce the same parametrized Cerf decompositions for $\Cobb$ if and only if the corresponding attaching spheres and attaching arcs coincide, while they are isotopic as framed spheres or framed arcs after replacing some framed spheres $\Fsphere$ in $\Ci(\Modi)$ with $\ovl{\Fsphere}$.
%\todo{AA I added the above two sentences to clear the ambiguity of Lemma 4.7.}

\begin{lem}
Let $\Cobb$ be a stable cobordism. Then every parametrized Cerf decomposition 
for $\Cobb$ is induced by a good Morse datum.
\end{lem}
\begin{proof}
Suppose that $\Cobb$ is an elementary cobordism  from $(M,T)$ to $(M',T')$. 

{\bf Case 1.} $\Cobb$ is parametrized by a framed sphere $\Fsphere$ in $(M,T)$ and a 
diffeomorphism $d:(M(\Fsphere),T)\to(M',T')$. It follows from the proof of   \cite[Lemma 2.14]{Juh-TQFT}, 
that one can define a Morse function $G':W(\Fsphere)\ra [0,1]$ together with an 
embedded gradient-like vector field 
$\xi'$ on $\Cobb(\Fsphere)=(W(\Fsphere),\sur(\Fsphere))$ such that 
$G'$  has a single critical point, 
while $g'=G'|_{\sur(\Fsphere)}$ has no critical points. Furthermore, the diffeomorphism 
induced by $\xi'$ on $(M(\Fsphere),T)$  is 
$\mathrm{Id}_{M(\Fsphere)}$. For instance, if $\Fsphere=\emptyset$, then for the cobordism 
$\Cobb(\Fsphere)=(M,T)\times [0,1]$ we let $G'=\pi_2$ and $\xi'=\del_t$. Then, 
$G=G'\circ D^{-1}$ and  $\xi=D_\star\xi'$ is the corresponding good Morse datum, 
where $D:\Cobb(\Fsphere)\to\Cobb$ is a diffeomorphism such that 
$D|_{(M,T)}=\mathrm{Id}$ and $D|_{(M(\Fsphere),T)}=d$.

{\bf Case 2.} $\Cobb$ is parametrized by an acceptable set of framed arcs $\Farc=\{\Farc_1,\ldots,\Farc_{n}\}$, and a diffeomorphism $d:(M,T(\Farc))\to(M',T')$. 
%Let 
%$N$ be a small open neighborhood of $D^2\times D^1\subset \R^3$ and 
%$H=N\times D^1$ and $B\subset H$ be defined as before.
For any $i$ consider a small neighborhood $N_i$ of
$\Farc_i$ in $(M,T)$. Let 
$$(W_{i},\sur_{i}):=\frac{\left(N_i,N_i\cap T\right)\times [0,1]
\coprod (H_i,B_i)}{(\Farc_i\times\{1\},(T\cap \Farc_i)\times \{1\})\sim \del_{-1}(H_i,B_i)}\subset \Cobb(\Farc).$$
Here $(H_i,B_i)$ denotes a copy of the standard saddle $(H,B)$ for 
$i=1,\ldots,n$. Let $N$ be a small neighborhood of $D^2\times D^1$ in $\R^3$ and $\ti{B}$ denote the corresponding saddle in $N\times D^1$.  After smoothing the corners obtained from attaching the 
saddles, $(W_{i},\sur_{i})$ is diffeomorphic with $(N\times D^1,\ti{B})$. Choose 
a diffeomorphism 
\[\phi_i:(W_{i},\sur_{i})\lra (N\times D^1,\ti{B})\] such that
\begin{align*}
&\phi_i\big(\del_1(H_i,B_i)\big)=
\big(D^2\times D^1\times \{1\},\ti{B}\cap(D^2\times D^1\times\{1\})\big)\ \ \text{and}\\
&\phi_i\big(N_i\times\{0\},(N_i\cap T)\times\{0\}\big)=
\big(N\times \{-1\},(\ti{B}\cap N)\times\{-1\}\big).
\end{align*}
Moreover, for a sufficiently small $\nu\in\R^+$ and every $t\in (0,1)$, 
if $(x,y,z)\in N$ belongs to the  
$\nu$-neighborhood of $\del N$, then 
\[\phi^{-1}(x,y,z,2t-1)\in N_i\times \{t\}.\] 
We define $G_i'$ by
\begin{align*}
G_i'(y):=\begin{cases}
\pi_2(y)\ &\text{if }\ y\in M\times [0,1]-\amalg_{i=1}^{n}W_{i}\\
\frac{\pi_2(\phi_i(y))+1}{2}\ &\text{if }\ y\in W_{i}
\end{cases}
\end{align*} 
This is a smooth function, by construction. We define the embedded gradient-like vector 
field $\xi'$ on $\Cobb(\Farc)$ by pulling back the vector field $2\del_t$ on 
$N\times D^1$ using  $\phi_i$ and extending it to the rest of 
$M\times [0,1]$ using $\del_t$.  It is now straightforward 
to check that the Morse function $G'$ together with the embedded gradient-like vector field 
$\xi'$ induces the identity diffeomorphism on $(M,T(\Farc))$. Hence, considering a diffeomorphism $D:\Cobb(\Farc)\to\Cobb$ for which $D|_{(M,T)}=\mathrm{Id}$ and $D|_{(M,T(\Farc))}=d$, the Morse function $G=G'\circ D^{-1}$ and the gradient-like vector field $\xi=D_{\star}\xi'$ is the corresponding Morse datum.

If $\Cobb$ is not an elementary cobordism, consider a parametrized Cerf decomposition 
\[\Cobb=\Cobb_1\cup_{(M_1,T_1)}\cdots \cup_{(M_{m-1},T_{m-1})}\Cobb_m\] for $\Cobb$. 
For each $\Cobb_i$ denote the corresponding Morse datum constructed as above by $(G_i,\xi_i)$.
Fix an $(m+1)$-tuple $\ub=(a=b_0<b_1<\cdots <b_m=b)$ of real numbers, and let 
$a_i:[0,1]\to [b_{i-1},b_i]$ be the diffeomorphism $a_i(t)=(1-t)b_{i-1}+tb_{i}$. 
Then, we can modify the Morse functions $a_i\circ G_i$ and the gradient-like vector fields $\xi_i$, on a collar 
neighborhood of $(M_i,T_i)$ such that they fit together to give a Morse function $G$ 
and an embedded gradient-like vector field $\xi$ on $\Cobb$. 
% 
%Let $G_i=a_i\circ G_i'\circ D_i^{-1}$ where $a_i:[0,1]\ra [b_{i-1},b_{i}]$ is the diffeomorphism 
%$a_i(t)=(1-t)b_{i-1}+tb_{i}$ and $D_i$ is the corresponding diffeomorphism from 
%$\Cobb(\Fsphere_i)$ or $\Cobb(\Farc_i)$ to $\Cobb_i$, which is equal to identity on 
%$(M_{i-1},T_{i-1})$ and to $d_i$ on $(M_{i-1}(\Fsphere_i),T_{i-1})$ or 
%$(M_{i-1},T_{i-1}(\Farc_{i}))$. 
%We can modify $G_i$ and $\xi=(D_i)_{\star}(\xi_i')$ on a collar 
%neighborhood of $(M_i,T_i)$ such that they fit together to give a Morse function $G$ 
%and an embedded gradient-like vector field $\xi$ on $\Cobb$. 
\end{proof}

\begin{lem} 
%\todo{EE: The statement is not clear.  AA Is the new statement better?}
Suppose $\Modi=(G,\ub,\xi)$ and $\Modi'=(G',\ub',\xi')$ are two good Morse data for a stable cobordism $\Cobb=(W,\sur)$ that induce the same parametrized Cerf decompositions. There exist diffeomorphisms $D:\Cobb\ra \Cobb$ and $\phi:\R\ra\R$ 
satisfying the followings:
\begin{enumerate}
\item $\ub'=\phi(\ub)$
\item $G'=\phi\circ G\circ D^{-1}$
\item for some $h\in C^{\infty}(W,\R^+)$ we have $D_*\xi=h.\xi'$
\item $D|_{M}=\mathrm{Id}_{M}$ and  $D|_{M'}=\mathrm{Id}_{M'}$.
\end{enumerate}
\end{lem}
\begin{proof}
See proof of \cite[Lemma 2.14]{Juh-TQFT}.
\end{proof}
\begin{prop}
For every stable cobordism  $\Cobb$  
there exists a good Morse datum $\Modi=(G,\ub,\xi)$, and 
therefore a parametrized Cerf decomposition.
\end{prop}
\begin{proof}
Let $\Cobb=(W,F)$ be a stable cobordism from $(M,T)$ to $(M',T')$ and let 
$g:\sur\ra [0,1]$ be a Morse function satisfying
$$g|_{T}\equiv 0,\ \ \ g|_{T'}\equiv 1,\ \ \ \text{and}\ \ \ g|_{\sur\cap \del_hW}=\pi_2.$$ 
Since every connected component of $\sur$ has non-empty intersection with either of 
$T$ and $T'$, we may assume that $g$ has no minimal or maximal index critical point i.e. all critical points of $g$ have index $1$. 
Further, we assume that it has the same value over all of its critical points. 
We extend $g$ to a tubular 
neighborhood  of $\sur$ such that it has no critical points and then extend it to a Morse function 
$G:W\ra [0,1]$. By a small perturbation in a neighborhood of the critical points of $G$, 
we may assume that $G$ has distinct critical values and these values  
are distinct from the critical value of $g$. 
Finally, %\todo{EE: Is this true? Shouldn't it be that the perturbation is supported in a neighborhood of F?} 
by a small perturbation in a neighborhood of the critical points of $g$ we can transform 
$G$ into a proper Morse function on $\Cobb$, such that for every $p,q\in\mathrm{Crit}(g)$ 
with $G(p)<G(q)$, $G$ has no critical point on $W$ above the interval  $(G(p),G(q))$. 
It is straightforward to show that there is an embedded 
gradient-like vector field for every Morse function on $\Cobb$.  Let $\xi$ be such a  
gradient-like vector 
field for $G$. Choose the ordered set 
\[\ub=\{0=b_0<b_1<\cdots <b_{m}=1\}\subset [0,1]\] 
of  regular  values for  $G$ such that for every $i=1,\ldots,m$, 
$\mathrm{Crit}_W(G)$ has at most one element in 
$G^{-1}[b_{i-1},b_{i}]$. Furthermore,  for exactly one 
$1\le j\le m$ all critical values of $g$ lie in $(b_{j-1},b_j)$, while $\mathrm{Crit}_W(G)$ does not intersect 
$G^{-1}[b_{j-1},b_{j}]$. The triple $(G,\ub,\xi)$ is a good Morse datum for $\Cobb$ and  it thus gives 
a parametrized Cerf decomposition of $\Cobb$.
\end{proof}
%%%%%%%%%%%%%%%%%%%%%%%%%%%%%%
\subsection{Cerf moves}
In this section, we describe {\emph{Cerf moves}} 
on Morse data and discuss how corresponding parametrized Cerf decompositions change under these moves. For this purpose, fix a cobordism $\Cobb=(W,F)$ from $(M,T)$ to $(M',T')$. Let $\Modi=(G,\ub,\xi)$ and $\Modi'=(G',\ub',\xi')$ be parametrized Morse data for $\Cobb$ so that $G,G':W\ra [a,b]$ are proper Morse functions on $\Cobb$.\\

{\bf Critical point cancellation/creation.} 
The Morse data $\Modi$ and $\Modi'$ are related by a {\emph{critical point cancellation}} if
\begin{enumerate}
\item $G$ and $G'$ are related by a critical point cancellation i.e. there exist a  family of smooth functions 
$\{G_t:W\ra [a,b]|t\in [-1,1]\}$ such that $G_{-1}=G$,  $G_1=G'$ and:
\begin{itemize}
\item For $t\in [-1,1]-\{0\}$,  $G_t$ is a proper Morse function on $\Cobb$ and $G_0$ 
has a \emph{death} singularity at some point $p_\circ\in W$,
\item The family is an elementary death ('chemin \'el\'ementaire de mort' in the sense 
of Cerf \cite[Section 2.3, p.71]{Cerf})  with support in a neighborhood $U$ of $p_\circ$. 
More precisely, $G_t$ does not depend on $t$ outside $U$ and for $t\in (0,1]$ it has no 
critical points in $U$. Furthermore, there exist  local coordinates $(x_1,\ldots,x_4)$ around
$p_\circ$, such that for $t\in [-1,1]$
$$G_t(x_1,\ldots,x_4)=G_0(p_\circ)+x_1^3+tx_1-x_2^2\cdots -x_k^2+x_{k+1}^2+\cdots +x_4^2$$
\end{itemize}

\item Assume $\ub=(a=b_0<b_1<\cdots <b_m=b)$. There is some $1\le j\le m-1$ such that $b_{j}=G_0(p_\circ)$, $\ub'=\ub-\{b_{j}\}$ while 
$$b_{j-1}<G_0(p_\circ)-2/3\sqrt{3}\ \ \ \text{and}\ \ \ b_{j+1}>G_0(p_\circ)+2/3\sqrt{3}.$$ 
\item Let $p$ and $q$ be the critical points of $G$ where $b_{j-1}<G(p)<b_j<G(q)<b_{j+1}$ 
and canceled against each other at $t=0$. The stable and unstable submanifolds $W^s(q)$ and $W^u(p)$ are 
transverse and intersect in a single flow line. Moreover, the neighborhood $U$ is in fact a 
neighborhood of 
$$(W^{u}(p)\cup W^s(q))\cap G^{-1}[b_{j-1},b_{j+1}],$$
and $\xi'$ coincide with $\xi$ outside $U$. 
\end{enumerate}
A {\emph{critical point creation}} is the reverse of a critical point cancellation. 

Suppose that the good Morse data $\Modi'$ is obtained from $\Modi$ by a critical point 
cancellation as above. The induced parametrized Cerf decomposition 
$\Ci(\Modi'):\Cobb'_1\cup\cdots \cup\Cobb'_{m-1}$ is related to 
$\Ci(\Modi):\Cobb_1\cup\cdots \cup\Cobb_m$ as follows. For any $i<j-1$ 
the parametrized elementary cobordisms $\Cobb'_i$ coincides with 
$\Cobb_{i}$, while for $i\ge j$ it coincides with $\Cobb_{i+1}$. 
Further, let $b(\Fsphere_j)$ in $(M_{j-1}(\Fsphere_j),T_{j-1})$ 
denote the belt sphere of the attached handle to $\Fsphere_j$. 
The framed sphere $\ti{\Fsphere}_{j+1}:=d_j^{-1}(\Fsphere_{j+1})$ 
intersects $b(\Fsphere_{j})$ in a single point, thus there is a 
diffeomorphism 
\[\phi:(M_{j-1},T_{j-1})\ra (M_{j-1}(\Fsphere_j)(\ti{\Fsphere}_{j+1}),T_{j-1})\] 
which is unique up to isotopy and fixes $(M_{j-1},T_{j-1})\cap 
(M_{j-1}(\Fsphere_j)(\ti{\Fsphere}_{j+1}),T_{j-1})$ 
(See \cite{Cerf}, \cite[Definition 2.17]{Juh-TQFT} and \cite[Theorem 5.4]{Mi}). 
Then, $\Cobb'_j$ is a cobordism from $(M'_{j-1},T'_{j-1})=(M_{j-1},T_{j-1})$ to 
$(M'_j,T'_j)=(M_{j+1},T_{j+1})$ parametrized by the  framed sphere 
$\Fsphere_j'=\emptyset$ and the diffeomorphism $d_j'$, which is 
isotopic to $d_{j+1}\circ \ti{d}_j\circ\phi$. Here,
\[\ti{d}_j:(M_{j-1}(\Fsphere_j)(\ti{\Fsphere}_{j+1}),T_{j-1})\lra 
(M_{j}(\Fsphere_{j+1}),T_{j})\]
is the diffeomorphism induced by $d_j$. For more details, see \cite[Lemma 2.15]{Juh-TQFT}.

{\bf Critical point switches.} The Morse data $\Modi$ and $\Modi'$ are related by 
{\emph{critical point switch}} if $\xi=\xi'$,  $\ub\setminus b_j=\ub'\setminus b_j'$ 
for some $j$, and $G$ is connected to $G'$ by a smooth 
family $\{G_t:W\ra [a,b]|{t\in [-1,1]}\}$ of Morse functions which
are proper for all but finitely many values of $t$ in $[-1,1]$ and $\mathrm{Crit}_{\Cobb}(G_t)$ is independent of $t$.  
Furthermore, depending on the 
type of critical point switch, the family $\{G_t\}$ satisfies one of the  
followings:

{\bf Type I.}   For critical points $ p,q\in \mathrm{Crit}_W(G)$ we have 
\[b_{j-1}<G(p)<b_{j}<G(q)<b_{j+1}\ \ \ \text{and}\ \ \ b_{j-1}<G'(q)<b'_{j}<G'(p)<b_{j+1}.\]

Then $G_0(p)=G_0(q)$ while $tG_t(p)>tG_t(q)$ 
for $t\neq 0$.  Further, the family $\{G_t\}_{t\in [-1,1]}$ is an elementary 
upward or downward switch ('chemin \'el\'ementaire de 
croisement, ascendant or descendente' in the sense of Cerf \cite[Chapter II, p.40]{Cerf})
in a neighborhood $U$ of 
\begin{align*}
&W_p^s(q):=W^s(q)\cap G^{-1} \left([G(p),G(q)] \right)\ \ \ \ \text{or}\ \ \ \ \
W^u_q(p):=W^u(p)\cap G^{-1}\left([G(p),G(q)]\right).
\end{align*}
In particular, $G_t$ is independent of $t$ outside $U$ while $G_t-G$ is constant in an open 
neighborhood containing the critical points inside $U$.

{\bf Type II.} For critical points $p\in\mathrm{Crit}_W(G)$ and $q_1,\ldots,q_{n}\in\mathrm{Crit}(g)$ we have 
\[b_{j-1}<G(p)<b_{j}<G(q_1)<\cdots <G(q_{n})<b_{j+1}\]
while $b_{j-1}<G'(q_1)<\cdots <G'(q_{n})<b_{j}<G'(p)<b_{j+1}$. Then 
$G_t(q_1)<\cdots <G_t(q_{n})<G_t(p)$ for some $\delta>0$ and every
$1-\delta<t\le 1$. 
Furthermore, the family $\{G_t\}_{t\in [-1,1]}$ is an elementary switch 
('chemin \'el\'ementaire de croisement' in the sense of  Cerf \cite[Chapter II, p.40]{Cerf})
with support in a neighborhood $U$ of
$$W^u_q(p):=W^u(p)\cap G^{-1}([G(p),G(q_{n})])$$
and $G_t|_{W-U}$ is independent of $t$.

{\bf Type III.} There are critical points $p,q\in\mathrm{Crit}(g)$ so that 
$b_{j-1}<g(p)<g(q)<b_{j}$, $g$ has no critical value in 
$(g(p),g(q))$ and $b_{j-1}<g'(q)<g'(p)<b_{j}$. In this case $G_0(p)=G_0(q)$, while $tG_t(p)>tG_t(q)$ for $t\neq 0$. 
Moreover, for a neighborhood $U$ of 
$$W^s(q)\cap G^{-1}([G(p),G(q)])\ \ \ \ \text{or}\ \ \ \ 
W^u(p)\cap G^{-1}([G(p),G(q)]),$$
$G_t$ is independent of $t$ in $W-U$. Furthermore, in a neighborhood $V\subset U$ of 
$p$ and $q$, $G_t-G$ is constant.

If the good Morse data $\Modi$ and $\Modi'$ are related by a critical point 
switch of type III, then it is straightforward to see that $\Ci(\Modi)$ and $\Ci(\Modi')$ 
are the same. However, if they are related by a critical point switch of type I or II, 
then $\Ci(\Modi)$ and $\Ci(\Modi')$ are related as in Lemma 2.16 of  \cite{Juh-TQFT}. 
Let us recall the statement of the aforementioned  Lemma (with a small modification)
for a critical point switch of type II. 

\begin{lem} (Critical point switch of type II) With the above notation fixed, if 
$\Modi=(G,\ub,\xi)$ and $\Modi'=(G',\ub',\xi')$ are related by a critical point switch of type II,
then the parametrized elementary cobordism $\Cobb_i$ coincides with 
$\Cobb_i'$ for any $i<j-1$ and $i>j$. Moreover, 
\begin{itemize}
\item $\Farc_{j+1}\cap d_j(b(\Fsphere_j))=\emptyset$ where $b(\Fsphere_j)$ denotes the belt sphere of the attached handle to $\Fsphere_j$,
\item $d_j(\Farc_j')=\Farc_{j+1}$ and $d_j'(\Fsphere_j)=\Fsphere_{j+1}'$,
\item The diagram 
\begin{diagram}
(M_{j-1}(\Fsphere_j),T_{j-1}(\Farc_j'))&\rTo{(d_j)^{\Farc_j'}}&(M_{j},T_{j}(\Farc_{j+1}))\\
\dTo{(d_j')^{\Fsphere_j}}&&\dTo{d_{j+1}}\\
(M_{j}'(\Fsphere_{j+1}'),T_{j}')&\rTo{d_{j+1}'}&(M_{j+1},T_{j+1}).
\end{diagram}
is commutative. Here, $(d_j)^{\Farc_j'}$ and $(d_{j}')^{\Fsphere_{j}}$ are induced by $d_{j}$ 
and $d_j'$, respectively.
\end{itemize}
\end{lem}

\begin{proof}
See proof of Lemma 2.16 in \cite{Juh-TQFT}.
\end{proof}

{\bf Isotopy on embedded gradient-like vector field.} We say that the Morse data 
$\Modi=(G,\ub,\xi)$ and $\Modi'=(G',\ub',\xi')$ are related by doing 
{\emph{isotopy on the embedded gradient-like vector field}}, if $G=G'$ and $\ub=\ub'$. 

If the good Morse data $\Modi$ and $\Modi'$ are related by doing isotopy on the 
embedded-gradient like vector fields, their corresponding parametrized Cerf 
decompositions,  possibly after  reversing some of the framed spheres, are 
related by ambient isotopies \cite[Remark 2.11]{Juh-TQFT}. More precisely, 
for any $j$ there is an ambient isotopy 
$\{\phi_t\}_{t\in[0,1]}$ of $(M_{j-1},T_{j-1})$ with $\phi_0=\mathrm{Id}_{M_{j-1}}$, 
so that if $\Cobb_j$ is parametrized by $(\Fsphere_j,d_j)$ (or $(\Farc_j,d_j)$) 
then $\phi_1(\Fsphere_j)=\Fsphere_j'$ and $d_j'=d_j\circ (\phi_1')^{-1}$ 
($\phi_1(\Farc_j)=\Farc_j'$ and $d_j'=d_j\circ (\phi_1')^{-1}$). Here, $\phi_1'$ 
is the diffeomorphism induced by $\phi_1$.
%\begin{lem}
%If $\Modi$ is related to $\Modi'$ by doing isotopy on the embedded gradient-like vector 
%field, then the induced parametrized Cerf decompositions, 
%after possibly reversing some of the framed spheres, 
% are related as follows. For every $j\in\{1,\ldots,m\}$, there is an ambient isotopy 
% $\{\phi_t\}_{t\in[0,1]}$ of $(M_{j-1},T_{j-1})$ with $\phi_0=\mathrm{Id}_{M_{j-1}}$,
% such that 
%\begin{itemize}
%\item If $G$ has one critical point on $W_j$, then $\phi_1(\Fsphere_j)=\Fsphere_j'$ and $d_j'=d_j\circ (\phi_1')^{-1}$
%where 
%%$(\Fsphere_j,d_j)$ and $(\Fsphere_j',d_j')$ are the induced parametrizations of 
%%$(W_j,\sur_j)$ and $(W_j',\sur_j')$ and 
%\[\phi_1':(M_{j-1}(\Fsphere_j),T_{j-1})\ra (M_{j-1}(\Fsphere_j'),T_{j-1})\] 
%is the diffeomorphism induced by $\phi_1$.
%
%\item If $g$ has critical points in $\sur_j$, then $\phi_1(\Farc_j)=\Farc_j'$ and $d_j'=d_j\circ (\phi_1')^{-1}$
%where
%% $(\Farc_j,d_j)$ and $(\Farc_j',d_j')$ are the induced parametrizations of 
%%$(W_j,\sur_j)$ and $(W_j',\sur_j')$ and 
%\[\phi_1':(M_{j-1},T_{j-1}(\Farc_j))\ra (M_{j-1},T_{j-1}(\Farc_j'))\] 
%is the diffeomorphism induced by $\phi_1$.

%\end{itemize}
%\end{lem}

%\begin{proof}
%Remark 2.11 of \cite{Juh-TQFT}.
%\end{proof}

{\bf{Adding/removing regular values.}}  The Morse data $\Modi=(G,\ub,\xi)$ 
and $\Modi'=(G',\ub',\xi')$ are related by 
{\emph{adding or removing regular values}}, if $G=G'$ and $\xi=\xi'$. 
Thus, $\Modi(\ub\cup\ub')=(G,\ub\cup\ub',\xi)$ is a Morse datum for $\Cobb$ obtained from 
$\Modi$ and $\Modi'$ by adding regular values. 
If both $\Modi$ and $\Modi'$ are good, the induced parametrized 
Cerf decomposition $\Ci(\Modi(\ub\cup\ub'))$ is obtained from $\Ci(\Modi)$ and 
$\Ci(\Modi')$ by {\emph{splitting}} product cobordisms and cobordisms parametrized by  
acceptable sets of framed arcs.
\begin{defn}
We say that a parametrized Cerf decomposition $\Ci'$ is obtained from $\Ci$ by a 
{\emph{splitting}}, if there is some $j$ such that for any $i<j$ the parametrized elementary 
cobordism $\Cobb_i'$ coincides with $\Cobb_i$ while for any $i>j+1$ it coincides with 
$\Cobb_{i-1}$. Furthermore, the cobordism $\Cobb_j$ splits as 
$\Cobb'_j\cup_{(M'_{j},T'_{j})}\Cobb'_{j+1}$ in $\Ci'$, such that one of 
$\Cobb'_j$ or $\Cobb'_{j+1}$, say $\Cobb_j'$, is either parametrized by 
$(\Fsphere_j'=\emptyset,d_j')$ or $(\Farc_j',d_j')$. If $\Cobb_j'$ is a product then, 
depending on the types of $\Cobb'_{j+1}$ and $\Cobb_j$, we have either 
\begin{align*}
&\Fsphere_j=d_j'^{-1}(\Fsphere'_{j+1})\ \ \ \text{and} \ \ \ 
d_j=d_{j+1}'\circ (d_j')^{\Fsphere_j},\ \ \  \text{or}\\
&\Farc_{j}=d_j'^{-1}(\Farc_{j+1}')\ \ \  \ \text{and} \ \ \ d_j=d'_{j+1}\circ (d'_j)^{\Farc_j}. 
\end{align*}
If $\Cobb_j'$ is parametrized by $(\Farc_j',d_j')$ then $\Cobb_{j+1}'$ is parametrized 
by $(\Farc_{j+1}',d_{j+1}')$. Moreover, 
\[\Farc_j=\Farc_j'\amalg d_j'^{-1}(\Farc_{j+1}')\ \ \ \text{and}\ \ \  
d_j=d_{j+1}'\circ (d_j')^{\Farc_j-\Farc_j'}.\]
The reverse of this move, is called {\emph{merging}}.
\end{defn} 
Therefore, for any two good Morse data $\Modi=(G,\ub,\xi)$ and $\Modi=(G,\ub',\xi)$ 
we may change $\Ci(\Modi)$ to $\Ci(\Modi')$ by first splitting and then merging.\\

{\bf{Left-right equivalence.}} 
Let $\Cobb=(W,\sur)$ be a cobordism from $(M,T)$ to $(M',T')$. 
We say that the Morse functions $G$ and $G'$ on $\Cobb$ are related by a 
{\emph{left-right equivalence}} if there are diffeomorphisms $\Phi:W\ra W$ and $\phi:\R\ra\R$
such that $\Phi|_{(M,T)}=\mathrm{Id}_{(M,T)}$, $\Phi|_{(M',T')}=\mathrm{Id}_{(M',T')}$, and 
$G'=\phi\circ G\circ\Phi^{-1}$. Moreover, we say that the good Morse data $\Modi=(G,\ub,\xi)$ 
and $\Modi'=(G',\ub',\xi')$ are related by a 
{\emph{left-right equivalence}}, if $G$ and $G'$ are related by a left-right equivalence and 
under the corresponding diffeomorphisms $\ub'=\phi\circ\ub$ and $\xi'=\Phi_{\star}(\xi)$. 
In this case, the parametrized Cerf decomposition $\Ci(\Modi')$ is obtained from 
$\Ci(\Modi)$ by a {\emph{diffeomorphism equivalence}}. This means that  
$\Cobb_i'=\Phi(\Cobb_i)$ as parametrized elementary cobordisms i.e. 
depending on the type of $\Cobb_i$, 
$\Fsphere_i'=\Phi(\Fsphere_i)$ or $\Farc_i'=\Phi(\Farc_i)$ and 
$d_i'=\Phi_{i}\circ d_i\circ (\Phi_{i-1}')^{-1}$, where $\Phi_i=\Phi|_{(M_i,T_i)}$ and $\Phi_i'$
is the map induced by $\Phi_i$  on $(M_i(\Fsphere_{i+1}),T_i)$ or on $(M_i,T_i(\Farc_{i+1}))$.

\subsection{Parametrized Cerf decomposition theorem} 
The goal of this subsection is to show that any two good Morse data associated with a stable 
cobordism can be related by a sequence of Cerf moves. As before, let $\Modi=(G,\ub,\xi)$ 
and $\Modi'=(G',\ub',\xi')$ be good Morse data for $\Cobb$ so that $G,G':W\ra [a,b]$ are proper.

\begin{lem}\label{lem:simple-path} If there exists a smooth family $\{G_t\}_{t\in [0,1]}$ of 
proper Morse functions on $\Cobb$ with $G_0=G$ and 
$G_1=G'$ then $G$ is related to $G'$ by a left-right equivalence.
\end{lem}
\begin{proof} The proof  is similar to the proof of Lemma 3.1 in \cite{GWW} with 
minor modifications.  For every $t\in [0,1]$, $G_t$ is proper.  Therefore, we have a 
smooth family $\{\phi_t:\R\ra \R\}_{t\in [0,1]}$ of diffeomorphisms of 
$\R$ such that $\phi_0=\mathrm{Id}$ and for 
every $t\in [0,1]$ , $\phi_t^{-1}\circ G_t$ has the same critical 
values as $G$. Moreover, if for $p\in \mathrm{Crit}_{\Cobb}(G)$ and 
$p_t\in \mathrm{Crit}_{\Cobb}(G_t)$ we have  $\phi_t^{-1} G_t(p_t)=G(p)$ then 
either both $p$ and $p_t$ are critical points in $W$ or both are critical points in $\sur$. 

Let $G_t'=\phi_t^{-1}\circ G_t$.
Consider a point $(x_0,t_0)\in W\times [0,1]$ such that 
$x_0\in\mathrm{Crit}_W(G_{t_0}')$. For sufficiently 
small values of $\epsilon>0$ and $\delta>0$, 
we can find the local coordinates $\theta_{t}:B_{\epsilon}\subset \R^{4}\ra W$ for every 
$t\in (t_0-\delta,t_0+\delta)$ such that $G_t'\circ\theta_t$ takes the normal form  
$$G_{t}'\circ\theta_t(x_1,\ldots,x_4)=G_{t_0}'(x_0)+\sum \pm x_i^2.$$ 
Similarly, if $x_0\in\mathrm{Crit}(g_t)$, for sufficiently small $\epsilon,\delta>0$,  we can find local coordinates  
$\theta_{t}:B_{\epsilon}\subset \R^{4}\ra W$ for $t\in (t_0-\delta,t_0+\delta)$, 
so that $\sur$ is given by $\{x_3=x_4=0\}$  in these coordinates and 
$$G_t'\circ\theta_t(x_1,\ldots,x_4)=G_{t_0}'(x_0)\pm x_1^2\pm x_2^2+x_3.$$  
Then, for any $(x,t)$ in a neighborhood of $(x_0,t_0)$ defined as above, let 
\[v(x,t):=\left(\frac{d}{ds}\theta_{t+s}(\theta_t^{-1}(x))\big|_{s=0},1\right).\]
It is straightforward to show that $\{G_t'\}_{t\in[0,1]}$ 
is constant along the flow lines of  $v(x,t)$ in this 
neighborhood.  Consider a finite set of pairs $\{(U_i,{v}_i)|i=1,\ldots,n\}$, where each 
pair consists of an  open neighborhoods $U_i\subset W\times [0,1]$ as above and 
the corresponding vector field ${v}_i$, 
such that 
\begin{align*}
\left\{(x,t)\in W\times [0,1]\ \big|\ x\in\mathrm{Crit}_{\Cobb}(G_t')\right\}
\subset \bigcup_{i=1}^nU_i.
\end{align*} 
Consider an open set $U_0\subset W\times I$ in the complement of the critical points 
such that  $\bigcup_{i=0}^nU_i$ covers $W\times [0,1]$. We define vector field $v_0$ on $U_0$ as 
$$v_0(x,t):=(-(\del_tG_t')(dG_t'(\xi_t))^{-1}\xi_t,1).$$
Here, $\{\xi_t\}_{t\in [0,1]}$ is a smooth family of vector fields on 
$W$ such that for every $t\in [0,1]$, $\xi_t$ is an embedded gradient-like vector 
field for $G_t'$. Note that  $\{G_t'\}_{t\in[0,1]}$ remain constant along the flow lines of ${v_0}$. 
Thus, we may patch the above local vector fields and construct a global vector field $v$ 
on $W\times[0,1]$ such that $\{G_t'\}_{t\in[0,1]}$ remain constant along its 
flow lines. Hence, $G_1'\circ \Phi_1=G$, where $\Phi_1$ is the time-one map of the 
flow of $v$. Therefore, $\phi_1^{-1}\circ G'\circ \Phi_1=G$.
\end{proof}

\begin{defn}
Given a stable cobordism $\Cobb=(W,\sur)$,  a proper Morse function $G$ on 
$\Cobb$ is called {\emph{almost ordered}} if
\begin{enumerate}
\item $G$ is ordered as a Morse function on 
$W$, i.e. for any $p,q\in \mathrm{Crit}_W(G)$, $\mathrm{ind}(p)<\mathrm{ind}(q)$ 
implies $G(p)<G(q)$.
\item  For every $p\in\mathrm{Crit}_W(G)$ with $\mathrm{ind}(p)<2$, 
$G(p)$ is smaller than the critical values of $g$, while for $p\in\mathrm{Crit}_W(G)$ 
with  $\mathrm{ind}(p)>2$,
$G(p)$ is greater than the critical values of $g$.
\end{enumerate}
The Morse function $G$ on $\Cobb$ is 
called  {\emph{ordered}} if $G$ is almost ordered and for every $p\in \mathrm{Crit}_W(G)$ 
with $\mathrm{ind}(p)=2$, $G(p)$ is greater than the critical values of $g$. 
A good Morse datum $\Modi=(G,\ub,\xi)$ for $\Cobb$ is called \emph{almost ordered} 
if $G$ is almost ordered and it is called \emph{ordered} if $G$ is ordered.
\end{defn}

\begin{lem}\label{lem:putorder}
Any good Morse datum $\Modi=(G,\ub,\xi)$ for $\Cobb=(W,\sur)$ 
can be connected by a sequence of Cerf moves 
to an ordered good Morse datum $\Modi'=(G',\ub',\xi')$ such that 
under these moves, the Morse function remains constant in an 
open neighborhood of $\sur$. \end{lem}

\begin{proof} Consider a consecutive pair of critical points $p,q\in\mathrm{Crit}_{\Cobb}(G)$ with wrong order, say  
$\mathrm{ind}(p)>\mathrm{ind}(q)$ while $G(p)<G(q)$.
After removing the extra regular values we may assume $(G(p),G(q))$ contains exactly one $b_j\in\ub$. 
%In order to prove this lemma, it is enough to show that for every consecutive pair of critical points 
%$p,q\in\mathrm{Crit}_{\Cobb}(G)$ with wrong order, say  
%$G(p)<G(q)$,  the Morse datum $\Modi=(G,\ub,\xi)$ can be connected by Cerf moves to a 
%Morse datum $\Modi'=(G',\ub',\xi')$ for which $G'(q)<G'(p)$ 
%and $G'|_U=G|_{U}$ where $U$ is a neighborhood of $\sur$ . 
Depending on the type of the critical points $p$ and $q$ one 
of the followings hold:

{\bf Case 1.} $p,q\in\mathrm{Crit}_W(G)$ and $\mathrm{ind}(p)>\mathrm{ind}(q)$. 
By a dimension count one can show that for a generic embedded 
gradient-like vector field $\xi$, 
\[W^{u}(p)\cap W^{s}(q)\cap G^{-1}(b_j)=\emptyset.\] 
Since $\xi$ is gradient-like, for some $\delta>0$, the submanifolds $W^u(p)$ and 
$W^s(q)$ are disjoint from $\sur$ in $G^{-1}\left((G(p)-\delta,G(q)+\delta)\right)$.
Therefore, we have  an elementary switch 
%a 'chemin \'el\'ementaire de croisnement' 
supported in a neighborhood of 
\[W^{u}(p)\cap G^{-1}\left((G(p)-\delta,G(q)+\delta)\right)\ \ 
\text{or}\ \ W^{s}(q)\cap G^{-1}\left((G(p)-\delta,G(q)+\delta)\right),\]
and disjoint from $\sur$, that connects $G$ to a proper Morse function $G'$ satisfying $G'(q)<G'(p)$. 
Furthermore, in a small neighborhood $U$ of $\sur$ we have $G|_U=G'|_U$. 
Let $\xi'=\xi$ and pick $\ub'$ such that $\ub'-b_{j}'=\ub-b_{j}$ and $G'(q)<b_j'<G'(p)$. 
Then the resulted good Morse datum $\Modi'=(G',\ub',\xi')$ satisfies the required 
conditions.

{\bf Case 2.} $p\in\mathrm{Crit}(g)$ and $q\in\mathrm{Crit}_W(G)$, with  $\mathrm{ind}(q)\le 1$. 
Let
\[\{p_1,\ldots,p_{n_j}=p\}\subset \mathrm{Crit}(g)\] 
be the set of critical points of $g$ such that 
\[b_{j-1}<g(p_1)<\cdots <g(p_{n_j})<b_j.\] 
Similar to Case 1, after changing $\xi$ to a generic embedded gradient-like 
vector field and by a dimension count, one can assume that 
$$(\bigcup_{i=1}^{n_j}W^u(p_i))\cap W^{s}(q)\cap G^{-1}(b_j)=\emptyset.$$
Moreover, for some $\delta>0$, $W^s(q)$ is disjoint from $\sur$ in $G^{-1}\left((G(p_1)-\delta,G(q))\right)$.
Thus, as before, there is an elementary critical point switch,
%a 'chemin \'el\'ementaire de croisnement' 
supported in a neighborhood of 
\[W^{s}(q)\cap G^{-1}(G(p_1)-\delta,G(q)),\] 
and disjoint from $\sur$,
which changes the order of $q$ and $\{p_1,\ldots,p_{n_j}\}$. 
The rest of the argument is as in Case 1 with no modifications.

{\bf Case 3.} $q\in\mathrm{Crit}(g)$ and $p\in\mathrm{Crit}_W(G)$, with $\mathrm{ind}(p)\ge 2$. 
This is the same as Case 2.

By induction on the number of pairs of critical points with wrong order, we are done.

\end{proof}

\begin{prop}\label{prop:Cerf-1}
Let $\Modi=(G,\ub,\xi)$ and $\Modi'=(G',\ub',\xi')$ be ordered, indefinite and good Morse 
data for a stable cobordism $\Cobb=(W,\sur)$. 
Assume that on a tubular neighborhood $U$ of $\sur$ we have $G|_U=G'|_U$ and 
$\xi|_U=\xi'|_U$. Then $\Modi$ can be connected to $\Modi'$ by a sequence 
of Cerf moves such that it stays indefinite and almost ordered throughout. 
Further, the Cerf moves can be chosen so that the Morse function and the gradient-like 
vector field are not changed on $U$. 
\end{prop} 

\begin{proof}
Suppose $\Cobb$ is a cobordism from $(M,T)$ to $(M',T')$. 
Let $a= G|_{M}=G'|_{M}$ and $b=G|_{M'}=G'|_{M'}$.
First, we show that $G$ can be connected to $G'$ by a smooth, \emph{generic} 
family $\{G_t\}_{t\in [0,1]}$, in the sense of \cite[Definition 2.3]{GK}, such that for all but 
finitely many values of $t$ in $[0,1]$, $G_t$ is a proper, indefinite and almost 
ordered Morse function. Moreover, for every $t\in [0,1]$ we have 
$G_t|_U=G|_U=G'|_U$. The argument is similar to the proof of Theorem 4.5 in \cite{GK}.

The Morse functions $G$ and $G'$ coincide on $U$, so there is a generic family $\{G_t\}_{t\in [0,1]}$ 
connecting $G$ to $G'$ such 
that $G_t|_U=G|_U=G'|_U$.  Associated with  $\{G_t\}$ consider a 
generic family of embedded gradient-like vector fields 
$\{\xi_t\}_{t\in[0,1]}$ connecting $\xi$ to $\xi'$ such that $\xi_t|_U$ does not depend on $t$. 
The family $\{G_t\}$ is called \emph{indefinite} if for all but finitely many values of 
$t$, $G_t$ is an indefinite Morse function.  If $\{G_t\}$ is not indefinite, consider 
$r\in [a,b]$ such that an index zero critical point is born at time $r$ in 
$p_r\in W$. Corresponding to this critical point, we have a path of critical points 
\[P=\left\{(t,p_t)\in [r,s]\times W\ \big|\ p_t\in\mathrm{Crit}_W(G_t)\ \ 
\text{and}\ \ \mathrm{ind}(p_t)=0 \ \ \text{for}\ \ t\in(r,s)\right\}\]  
in $[0,1]\times W$ such that at $p_s$ the index zero 
critical point is canceled 
against an index one critical point.  Since $W$ is connected, for every $t\in (r,s)$ there 
is  an index one critical point $q_t$ of $G_t$ 
which cancels $p_t$, i.e. $W^s(q_t)\cap W^u(p_t)$ is a single flow line. 
Therefore, for some $\delta>0$, and an ordered sequence 
\[(r=t_0<t_1<\cdots <t_n=s),\] 
we have paths of index one critical points 
\[Q^i=\left\{(t,q_t^i)\in I_i\times W\ \big|\ q_t^i\in\mathrm{Crit}_W(G_t)\right\},\ \ \ \ i=1,\ldots,n,\] 
such that $q^i_t$ cancels $p_t$ for $t\in I_i$, where
\[I_1=[r,t_1+\delta), I_2=(t_1-\delta,t_2+\delta),\ldots,  I_n=(t_{n-1}-\delta,s].\]
Moreover, $q^1_r=p_r$ 
and $q^n_s=p_s$. Since $\{\xi_t\}$ is generic, for any  $1\le i\le n$ and $t\in I_i$, the submanifold
\begin{equation}\label{def:desdisk}
W_P^s(q^i_t)=W^s(q^i_t)\cap G_t^{-1}\left([G_t(p_t),G_t(q^i_t)]\right)
\end{equation}
is disjoint from $\sur$. Thus we can use the {\emph{Unmerge Lemma}} \cite[Lemma 4.6]{GK} to cancel 
$P$ against $Q^i$ on the non-overlapping parts of the intervals $I_i$ for $i=2,\ldots,n-1$. 
Then we use either {\emph{Eye Death Lemma}} \cite[Lemma 4.7]{GK} or 
{\emph{Swallowtail Death Lemma}} 
\cite[Lemma 4.8]{GK} to cancel over overlaps and on the intervals $I_1$ and $I_n$. 

Next, we make the family $\{G_t\}$ almost ordered. It is  straightforward that one can 
use generic homotopies which pass through {\emph{cusp-fold}} crossings 
to move all critical point creations before all critical point switches and all 
critical point cancellations after all critical point switches. See \cite[Section 2]{GK} for 
the exact definition of a cusp-fold crossing homotopy, see Figure \ref{fig:cusp-fold}.  
\begin{figure}[ht]
\def\svgwidth{9cm}
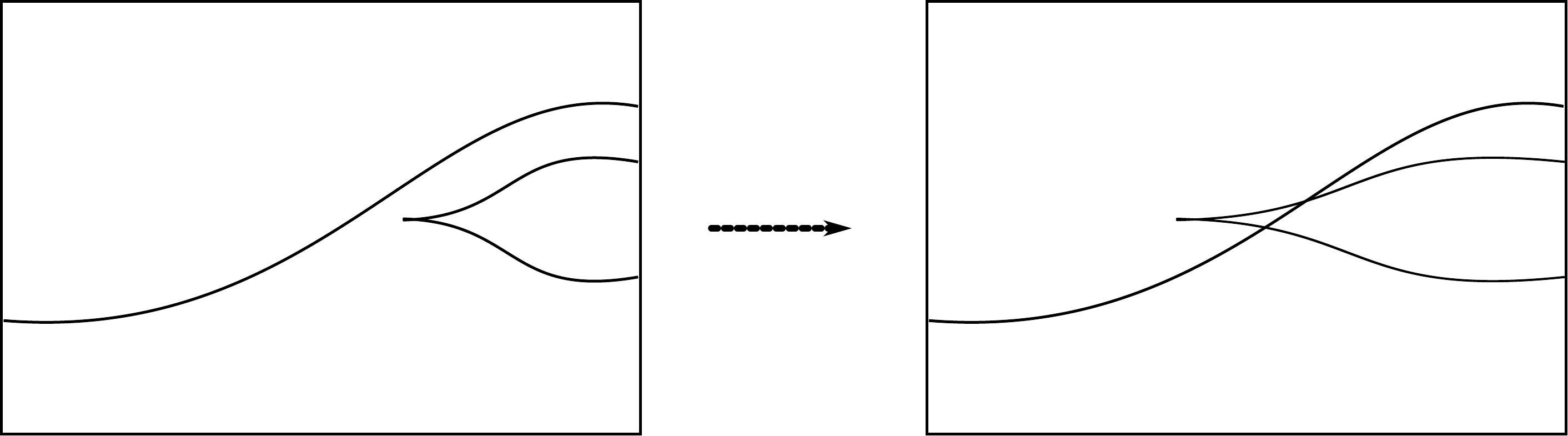
\caption{In a cusp-fold homotopy, the Cerf graphic, which is given by 
$\left\{(t,G_t(p))\subset [0,1]\times [a,b]\ \big|\ p\in \mathrm{Crit}(G_t)\right\}$,
changes as above.}\label{fig:cusp-fold}
\end{figure}

Then, we modify $\{G_t\}$ such that if an index one/two (two/three) critical 
point creation/cancellation happens at a point $p_t\in W$ and time $t$, then $G_t(p_t)$ 
is larger than the values of index one (two) critical points of $G_t$ and smaller than 
the values of index two (three) critical points of $G_t$. Since $G$ and $G'$ are ordered, such
modifications of $\{G_t\}$ can be achieved through generic 
homotopies supported in a neighborhood of an arc disjoint from $\sur$.

Suppose that for paths of critical points 
\begin{align*}
& P=\left\{(t,p_t)\in I_{pq}\times W \ \big|\ p_t\in\mathrm{Crit}_W(G_t)\right\}\ \   \text{and}\ \
 Q=\left\{(t,q_t)\in I_{pq}\times W\ \big|\ q_t\in \mathrm{Crit}_W(G_t)\right\}
\end{align*} 
with  $\mathrm{ind}(P)<\mathrm{ind}(Q)$ we have $G_t(p_t)>G_t(q_t)$ 
for any $t\in I_{pq}$.  Further, suppose that $G_t$ has no critical value in 
$(G_t(p_t),G_t(q_t))$. It follows from the definition of a gradient-like vector field 
that $W^s(p_t)$ and $W^u(q_t)$ are disjoint from $\sur$ in 
$G_t^{-1}\left([G_t(q_t),G_t(p_t)]\right)$. Moreover, since $\xi_t$ 
is generic, by counting dimensions we conclude that $W^s(p_t)$ does 
not intersect $W^u(q_t)$. Therefore, there is a homotopy supported in a 
neighborhood of the $1$-parameter family of descending disks 
$\{W^{s}_Q(p_t)\}$ or the $1$-parameter family of ascending disks 
$\{W_{P}^u(q_t)\}$ that pulls $P$ below $Q$. Note that $W^{s}_Q(p_t)$ 
and $W_{P}^u(q_t)$ are defined as in Equation \ref{def:desdisk}.

The other possibility is that,  for a path of index $1$ critical points 
\[P=\left\{(t,p_t)\in I_p\times W\ \big|\ p_t\in\mathrm{Crit}_W(G)\right\},\]
$g(q)<G_t(p_t)$ for some $q\in\mathrm{Crit}(g)$ and $G_t$ has no critical value in  
$(g(q),G_t(p_t))$. As before, we can modify $\{G_t\}$ by a
homotopy disjoint from $\sur$ and supported in a neighborhood of 
$\{W^{s}_q(p_t)\}$ to pull $P$ below $q$. Similarly, we can modify   
$\{G_t\}$ through homotopies such that for every 
index three critical point $p_t$ of $G_t$ (with $t\in[0,1]$) and every 
$q\in\mathrm{Crit}(g)$, we have $G_t(p_t)>g(q)$.\\ 

%
%
%Similarly, we  may change $\{G_t\}_t$ through homotopies such that for every 
%index three critical point $p_t$ of $G_t$ (with $t\in[0,1]$) and every 
%$q\in\mathrm{Crit}(g)$, we have $G_t(p_t)>g(q)$.\\ 

As a result, we obtain a generic, smooth path $\{G_t\}_{t\in[0,1]}$ of Morse 
functions which are proper, indefinite and almost ordered for all but 
finitely many values of $t$. Let $m$ and $m'$
denote the smallest and largest critical values of $g$. 
The next step is to modify $\{G_t\}$ through generic homotopies such that 
there exists an ordered set 
\[(0=t_0<t_1<\cdots <t_{k}=1)\subset [0,1]\] 
for which $G_{t_i}$ is proper, indefinite and almost ordered with no critical 
points in $W$ above the interval $[m,m']$. Moreover, for any  $1\le i\le k$ 
the family $\{G_t\}_{t\in [t_{i-1},t_{i}]}$ satisfies one of the followings:
\begin{enumerate}
\item For every $t\in [t_{i-1},t_{i}]$, $G_t$ is a proper Morse function on $\Cobb$.
\item $\{G_t\}_{t\in [t_{i-1},t_{i}]}$ corresponds to  a critical point creation/cancellation 
connecting $G_{t_{i-1}}$ and $G_{t_{i}}$.
\item $\{G_ t\}_{t\in [t_{i-1},t_{i}]}$ corresponds to switching  two critical points of 
$G_{t_{i-1}}$ on $W$ with equal index, thus to a critical point switch of type I. 
\item $\{G_ t\}_{t\in [t_{i-1},t_{i}]}$ corresponds to switching some
$p\in \mathrm{Crit}_W(G_{t_{i-1}})$ with all of the 
critical points of $g$, thus a critical point switch of type II. 
\end{enumerate}

In order to do so, we apply generic homotopies similar to Reidemeister II and 
Reidemeister III fold-crossings in \cite[p.11, p.12]{GK} [See Figures~\ref{fig:ReidmII} 
and~\ref{fig:ReidmIII}]. The only 
difference is that we also have paths consisting of the critical points of $g$. 
Assume that for some $c\in [0,1]$, there are index $2$ critical points $p$ 
and $p'$ of $G_c$, such that $G_c(p)=G_c(p')=d$ and $d<m'$. There is a 
critical point $q\in\mathrm{Crit}(g)$ so that $G_c$ has no critical value in the 
interval $(d,g(q))$. For sufficiently small $\delta>0$, consider paths of index $2$ critical points
\begin{align*}
&P=\{(t,p_t)\in (c-\delta,c+\delta)\times W \ \big|\ p_t\in\mathrm{Crit}_{W}(G_t)\}\ \ \ \text{and}\\
&P'=\{(t,p'_t)\in (c-\delta,c+\delta)\times W \ \big|\ p'_t\in\mathrm{Crit}_{W}(G_t)\}
\end{align*}
so that $p=p_c$ and $p'=p'_c$. Since $\{\xi_t\}$ is generic, we may arrange for 
descending disks $W^s(p)$ and $W^s(p')$ to be disjoint from $q$. Thus, we may 
modify $\{G_t\}_{t\in (s-\delta,s+\delta)}$ through generic homotopies, disjoint 
from $\sur$, to make $G_t(p_t),G_t(p'_t)>g(q)$ for any $t\in(c-\delta,c+\delta)$. 
We continue these modifications, until for any two critical point  
$p,p'\in\mathrm{Crit}_W(G_t)$ with $d=G_t(p)=G_t(p')$ either $d<m$ or $d>m'$. 

With a similar argument, we may use generic homotopies through 
Reidemeister II fold-crossings, that are disjoint from $\sur$, 
and make $\{G_t\}_{t\in[0,1]}$ to satisfy the required
conditions. Note that the family $\{G_t\}_{t\in[0,1]}$ remains indefinite and almost ordered 
under these homotopies.

\begin{figure}[ht]
\def\svgwidth{7cm}
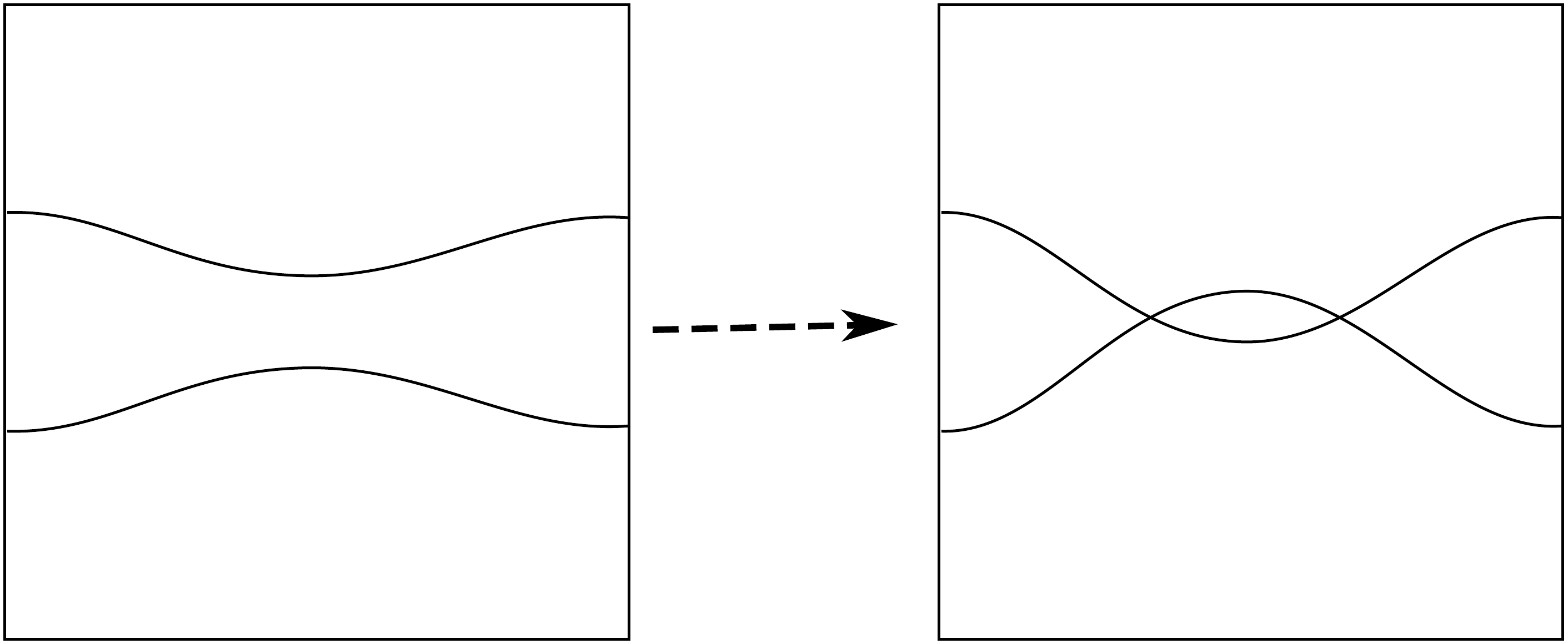
\caption{Reidemeister II fold-crossing.}\label{fig:ReidmII}
\end{figure}

\begin{figure}[ht]
\def\svgwidth{7cm}
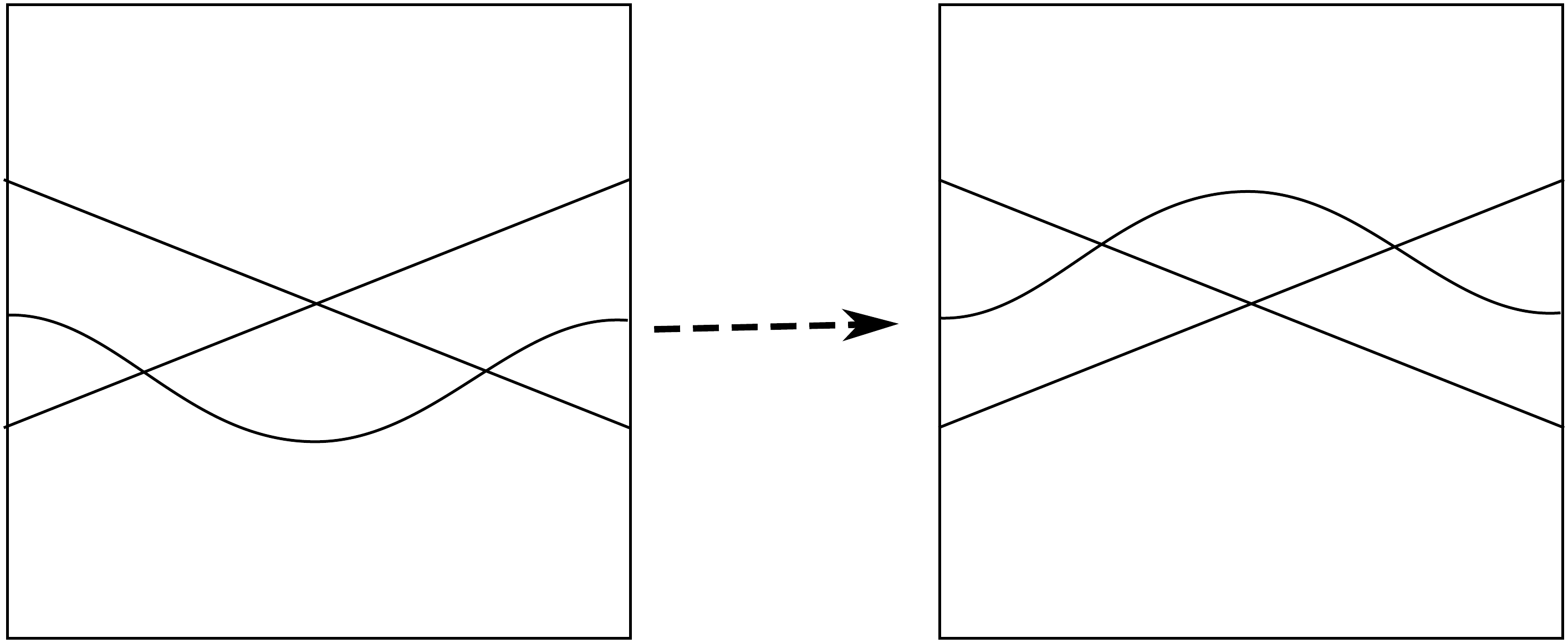
\caption{Reidemeister III fold-crossing.}\label{fig:ReidmIII}
\end{figure}

For any $1\le i\le k$, if $\{G_t\}_{t\in[t_{i-1},t_i]}$ is of type $(2)$ or $(3)$, 
we may choose the interval $[t_{i-1},t_i]$ sufficiently small so that we can 
obtain an elementary critical point creation/cancelation or critical point switch by a 
perturbation of $\{G_t\}_{t\in[t_{i-1},t_i]}$.
%\todo{AA: I don't know about 
%type 4. EE: Do we have a problem with type 4 here? I am not sure if I see your point.} 

Let us assume that an embedded gradient-like vector field $\xi_i$, together with 
an ordered set of regular values $\ub_i$ for $G_{t_i}$ is given,  
such that $\Modi_i=(G_{t_i},\xi_i,\ub_i)$ 
is a good Morse datum for $\Cobb$ for some $i\in\{0,1,..,k-1\}$. 
Moreover, assume that 
$\ub_i\cap [m,M]=\emptyset$. Then, depending on the type of the family 
$\{G_t\}_{t\in [t_i,t_{i+1}]}$ we may construct a good Morse datum 
$\Modi_{i+1}=(G_{t_{i+1}},\xi_{i+1},\ub_{i+1})$. 
In fact, if the family $\{G_t\}_{t\in[t_i,t_{i+1}]}$ satisfies (1), (2), (3) or (4), then 
$\Modi_{i+1}$ is obtained from $\Modi_i$  
by a left-right equivalence, a critical point creation/cancellation, a
critical point switch of type I or a  critical point switch of type II, respectively.  \\

Inductively, we get a sequence of good Morse data $\Modi_{i}=(G_{i}, \xi_{i},\ub_{i})$ 
for $i=1,2,\hdots,k$ with $G_k=G'$, which 
is obtained by applying a sequence of Cerf moves to $\Modi=\Modi_0$. 
Then, $\Modi_{k}$ is related to $\Modi'$ by isotopies of the 
gradient-like vector field and adding or removing regular values, and we are done.
\end{proof}

\begin{lem}\label{lem:sur-Mors}
Let $\Cobb=(W,\sur)$ be a stable cobordism and suppose that
$g,g':\sur\ra [a,b]$ are indefinite Morse functions over $\sur$. Then $g$ can 
be connected to $g'$ by a generic family $\{g_t\}_{t\in [0,1]}$ of indefinite Morse functions. 
\end{lem}
\begin{proof}
This is a corollary of Theorem 4.5 in \cite{GK}.
\end{proof}

\begin{prop}\label{prop:Cerf-2}
Let $\Modi=(G,\ub,\xi)$ be an indefinite, ordered and good Morse datum for 
the stable cobordism $\Cobb=(W,\sur)$. Fix a proper, indefinite and ordered 
Morse function $G'':W\ra [a'',b'']$ on $\Cobb$. Then $\Modi$ can be connected 
by a sequence of Cerf moves
to a Morse datum $\Modi'=(G',\ub',\xi')$, such that the Morse datum stays ordered throughout and $G'$ coincides 
with $G''$ in a neighborhood of $\sur$.
\end{prop}

\begin{proof} 
After removing some regular values, we may assume that for some
integer $n$, all critical values of $g$ lie in the interval $(b_{n}, b_{n+1})$. 
Choose the regular values $m$ and $m'$ for $G''$ such that for any $p\in\mathrm{Crit}_W(G'')$
\begin{align*}
\begin{cases}
G''(p)<m\ \ \ &\text{if}\ \ \mathrm{ind}(p)=1\\
G''(p)>m'\ \ \ &\text{if}\ \ \mathrm{ind}(p)=2\\
\end{cases}.
\end{align*}
Moreover, we require that every critical point $p\in\sur$ of $g''=G''|_{\sur}$ satisfies 
$m<G''(p)<m'$. Let $a=\min(\ub)$ and $b=\max(\ub)$. Consider a diffeomorphism 
$\phi:[a,b]\ra [a'',b'']$ with 
\begin{align*}
\phi(a)=a'',\  \ \phi(b)=b'',\ \   
\phi(b_n)=m\  \ \ \text{and}\ \ \  
\phi(b_{n+1})=m'.
\end{align*}
We apply the left-right equivalence move defined by $\phi$ and 
$\mathrm{Id}_{W}$ on the Morse datum $(G,\ub,\xi)$ to get 
$(\phi\circ G,\phi(\ub),\xi)$. Lemma \ref{lem:sur-Mors} implies that 
$\phi\circ g$ can be connected to $g''$ by a generic family $\{g_t\}_{t\in [0,1]}$ 
of indefinite Morse functions on $\sur$ which fails to be proper at the times
\[0<c_1<c_2<\cdots  <c_l<1.\] 
Moreover, we may assume that for every $t\in[0,1]$ the
critical values of $g_t$ lie in interval $(m,m')$. For any $1\le i\le l$,
consider a sufficiently small $\delta_i>0$ such that 
\[0<t_1<c_1<t_2<t_3<c_2<\cdots <t_{2l-1}<c_l<t_{2l}<1\] 
where $t_{2i-1}=c_i-\delta_i$ and
$t_{2i}=c_i+\delta_i$. Let $p^i$ and $q^i$ denote the critical points of $g_{c_i}$ for which
\[g_{c_i}(p^i)=g_{c_i}(q^i),\ \ \ g_{2i-1}(p^i)<g_{2i-1}(q^i)\ \ \ 
\text{and}\ \ \  g_{2i}(p^i)>g_{2i}(q^i).\] 
Here $g_i=g_{t_i}$.
Since $\delta_i$ is sufficiently small, we 
may perturb $\{g_t\}_{t\in [t_{2i-1},t_{2i}]}$ and change it to 
an elementary switch. In particular, there is a vector field $\ti{\xi}_i$ on 
$\sur$ which is gradient-like for any $g_t$ and the homotopy $\{g_t\}_{t\in [t_{2i-1},t_{2i}]}$  is 
supported in a neighborhood $V_i$ of 
\[W^s\left(p^i\right)\cap g_{2i-1}^{-1}\left[g_{2i-1}(p^i),g_{2i}(p^i)\right]\ \ \ \ \text{or}\ \ \ \ 
W^u\left(q^i\right)\cap g_{2i-1}^{-1}\left[g_{2i}(q^i),g_{2i-1}(q^i)\right].\]
Without loss of generality, assume $V_i$ is a neighborhood of 
\[W^s\left(p^i\right)\cap g_{2i-1}^{-1}\left[g_{{2i-1}}(p^i),g_{{2i}}(p^i)\right].\]
For a bump function $\omega_i$ supported in $V_i$ we have
$$g_t=g_{{2i-1}}+(t-t_{2i-1})\omega_i\ \ \ \ \text{for}\ \ \ \ t\in [t_{2i-1},t_{2i}].$$
Furthermore, $\omega_i$ is constant in a neighborhood of $p^i$.

%Let $\ti{\xi}_{i}$ be an embedded gradient-like vector field for $g_{i}=g_{t_i}$. 
%For every $i=1,\ldots,l$, let $p^i$ and $q^i$ be the critical points of $g_{c_i}$ such that 
%\[g_{c_i}(p^i)=g_{c_i}(q^i),\ \ \ g_{2i-1}(p^i)<g_{2i-1}(q^i)\ \ \ 
%\text{and}\ \ \  g_{2i}(p^i)>g_{2i}(q^i).\] 
%
%Since $\delta_i$ is sufficiently small, we 
%may perturb $g_t$ for $t\in [t_{2i-1},t_{2i}]$ and change it to an elementary switch.
%%'chemin \'el\'ementaire de croisement'. 
%Moreover, we may arrange for $\ti{\xi}_{2i-1}=\ti{\xi}_{2i}$ such 
%that the elementary switch
%%'chemin \'el\'ementaire de croisement' family 
%$\{g_t\}_{[t_{2i-1},t_{2i}]}$ is 
%supported in a neighborhood $V_i$ of 
%\[W^u\left(p^i\right)\cap g_{2i-1}^{-1}\left[g_{2i-1}(p^i),g_{2i}(p^i)\right]\ \ \ \ \text{or}\ \ \ \ 
%  W^s\left(q^i\right)\cap g_{2i-1}^{-1}\left[g_{2i-1}(p^i),g_{2i}(p^i)\right].\]

Let $G_{2i-1}$ be an extension of $g_{2i-1}$ to an ordered Morse function on 
$\Cobb$. Consider an embedded gradient-like vector field $\xi_{i}$ for $G_{2i-1}$ 
such that $\xi_{i}|_{\sur}=\ti{\xi}_{i}$. 
%Without loss of generality we may assume 
%that $V_i$ is a neighborhood of 
%\[W^u\left(p^i\right)\cap g_{2i-1}^{-1}\left[g_{{2i-1}}(p^i),g_{{2i}}(p^i)\right].\]
We may extend 
$\omega_i$ to a bump function on $W$, denoted by $\Omega_i$ and supported in an 
open neighborhood $\ovl{V}_i$ of 
\[W^s\left(p^i\right)\cap G_{2i-1}^{-1}\left[G_{2i-1}(p^i),G_{2i}(p^i)\right],\]
and is constant on a neighborhood of $p^i$ in $W$. Moreover, for every $t\in[t_{2i-1}, t_{2i}]$ 
\[G_t=G_{2i-1}+(t-t_{2i-1})\Omega_i\]
is a Morse function on $\Cobb$ with $\mathrm{Crit}_{\Cobb}(G_t)
=\mathrm{Crit}_{\Cobb}(G_{2i-1})$ and corresponds to the embedded-gradient 
like vector field $\xi_{i}$. Thus, $\{G_t\}_{t\in [t_{2i-1},t_{2i}]}$ is
an elementary critical point switch. Further, 
$\Modi_{2i}=(G_{2i},\phi\circ\ub,\xi_{i})$ is an ordered, good Morse datum, 
obtained from $\Modi_{2i-1}=(G_{2i-1},\phi\circ\ub,\xi_{i})$ by a critical point switch of type III.
%is 'chemin \'el\'ementaire de croisement'. 

On the other hand, we may find a family 
\[\left\{\psi_t\ \big|\ t\in [t_{2i},t_{2i+1}]\right\}\ \ \ \text{for }\  i=1,\ldots,l-1,\] 
of diffeomorphisms of $\R$ such that $\psi_{t_{2i}}=\mathrm{Id}$ and $\psi_t\circ g_t$ have 
the same critical values as $g_{{2i}}$.  With the same argument as in the proof of Lemma 
\ref{lem:simple-path}, we may define a vector field $v(x,t)$
on $\sur\times [t_{2i},t_{2i+1}]$ such that the corresponding diffeomorphism 
\[\Psi_t:\sur\times\{t_{2i}\}\ra\sur\times \{t\}\] satisfies $\psi_t\circ g_t\circ\Psi_t|_{\sur\times \{t_{2i}\}}=g_{{2i}}$.
After choosing an arbitrary connection 
%\todo{EE: Why do we need a connection here? AA: Doesn't connection tell us how to lift the vector field? } 
on a tubular neighborhood  $\mathrm{nd}(\sur)$ of $\sur$ in $W$
we may extend $v$ to a vector field over $\mathrm{nd}(\sur)\times [t_{2i},t_{2i+1}]$ 
and correspondingly, extend
$\Psi_{t}$ from $\sur\times \{t_{2i}\}$ to a diffeomorphism 
\[\Psi_t:\mathrm{nd}(\sur)\times\{t_{2i}\}\lra \mathrm{nd}(F)\times\{t\}.\]%W\times [t_{2i},t_{2i+1}].\]

Suppose that we have an extension of $g_{{2i}}$ to an ordered Morse function 
$G_{{2i}}$ on $\Cobb$. Then for any $t\in [t_{2i},t_{2i+1}]$ the Morse function  
$G_{{2i}}\circ \Psi_{t}^{-1}$ is defined on the tubular neighborhood
$\mathrm{nd}(\sur)\times\{t\}\subset W$ of $\sur\times\{t\}$ and has no critical points. 
Fix the vector field $\del_t$ on the complement of $\sur\times [t_{2i},t_{2i+1}]$.
By patching the vector fields $v$ and $\del_t$ using a partition of unity, we get a
global vector field, 
still denoted by $v(x,t)$, on $W\times [t_{2i},t_{2i+1}]$. 
The flow of $v$ defines a family of diffeomorphisms 
\[\left\{\Psi_t:W\times \{t_{2i}\}\ra W\times\{t\}\ \big|\ t\in [t_{2i},t_{2i+1}]\right\}.\] 
Furthermore, the family
\[\left\{G_{t}:=G_{{2i}}\circ \Psi^{-1}_t:W\times\{t\}\ra \R\ \big|\ t\in[t_{2i},t_{2i+1}]\right\}\] 
is a family of 
ordered Morse function on 
$\Cobb$ such that $G_t|_{\sur}=\phi_t\circ g_t$. Therefore, any ordered 
Morse function $G_{{2i}}$ on $\Cobb$ with $G_{{2i}}|_{\sur}=g_{{2i}}$ can 
be connected by left-right equivalence to an ordered Morse function 
$G_{{2i+1}}$ on $\Cobb$ such that $G_{{2i+1}}|_{\sur}=g_{{2i+1}}$ 
and $G_{{2i+1}}|_{W-\mathrm{nd}(\sur)}=G_{{2i}}|_{W-\mathrm{nd}(\sur)}$.

We may thus connect $\Modi=(G,\xi,\ub)$ to a Morse datum 
$\ti{\Modi}=(\ti{G},\ti{\xi},\ub)$ by a sequence of Cerf moves, such that 
$$\ti{G}|_{\sur}=G''|_{\sur}\ \ \ \ \text{and}\ \ \ \ \ti{G}|_{W-U}=G|_{W-U}$$ 
for an open neighborhood $U$ of $\sur$.
Moreover, $\mathrm{Crit}_W(\ti{G})=\mathrm{Crit}_W(G'')$ is a subset of 
$W\setminus U$ and the Morse datum remains good, indefinite and ordered throughout.

Finally, after a small perturbation in a neighborhood of $\sur$, 
we may arrange for $t\ti{G}+(1-t)G''$ to be a family of Morse functions 
with no critical points in a neighborhood $\ti{U}\subset U$ of $\sur$.  
Let $v'(x,t)$ be the corresponding vector field on $\ti{U}\times [0,1]$ as in the proof of 
Lemma \ref{lem:simple-path}. Define a global vector field $v'$ on $W\times [0,1]$ 
by patching $v'$ in $\ti{U}$
with the vector field $\del_t$ on $(W-\sur)\times [0,1]$ using a partition 
of unity. Denote the time-one flow of $v'$ by $\Phi_1$. Then, the Morse datum 
$\Modi'=(G',\ub',\xi')$ is obtained from $\ti{\Modi}$ by the left-right equivalence 
corresponding to $\mathrm{Id}_\R$ and $\Phi_1^{-1}$ (so that in particular, we have 
$G'=\ti{G}\circ \Phi_1^{-1}$), and satisfies the required conditions.
\end{proof}

Combining Propositions ~\ref{prop:Cerf-1} and ~\ref{prop:Cerf-2}, we deduce that:

\begin{cor}\label{cor:PCD+order}
Let $\Modi=(G,\ub,\xi)$ and $\Modi'=(G',\ub',\xi')$ be ordered, 
good and indefinite Morse data for a stable cobordism 
$\Cobb=(W,F)$. Then, $\Modi$ and $\Modi'$ can be connected by 
a sequence of Cerf moves. Moreover, we can keep the Morse data 
almost ordered, and avoid index zero and four critical points 
throughout the sequence.
\end{cor}

Lemma ~\ref{lem:putorder} and Corollary~\ref{cor:PCD+order} 
implies the following theorem, which will be called the  
{\emph{parametrized Cerf decomposition theorem}} in this 
paper.

\begin{thm}\label{thm:Cerf}
Any two good and indefinite Morse data for an stable cobordism 
$\Cobb=(W,F)$, may be connected by a sequence of Cerf moves. 
Moreover, we can avoid index zero and four critical points 
throughout the sequence.
\end{thm}

\subsection{Ordered Morse data and parametrized decompositions}
Let $\Cobb=(W,\sur)$ be a stable cobordism from the tangle 
$(M,T)$ to $(M',T')$. Suppose that $G:\Cobb\to [a,b]$ is an 
indefinite, ordered Morse function. There is an ordered set of 
regular values $b=(b_0=a<b_1<b_2<b_3=b)$ so that for any 
$p\in\mathrm{Crit}_{\Cobb}(G)$, one of the following holds:
\begin{itemize}
\item if $p\in\mathrm{Crit}_W(G)$ and $\ind(p)=i$ then 
$b_{i-1}<G(p)<b_i$,
\item if $p\in\mathrm{Crit}(g)$ then $b_1<g(p)<b_2$.
\end{itemize}
Any such set of regular values gives a decomposition 
\[\Cobb=\Cobb_1\cup_{(M_1,T_1)}\Cobb_2\cup_{(M_2,T_2)}\Cobb_3\]
where $(M_i,T_i)=\Cobb\cap G^{-1}(b_i)$ and 
$\Cobb_i=\Cobb\cap G^{-1}([b_{i-1},b_i])$. Choose a 
gradient-like vector field $\xi$ for $G$. It induces a 
parametrization on each $\Cobb_i=(W_i,\sur_i)$. More precisely, 
for $i=1,3$, $\xi$ specifies a set 
$\Fsphere_i\subset M_{i-1}\setminus T_{i-1}$ of pairwise 
disjoint, framed $(i-1)$-spheres such that $W_i$ is 
diffeomorphic to $W(\Fsphere_{i})$, the cobordism obtained by 
attaching $i$-handles to $M_{i-1}$ along $\Fsphere_i$. 
Moreover, it determines a diffeomorphism 
$d_i:(M_{i-1}(\Fsphere_i),T_{i-1})\to (M_i,T_i)$. 
Similarly, $\xi$ determines a framed link $\Fsphere_2$ and an 
acceptable set of framed arcs $\Farc$ in $(M_1,T_1)$, along 
with a diffeomorphism 
$d_2:(M_1(\Fsphere_2),T_1(\Farc))\to(M_2,T_2)$. Any 
parametrized decomposition of this form is called 
\emph{indexed}. 
For any indexed parametrized decomposition of $\Cobb$ as above, 
there is an indefinite, ordered Morse function $G$, a 
gradient-like vector field $\xi$ and a $4$-tuple of regular 
values that specifies it. Any such triple $(G,\xi,b)$ is 
called a \emph{simplified Morse datum}.

%%%%%%%%%%%%%%%%%%%%%%%%
%%%%%%%%%%%%%%%%%%%%%%%%%

Suppose that \[\ti{\Cerf}:\Cobb=\Cobb_1\cup_{(M_1,T_1)}
\Cobb_2\cup_{(M_2,T_2)}\Cobb_3\quad\text{and}\quad 
\ti{\Cerf}':\Cobb=\Cobb_1'\cup_{(M'_1,T'_1)}
\Cobb_2'\cup_{(M'_2,T'_2)}\Cobb_3'\] are indexed parametrized 
decompositions for $\Cobb$. For $i=1,3$, let $(\Fsphere_i,d_i)$ 
and $(\Fsphere_i',d_i')$ denote the parametrization of 
$\Cobb_i$ and $\Cobb_i'$, respectively. In addition, assume 
$(\Fsphere_2,\Farc,d_2)$ and $(\Fsphere_2',\Farc',d_2')$ are 
the parametrization of $\Cobb_2$ and $\Cobb_2'$, respectively.  
Then, $\ti{\Cerf}'$ is obtained from $\ti{\Cerf}$ by creation 
of an index one/two critical point if:
\begin{itemize}
\item $\Cobb_3$ coincides with $\Cobb_3'$ as parametrized 
cobordisms.
\item $\Fsphere_1'=\Fsphere_1\amalg s$, where 
$s\subset M\setminus (T\cup\Fsphere_1)$ is a framed $0$-sphere, 
\item $\Fsphere_2'=d_1'\circ\ti{d}_1^{-1}(\Fsphere_2)\amalg k$ 
where $k\subset M_1'\setminus T_1'$ is a framed knot disjoint 
from $d_1'\circ\ti{d}_1^{-1}(\Fsphere_2)$. Here, 
$\ti{d}_1:M(\Fsphere_1')\to M_1(s)$ is the diffeomorphism 
induced by $d_1$. Moreover, 
$\Farc'=d_1'\circ\ti{d}_1^{-1}(\Farc)$
\item The framed knot $\ti{k}=\ti{d}_1\circ(d_1')^{-1}(k)$ 
intersects the belt sphere of $s$ at one point. As a result, 
there is a corresponding diffeomorphism 
\[\phi:(M_1,T_1)\to (M_1(s)(\ti{k}),T_1).\]
Note that $\phi(\Fsphere_2)=\Fsphere_2$ and 
$\phi(\Farc)=\Farc$.
\item $d_2=d_2'\circ D\circ \Phi$ where 
\begin{align*}
&\Phi:(M_1(\Fsphere_2),T_1(\Farc))
\to (M_1(s)(\ti{k})(\Fsphere_2),T_1(\Farc))\quad\text{and}\\
& D:(M_1(s)(\ti{k})(\Fsphere_2),T_1(\Farc))\to 
(M_1'(\Fsphere_2'),T(\Farc'))
\end{align*}
are diffeomorphisms induced by $\phi$ and 
$d_1'\circ(\ti{d}_1)^{-1}$, respectively.
\end{itemize}
The inverse of this move describes cancellation of a pair of 
index one/two critical points. Similarly, one may describe 
creation and cancellation of a pair of index two/three critical 
points. 

\begin{thm}\label{thm:inddecom}
With the above notation fixed, any two indexed parametrized 
decompositions \[\ti{\Cerf}:\Cobb=\Cobb_1\cup_{(M_1,T_1)}
\Cobb_2\cup_{(M_2,T_2)}\Cobb_3\ \ \ \text{and}\ \ \ 
\ti{\Cerf}':\Cobb=\Cobb_1'\cup_{(M'_1,T'_1)}
\Cobb_2'\cup_{(M'_2,T'_2)}\Cobb_3'\]
can be connected by a sequence of the following moves: 
\begin{enumerate}
\item Sliding one component of $\Fsphere_i$ on another 
component of $\Fsphere_i$, for $i=1,2,3$,
\item Sliding one component of $\Farc$ on another component of 
$\Farc\amalg \Fsphere_2$,
\item Sliding one component of $\Fsphere_2$ on a component of 
$\Farc$,
\item Creation and cancellation of index one/two or 
two/three critical points,
\item Diffeomorphism equivalences.  
\end{enumerate}
\end{thm}

\begin{proof}
The proof is a straightforward application of 
Corollary \ref{cor:PCD+order}.
\end{proof}

\newpage

\section{One-handles, Three-handles and the cobordism maps}\label{sec:connected-sum}
\subsection{Adding one-handles}\label{sec:1-handlediag}
Fix an algebra $\Ring$ over $\F=\Z/2\Z$ as before. Let $\Tangle=[M,T,\spinc,\la]$ be an $\Ring$-tangle and $\Fsphere\subset M\setminus T$ be a framed $0$-sphere. Recall that $M'=M(\Fsphere)$ denotes the $3$-manifold obtained from $M$ after attaching a $1$-handle along $\Fsphere$.  If components of $\Fsphere$ lie in the same component of $M$, then $M'$ is diffeomorphic to $M\# (S^1\times S^2)$. Otherwise, $M'$ has one component which is the connected sum of the corresponding components of $M$. Let $T'=T$ and $\la'=\la$. It is straightforward that $(M',T')$ is a balanced tangle with $\Ring$-coloring $\la'$. 

Associated with $\Fsphere$, we obtain a cobordim
$W(\Fsphere)$ from $M$ to $M'$, by attaching a $1$-handle along $\Fsphere\times\{1\}$ to $M\times [0,1]$. 
Let $\sur=T\times [0,1]\subset W(\Fsphere)$ and regard $\la$ also as a map  
$\la_\sur:\pi_0(\sur)=\pi_0(T)\ra \Ring$. Then, $(W(\Fsphere),\sur)$ is a stable cobordism from $(M,T)$ to $(M',T')$ and $\la_{\sur}$ induces the $\Ring$-colorings $\la$ and $\la'$ on $(M,T)$ and $(M',T')$, respectively. Given any $\SpinC$ class $\spinct$ on $W(\Fsphere)$ such that $\spinc=\spinct|_{M}$, we get an $\Ring$-cobordism $\Cob=[W(\Fsphere),\sur,\spinct,\la_\sur]$ from $\Tangle$ to $\Tangle'=[M',T',\spinc',\la']$ where $\spinc'=\spinct|_{M'}$. In fact, the $\SpinC$ structure $\spinc$ always determines $\spinct$ and thus $\spinc'$. Let $\Cob(\Fsphere)=\Cob$ and $\Tangle(\Fsphere)=\Tangle'$.

%Let
%When the feet of the one-handle are on the 
%same connected component of $M$, 
%$M'\simeq M\#(S^1\times S^2)$, and if we further require that the evaluation 
%of $c_1(\spinc')$ over the homology class $[\{pt\}\times S^2]$ is trivial, $\spinct$
%is again determined by $\spinc$. If this is not the case, we define the cobordism 
%map $\fmap_{\Cob}^\M$ associated with $\Cob$ to be the trivial map. In order to complete 
%our definition of the map $\fmap_\Cob^\M$ we thus need to focus on the case when the
%above condition on $c_1(\spinc)$ is satisfied or the feet of the one-handle are placed 
%on different connected components of $M$, when $\spinct$ and $\spinc'$ are both 
%determined by $\spinc$. 
%We will assume that this is the case in the 
%following discussions. With these assumptions on the $\SpinC$ structures fixed, let 
%\[\Tangle'=\Tangle(\Fsphere)=[M(\Fsphere),T,\la,\spinc']\ \ \ \text{and}\ \ \ 
%\Cob=\Cob(\Fsphere)=[W,\sur,\la_\sur,\spinct].\]

We may choose a 
Heegaard surface $\Sig$ for $(M,T)$ so that it cuts each one of the two disjoint balls 
in a disk. Denote the boundary curves of these two disks by $C_1$ and $C_2$, and 
their centers by $w_1$ and $w_2$, respectively. Further, each connected 
component $T_i$ of $T$ cuts $\Sig$ in a single transverse point $z_i$, let $\z=\{z_1,\ldots,z_{\ell}\}$. For appropriate collections $\alphas$ and $\betas$ of pairwise disjoint circles on $\Sig\setminus\z$ that bound disks on the 
two sides of $\Sig$ in $M\setminus T$, we get a Heegaard diagram  
\[\HD=(\Sig,\alphas,\betas,\la:\z=\{z_1,\ldots,z_{\ell}\}\ra \Ring,\spinc)\]
for $\Tangle$. 
%We may assume that each connected 
%component $T_i$ of $T$ cuts $\Sig$ in a single transverse point $z_i$ and that 
%the collections $\alphas$ and $\betas$ of simple closed curves on $\Sig$,
%which bound disks in $M-T$, give the $\Ring$-diagram
%\[H=(\Sig,\alphas,\betas,\la:\z=\{z_1,\ldots,z_n\}\ra \Ring)\]
%for $\Tangle$.\\
In this situation, a Heegaard diagram for $\Tangle'$ may 
be constructed as follows. There is a properly embedded cylinder $S$ on the one-handle 
attached to $M$ with boundary circles $C_1$ and $C_2$. If we remove the 
two disks with centers $w_1$ and $w_2$ from $\Sig$ and glue $S$ to the resulting 
surface (which has $C_1$ and $C_2$ as its two boundary components) we obtain a 
Heegaard surface $\Sig'$ for $M'$. Let $\alpha$ and 
$\beta$ be circles on $S$ which bound disks on the 
two sides of $\Sig'$, and cut each other in a pair of canceling intersection points, 
i.e. are Hamiltonian isotopes of each other. Then, 
$$(\Sig',\alphas'=\alphas\cup\{\alpha\},\betas'=\betas\cup\{\beta\},\z)$$
is a Heegaard diagram for $(M',T')$. Moreover, if $\HD$ 
is $\spinc$-admissible in the strong sense of \cite[Remark 4.6]{AE-1},
the diagram 
\[\HD'=(\Sig',\alphas',\betas',\la:\z\ra \Ring,\spinc')\]   
is a Heegaard diagram for $\Tangle'$. We will thus assume the above stronger form 
of admissibility for the Heegaard diagram $\HD$. 

The union of the cylinder $S$ 
and the aforementioned disks with centers $w_1$ and $w_2$ is a sphere, which 
will be denoted by $\ovl{S}$. 
%We identify $(\ovl{S},w_1,w_2)$ with 
%$(\mathbb{P}^1,0,\infty)$. 
We may label the two intersection 
points between $\alpha$ and $\beta$ by $\theta_{\alpha\beta}$ and
$\theta_{\beta\alpha}$, so that the bigons on the cylinder $S$ connect 
$\theta_{\alpha\beta}$ to $\theta_{\beta\alpha}$ as the domains 
of the Whitney disks for the Heegaard diagram $H_S=(\ovl{S},\alpha,\beta, w_1,w_2)$.

With the above two Heegaard diagrams fixed, and
given an $\Ring$-module $\M$, we construct a cobordism map from 
$\HFT^\M(\Tangle)$ to $\HFT^\M(\Tangle')$
associated with adding the $1$-handle along the framed $0$-sphere $\Fsphere$.
A completely similar construction would give a cobordism map associated 
with attaching a $3$-handle along a framed $2$-sphere. 

\subsection{Stretching the necks in Heegaard diagrams} 
The construction of the cobordism map for $1$-handles rests on a slight 
generalization of  \cite[Theorem 5.1]{OS-link}, which will be discussed in 
the present subsection.
\begin{prop}\label{prop:limit-moduli-space}
Fix the Heegaard diagrams 
\[\HD^i=(\Sig^i,\alphas^i,\betas^i,\z^i),\ \ \ i=1,2,\]
with extra marked points $\w^i=\{w^i_1,\ldots,w^i_l\}$ on $\Sig^i\setminus (\alphas^i\cup\betas^i\cup\z^i)$ for $i=1,2$. Let $\Sig$ denote the surface obtained from $\Sig^1$ and $\Sig^2$ by 
attaching  $l$ one-handles (necks) 
which connect $w^1_j$ to $w^2_j$, for $j=1,\ldots,l$.
Denote the number of curves in $\alphas^i$ and $\betas^i$ by $d^i$ and the 
genus of $\Sig^i$ by $g^i$. For $i=1,2$, choose a Whitney disk $\phi^i$ for $\HD^i$, with
$n_{w^1_j}(\phi^1)=n_{w^2_j}(\phi^2)=k_j$ and let $\phi=\phi^1\star\phi^2$ 
be the homotopy class of the Whitney disk obtained by joining
$\phi^1$ and $\phi^2$ along the necks.

\begin{enumerate}
\item $\mu(\phi)=\mu(\phi^1)+\mu(\phi^2)-2(k_1+\cdots  +k_l).$
\item Let $J_t$ be a path of almost-complex structures on $\Sig$, which are stretched along the necks, so that all necks have length $t$, and as $t\ra\infty$ converges to a degenerate path of almost complex structures on 
%\todo{AA I couldn't find the right notation for this. EE I think this is fine}
$$\Sig^1\vee\Sig^2=\frac{\Sig_1\amalg\Sig_2}{\{w^1_i\sim w^2_i\ |\ i=1,\ldots,l\}}$$
giving the path $J^i$ of almost complex structures on $\Sig^i$.
If for any $R\in\R$, there is some $t>R$ so that the moduli space $\Mod(\phi)$ is non-empty for $J_{t}$, then the moduli spaces of broken pseudo-holomorphic flowlines representing $\phi^1$ and $\phi^2$ are non-empty.
\item Suppose $\HD^i$ and $\phi^i$ satisfy the followings:
\begin{enumerate}
\item $\mu(\phi^1)=1$,  $\mu(\phi^2)=2(k_1+\cdots  +k_l)$ and $\la(\phi^1)\neq 0$,
\item All components of $\Sig^2\setminus \alphas^2$ and $\Sig^2\setminus \betas^2$ are punctured spheres and $d^2>g^2$.
\end{enumerate}
If $J^1$ and $J^2$ are generic, then for sufficiently large $t$, the moduli space  $\Mod(\phi)$ may be identified with the fiber product
\begin{align*}
\Mod(\phi^1)&\times_{\Sym^{k_1}(\D)\times \cdots   \times \Sym^{k_l}(\D)}
\Mod(\phi^2)
=\{(u^1,u^2)\ |\ u^i\in\Mod(\phi^i), \rho^1(u^1)=\rho^2(u^2)\}
\end{align*}
where $\rho^i=\rho^i_{1}\times\cdots \times\rho^{i}_l$ denotes the product of evaluation maps
%\todo{AA I understand that this is OS's notation, but inside the parenthesis looks wrong. EE: Why is that?}
\begin{align*}
\rho^i_j&=\rho_{w^i_j}:\Mod(\phi^i)\lra \Sym^{k_j}(\D),\ \ \ 
\rho^i_j(u)=u^{-1}(\{w^i_j\}\times \Sym^{d_j-1}).
\end{align*}
%
%
%
%Let $\mu(\phi^1)=1$, $\mu(\phi^2)=2(k_1+\cdots  +k_l)$ and suppose that 
% $\Sig^2\setminus \alphas^2$ and $\Sig^2\setminus \betas^2$ only have
% components of genus zero, while $d^2>g^2$. Consider 
%\todo{AA I understand that this is OS's notation, but inside the parenthesis looks wrong.}
%\begin{align*}
%&\rho^i:\Mod(\phi^i)\lra \Sym^{k_1}(\D)\times \cdots   \times \Sym^{k_l}(\D),\\
%&\rho^i(u):=\left(u^{-1}(\{w^i_1\}\times \Sym^{d_1-1})\right)\times\cdots \times \left(u^{-1}(\{w^i_l\}\times \Sym^{d_l-1})\right).
%\end{align*}
%If $J^1$ and $J^2$ are generic, then $\Mod(\phi)$ may be identified with the fiber product
%\begin{align*}
%\Mod(\phi^1)&\times_{\Sym^{k_1}(\D)\times \cdots   \times \Sym^{k_l}(\D)}
%\Mod(\phi^2)\\ 
%&=\{(u^1,u^2)\ |\ u^i\in\Mod(\phi^i), \rho^1(u^1)=\rho^2(u^2)\}
%\end{align*}
%for sufficiently large $t$.
\end{enumerate}
\end{prop}
\begin{proof}
The proof of Proposition~\ref{prop:limit-moduli-space} is basically the same 
as the proof of  \cite[Theorem 5.1]{OS-link}, and uses the cylindrical 
reformulation of Heegaard Floer homology by Lipshitz \cite{Robert-cylindrical}.
We only need to make a small modification to the last part of the argument 
of Ozsv\'ath and Szab\'o. In fact, the proof of the first two claims remains 
unchanged. We will thus focus on the proof of the last claim.

Consider a sequence $\{t_j\}$ of real numbers which converges to $\infty$ such that $\Mod(\phi)$ is nonempty for each $J_{t_j}$. In Lipshitz reformulation, as we stretch the necks (i.e. as $j\ra \infty$), 
every sequence $\{v_j\}_{j\in\Z^+}$ 
of curves with $v_j\in \Mod_{J_{t_j}}(\phi)$ has a subsequence which converges
to a pseudo holomorphic curve in the symplectic manifold
\begin{align*}
W(\infty)&=\left((\Sig^1-\w^1)\times [0,1]\times \R\right)\coprod
\left((\Sig^2-\w^2)\times [0,1]\times \R\right)\\
&=\left(W^1-\w^1\times [0,1]\times \R\right)\coprod
\left(W^2-\w^2\times [0,1]\times \R\right).
\end{align*}
which may be completed to a $J^1$-holomorphic curve in the symplectic manifold 
$W^1=\Sig^1\times[0,1]\times \R$ and 
a $J^2$-holomorphic curve in $W^2=\Sig^2\times[0,1]\times\R$. 
The components of this limit consist 
of pre-glued flowlines, boundary degenerations and nodal curves supported 
entirely inside the fibers of the projection map to $\D\simeq [0,1]\times\R$,
which are identified with $\Sig^1\vee\Sig^2$. By ignoring the matching 
conditions for the pre-glued flowlines we obtain the representative in 
$\Mod(\phi^1)$ and $\Mod(\phi^2)$ (in fact, this proves the second claim).

Let us now assume that $\mu(\phi^1)=1$. Since $\la(\phi^1)\neq 0$, in the above Gromov limit we  
obtain a $J^1$-holomorphic representative $u^1$ of $\phi^1$ and possibly a 
broken flow-line representative of $\phi^2$. If this latter representative has no 
closed components, a component $u^2$ of it forms a pre-glued flowline together
with $u^1$. Since other possible components of $\phi^2$ (including the 
boundary degenerations) each correspond to a positive share of the 
Maslov index, it follows that $\mu(u^2)\leq 2(k_1+\cdots  +k_l)$, with equality 
happening only if $u^2$ represents $\phi^2$ and there are no other components.
For a generic point 
\[\Delta=\rho^1(u^1)\in  \Sym^{k_1}(\D)\times \cdots   \times \Sym^{k_l}(\D)\]
the moduli space $(\rho^2)^{-1}(\Delta)$, which contains $u^2$, is of expected 
dimension $\mu(u^2)-2(k_1+\cdots  +k_l)$. Thus,
$u^2$ is forced to represent $\phi^2$ if $J^1$ (and thus $\Delta$) is generic.

The difference with the argument of Ozsv\'ath and Szab\'o appears when we 
consider the possibility of having closed components in the broken flowline 
representing $\phi^2$. Let us assume that these closed components represent 
$m[\Sig^2]$. After deleting these components we obtain a new class 
$\psi^2$, represented by $u^2$, with 
\[\mu(\psi^2)=\mu(\phi^2)-2m(d^2-g^2+1).\]
Moreover, each component $\Delta_j'$ of 
\[\rho^2(u^2)=\Delta'=(\Delta_1',\ldots,\Delta_l')\in\Sym^{k_1}(\D)\times 
\cdots  \times \Sym^{k_l}(\D)\] 
is obtained from the component $\Delta_j$ of $\rho^1(u^1)=\Delta$ by deleting 
$m$ points $\{p_1,\ldots,p_m\}$ from it. Note that these 
$m$ points are determined by the projection of the closed 
components in $W^2$ over $\D$, and are thus the same for 
$\Delta_1,\ldots,\Delta_l$.  
This means that the components of $\Delta$ have the points $p_1,\ldots,p_m$
in common. The subset of $\Mod(\phi^1)$ consisting of 
$J^1$-holomorphic curves satisfying this condition is of expected dimension
\[\mu(\phi^1)-2m(l-1)=1-2m(l-1).\]
If $m>0$, $l>1$, and $J^1$ is generic, this subset of $\Mod(\phi^1)$ 
is empty. The only remaining case is thus the case where $l=1$ and $m>0$.\\

In this latter case, the argument of Ozsv\'ath and Szab\'o may be used. 
The moduli space of all $u^2\in\Mod(\phi^2)$ with $\rho^2(u^2)=\Delta'$ 
is of expected dimension
\begin{align*}
\mu(\phi_2)-2(k_1-m)&=2k_1-2m(d_2+1-g_2)-2(k_1-m)
=-2m(d_2-g_2)<0
\end{align*}
and hence is empty for a generic choice of $J^1$ and $J^2$.

It follows that every weak limit of the curves in $\Mod(\phi)$, as we stretch the 
necks, is in correspondence with a pre-glued curve in 
\begin{align*}
&\{(u^1,u^2)\ |\ u^i\in\Mod(\phi^i), \rho^1(u^1)=\rho^2(u^2)\}
=\Mod(\phi^1)\times_{\Sym^{k_1}(\D)\times \cdots   \times \Sym^{k_l}(\D)}
\Mod(\phi^2).
\end{align*}
Moreover, from a pre-glued curve we may obtain an actual $J_{t_j}$-holomorphic
curve if $j$ is sufficiently large. This completes the proof.
\end{proof}

\subsection{Construction of the cobordism map}\label{sec:1-handlemap}
%The cobordism map is constructed from a chain map
%\[\fmap_{\Fsphere}:\CFT(\Tangle)\lra 
%\CFT(\Tangle').\]
Choose a generic path $J$ of almost complex structures 
on the Heegaard surface $\Sig$. We identify $(\ovl{S},w_1,w_2)$ with $(\mathbb{P}^1,0,\infty)$ and denote the induced complex structure on $\ovl{S}$ by $J_S$.  Let $J'$ be a generic path of almost-complex structures on $\Sig'$ which is sufficiently close to the join of $J$ and $J_S$.

%and choose the generic path $J'$ of almost 
%complex structures associated with $\Sig'$ sufficiently close to the join of 
%$J$ and the standard complex structure $J_S$ on the sphere $\ovl{S}$,
%which is identified with $\mathbb{P}^1$.
%Corresponding to $S$ we thus obtain the Heegaard diagram 
%$H_S=(\ovl{S},\alpha,\beta,\{w_1,w_2\})$.
Let $f_{\Fsphere}:\CFT_J(\HD)\ra\CFT_{J'}(\HD')$ be the homomorphism defined as 
\[f_{\Fsphere}(\x):=\x\times\{\theta_{\alpha\beta}\}\in
\CFT_{J'}(\HD')\]
for any generator $\x\in\Ta\cap\Tb$ representing the $\SpinC$ class $\spinc$.

%where $\CFT(\Tangle')$ is defined using the Heegaard diagram $\HD'$ and 
%the path $J'$ of almost complex structures.

\begin{prop}\label{prop:connected-sum}
The homomorphism $f_{\Fsphere}$ is a chain map.
%\begin{displaymath}
%\begin{split}
%\fmap_{\Fsphere}:\CFT(\Tangle)\lra
%\CFT(\Tangle')
%\end{split}
%\end{displaymath}
%defined above is a chain map.
\end{prop}
\begin{proof}
Let $\x\in\Ta\cap\Tb$ be an intersection point corresponding to the $\SpinC$ class 
$\spinc\in\SpinC(M)$. Suppose
\[\y\times \theta\in\mathbb{T}_{\alpha'}\cap\mathbb{T}_{\beta'}=
(\Ta\cap\Tb)\times (\alpha\cap\beta)\] is a generator contributing to $df_{\Fsphere}(\x)$, via 
$u'\in\Mod(\phi')$, where 
$\phi'\in\pi_2(\x\times\theta_{\alpha\beta},\y\times\theta)$.
Thus, $\mu(\phi')=1$ and $\la'(\phi')\neq 0$. The class $\phi'$ is the join of classes $\phi\in\pi_2(\x,\y)$ 
(corresponding to the Heegaard diagram $\HD$) and 
$\phi_S\in\pi_2(\theta_{\alpha\beta},\theta)$ (corresponding to the Heegaard 
diagram $\HD_S$), along a pair of necks corresponding to 
$w_1$ and $w_2$ which connect the two Heegaard diagrams. Further, $\la'(\phi')\neq 0$ implies that $\la(\phi)\neq 0$.

Let $|\theta|$ denote the number of intersection points in $\theta$ which are different from 
$\theta_{\alpha\beta}$. Thus, 
$|\theta|\in\{0,1\}$.
For a class $\phi'$ as above with $\mu(\phi')=1$, 
%if $\la(\phi)\neq 0$ 
we have
%\todo{AA I removed $\la(\phi)\neq 0$ from here. EE: fine}
\begin{align*}
1&=\mu(\phi')
=\mu(\phi)+\mu(\phi_S)-2n_{w_1}(\phi_S)-2n_{w_2}(\phi_S)
= \mu(\phi)+|\theta|.
\end{align*} 

Since the path $J'$ of almost complex structures is chosen 
close to the join of $J$ and $J_S$, and $\la(\phi)\neq 0$, Proposition~\ref{prop:limit-moduli-space} implies that
every pseudo-holomorphic representative $u'$ of $\phi'$ is 
in correspondence  with a
$J$-holomorphic curve $u\in\Mod(\phi)$ and a $J_S$-holomorphic curve 
$u_S\in\Mod(\phi_S)$, which are pre-glued.

If $\theta\neq \theta_{\alpha\beta}$, then $\mu(\phi)=0$ and $u$ is the constant 
map. It follows that $n_{w_i}(\phi_S)=n_{w_i}(\phi)=0$ and $\phi_S$ corresponds 
to one of the two bigon connecting $\theta_{\alpha\beta}$ to 
$\theta_{\beta\alpha}$.
The total contribution of such $u'$ to $df_\Fsphere(\x)$ 
is thus zero.

The remaining contributions to $df_{\Fsphere}(\x)$ come from 
the classes $\phi'$ such that 
$\mu(\phi)=1$, while  $\theta=\theta_{\alpha\beta}$ and 
$n_{w_i}(\phi_S)=n_{w_i}(\phi)=k_i$ for $i=1,2$. 
We would like to show that the total contribution to $df_{\Fsphere}(\x)$
from such $u'$ is equal to the corresponding contribution from $u\in\Mod(\phi)$ 
to $d(\x)$. It follows from Proposition~\ref{prop:limit-moduli-space}, that 
if the necks connecting 
$S$ to $\Sig$  are sufficiently 
stretched, then for $\phi'=\phi\star\phi_S$ the moduli space $\Mod(\phi')$ maybe identified with the fiber product $\Mod(\phi)\times_{\Sym^{k_1}(\D)\times \Sym^{k_2}(\D)}\Mod(\phi_S)$. In particular, any $u'\in\Mod(\phi')$ corresponds to the degeneration $u\star u_S$
with $u\in\Mod(\phi)$ and $u_S\in\Mod(\phi_{S})$ such that
%\begin{displaymath}
\[\rho_{\Sig}(u)=\rho_{S}(u_S)\in\Sym^{k_1}(\D)\times \Sym^{k_2}(\D),\]
where 
\begin{align*}
&\rho_{\Sig}=\rho_{w_1}\times\rho_{w_2}:\Mod(\phi)\ra \Sym^{k_1}(\D)\times \Sym^{k_2}(\D)\ \ \text{and}\\ 
&\rho_S=\rho_{w_1}\times\rho_{w_2}:\Mod(\phi_{S})\ra \Sym^{k_1}(\D)\times \Sym^{k_2}(\D)
\end{align*} 
are the evaluation maps.

%Let us denote the union of all such classes $\phi_S\in\pi_2(\theta_{\alpha\beta},
%\theta_{\alpha\beta})$ with $n_{w_i}(\phi_S)=k_i$ for $i=1,2$ by 
%$\phi_{k_1,k_2}$. Consider
% the evaluation maps
%\begin{align*}
%&\rho_{\Sig}=\rho_{w_1}\times\rho_{w_2}:\Mod(\phi)\ra \Sym^{k_1}(\D)\times \Sym^{k_2}(\D)\ \ \text{and}\\ 
%&\rho_S=\rho_{w_1}\times\rho_{w_2}:\Mod(\phi_{k_1,k_2})\ra \Sym^{k_1}(\D)\times \Sym^{k_2}(\D).
%\end{align*}
%If $u'\in\Mod(\phi\star \phi_S)$ correspond to the degeneration $u\star u_S$
%with $u\in\Mod(\phi)$ and $u_S\in\Mod(\phi_{k_1,k_2})$, we further find
%%\begin{displaymath}
%\[\rho_{\Sig}(u)=\rho_{S}(u_S)\in\Sym^{k_1}(\D)\times \Sym^{k_2}(\D).\]
%%\end{displaymath} 
%There are several classes $\phi_S$ which form $\phi_{k_1,k_2}$.  
%In fact such triangle classes are in correspondence with 
%the pairs $(a,b)$ of non-negative integers with $a+b=k_1+k_2$. 
%We thus obtain a map
%\begin{displaymath}
%\imath_{\phi}:\Mod(\phi')\simeq
%\Mod(\phi\star \phi_{k_1,k_2})
%\lra \Mod(\phi)\times_{\Sym^{k_1}(\D)\times \Sym^{k_2}(\D)}\Mod(\phi_{k_1,k_2}).
%\end{displaymath}
%By Proposition~\ref{prop:limit-moduli-space}, 
%if the necks connecting 
%$S$ to $\Sig$  are sufficiently 
%stretched and the paths of almost complex structures are generic 
%$\imath_{\phi}$ 
%is a bijection, reducing the proof  to the first part of the following lemma,
This reduces the proof to the first part of the following Lemma, which will be proved in Subsection~\ref{subsec:lemma}. With this 
lemma in place, the proof of the proposition is complete.
\end{proof}

\begin{lem}\label{lem:three-curves-general}
Let $\alpha,\beta,\gamma$ denote three 
curves on $\PP^1$ which are small Hamiltonian isotopes of one another and 
denote the corresponding top intersection points by $\theta_{a}\in\beta\cap\gamma,
\theta_b\in\alpha\cap\gamma$ and $\theta_c\in\alpha\cap\beta$. Let
$w_1$ and $w_2$ denote markings on the two domains in the complement 
of the isotopy regions. Let $\phi_{k_1,k_2}$ denote the set of Whitney disks 
$\phi\in \pi_2(\theta_c,\theta_c)$  with 
$n_{w_i}(\phi)=k_i,\ i=1,2$
and $\Delta_{k_1,k_2}$ denote the union of the triangle classes  
$\Delta\in \pi_2(\theta_c,\theta_a,\theta_b)$
with $n_{w_i}(\Delta)=k_i,\ i=1,2$. Then for generic 
\[(p_1,p_2)\in \Sym^{k_1}(\D)\times\Sym^{k_2}(\D)\] 
and a generic path of almost complex structures
%\todo{ AA why? What is the point of considering $\PP^1$? EE: It is just a notation fixed in the previous pages. Otherwise, nothing!}} 
on $\PP^1$ we have
\begin{align*}
&(i)\ \ n_{p_1,p_2}(\phi_{k_1,k_2}):=\#\left\{u\in\Mod(\phi_{k_1,k_2})\ \big|\ 
\rho_{w_i}(u)=p_i,\ i=1,2\right\}=1\ \ \text{and}\\
&(ii)\ n_{p_1,p_2}(\Delta_{k_1,k_2}):=
\#\left\{u\in\Mod(\Delta_{k_1,k_2})\ \big|\ \rho_{w_i}(u)=p_i,\ i=1,2\right\}=1.
\end{align*}
\end{lem}

\subsection{Proof of Lemma~\ref{lem:three-curves-general}}
\label{subsec:lemma} 
The proof is closely related to  the proof of \cite[Lemma 6.4]{OS-link}. 
We will prove the second claim. The first claim is in fact easier, and its proof 
is completely similar.
The first step
is to show that the number $n_{p_1,p_2}(\Delta_{k_1,k_2})$ does not depend 
on the generic choice of 
$(p_1,p_2)$. Given a generic path $\{(p_1^t,p_2^t)\}_{t\in[0,1]}$ in 
$\Sym^{k_1}(\D)\times \Sym^{k_2}(\D)$
connecting the generic points $(p_1^0,p_2^0)$ and $(p_1^1,p_2^1)$,  
consider the moduli space 
\begin{displaymath}
\left\{(u,t)\ \big|\ t\in[0,1],\ u\in\Mod(\Delta_{k_1,k_2}), \rho_{w_i}(u)=p_i^t,
\ i=1,2\right\},
\end{displaymath}
which is a smooth $1$-dimensional moduli space with ends determined by 
the Gromov limits of its points. 
Picture $\D$ as a $Y$ shape domain, and assume that 
the complex structure is translation invariant as we move towards infinity in any of 
three directions corresponding to $v_a,v_b$ and $v_c$.
Three types of the boundary points correspond to a degenerations of 
the domain into a disk $\phi$ and a triangle class $\Delta_\phi$. 
Since the path $\{(p_1^t,p_2^t)\}_t$ remains in a 
compact subset of the domain, the disk $\phi$ can not contain any of the 
pre-images
of $w_1$ or $w_2$, i.e. $n_{w_1}(\phi)=n_{w_2}(\phi)=0$. 
There are no holomorphic disks to $\theta_b$ with coefficient $0$ at 
$w_1$ and $w_2$. Thus,  $\phi$  corresponds to the 
disks contributing to $\ov\del(\theta_a)$ or $\ov\del(\theta_c)$
which are all zero, i.e. the total number of such ends is always zero.
The remaining ends  correspond to  
\begin{align*}
-&\left\{u\in\Mod(\Delta_{k_1,k_2})\ \big|\ \rho_{w_i}(u)=p_i^0,\ i=1,2\right\}\coprod
\left\{u\in\Mod(\Delta_{k_1,k_2})\ \big|\ \rho_{w_i}(u)=p_i^1,\ i=1,2\right\},
\end{align*}
implying the independence of $n_{p_1,p_2}(\Delta_{k_1,k_2})$ from the 
generic choice of 
$(p_1,p_2)$ in  the product $\Sym^{k_1}(\D)\times\Sym^{k_2}(\D)$.

Consider one of the branches of the $Y$-shape domain $\D$, for instance the one corresponding to $v_a$, and identify 
the branch with $[0,1]\times (0,\infty)$.
Denote the projection over the second factor by $\pi_\R$.
Choose the generic path $(p_1^T,p_2^T)$ of points in $\Sym^{k_1}(\D)
\times \Sym^{k_2}(\D)$  so that $p_1^T$ is 
a union of $k_1$ points $p_{1,1}^T,\ldots, p_{1,k_1}^T$ 
in the above branch so that 
$$\pi_\R(p_{1,1}^T)>T\ \ \text{and}\ \ \pi_\R(p_{1,i+1}^T)-\pi_\R(p_{1,i}^T)>T.$$
Similarly, assume that $p_2^T$ is 
a union of $k_2$ points $p_{2,1}^T,\ldots, p_{2,k_2}^T$ 
in the above branch so that 
$$\pi_\R(p_{2,1}^T)-\pi_\R(p_{1,k_1}^T)>T\ \ \text{and}\ \ 
\pi_\R(p_{2,i+1}^T)-\pi_\R(p_{2,i}^T)>T.$$
One may then consider the ends of the smooth $1$-dimensional moduli space 
\begin{displaymath}
\Ncal=\coprod_{T\in[1,\infty)}\left\{u\in\Mod(\Delta_{k_1,k_2})\ \big|\ 
\rho_{w_i}(u)=p_i^T,\ i=1,2\right\}.
\end{displaymath}
Two types of boundary ends correspond to degenerations of $u$ to a triangle
and a disk $\phi$ which goes to $\theta_b$ or starts from $\theta_c$.
Since the path of points remains in a compact subset of the branch corresponding
to vertices $v_b$ and $v_c$, we find $n_{w_1}(\phi)=n_{w_2}(\phi)=0$. There are no 
holomorphic disks to $\theta_b$ with coefficients $0$ at $w_1$ and $w_2$, and the 
disks $\phi$ which contribute to $\ov\partial(\theta_c)$ come in canceling pairs. 
Such pairs correspond to  pairs of points in the boundary of $\Ncal$
with canceling contributions. The remaining boundary components of $\Ncal$ are in correspondence with   
the points in union of 
\[-\left\{u\in\Mod(\Delta_{k_1,k_2})\ \big|\ \rho_{w_i}(u)=p_i^1,\ i=1,2\right\}\]
and the product
%\todo{AA this last paragraph is wrong, we are not considering all of the possibilities.
%EE I have added a paragraph to the proof. Is this what you mean?}
\begin{align*}
\Mod(\Delta_{0,0})&\times \left(\prod_{i=1}^{k_1}
\left\{u\in\Mod(\phi)\big|\rho_{w_1}(u)=t_{1,i}\right\}\right)
\times \left(\prod_{i=1}^{k_2}
\left\{u\in\Mod(\psi)\big|\rho_{w_1}(u)=t_{2,i}\right\}\right).
\end{align*}
Here $\phi\in\pi_2(\theta_a,\theta_a)$ denotes the homotopy types of the 
two disks connecting $\theta_a$ to itself with boundary on $\beta\amalg\gamma$,
$n_{w_1}(\phi)=1$ and $n_{w_2}(\phi)=0$. 
Similarly, $\psi\in\pi_2(\theta_a,\theta_a)$ denotes the homotopy types of the 
two disks connecting $\theta_a$ to itself with boundary on $\beta\amalg\gamma$,
$n_{w_1}(\phi)=0$ and $n_{w_2}(\phi)=1$. 
Moreover, $t_{1,i}$ and $t_{2,i}$ are arbitrary point on 
$[0,1]\times \R$. It is implied by \cite[ Lemma 7.3]{AE-1}  that 
the total number of points (counted with sign) in this latter end of the moduli space
is $1$. Consequently, $n_{p_1^1,p_2^1}(\Delta_{k_1,k_2})=1$ and the proof  
is complete.  

\subsection{Invariance of the cobordism map for one-handles}\label{subsec:CS-2} 
%\todo{AA how about we change the title of this section?  EE any suggestions? I changed it as above, but if you have a better suggestion, feel free to apply it}
Let $\Tangle=[M,T,\la,\spinc]$ be an $\Ring$-tangle and $\Fsphere\subset M\setminus T$ be a framed $0$-sphere. As before, $\Cob=\Cob(\Fsphere)$ denotes the parametrized $\Ring$-cobordism from $\Tangle$ to $\Tangle'=\Tangle(\Fsphere)$, obtained by attaching a $1$-handle to $\Tangle\times [0,1]$  along $\Fsphere\times\{1\}$.

%The goal of this section is to prove that $\fmap_{\Fsphere}$ induces an invariant homomorphism as follows.
%
%In view of the discussion of naturality for tangle Floer homology in 
%Section~\ref{sec:naturality}, we now prove:

\begin{thm}\label{thm:one-handle}
%Let $\Tangle$ be an $\Ring$-tangle and 
%\[\Cob=\Cob(\Fsphere):\Tangle\leadsto \Tangle'=\Tangle(\Fsphere)\] 
%denote the parametrized elementary $\Ring$-cobordism obtained from the product cobordism $\Tangle\times [0,1]$ by attaching a $1$-handle along a framed $0$-sphere $\Fsphere$ as before. 
For any $\Ring$-module $\Rin$, the chain map $f_{\Fsphere}$ from Proposition~\ref{prop:connected-sum} induces a homomorphism 
\[\fmap_{\Cob}^{\Rin}:\HFT^\Rin(\Tangle)
\lra \HFT^\Rin(\Tangle').\]

\end{thm} 
\begin{proof}
Any Heegaard surface for $\Tangle$ may be modified to intersect each component of $\Fsphere$ in a disk, by an isotopy supported in a neighborhood of two arcs in $M$. So we may just consider Heegaard diagrams for $\Tangle$ that the underlying Heegaard surface satisfies in this property. %\todo{AA I am not sure if it is really important.} 
From any Heegaard diagram $\HD$ for $\Tangle$ with such an underlying Heegaard surface $\Sig$, one may construct a Heegaard diagram $\HD'$ for $\Tangle'$ as described in Section \ref{sec:1-handlediag}.  Any choice of a path $J$ of almost complex structures for the Heegaard surface
$\Sig$ of $\Tangle$ gives a degenerate path of almost complex structures corresponding to 
the join of $\Sig$ and the standard model $\ovl{S}$ 
at $w_1$ and $w_2$. This path may be perturbed to a path $J'$ of almost complex structures corresponding to
$\Sig'$, the Heegaard surface for $\Tangle'$. Let $J_1$ and $J_2$ be paths of almost complex structures for $\Sig$ and $J_1'$ and $J_2'$ be the corresponding paths for $\Sig'$. It follows from standard techniques (involving a 
path connecting different choices of paths of almost complex structures) that \[f_{\Fsphere}\circ\Phi_{J_1\ra J_2}=\Phi_{J_1'\ra J_2'}\circ f_{\Fsphere}.\]
%We will
%thus drop $J$ and $J'$ from the notation for simplicity.\\

%The Heegaard surface $\Sig'$ for $(M',T')$ cuts the sphere corresponding to the 
%$1$-handle  in a simple closed curve. Every such surface in $(M',T')$ determines 
%the Heegaard surface $\Sig$  for $M$ and a cyliner $S$, as a subsurface of $\Sig'$.
%Moreover, we obtain the disk $D_i$, $i=1,2$ on $\Sig$ with center at a marked 
%point $w_i$.
%Corresponding to any choice of the paths $J$ of almost complex structures corresponding to 
%$\Sig$ we obtain a degenerate path of almost complex structures corresponding to 
%the join of the two surfaces $\Sig$ and the standard model $\ovl{S}$ 
%at $w_1$ and $w_2$.
%This path may be perturbed to a path $J'$ of almost complex structures corresponding to
% $\Sig'$. The independence of the isomorphism from the choice of these 
% paths of almost complex structures is proved by  standard techniques (involving a 
% path connecting different choices of paths of almost complex structures). So, we will
% thus drop $J$ and $J'$ from the notation for simplicity.\\

Let $\HD_0$ and $\HD_1$ be Heegaard diagrams for $\Tangle$ such that $\HD_1$ is obtained from $\HD_0$ by one Heegaard move (i.e. isotopy on an $\alpha$ or $\beta$ curve, handleslide on an $\alpha$ or $\beta$ curve, stabilization or destabilization) denoted by $\hmove$. We need to prove that the diagram
\begin{displaymath}
\begin{diagram}
\HFT^\Rin(\HD_0)
&\rTo{f_{\Fsphere}}&\HFT^\Rin(\HD_0')\\
\dTo{\Phi_{\hmove}}&&
\dTo{\Phi_{\hmove'}}\\
\HFT^\Rin(\HD_1) &\rTo{f_{\Fsphere}}&
\HFT^\Rin(\HD_1').
\end{diagram}
\end{displaymath}
is commutative. Here, $\Phi_e$ is the isomorphism associated with the Heegaard move, and $\HD_i'$ is the Heegaard diagram for $\Tangle'$ obtained from $\HD_i$ after joining it with $\HD_S=(\mathbb{P}^1,\alpha,\beta, w_1,w_2)$ at $w_1$ and $w_2$. Furthermore, $\HD_1'$ is obtained from $\HD_0'$ by a Heegaard move, denoted by $e'$.
% 
%
%
%So we need to prove that $\fmap_{\Fsphere}$ commutes with the isomorphisms corresponding to isotopies over the $\alpha$ or $\beta$ curves, handle-slides . This implies that $\fmap_{\Fsphere}$ induces a homomorphism from each vertex of the oriented graph corresponding to $\Tangle$ to the oriented graph corresponding to $\Tangle'$. Then, we should prove that 
% 
%Each  $\Ring$-diagram $H$ for $\Tangle$ corresponds to an isomorphism
%$$\Phi_{H}:\HFT^\Rin(H)\lra \HFT^\Rin(M,T,\la,\spinc),$$
%and the Heegaard moves $\underline\hmove$ changing $H_0$ to $H_1$ give 
%the isomorphims $\Phi_{{\underline\hmove} }$ such that
%$\Phi_{H_1}\circ \Phi_{{\underline\hmove}}=\Phi_{H_0}$.
%Correspondingly, we obtain a sequence of moves $\hmove$ which changes the 
%diagram $H_0'=H_0\#H_S$ to $H_1'=H_1\#H_S$, 
%where $H_S$ is the simple diagram corresponding to the standard
%model  $\ovl{S}$ in the above construction.
%Let us denote the chain map constructed in 
%Proposition~\ref{prop:connected-sum} from the diagram $H_0$  by 
%$\fmap_{H_0}$ and the corresponding map for $H_1$ 
%by $\fmap_{H_1}$. We then need to show that the  diagram 
%\begin{displaymath}
%\begin{diagram}
%\HFT_\Rin(H_0,\spinct|_{M})
%&\rTo{\fmap_{H_0}}&\HFT_\Rin(H_0'=H_0\# H_S,\spinct|_{M'})\\
%\dTo{\Phi_{\underline\hmove}}&&
%\dTo{\Phi_{\underline\hmove}}\\
%\HFT_\Rin(H_1,\spinct|_{M}) &\rTo{\fmap_{H_1}}&
%\HFT_\Rin(H_1'=H_1\# H_S,\spinct|_{M'}).
%\end{diagram}
%\end{displaymath}
%is commutative. 
%This is in turn
% reduced to the  case where  $\underline\hmove$ 
% consists of a single Heegaard 
%move. 

When $\hmove$ is a stabilization or destabilization,
the proof is straightforward, and the above diagram is in fact commutative in the level 
of chain complexes if the almost complex structure is chosen correctly. The remaining
cases are thus the cases where $\hmove$ is an isotopy (which may pass over
either of $w_1$ and $w_2$), or a handle slide.

We present the argument in the case where  
$\hmove$ is a $\beta$-isotopy or a handle-slide on $\betas$.
%In this case, $\HD_1'$ is obtained 
% from $\HD_0'$ by a Heegaard move, 
%which may also be denoted by 
%$\hmove$. 
Let $\gammas$ be 
a collection of curves which is obtained from $\betas$ by applying the Heegaard 
move $\hmove$, followed by small Hamiltonian isotopies. 
Choose a small isotope $\gamma$ of $\beta$ as well. Let 
$\gammas'=\gammas\amalg\{\gamma\}$. There is a generator associated with the Heegaard diagram $(\Sig,\betas,\gammas,\la,\spinc)$, denoted by $\Theta$, which represents the homology class $\Theta_{\beta\gamma}$. Furthermore, $\Theta'=\Theta\times\{\theta_{\beta\gamma}\}$ generates the top homology class $\Theta_{\beta'\gamma'}$ in $\HFT(\Sig',\beta',\gamma',\la',\spinc')$. 

% 
% 
% either of the Heegaard diagrams $(\Sig,\betas,\gammas,\la,\spinc)$ and 
% $(\Sig',\betas',\gammas',\la',\spinc')$, which represents the top generators $\Theta_{\beta\gamma}$ and $\Theta_{\beta'\gamma'}$, respectively.  
% 

% 
% 
% which will be denoted by 
% \[\Theta=\Theta_{\beta\gamma}\ \  \text{and}\ \  
% \Theta'=\Theta_{\beta'\gamma'}=\Theta\times \{\theta_{\beta\gamma}\},\] 
% respectively. 
The Heegaard triple $(\Sig,\alphas,\betas,\gammas,\la,\spinc)$ determines 
a holomorphic triangle map, and together with the generator $\Theta$, this 
gives a chain map
\[f=f_{\alpha\beta\gamma}:\CFT(\Sig,\alphas,\betas,\la,\spinc)\lra 
\CFT(\Sig,\alphas,\gamma,\la,\spinc).\]
Similarly, the Heegaard triple $(\Sig',\alphas',\betas',\gammas',\la',\spinc')$ and 
the generator $\Theta'$ determine 
a chain map
\[f'=f_{\alpha'\beta'\gamma'}:\CFT(\Sig',\alphas',\betas',\la',\spinc')\lra 
\CFT(\Sig',\alphas',\gamma',\la',\spinc').\]
To complete the proof, it is enough to show that 
$f'\circ f_{\Fsphere}=f_{\Fsphere}\circ f$.

For this purpose, fix the generator 
$\x$ of $\CFT( \Sig,\alphas,\betas,\la,\spinc)$ and suppose 
$$\Delta'\in\pi_2(\x\times \theta_{\alpha\beta},\Theta\times\theta_{\beta\gamma},
\y\times \theta)$$
is a triangle class contributing to $f'\circ f_{\Fsphere}$, where $\theta$ 
is one of the two intersection points in $\alpha\cap\gamma$.
$\Delta'$ gives the triangle classes 
\begin{displaymath}
\begin{split}
&\Delta\in\pi_2(\x,\Theta,\y)\ \ \ \  
\text{and} \ \ \ \ 
\Delta_S\in\pi_2(\theta_{\alpha\beta},\theta_{\beta\gamma},\theta)
\end{split}
\end{displaymath}
which support holomorphic representatives. Moreover we have
\begin{displaymath}
\begin{split}
&n_{w_i}(\Delta)=n_{w_i}(\Delta_S)\ \ \text{for}\ i=1,2
\quad \quad\text{and}\quad\quad 
0=\mu(\Delta')=\mu(\Delta)+\mu(\Delta_S)-2n_{w_1}(\Delta_S)
-2n_{w_2}(\Delta_S).
\end{split}
\end{displaymath}
Since $\mu(\Delta_S)=|\theta|+2n_{w_1}(\Delta_S)+2n_{w_2}(\Delta_S)$
we find $\mu(\Delta')=\mu(\Delta)+|\theta|$. From here, we conclude that
$\theta$ is the top generator $\theta_{\alpha\gamma}$
and that $\mu(\Delta)=0$.

Let  $k_i=n_{w_i}(\Delta_S)$ for $i=1,2$ and consider
the evaluation maps
\begin{align*}
&\rho_{\Sig}=\rho_{w_1}\times\rho_{w_2}:\Mod(\Delta)\ra \Sym^{k_1}(\D)\times \Sym^{k_2}(\D)\ \ \text{and}\\ 
&\rho_S=\rho_{w_1}\times\rho_{w_2}:\Mod(\Delta_S)\ra \Sym^{k_1}(\D)\times \Sym^{k_2}(\D).
\end{align*}
If $u'\in\Mod(\Delta')$ corresponds to the degeneration $u\star u_S$
with $u\in\Mod(\Delta)$ and $u_S\in\Mod(\Delta_S)$, we find
%\begin{displaymath}
\[\rho_{\Sig}(u)=\rho_{S}(u_S)\in\Sym^{k_1}(\D)\times \Sym^{k_2}(\D).\]
%\end{displaymath} 
There are several classes $\Delta_S$  with the property that 
$n_{w_i}(\Delta_S)=k_i$. Let 
$\Delta_{k_1,k_2}$ denote the set of all the above classes.
We thus obtain a map
\begin{displaymath}
\imath_{\Delta'}:
\Mod(\Delta')\lra \Mod(\Delta)\times_{\Sym^{k_1}(\D)\times \Sym^{k_2}(\D)}
\Mod(\Delta_{k_1,k_2}).
\end{displaymath}

By the argument of Proposition~\ref{prop:limit-moduli-space}, if the neck is sufficiently 
stretched and the paths of almost complex structures are generic 
$\imath_{\Delta'}$ 
is a bijection, reducing the
proof  to the second claim in Lemma~\ref{lem:three-curves-general}.
This completes the proof of Theorem~\ref{thm:one-handle}.
\end{proof}

\subsection{Three-handles}
Let $\Tangle=[M,T,\la,\spinc]$ be an $\Ring$-tangle as before, and $\Fsphere\subset M\setminus T$ be a framed $2$-sphere. Further, assume that $\langle c_1(\spinc),[a(\Fsphere)]\rangle=0$, where $a(\Fsphere)$ is the attaching sphere of $\Fsphere$. Then, $\Fsphere$ specifies an $\Ring$-cobordism
\[\Cob=\Cob(\Fsphere)=[W(\Fsphere),F(\Fsphere)=T\times[0,1],
\la_\sur:\pi_0(T)=\pi_0(F(\Fsphere))\ra \Ring,\spinct_\spinc],\]
where $W(\Fsphere)$ is obtained from $M\times[0,1]$ by attaching a $3$-handle along 
the framed sphere $\Fsphere\times\{1\}\subset M\times \{1\}$. The $\SpinC$ class $\spinct=\spinct_{\spinc}$ is the $\SpinC$ structure on $W=W(\Fsphere)$ determined by $\spinc$ i.e. $\spinct|_{M}=\spinc$.
%We require that $\langle c_1(\spinc),[\Fsphere]\rangle=0$, and then 
%the $\SpinC$ structure $\spinct$ over $W=W(\Fsphere)$ is determined by $\spinc$.
The cobordism 
$\Cob(\Fsphere)$ connects $\Tangle$ to 
\[\Tangle'=\Tangle(\Fsphere)=[M'=M(\Fsphere),T'=T\times\{1\},
\la':\pi_0(T')=\pi_0(T)\ra\Ring,\spinc'=\spinct|_{M'}].\]
If we reverse the cobordism $\Cob(\Fsphere)$ we obtain an 
$\Ring$-cobordism from $\Tangle'$ to $\Tangle$ which corresponds to attaching
a $1$-handle.

Consider a Heegaard surface $\Sig$ for $M$ that cuts the framed sphere 
$\Fsphere$ in a cylinder $S\subset \Sig$ with boundary components $C_1$ and $C_2$.  
A Heegaard surface $\Sig'$ for $M'$ is then obtained from $\Sig$ by removing 
$S$, and gluing a pair of disks with centers $w_1$ and $w_2$ to $C_1$ and $C_2$,
respectively. Let $\HD'=(\Sig',\alphas',\betas',\la':\z\ra \Ring,\spinc')$ be a
Heegaard diagram for $\Tangle'$. Consider a pair of Hamiltonian isotopic curves $\alpha$ and $\beta$ on $S$ that are both isotopic to the core of the cylinder and 
cut each other transversely in $\theta_{\alpha\beta}$ and $\theta_{\beta\alpha}$. Then,
\[\HD=(\Sig,\alphas=\alphas'\cup\{\alpha\},\betas=\betas'\cup\{\beta\},
\la:\z\ra \Ring,\spinc)\]
is a Heegaard diagram for $\Tangle$, provided that we use the stronger form 
of admissibility for $\HD'$, as discussed in   \cite[Remark 4.6]{AE-1}. 
%We thus obtain a chain map
%\[\fmap_{\Fsphere}:\CFT(\Tangle)\lra 
%\CFT(\Tangle').\]

Choose a generic path $J'$ of almost complex structures 
associated with the surface $\Sig'$. On the sphere $\ovl{S}$, obtained by attaching a pair of disks with centers $w_1$ and $w_2$ to $S$, consider a complex structure $J_S$, so that 
$(\ovl{S},w_1,w_2)$ is identified with $(\PP^1,0,\infty)$. Then, let $J$ be a generic path of almost 
complex structures on $\Sig$ sufficiently close to the join of 
$J$ and $J_S$.
%the standard complex structure $J_S$ on the sphere $\ovl{S}$ 
%obtained from $S$
%by attaching a pair of disks with centers $w_1$ and $w_2$ to $S$, so that 
%$(\ovl{S},w_1,w_2)$ is identified with $(\PP^1,0,\infty)$.
The chain map 
\[f_{\Fsphere}:\CFT_{J}(\HD)\lra \CFT_{J'}(\HD').\]
is defined as follows. For a generator 
\[\x=\x'\times \theta\in\Ta\cap\Tb=(\mathbb{T}_{\alpha'}\cap
\mathbb{T}_{\beta'})\times (\alpha\cap \beta)\] 
let 
\[f_{\Fsphere}(\x)=\begin{cases} \x'\ \ \ \ \text{if}\ \ \ \ \theta=\theta_{\beta\alpha}\\
0\ \ \ \ \ \text{otherwise}
\end{cases}\]
%
%(where this latter complex is defined using the Heegaard diagram $\HD$ and the path 
%$J$ of almost complex structures), define
%$\fmap_{\Fsphere}(\x):=\x'$ in 
%$\CFT(\Tangle')$ if $\theta=\theta_{\beta\alpha}$ is the intersection point
%with lower homological grading (the bottom intersection point), and 
%define $\fmap_{\Fsphere}(\x)=0$, otherwise.
%Here we assume that $\CFT(\Tangle')$ is defined using the 
%Heegaard diagram $H'$ and the path $J'$ of almost complex structures.\\

An argument, which is basically the dual of the argument used for $1$-handles 
implies the following theorem.

\begin{thm}\label{thm:three-handle}
With the above notation fixed, $f_{\Fsphere}$ is a chain map, and for any $\Ring$-module $\Rin$, it induces a natural homomorphism 
\begin{displaymath}
\fmap^\M_{\Fsphere}=\fmap^{\M}_{\Cob}:\HFT^\Rin(\Tangle)
\lra \HFT^\Rin(\Tangle').
\end{displaymath}

%Let $\Ring$ denote a commutative algebra over $\Z/2\Z$, 
% $\Rin$ denote an $\Ring$-module and $\Tangle=[M,T,\spinc,\la]$ denote an 
%$\Ring$-tangle. Let $\Fsphere\subset M-T$ be a framed $2$-sphere such that
%$\langle c_1(\spinc),[\Fsphere]\rangle=0$ and 
%\[\Cob=\Cob(\Fsphere):\Tangle\leadsto \Tangle'
%=\Tangle(\Fsphere)\] denote the $\Ring$-cobordism
%which is obtained from the product  $\Tangle\times[0,1]$ by adding 
%a $3$-handle along  $\Fsphere$. Then the
%map $\fmap_{\Cob}=\fmap_{\Fsphere}$ defined above is a 
%chain map and  induces a natural
%homomorphism
%\begin{displaymath}
%\begin{split}
%\fmap^\M_{\Cob}:\HFT^\Rin(\Tangle)
%\lra \HFT^\Rin(\Tangle').
%\end{split}
%\end{displaymath} 
\end{thm}

\newpage

\section{Framed arcs, framed knots and cobordism maps}\label{sec:map}
%\todo{AA Write an intro paragraph for this Section}
\subsection{Special tangles corresponding to cobordisms}\label{SpTangle}
Assume $\Cob=[W,\sur,\spinct,\la_\sur]$ is an $\Ring$-cobordism 
from the $\Ring$-tangle
$\Tangle= [M,T,\spinc,\la]$ to another $\Ring$-tangle 
$\Tangle'=[M',T',\spinc',\la']$.
In this subsection we introduce an $\Ring$-tangle $\Tangle_\sur$ associated with
$\Cob$ playing the role of $\#^kS^1\times S^2$ in 
defining the chain maps in Heegaard Floer theory. The important 
feature of $\Tangle_\sur$ is the existence of a distinguished generator
$\Theta_\sur\in \HFT(\Tangle_\sur)$, which plays the role of the 
{\emph{top generator}} in $\HFT^-(\#^kS^1\times S^2)$.

Denote the stable cobordism $(W,\sur)$ by $\Cobb$. Recall that the positive boundary of $\Cobb$,  denote by 
$\del^+\Cobb=(M^+,T^+)$, is identified with the product tangle  
\[(M^+,T^+)=(\del^+M\times [0,1], \del^+ T\times [0,1]).\]
Thus, denote
\begin{align*}
&(\del^+M^+,\del^+T^+)=(\del^+M\times\{1\},\del^+T\times \{1\})\ \ \text{and}\ \ 
(\del^-M^+,\del^-T^+)=(\del^+M\times\{0\},\del^+T\times \{0\}).
\end{align*}

Let $\arcj\subset \sur$ be a properly embedded, simple arc such that
$\del\arcj\subset T^+$. Associated with $\arcj$ we define a 
pair $(M_{\arcj},T_{\arcj})$ by doing surgery on $(M^+,T^+)$ at  
points $\arcj\cap T^+$ with the framing induced by $\sur$. More precisely, 
let $\mathrm{nd}(\arcj)\subset W$ be a small tubular neighborhood of 
$\arcj$ in $W$. The intersection $\sur\cap\del\mathrm{nd}(\arcj)$ 
induces a framing on $\del\mathrm{nd}(\arcj)$. Using this framing, 
we define
\begin{align*}
&M_{\arcj}= \left(M^+\setminus M^+\cap \nd(\arcj)\right)\cup \partial \nd(\arcj)
\quad\text{and}\quad
T_{\arcj}=\del(\sur\setminus\nd(\arcj))\setminus\left(\del\sur\cap(\del W\setminus M^+)\right).
\end{align*}

If the end points of $J$ are on distinct connected components of $T^+$
then $(M_J,T_J)$ is a balanced tangle. In this case,
consider a small product neighborhood of $M^+$, and an identification of 
it with 
$M^+\times [0,\epsilon]$, such that 
$\sur\cap \left(M^+\times [0,\epsilon]\right)=T^+\times [0,\epsilon]$. Denote the 
$4$-manifold obtained from attaching the one-handle $\nd(\arcj)$ to 
$M^+\times [0,\epsilon]$ by 
$W_{\arcj}$. Let $\sur_\arcj=\sur\cap W_{\arcj}$ and note that 
$\Cobb_{\arcj}=(W_{\arcj},\sur_{\arcj})$ is a cobordism from 
$(M^+,T^+)$ to $(M_{\arcj},T_{\arcj})$. 

\begin{defn}
A set $\Sarc=\{\arcj_1,\ldots,\arcj_n\}$ of pairwise disjoint, properly embedded, 
simple arcs on $\sur$ satisfying $\del \arcj_i\subset T^+$ for $i=1,\ldots,n$ 
is called a \emph{spanning} set if each connected component of 
$\sur\setminus\left(\coprod_{i=1}^n\arcj_i\right)$ is a disc and contains exactly one connected 
component of $T$ and one connected component of $T'$.
\end{defn}

Consider a spanning set $\Sarc=\{\arcj_1,\ldots,\arcj_n\}$ of arcs on 
$\sur$. After doing surgery on 
$(M^+,T^+)$ along the elements of $\Sarc$, we get 
a balanced tangle, which is denoted by $(M_{\Sarc},T_{\Sarc})$.
In particular, if $\arcj$ is a single arc, 
$(M_{\{\arcj\}},T_{\{\arcj\}})=(M_\arcj,T_\arcj)$. The diffeomorphism type of $(M_J,T_J)$ does not change under an \emph{arc slide} move, i.e.  sliding one foot of an arc $J_i$ over another arc $J_j$, while fixing the rest of the arcs.  
\begin{lem}\label{lem:arc-slide}
The diffeomorphism type of the tangle $(M_{\Sarc},T_{\Sarc})$ is independent of the choice of $\Sarc$.
\end{lem}
\begin{proof}
Associated with any spanning set $\Sarc$ of arcs on $\sur$, we may define a Morse function $g:\sur\ra [0,1]$ 
and a gradient like vector field $\xi$, such that
\begin{enumerate}
\item $g|_{\sur\cap\del^{-}W}\equiv 0$, $g|_{\sur\cap\del^+W}\equiv 1$ and $g$ has no 
critical point in a neighborhood of $\del\sur$.
\item All critical points of $g$ have index one. 
\item The arcs in $J$ are the unstable manifolds of the 
critical points of $g$.
\end{enumerate}
It follows from \cite[Theorem 4.5]{GK} that any two Morse data $(g,\xi)$ and $(g',\xi')$ as 
above, can be connected by a sequence of critical point switches and isotopies of the 
gradient-like vector fields. Thus, the corresponding spanning sets can be connected by a 
sequence of arc slides. So, $(M_J,T_J)$ does not depend on the choice of $J$.
%
%Let $\Sarc$ and $\Sarc'$ be two sets of disjoint, properly embedded simple 
%arcs on $\sur$ satisfying the above conditions. Consider the Morse functions $g$ 
%and $g'$ together with gradient-like vector fields $\xi$ and $\xi'$ associated 
%with $\Sarc$ and $\Sarc'$, respectively. Theorem 4.5 in \cite{GK} implies that 
%the Morse data $(g,\xi)$ can be connected to $(g',\xi')$ by a sequence of critical 
%point switches and isotopies of the gradient-like vector field. Thus $\Sarc$ can 
%be connected to $\Sarc'$ by a sequence of arc slides. Here, an arc slide of $\arcj_1$ 
%over $\arcj_2$ is replacing $\arcj_1$ with the arc obtained by sliding one of the 
%feet of $\arcj_1$ over $\arcj_2$. As a result, since 
%$(M_{\Sarc},T_{\Sarc})$ does not change under arc slides, it is independent of  
%$\Sarc$. 
\end{proof}

We may thus denote the balanced tangle $(M_{\Sarc},T_{\Sarc})$ by 
$(M_{\sur},T_{\sur})$ and the cobordism 
$\Cobb_{\Sarc}$ 
from $(M^+,T^+)$ to $(M_{\sur},T_{\sur})$ 
by $\Cobb_{\sur}$. 

\begin{defn}
A spanning set $J=\{J_1,\ldots,J_n\}$ of arcs is called \emph{ordered}, if 
\begin{enumerate}
\item endpoints of each $J_i$ lie on distinct components of $T^+$, 
\item if for some $j<k$ an endpoint $e_j$ of $\arcj_j$ and an endpoint $e_k$ of 
$\arcj_k$ are on the same component of $T^+$, then $\pi(e_j)>\pi(e_k)$. Here, $\pi:T^+=\amalg_{i=1}^\el(\{p_i\}\times [0,1])\to [0,1]$ denotes the projection on $[0,1]$.
\end{enumerate}
\end{defn}

If a spanning set $J=\{J_1,\ldots,J_n\}$ of arcs is ordered, doing surgery along $J_1,\ldots,J_i$ on $(M^+,T^+)$ would result in a balanced tangle, for any $1\le i\le n$.

\begin{lem}
For any stable cobordism $(W,\sur)$, there is an ordered spanning set of arcs on $\sur$.
\end{lem}
\begin{proof}
The proof is straightforward. 
\end{proof}

The $\Ring$-coloring $\la_\sur$ induces an $\Ring$-coloring $\la_\sur\circ\imath$ on 
$(M_\sur,T_\sur)$, where $\imath:\pi_0(T_\sur)\to\pi_0(\sur)$ is defined by inclusion. 
Abusing the notation we denote this coloring by $\la_\sur$. Let 
$H_i\in \mathrm{H}_2(M_\sur,\Z)$ be the homology class 
represented by the belt sphere of the attached $1$-handle 
corresponding to $J_i$. Moreover, associated with every component $\partial^+_kM$ 
of $\partial^+M$, with $k=1,\ldots,\ell$, we obtain a homology class 
$H_k'\in \mathrm{H}_2(M_\sur,\Z)$.
Consider the $\SpinC$ class 
$\spinc_0\in\SpinC(M_\sur)$ such that $\langle c_1(\spinc_0),H_i\rangle=
\langle c_1(\spinc_0),H'_k\rangle=0$ for every $1\le i\le m$ and every 
$1\le k\le \ell$. We then define  $\Tangle_\sur=[M_\sur,T_\sur,\spinc_0,\la_\sur]$.

%
%Let $H_i\in \mathrm{H}_2(M_\sur,\Z)$ be the homology class 
%represented by the belt sphere of the attached $1$-handle 
%corresponding to $J_i$. Consider the $\SpinC$ class 
%$\spinc_0\in\SpinC(M_\sur)$ such that $\langle c_1(\spinc_0),H_i\rangle=0$ for every $1\le i\le m$. Then, we define  $\Tangle_\sur=[M_\sur,T_\sur,\spinc_0,\la_\sur\circ\imath]$ denote 
%the $\Ring$-tangle obtained from $(M_\sur,T_\sur)$, the $\SpinC$ 
%class $\spinc_0$ and the representation $\la_\sur\circ\imath$,
%where $\imath:\pii(T_\sur)\ra \pii(\sur)$ is induced by inclusion.
Similar to $\Tangle_\sur$, we define another tangle $\Tangle_F^b$ as follows.  Let us assume that 
$\del^+T=\{p_1,\ldots,p_\el\}\subset \del^+M$ and correspondingly, 
$T^+$ is a union of components  
\[T^+=\coprod_{i=1}^{\el}T^+_i=\coprod_{i=1}^{\el} \left(\{p_i\}\times [0,1]\right)
\subset M^+.\]
For every
$1\le i\le \el$ consider a point $\ovl{p}_i\in \del^+M$ close to $p_i$ and 
let $\ovl{T}_i^+=\ovl{p}_i\times [0,1]$ and 
$\ovl{T}^+:=\amalg_{i=1}^{\el}\ovl{T}_i^+$ in $M^+$. Setting $T_\sur^b=T_\sur\amalg\ovl{T}^+$ one gets a balanced tangle $(M_\sur,T_\sur^b)$.

Suppose that $\sur=\coprod_{i=1}^m\sur_i$ and $\del^+M=\amalg_{k=1}^\ell\del^+_kM$. Associated to each component $\del^+_kM$ of $\del^+M$, we define a monomial 
\[\law_k=\prod_{p_j\in \del_k^+M}\left(\la(p_j)\zet_j\right)
\in\F[\la_1,\ldots,\la_m,\zet_1,\ldots,\zet_\el]\]   
where $\la(p_j)=\la_i$ if {$p_j\in \sur_i$.
%Consider a variable $\la_i$ associated with each component $\sur_i$
%of $\sur$ and a variable $\zet_j$ associated with each component of 
%$\ovl{T}=\coprod_{j=1}^\el \ovl{T}_j$. Furthermore, assume that 
%$\del^+M=S=\coprod_{k=1}^\ell S_k$ are the connected components of 
%$\del^+M$, while $T_j$ is the connected component of $T_\sur$ which includes 
%$(p_j,0)\subset \del^+M\times[0,1]$  for $j=1,\ldots,\el$. 
%Define $\la(p_j)=\la_i$ if $p_j\in \sur_i$ and set
%\[\la^k=\prod_{p_j\in S_k}\left(\la(p_j)\zet_j\right)
%\in\F[\la_1,\ldots,\la_m,\zet_1,\ldots,\zet_\el].\]   
Denote the genus of $\del_k^+M$ by $g_k$. With this notation fixed, let  
\[\Ring_{\sur}^b:=\frac{\F[\la_1,\ldots,\la_m,\zet_1,\ldots,\zet_{\el}]}
{\langle \law_k\ |\ g_k>0,\ k=1,\ldots,\ell\rangle}.\]
The map $\la_{\sur}^b:\pi_0(T_{\sur}^b)\ra\Ring_{\sur}^b$ defined by
\begin{displaymath}
\begin{cases}
\la_\sur^b(\ovl{T}_j^+)=\zet_j\ \ \ \ \ &1\le j\le\el\\
\la_\sur^b(T_\sur^j)=\la(p_j) \ \ \ \ \ &1\le j\le \el,
\end{cases}
\end{displaymath}
is an $\Ring_\sur^b$-coloring of $(M_\sur,T_\sur^b)$. Here, $T_\sur^j$ denotes the component of $T_\sur$ that intersects $\del^+M$ in $p_j$. As a result, we get an $\Ring_F^b$-tangle $\Tangle^b=[M_\sur,T_\sur^b,\spinc_0,\la_\sur^b]$.

\begin{remark}
Let $\Ring$ denotes the $\Z$-algebra associated with the sutured manifold corresponding to $(M_\sur,T_\sur^b)$, see Section \ref{sec:alg}. Then $\Ring_\sur^b$ is the quotient of $\Ring\otimes\F$ by the ideal generated by the binomials $\la_i-\la_j$
for all $i$ and $j$, so that $T^i_\sur$ and $T^j_\sur$ lie on the boundary of the same component of $\sur$.
\end{remark}

The algebra $\Ring$ has a natural $\Ring_\sur^b$-module structure given by the homomorphism $\phi:\Ring^b_\sur\to\Ring$ defined as 
\[
\begin{cases}
\begin{split}
&\phi(\la_i)=\la_\sur(\sur_i)&1\le i\le m\\
&\phi(\zet_j)=1&1\le j\le\el.
\end{split}
\end{cases}
\]
Further, considering this module structure on $\Ring$ it is straightforward that
\[\HFT^\Ring(\Tangle^b)=\HFT(\Tangle_\sur).\]

\subsection{The distinguished generator}
For every $1\le i\le\el$, consider a product disk 
$(D_i,\del D_i)$ in the balanced tangle $(M^+,T^+\amalg\ovl{T}^+)$ 
such that 
\[\ovl{T}^+_i\amalg-T^+_i\subset \del D_i\ \ \ \text{and}\ \ \  
\del D_i\setminus(\ovl{T}_i^+\amalg-T^+_i)\subset\del M^+.\] 
Abusing the notation, we 
denote the intersection of $D_i$ with $M_{\sur}$ by $D_i$ as well. Let
$D=\amalg_{i=1}^{\el}D_i$. For every $1\le j\le n$, we may attach an 
oriented one handle $D'_j$ (i.e. a band) to $D$ such that $D'_j\subset M_{\sur}$ is 
embedded in the one handle associated with $\arcj_j$ and 
\[\del D'_j\setminus(\del D'_j\cap D)\subset \sur.\] Denote the resulting embedded, 
oriented surface in $(M_{\sur},T_{\sur}^b)$ by $\sur'$. 
Note that 
\[(\sur',\del\sur')\subset\left(M_{\sur},\del M_{\sur}
\cup T_\sur^b\right)\]
and that this pair is uniquely determined up to isotopy. 
%\todo{AA This looks correct to me, but could you also double check it please.}

Let $(X_{\sur},\tau_{\sur}^b)$ be the balanced 
sutured manifold associated with the tangle 
$(M_{\sur},{T}_{\sur}^b)$. The surface 
$\sur'\cap X_{\sur}\subset X_{\sur}$ is a decomposing surface for $(X_\sur,\tau_\sur^b)$, in the sense of \cite[Definition 2.7]{Juh-surface}. It is straightforward to check that 
this surface  decomposes $(X_\sur,\tau_\sur^b)$ into a product sutured manifold. Thus, 
there exists a unique relative 
$\SpinC$ structure 
$\relspinc_{\sur'}\in\SpinC(X_{\sur},\tau_{\sur}^b)$ 
which is \emph{outer} with respect to 
$\sur'$. Recall that 
$\relspinc\in\SpinC(X_{\sur},\tau_{\sur}^b)$ 
is called {outer} with respect to $\sur'$, if it is represented by a unit vector field 
$v$ on $X_{\sur}$ such that $v_{p}\neq -(\nu_{\sur'})_p$ for every $p\in\sur'$. 
Here $\nu_{\sur'}$ is the unit normal vector field of $\sur'$ with respect 
to some Riemannian metric on $X_{\sur}$, See \cite[Definition 1.1]{Juh-surface}.
Note that $[\relspinc_{\sur'}]=\spinc_0$.

If the components $T_\sur^{j_1}$ and $T_\sur^{j_2}$ of $T_\sur$ lie on the same component $\sur_i$ of $\sur$, then 
\[[\tau^{j_1}_\sur]=[\tau^{j_2}_\sur]\in \Ht_1(X_{\sur},\Z)\]
where $\tau_\sur^j$ denotes the suture associated with $T_\sur^j$.  Thus, for any component $\sur_i$ of $\sur$, there is a well-defined homology class $h_i=[\tau^j_\sur]\in \Ht_1(X_{\sur},\Z)$ where $T_{\sur}^j\subset F_i$.  As a result, the algebra $\Ring_{\sur}^b$ admits a filtration by 
$\Hbb=\Ht^2(X_{\sur},\del X_{\sur},\Z)$, which is defined by
\begin{displaymath}
\begin{split}
&\chi:G(\Ring_{\sur}^b)\rightarrow \Hbb=\Ht^2(X_{\sur},\del X_{\sur},\Z)\\
&\chi(\prod_{i=1}^{m}\la_i^{a_i}\prod_{j=1}^\el\zet_j^{b_j}):
=\sum_{i=1}^m a_i\PD(h_i)+\sum_{j=1}^\el b_j\PD(\ovl{\tau}_j^+)
\end{split}
\end{displaymath}
where $\ovl{\tau}_j^+$ is the suture corresponding to
$\ovl{T}_j^+\subset T^b_\sur$. 
Therefore, for every
$\spinc\in\SpinC(M_\sur)$
$$\CFT(M_{\sur},T_{\sur}^b,\la_{\sur}^b,\spinc)$$
may be decomposed into sub-complexes associated with 
relative $\SpinC$ structures. 
More precisely, let 
$\HD=(\Sig,\alphas,\betas,\la_\sur^b:\z\amalg\ovl{\z}\ra\Ring_\sur^b,\spinc_0)$ 
be an $\Ring_{\sur}^b$-diagram for $\Tangle^b_\sur$, where 
$\z=\{z_1,\ldots,z_\el\}$ and $\ovl{\z}=\{\ovl{z}_1,\ldots,\ovl{z}_\el\}$, while 
$z_i=T_{\sur}^i\cap\Sig$ and $\ovl{z}_i=\ovl{T}_i^+\cap\Sig$. Then 
$$\CFT(\Sig,\alphas,\betas,\la_\sur^b,\spinc_0)
=\bigoplus_{\relspinc\in\spinc_0\subset\SpinC(M_{\sur},
T_{\sur}^b)}\CFT(\Sig, \alphas,\betas,\la_{\sur}^b,\relspinc)$$
and $\CFT(\Sig, \alphas,\betas,\la_\sur^b,\relspinc)$ is generated by the 
elements $\la.\x$, where $\la\in G(\Ring_{\sur}^b)$ is a monomial in $\Ring_\sur^b$ 
and $\x\in\Ta\cap\Tb$ is an intersection point satisfying 
$$\relspinc(\la.\x):=\relspinc(\x)+\chi(\la)=\relspinc.$$

For a relative class $\relspinc\in\spinc_0$, let 
\[\la=\prod_{i=1}^m\la_i^{a_i}\prod_{j=1}^{\el}\zet_j^{b_j}\ \  \text{and}\ \  
\la'=\prod_{i=1}^m\la_i^{c_i}\prod_{j=1}^\el\zet_{j}^{d_j}\] be non-zero monomials 
in $G(\Ring_{\sur}^b)$ and $\x,\y\in\Ta\cap\Tb$ be intersection points such that 
$\relspinc(\la.\x)=\relspinc(\la'.\y)=\relspinc$. Then, we say $\phi\in\pi_2(\x,\y)$ 
connects $\la.\x$ to $\la'.\y$ if 
\begin{displaymath}
\begin{cases}
a_i+\sum_{T^j_\sur\subset \sur_i}n_{z_j}(\phi)=c_i\ \ \ \ \ &1\le i\le m\\
b_j+n_{\ovl{z}_j}(\phi)=d_j\ \ \ \ \ &1\le j\le\el.
\end{cases}
\end{displaymath}
Moreover, we define the relative grading of $\la.\x$ and $\la'.\y$ by
$$\gr(\la.\x,\la'.\y)=\mu(\phi)$$
where $\phi\in\pi_2(\x,\y)$ is a disk connecting $\la.\x$ to $\la'.\y$. 

\begin{lem}
The relative grading $\gr$ is well-defined and is independent of the choice of 
$\phi$. It induces a relative grading on the chain complex $\CFT(M_\sur,T_{\sur}^b,\la_{\sur}^b,\relspinc)$
and the differential lowers this grading by one.
\end{lem}
\begin{proof}
Suppose $\phi,\phi'\in\pi_2(\x,\y)$ connect $\la.\x$ to $\la'.\y$. 
Then, $\Dcal=\Dcal(\phi)-\Dcal(\phi')$ is a periodic domain such that \[\sum_{T_\sur^j\subset\sur _i}n_{z_j}(\Dcal)=0\ \ \ \ \ \text{and}\ \ \ \ \ n_{\ovl{z}_j}(\Dcal)=0\]
for any $i$ and $j$. Thus $\Dcal=\sum_{i=1}^nm_i\Pcal_i$ where $\Pcal_i$ is the periodic domain associated with $H_i$. Since,  $\langle c_1(\spinc_0),H_i\rangle=0$ for any $i$, 
\[\mu(\phi)-\mu(\phi')=\langle c_1(\spinc_0), H(\Dcal)\rangle =0. \]
%
%
% and 
%$\Sig-\{\alphas\}=\amalg_{k=1}^l A_k$. From here we find
%\begin{align*}
%\mu(\phi)-\mu(\phi')&=\sum_{i=1}^nm_i\langle 
%c_1(\spinc_0),H_i\rangle+\sum_{k=1}^ln_k\langle c_1(\spinc_0),H(A_k)\rangle\\
%&=\sum_{k=1}^ln_k\langle c_1(\spinc_0),H(A_k)\rangle.
%\end{align*}
%Moreover, $\sum_{T_j\subset\sur_i}n_{z_j}(\Pcal_k)=0$ for every $i$ and $k$. 
%Hence, $n_k=0$ for all $k$, which implies that $\mu(\phi)=\mu(\phi')$.
\end{proof}

\begin{prop}\label{prop:generator}
With respect to the relative grading defined above, the \emph{top-dimensional} homology 
group in $\HFT(M_\sur,T_{\sur}^b,\la_{\sur}^b,\relspinc_{\sur'})$ is 
isomorphic to $\F$. If $\Theta_{\sur}^b$ denotes the generator of this 
homology group, then for the homomorphism
\begin{align*}
\phi:\HFT(M_\sur,T_{\sur}^b,\la_\sur^b,
\relspinc_{\sur'})\lra&\HFT^{\F}(M_\sur,T_{\sur}^b,\la_{\sur}^b,
\relspinc_{\sur'})=\mathrm{SFH}(X_\sur,\tau_{\sur}^b,\relspinc_{\sur'})=\F,
\end{align*}
 we have $\phi(\Theta_{\sur}^b)=1$.  
Here, $\F$ is an $\Ring_{\sur}^b$-module  with trivial module structure and 
the homomorphism $\phi$ is induced from the surjection $\Ring_\sur^b\ra \F$,
which sends all variables to zero.
\end{prop}

\begin{proof}
Let $\pi:T^+=\amalg_{i=1}^\el(\{p_i\}\times[0,1])\ra [0,1]$ denote the projection on $[0,1]$. 
Consider an ordered set of spanning arcs  $\Sarc=\{\arcj_1,\cdots .,\arcj_n\}$ on $\sur$.
%{\todo{You have defined the ordered set of spanning arcs before.
%}}
The result of surgery on $(M^+,T^+\amalg\ovl{T}^+)$ along 
$\arcj_1,\ldots,\arcj_i$ is a balanced tangle for any $i$, denoted by  $(M^{i},T^{i}\amalg\ovl{T}^{i})$. Note that $\ovl{T}^i=\ovl{T}^+$ for any $i$, and $(M^n,T^n\amalg\ovl{T}^n)=(M_\sur,T_{\sur}^b)$.

Let $\spinc^i_0$ denote the $\SpinC$ structure on $M^{i}$ 
such that for any $1\le j\le i$ we have $\langle c_1(\spinc^i_0),H_j\rangle=0$ and 
for $1\le k\le \ell$ we have $\langle c_1(\spinc^i_0),H_k'\rangle=0$.
%{\todo{I added an extra condition, as explained before. Please check!}} 
Thus, $\spinc^n_0=\spinc_0$.
Assume that the components of $T^{i}=\amalg_{j=1}^{\el}T^i_j$ are 
labeled such that $T_j^i\cap \del^+M=\{p_j\}$. Inductively, we define an $\F$-algebra  $\Ring^i$ and a map 
$\la^i:\pi_0(T^i\amalg\ovl{T}^i)\ra \Ring^i$, such that 
$\Tangle^i=[M^i,T^i\amalg\ovl{T}^i,\spinc_0^i,\la^i]$ becomes an $\Ring^i$-tangle. For $i=0$, we have
$(M^0,T^0\amalg\ovl{T}^0)=(M^+,T^+\amalg\ovl{T}^+)$. Let  
$$\Ring^0:=\frac{\Z_2[\zet_1',\ldots,\zet_{\el}',\zet_1,\ldots,\zet_\el]}
{\langle \la^k\ |\ g_k>0  \rangle}$$
where $\la^k=\prod_{p_j\in \del^+_kM}(\zet_j\zet_j')$. An $\Ring^0$-coloring $\la^0:\pi_0(T^0\amalg\ovl{T}^0)\ra \Ring^0$ is defined as 
\[\la^0(T^0_j)=\zet_j'\ \ \ \text{and}\ \ \  \la^0(\ovl{T}^0_j)
=\zet_j\ \ \ \ \ \ \ \ \ \ \text{for~every}\ \ \ 1\le j\le \el.\]
Then, if $\arcj_i$ connects $T^{i-1}_a$ and $T^{i-1}_b$,
we define 
$$\Ring^i:=\frac{\Ring^{i-1}}{\langle 
\la^{i-1}(T^{i-1}_a)-\la^{i-1}(T^{i-1}_b) \rangle}$$
and $\la^{i}:\pi_0(T^{i}\amalg\ovl{T}^{i})\ra\Ring^i$ is 
the map induced by $\la^{i-1}$.
Note that $\Ring^n=\Ring_{\sur}^b$.

We may also construct an $\Ring^i$-diagram 
\[\HD^i=(\Sig^i,\alphas^i,\betas^i,\la^{i}:\z^i\amalg\ovl{\z}^i\to\Ring^i,\spinc_0^i)\] for $\Tangle^i$, inductively. For $i=0$, $\Tangle^0$ is a product tangle so let $\Sig^0=\del^+M\times \{1/2\}$, $\alphas^0=\betas^0=\emptyset$, $\z=\{p_1,\ldots,p_{\el}\}\times\{1/2\}$ and $\ovl{\z}=\{\ovl{p}_1,\ldots,\ovl{p}_{\el}\}\times\{1/2\}$. Given an $\Ring^{i-1}$-diagram 
$\HD^{i-1}$ for $\Tangle^{i-1}$, we construct an $\Ring^i$-diagram
for $\Tangle^i$ as follows. As before, suppose $J_i$ connectes $T^{i-1}_a$ to $T^{i-1}_b$.

\begin{enumerate}
\item The Heegaard surface $\Sig^i$ is obtained from $\Sig^{i-1}$ by adding a one-handle with feet 
near $z_a$ and $z_b$, where $z_a=T^{i-1}_{\bullet}\cap \Sig^{i-1}$ for $\bullet=a,b$.
\item $\alphas^i=\alphas^{i-1}\cup \{\alpha_i\}$ where $\alpha_i$ is the core of the attached one-handle.
\item $\betas^i=\betas^{i-1}\cup\{\beta_i\}$ where $\beta_i$ is a Hamiltonian  translate of $\alpha_i$, intersecting $\alpha_i$ in a pair of canceling 
intersection points.  The area bounded between $\alpha_i$ and $\beta_i$ gives 
a $2$-chain $\Pcal=D_i^+-D_i^-$ with $\del \Pcal=\alpha_i-\beta_i$, the components 
$D_i^+$ and $D_i^-$ 
of $\Sig^i-\alphas^i-\betas^i$ are bigons  such that one of them intersects 
$\z$ in $z_a$ while the other one intersects $\z$ in $z_b$. 
\item $\z^i=\z^{i-1}$ and $\ovl{\z}^i=\ovl{\z}^{i-1}$.
\end{enumerate}

%
%We may construct an $\Ring_i^+$-diagram for 
%$\Tangle^i=[M^i,T^i\amalg\ovl{T}^i,\spinc_0^i,\la_i^+]$ using an 
%$\Ring_{i-1}^+$-diagram 
%\[H^{i-1}=(\Sig^{i-1},\alphas^{i-1},\betas^{i-1},\la_{i-1}^+:
%\z^{i-1}\amalg\ovl{\z}^{i-1}\ra\Ring^+_{i-1},\spinc_0^{i-1})\]
%for $\Tangle^{i-1}$ by setting  
%$$H^i=(\Sig^i,\alphas^i=\alphas^{i-1}\cup\{\alpha_i\}, 
%\betas^{i}=\betas^{i-1}\cup\{\beta_i\},\la_i^+:\z^i\amalg\ovl{\z}^i\ra\Ring_i^+,
%\spinc_0^i),$$
% $\alpha_i$ is the core of the one-handle and 
%$\beta_i$ is a Hamiltonian  translate of $\alpha_i$.  Moreover, 
%$\z^{i}=\z^{i-1}$, $\ovl{\z}^i=\ovl{\z}^{i-1}$ and $\la_{i}^+$ is the map induced by 
%$\la_{i-1}^+$. Furthermore, $\beta_i$ intersects $\alpha_i$ in a pair of canceling 
%intersection points and the area bounded between $\alpha_i$ and $\beta_i$ gives 
%the $2$-chain  $\Pcal=D_i^+-D_i^-$ with $\del \Pcal=\alpha_i-\beta_i$, 
%where the  connected components 
%$D_i^+$ and $D_i^-$ 
%of $\Sig^i-\alphas^i-\betas^i$ are bigons  such that one of them intersects 
%$\z$ in $z_a$ while the other one intersects $\z$ in $z_b$. \\

As a result, we get an 
$\Ring_\sur^b$-diagram 
$$\HD:=\HD^n=(\Sig,\alphas=\{\alpha_1,\ldots,\alpha_n\},\betas=\{\beta_1,\ldots,\beta_n\},
\la^b_\sur:\z\amalg\ovl{\z}\ra\Ring^b_\sur,\spinc_0)$$
for $\Tangle^b$.
Any pair of curves $(\alpha_i,\beta_i)$ intersect in a pair of points $x_{i}^+$ and  
$x_{i}^-$, so that the bi-gons $D_i^+$ and $D^-_i$ connect $x_i^+$ to $x_i^-$. 
Thus, corresponding to any map $\epsilon:\{1,\ldots,n\}\ra\{+,-\}$ we have an 
intersection point 
$$\x^\epsilon=\{x_1^{\epsilon(1)},\ldots,x_n^{\epsilon(n)}\}\in\Ta\cap\Tb.$$
By an argument similar to the argument of Proposition \ref{prop:connected-sum},
the map 
$$f_i:\CFT(\Sig^{i-1},\alphas^{i-1},\betas^{i-1},
\la^{i-1},\spinc_0^{i-1})\otimes\Ring^{i}\ra\CFT(\Sig^i,\alphas^i,\betas^i,\la^{i},
\spinc_0^i)$$
which is defined by $f_i(\y)=\y\times\{x_{i}^+\}$, is a chain map if the path of 
almost complex structures is chosen correctly. 
In particular, the intersection point $\theta^+:=\{x_1^+,\ldots,x_{n}^+\}$ is closed in 
the chain complex
\[\CFT(\Sig,\alphas,\betas,\la^b_\sur,\spinc_0)=\CFT(\Tangle^b),\] 
i.e. we have $\del\theta^+=0$.\\

For every $1\le i\le n$, let 
$(\sur_i',\del\sur_i')\subset (M^i,\del M^i\cup T^i\cup\ovl{T}^i)$ denote the 
embedded, oriented surface obtained from $D$ by attaching the embedded one 
handles $D'_1,\ldots,D'_i$. Each $\sur_i'$ is then a decomposing surface in 
$(M^i,T^i\amalg\ovl{T}^i)$. We may inductively construct a closed subsurface 
$P^i\subset \Sig^i$ such that the Heegaard diagram  $\HD^i$, together with $P^i$, is 
a diagram adapted to $\sur_i'$ in the sense of \cite[Definition 4.3]{Juh-surface}. 
More precisely, $P^i$ is a closed subsurface of $\Sig^i$ such that the boundary of 
$P^i$ is a union of polygons, whose vertices are $P^i\cap (\z\amalg\ovl{\z})$ and 
its edges are decomposed as $\del P^i=A^i\cup B^i$ where 
\[A^i\cap B^i\subset \z\amalg\ovl{\z},\  \ \ \alphas^i\cap B^i=\emptyset\ \ \  \text{and} 
\ \ \ \betas^i\cap A^i=\emptyset.\] 
Finally, the equivalence class of $\sur_i'$ is given by 
smoothing the corners of 
\[\left(P^i\times\left\{{1}/{2}\right\}\right)\cup
\left(A^i\times \left[{1}/{2},1\right]\right)\cup
\left(B^i\times \left[0,{1}/{2}\right]\right)\subset 
\left(M^i,T^i\amalg\ovl{T}^i\right).\]

For every $i$, suppose that $\arcj_i$ connects $T^{i-1}_{a_i}$ to $T^{i-1}_{b_i}$. 
For $i=1$, $\sur_1'$ is obtained from $D$ by attaching a one-handle to $T^0_{a_1}$ and 
$T^0_{b_1}$. Thus, $P^1$ is a union of a rectangle whose vertices are 
$z_{a_1}, z_{b_1}, \ovl{z}_{a_1}$ and $\ovl{z}_{b_1}$ and contains the intersection 
point $x_1^-$, together with $n-2$ bigons disjoint from $\alphas^1$ and $\betas^1$, 
whose vertices are $z_j$ and $\ovl{z}_j$ for $j\neq a_1,b_1$.  For any 
$i>1$, the subsurface $P^i$ may be constructed from $P^{i-1}$ by attaching an 
embedded one-handle in $\Sig^i$ which intersects $\alpha_i$ and $\beta_i$ and 
contains $x_i^-$, as illustrated in Figure \ref{fig:Surface-proj}.\\

\begin{figure}%[ht!]
\def\svgwidth{12cm}
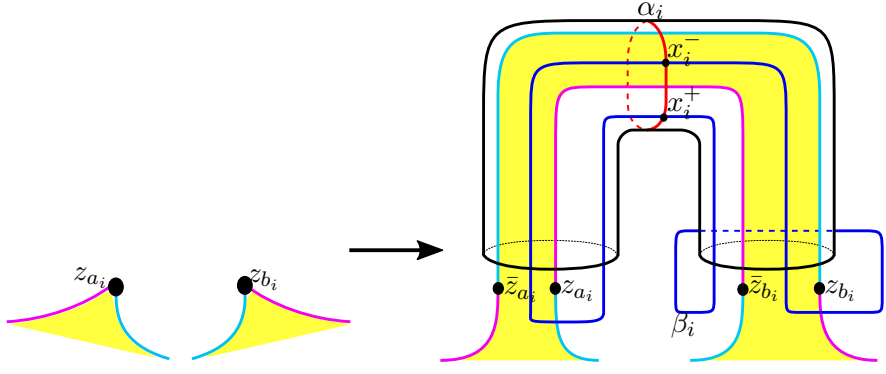
\caption{
The $i$-th one-handle is attached over the marked points $z_{a_i}$ and $z_{b_{i}}$.
The curves $\alpha_i$ and $\beta_i$, their intersection points $x_i^+$ and $x_i^-$ and 
the modification changing $P^{i-1}$ to $P^i$ are illustrated.
}\label{fig:Surface-proj}
\end{figure}

Thus, for $i=n$, ${\theta}^+$ is the only intersection point that does not 
intersect $P^n$. As a result, ${\theta}^+$ is the only intersection point for 
which $\relspinc({\theta}^+)=\relspinc_{\sur'}$ and thus ${\theta}^+$ 
is the generator of 
\[\mathrm{SFH}(X_\sur,\tau^b_{\sur},\relspinc_{\sur'})=\F.\]

For every $\epsilon:\{1,\ldots,n\}\ra\{+,-\}$, there is a positive disk 
\[\phi\in\pi_2(\theta^+,\x^{\epsilon})\ \ \  \text{with}\ \ \  
\mu(\phi)=\#\{i\ |\ \epsilon(i)=-\}.\] 
Furthermore, $\relspinc(\la_\sur^b(\phi).\x^{\epsilon})=\relspinc_{\sur'}$ and if for a 
monomial $\la\in\Ring_{\sur}^b$ we have 
$\relspinc(\la.\x^{\epsilon})=\relspinc_{\sur'}$ then $\la={\la}_{\sur}^b(\phi). \ti{\la}$ 
for some  $\ti{\la}\in\Ring_{\sur}^b$. Hence $\Theta_\sur^b={\theta}^+$,
which is closed by our earlier considerations, generates the 
top-dimensional homology group in 
\[\HFT(M_{\sur},T_{\sur}^b,{\la}^b_{\sur},\relspinc_{\sur'})\] 
with respect to the relative grading defined above. The above observations complete the 
proof of the proposition.
\end{proof}

The algebra $\Ring$ has a natural $\Ring_\sur^b$-module structure given by the homomorphism $\phi:\Ring^b_\sur\to\Ring$ defined as 
\[
\begin{cases}
\begin{split}
&\phi(\la_i)=\la_\sur(\sur_i)&1\le i\le m\\
&\phi(\zet_i)=1&1\le i\le\el.
\end{split}
\end{cases}
\]
Further, considering this module structure on $\Ring$ it is straightforward that
\[\HFT^\Ring(\Tangle^b)=\HFT(\Tangle_\sur).\]
%Let $\Tangle^+$ be the  
%$\Ring_{\sur}^+$-tangle corresponding to an $\Ring$-cobordism 
%\[\Cob(\Farc,\Fsphere,\spinct)=[W,\sur,\spinct=\spinct,\la_\sur]:
%\Tangle=[M,T,\spinc,\la]\leadsto\Tangle'=[M',T',\spinc'=\spinc,\la']\] 
%which corresponds to some acceptable set 
%$\Farc$ of framed arcs, a framed link $\Fsphere$ and a $\SpinC$ 
%structure $\spinct\in \SpinC(W)$. 
%Then, $\Ring$ is an $\Ring_{\sur}^+$-module via a 
%homomorphism 
%$\phi:\Ring_{\sur}^+\ra\Ring$ defined by
%\begin{displaymath}
%\phi(\la_i)=\la_{\sur}(\sur_i)\  \ \ \text{for}\ 1\le i\le m\ \ \ \ \ \text{and}\ \ \ \ \ 
%\phi(\zet_j)=1\ \ \ \text{for}\  1\le j\le \el.
%\end{displaymath}
%On the other hand, if we abuse the notation and denote $\la_F\circ \phi$ by 
%$\la_F$ as well, \[\Tangle_F=[M_{\sur},T_{\sur},\spinc_0,\la_F:\pi_0(T_\sur)\ra \Ring]\] 
%is an $\Ring$-tangle. Furthermore, 
\begin{defn}\label{def:distinguished-generator}
The image of the homolog class $\Theta_\sur^b\otimes 1$ under the homomorphism
\[\HFT(\Tangle^b)\otimes_\phi\Ring\to\Ht_\star(\CFT(\Tangle^b)\otimes_\phi\Ring)=\HFT^\Ring(\Tangle^b)=\HFT(\Tangle_\sur)\]
is denoted by $\Theta_\sur$ and is called the  {\emph{distinguished generator}}  of 
$\HFT(\Tangle_\sur)$.\\ 
\end{defn}

\subsection{Framed arcs, framed knots and  the cobordism map} 
Let $(M,T)$ be a balanced tangle, $\Farc=\{\Farc_1,\ldots,\Farc_n\}$ be an acceptable 
set of framed arcs in $(M,T)$ and $\Fsphere=\{\Fsphere_1,\ldots,\Fsphere_m\}$ be a 
set of framed circles in $M\setminus (\Farc\cup T)$. In other words, each $\Fsphere_i$ 
is determined by a framing on a knot $K_i$ and each $\Farc_i$ is determined 
by a framing on an arc $I_i$. Let $L=\amalg_{i=1}^m K_i$ and $I=\amalg_{i=1}^nI_i$. As discussed in Section \ref{subsec:elem-cob}, $(\Farc,\Fsphere)$ specifies a cobordism $\Cobb=(W,F)=\Cobb(\Farc,\Fsphere)$ from $(M,T)$ to $(M',T')=(M(\Fsphere),T(\Farc))$. Here, $T(\Farc)$ is constructed from band surgery on $T$ along $\Farc$ and
$M(\Fsphere)$ is constructed from $M$ by surgery along the framed link $\Fsphere$.\\

%
%
%Moreover, $(M',T')=(M(\Fsphere),T(\Farc))$
%where $T(\Farc)$ is constructed from band surgery on $T$ along $\Farc$ and
%$M(\Fsphere)$ is constructed from $M$ by surgery along the framed link $\Fsphere$.\\

\begin{defn} With the above notation fixed, a \emph{stalk} $s(\Fsphere)=\{s_1,\ldots,s_m\}$ for the framed link $\Fsphere$ (disjoint from $\Farc$) in $(M,T)$ is a set of embedded arcs in $M\setminus (T\cup\Farc)$ such that $s_i$ 
connects $K_i$ to $\del^+M\setminus T$.
\end{defn}

For every $i$, if $\Farc$ intersects $T_i$, let $r_i\subset T_i$ be the segment where 
\[\del^+ r_i\subset \del^+M,\ \ \  \del^-r_i\subset T_i\cap\Farc\ \ \  
\text{and}\ \ \  (T_i-r_i)\cap\Farc=\emptyset.\]  
If $T_i\cap\Farc=\emptyset$, let $r_i=\emptyset$ and denote $r(\Farc)=\amalg_{i=1}^\el r_i$. Let
\[B(\Farc)=I\cup r(\Farc)\ \ \ \ \text{and}\ \ \ \  B(\Fsphere)=L\cup s(\Fsphere).\]

%
%\[B=B(\Farc)\cup B(\Fsphere,s(\Fsphere))=
%\left((\coprod_{i=1}^n I_i)\cup(\coprod_{i=1}^\el r_i)\right)\cup
%\left(\coprod_{i=1}^m (K_i\cup s_i)\right).\] 

Consider small tubular 
neighborhoods $\nd(r(\Farc))$ and $\nd(s(\Fsphere))$ of $r(\Farc)$ and 
$s(\Fsphere)$, respectively and let
\[N=\Farc\cup\Fsphere\cup\nd(r(\Farc))\cup\nd(s(\Fsphere))\]
be the resulting neighborhood of $B=B(\Farc)\cup B(\Fsphere)$. For each $i=1,\ldots,n$, 
let $A_i\subset \del N$ be a sphere with $4$ boundary components resulted 
from the intersection of an
enlarged neighborhood of $I_i$ with $\del N$. Such neighborhoods are 
illustrated in Figure~\ref{fig:framed-arcs}. 
%Similarly, for each $i=1,\ldots,m$, let $B_i=\del \Fsphere_i\setminus (\Fsphere_i\cap\nd(s(\Fsphere)))$, a punctured torus corresponding to $K_i$.  Such neighborhoods are 
%illustrated in Figure~\ref{fig:framed-arcs}. 
The intersection of $T$ with 
$M''=M-N$ defines a new tangle $(M'',T''=T\cap M'')$.

\begin{figure}%[ht]
\def\svgwidth{10cm}
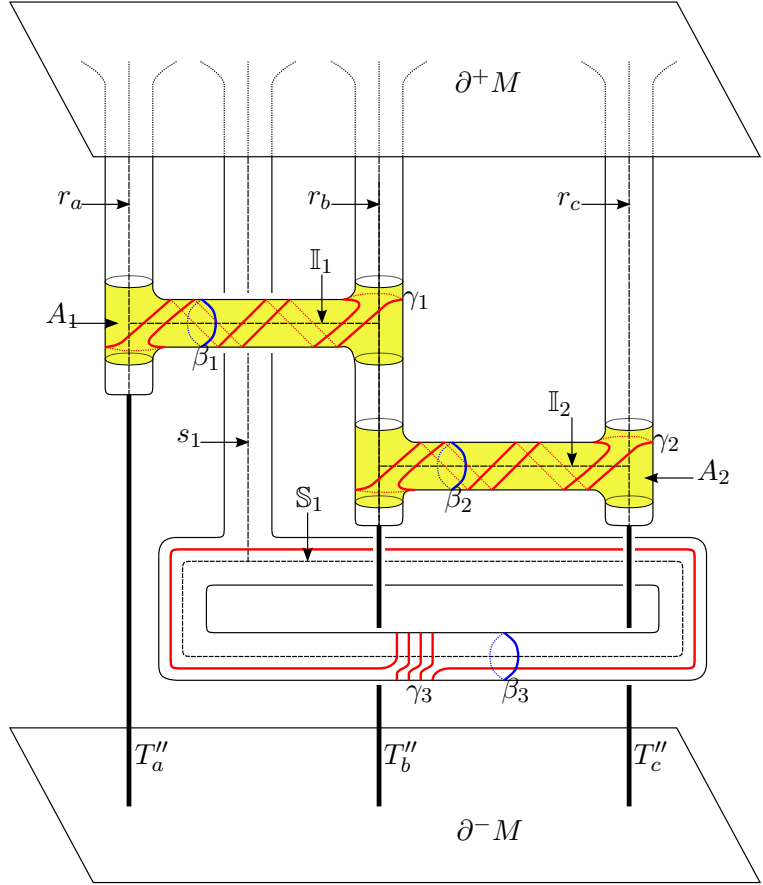
\caption{A Heegaard diagram subordinate to a framed knot $\Fsphere_1$ and 
a pair of framed arcs
$\Farc_1$ and $\Farc_2$ with one end point on the same strand $T_b$ of $T$,
and the other ends on the strands $T_a$ and $T_c$. A tubular neighborhood of 
the union of $K_1,s_1,I_1,I_2,r_a,r_b$ and $r_c$ is deleted to obtain the tangle 
$(M'',T'')$. Attaching disks to the meridians $\beta_{1}$ and $\beta_2$ 
of $\Farc_1$ and $\Farc_2$ and the meridian $\beta_{3}$ of $\Fsphere_1$ 
gives a Heegaard diagram for $(M,T)$ while 
the framings of $\Farc_1,\Farc_2$ and $\Fsphere_1$ determine the curves 
$\gamma_{1},\gamma_{2}$ and 
$\gamma_3$. The curves $\gamma_1$ and $\gamma_2$ 
live in the $4$-punctured spheres 
$A_1$ and $A_2$, respectively.  
}\label{fig:framed-arcs}
\end{figure}

\begin{defn}\label{def:subordinate-HD}
A Heegaard triple {\emph{subordinate}} to the framed arcs $\Farc$,  
the framed link $\Fsphere$, and the stalk $s(\Fsphere)$ for the 
balanced tangle $(M,T)$ is a Heegaard triple
$$(\Sig,\alphas=\{\alpha_1,\ldots,\alpha_{\ell}\},\betas=\{\beta_1,\ldots,\beta_{\ell}\},
\gammas=\{\gamma_1,\ldots,\gamma_{\ell}\},\z)$$
satisfying the following conditions:
\begin{enumerate}
\item $(\Sig,\alphas,\{\beta_{n+m+1},\ldots,\beta_{\ell}\},\z)$ is a Heegaard diagram for 
$(M'',T'')$. Fix an identification of $\Sig[\beta_{n+m+1},\ldots,\beta_\ell]$ with $\del^+M''$. 
\item For $i=m+n+1,\ldots, \ell$, $\gamma_i$ is obtained by a  small Hamiltonian isotopy 
from $\beta_i$ supported away from the marked points,
so that $|\beta_{i}\cap\gamma_{i}|=2$.
\item For any $i=1,\ldots,m$, the curves $\beta_{n+i}$ and $\gamma_{n+i}$ lie on the punctured torus $\del \Fsphere_i\cap\del ^+M''$, representing the meridian and the framing of $K_i$, respectively. Further, they meet in a single transverse intersection point.

%
%
%$\del\Fsphere_i-\left(\Fsphere_i\cap\nd(s(\Fsphere))\right)$.
%The curve $\beta_{n+i}$ represents the meridian of $K_i$ while 
%$\gamma_{n+i}$ represents the framing, and they meet
%in a single transverse intersection point.
\item For any $j=1,\ldots,n$, the curves $\beta_{j}$ and 
$\gamma_{j}$ lie on the punctured sphere $A_{j}$. Moreover, $\beta_{j}$ represents the meridian of $\Farc_j$ and meets 
$\gamma_{j}$ in two transverse intersection points, 
while $\gamma_{j}$ is obtained from 
$\beta_{j}$ by an isotopy corresponding to the framing of $\Farc_j$ 
(which crosses two of the boundary components of $A_j$), as illustrated 
in Figure~\ref{fig:framed-arcs}.
\item The Heegaard diagrams
$(\Sig,\alphas,\betas,\z)$ and $(\Sig,\alphas,\gammas,\z)$ are diagrams for  
$(M,T)$ and $(M',T')$, respectively.
\end{enumerate}
We say that a Heegaard triple is \emph{subordinate} to the framed arcs $\Farc$ and 
the framed link $\Fsphere$ if it is subordinate to the framed arcs $\Farc$ and 
the framed link $\Fsphere$ and some stack $s(\Fsphere)$ for $\Fsphere$.
\end{defn}

The existence of Heegaard triples subordinate to an acceptable set of framed 
arcs $\Farc$ and a framed link $\Fsphere$ 
for a stable cobordism $\Cobb$ as above
%\[\Cob=\Cob(\Farc,\Fsphere)=[W=W(\Fsphere),\sur=\sur(\Farc),\la_\sur]:
%\Tangle=[M,T,\la]\leadsto \Tangle'=[M',T',\la']\] 
and  the correspondence between different such Heegaard triples is 
addressed in the following lemma. 

\begin{lem}\label{lem:HDs-for-framed-arcs}
Let $\Farc=\{\Farc_1,\ldots,\Farc_n\}$ be an acceptable set of framed arcs in $(M,T)$ and $\Fsphere=\{\Fsphere_1,\ldots,\Fsphere_m\}$ be a 
framed link in $M\setminus(T\cup\Farc)$.
There is a Heegaard triple subordinate to $\Farc$ and $\Fsphere$. 
Further, every two such triples may be connected (after composing 
with a diffeomorphism of  the diagram) by a sequence of  following moves,
all supported away from the set $\z$  of marked points:
\begin{enumerate}
\item Isotopies and handle slides among $\{\alpha_1,\ldots,\alpha_{\ell}\}$,
\item Isotopies and handle slides among $\{\beta_{n+m+1},\ldots,\beta_{\ell}\}$ 
while carrying the corresponding isotopy or handle slide among 
$\{\gamma_{n+m+1},\ldots,\gamma_{\ell}\}$,
\item Stabilization (and destabilization); 
i.e. taking the connected sum of the Heegaard triple with a 
triple $(E,\alpha,\beta,\gamma)$, where $E$ is a surface of genus one, 
$|\alpha\cap\beta|=1$, and $\gamma$ is obtained 
by a small Hamiltonian isotopy from  $\beta$ such that 
$|\beta\cap\gamma|=2$,
\item Isotopies or handle slides of $\beta_{n+j}$ along the curves in
$\{\beta_{n+m+1},\ldots,\beta_{\ell}\}$ for $j=1,\ldots,m$,
\item Isotopy or handle slides of $\gamma_{n+j}$ along the curves in 
$\{\gamma_{n+m+1},\ldots,\gamma_{\ell}\}$ for $j=1,\ldots,m$, 
\item For $i=1,\ldots,n$, isotopy or handle slide of $\beta_i$ along the curves $\{\beta_{n+m+1},\ldots,\beta_{\ell}\}$, while carrying the corresponding isotopy or handle slide on $\gamma_i$,
%\todo{AA I added point (6), does it look ok? EE. Fine!}
\item Handle slide of a curve in $\{\beta_{n+m+1},\ldots,\beta_{\ell}\}$ 
along some $\beta_{n+j}$ for $j=1,\ldots,m$, while doing a 
handle slide of the corresponding  curve in $\{\gamma_{n+m+1},\ldots,\gamma_{\ell}\}$ 
along $\gamma_{n+j}$.
\end{enumerate}
\end{lem}

\begin{proof}
Given the stack $s(\Fsphere)$ for $\Fsphere$, the proof of Lemma 4.5 from 
\cite{OS-4mfld} may be used to show that  Heegaard diagrams subordinate to 
$\Farc$, $\Fsphere$ and $s(\Fsphere)$ exist and that
for every pair $H$ and $H'$ of such diagrams, $H$ may be changed to $H'$ via 
a sequence of moves of types $1,2,3,4,5$ and $6$ in the statement of the lemma.
If the stacks $s(\Fsphere)$ and $s'(\Fsphere)$ are different, the proof of  
\cite[Lemma 4.8]{OS-4mfld} implies that there is a Heegaard diagram $H$ as above 
subordinate to $\Farc$, $\Fsphere$ and $s(\Fsphere)$ and a Heegaard diagram 
$H'$ subordinate to $\Farc$, $\Fsphere$ and the stalk $s'(\Fsphere)$ such that 
the following is true. There is  a sequence of handle slides of 
some  particular curves in $\{\beta_{n+m+1},\ldots,\beta_{\ell}\}$ over the curves  
$\beta_{n+j}$ for $j=1,\ldots,m$ and other curves in 
$\{\beta_{n+m+1},\ldots,\beta_{\ell}\}$
(and a corresponding sequence of handle slides for $\gammas$) which change 
$H$ to $H'$. The lemma then follows.  
\end{proof}

Let $H=(\Sig,\alphas,\betas,\gammas,\z)$ 
be a Heegaard triple subordinated to 
the framed arcs $\Farc$ and the framed link $\Fsphere$ as above. 
Associated with $H$ we have a cobordism 
$$\Cobb_{H}=(W_H,\sur_{H}):(M,T)\amalg(M_{\beta\gamma},
T_{\beta\gamma})\leadsto (M({\Fsphere}),T({\Farc})),$$
where $(M_{\beta\gamma},T_{\beta\gamma})$ is the balanced tangle 
determined by the Heegaard diagram $(\Sig,\betas,\gammas,\z)$. This cobordism is related to $\Cobb(\Farc,\Fsphere)=(W,\sur)$ as follows. 

%Let $(M^+,T^+):=\del^+\Cobb(\Farc,\Fsphere)$ and $(M_{\sur},{T}_{\sur})$ 
%denote the associated balanced tangles as in Subsection \ref{SpTangle}. The tangle 
%$(M_{\beta\gamma},T_{\beta\gamma})$ is constructed from $(M_{\sur},{T}_{\sur})$ 
%by attaching $\ell-n-m$ one-handles to $M_\sur-T_{\sur}$. Let the spheres 
%$S_1,\ldots,S_{\ell-n-m}$ denote the cores of these one-handles.\\

%
%
%Associated with every framed arc $\Farc_i$ a properly embedded 
%arc $J_i\subset \sur$  is obtained as follows. 

Each $\Farc_i$ determines an embedded arc $J_i$ on 
$\sur$ with endpoints on $T$. After applying a smooth isotopy supported in a neighborhood 
of $\del\sur$ on $J=\{J_1,\ldots,J_n\}$ which moves the endpoints to $\partial^+\sur$, 
it becomes a spanning set of arcs on $\sur$. 
%
%These endpoints may first be moved to 
%$\del^+T$ and then to the interior of $T^+$ by an isotopy. More precisely,
%after applying a smooth isotopy supported in a 
%neighborhood of $\del \sur$ on the arcs $J=\{J_1,\ldots,J_{n}\}$ we may 
%assume that 
%$$\coprod_{i=1}^{n}\del J_i\subset \sur\cap M^+={T}^+.$$
%As before, $(M^+,T^+)=\del^+\Cobb(\Farc,\Fsphere)$.
%We call $\Farc$ {\emph{ordered}} if 
%the corresponding spanning set $J$ is ordered. It is relatively easy to show that 
%every acceptable set of framed arcs $\Farc$ may be changed to an ordered 
%set of framed arcs by performing a number of arc slides.
Let  $\nd(J)$ be a tubular 
neighborhood of $J$. Then,
$$(W-\nd(J),\sur\cap (W-\nd(J)))$$
gives a cobordism from $(M,T)\amalg (M_\sur,{T}_{\sur})$ to $(M({\Fsphere}),T({\Farc}))$.

\begin{lem}\label{HT:Arcs}
Under the above assumptions, after attaching $3$-handles along 
 $S_{1},\ldots,S_{\ell-n-m}\subset M_{\beta\gamma}$ to $\Cobb_{H}$, we obtain
the cobordism  $$(W-\nd(J),\sur\cap (W-\nd(J))).$$
\end{lem}

\begin{proof}
Denote the compression body obtained by attaching disks to $\Sig\times[0,1]$ along $\betas\times\{1\}$ by $C(\betas)$. 
Without loss of generality, we may assume that   
\[\Farc_i,\Fsphere_j\subset C(\betas)\ \ \ \ \ \text{for }\ i\in\{1,\ldots,n\},\ j\in\{1,\ldots,m\} .\] 
Furthermore, we may consider an identification 
\[W=\left(M\times [0,1]\right)\cup_{\{1\}\times \Fsphere}\left(\bigcup_{i=1}^m
D^2\times D^2\right)\] 
such that $\pi_2|_{\sur}$ is a Morse function with all critical points of index one. Here, 
$\pi_2$ is the projection map from $M\times[0,1]\setminus \{1\}\times\Fsphere$ 
onto $[0,1]$.  Moreover, we assume that 
\[\mathrm{Crit}(\pi_2|_{\sur})\subset M\times \{1/2\}.\]  
Thus, $T_{1/2}=\sur\cap(M\times\{1/2\})$ is a properly embedded, 
oriented, singular 1-dimensional submanifold of $M\times\{1/2\}$. For every framed arc $\Farc_i$ we have a singular point 
\[q_i\in T_{1/2}\cap \left(C(\betas)\times\{1/2\}\right)\subset \sur.\] 
Let $B'(\Farc)\subset T_{1/2}$ be the subspace corresponding to $B(\Farc)$. 
Let $N',N_{1/2}\subset W$ be tubular neighborhoods of $B'(\Farc)$ and 
$C(\betas)\times \{1/2\}$, respectively. Then
\begin{align*}
&\left(W-N',\sur\cap (W-N')\right)=\left(W-\nd(J),\sur\cap (W-\nd(J))\right)
\ \ \ \text{and}\\ 
&\left(W-N_{1/2},\sur\cap (W-N_{1/2})\right)=\Cobb_H.
\end{align*}

Thus, $\left(W-\nd(J),\sur\cap (W-\nd(J))\right)$ is obtained from $\Cobb_H$ by 
attaching $3$-handles along the spheres 
$S_1,\ldots,S_{\ell-n-m}\subset M_{\beta\gamma}$.
\end{proof}

It follows from the proof of Lemma \ref{HT:Arcs} that $(M_\sur,T_\sur)$ is obtained by 
surgery on  $(M_{\beta\gamma},T_{\beta\gamma})$ along the $2$-spheres 
$S_1,\ldots,S_{\ell-m-n}$. Abusing the notation, let $\spinc_0\in\SpinC(M_{\beta\gamma})$ 
denote the $\SpinC$ class obtained from $\spinc_0\in\SpinC(M_\sur)$ which satisfies
\[\langle c_1(\spinc_0), S_i\rangle =0\ \ \ \text{for~any}\ \ i\in\{1,2,\ldots,\ell-n-m\}.\]
For any Heegaard tripe $H$ subordinated to $\Farc$ and $\Fsphere$, let 
\[r:\SpinC(\Cobb({\Farc},\Fsphere))\ra\SpinC(\Cobb_H).\] 
denote the restriction map.

%
%In the above situation, 
%Lemma \ref{HT:Arcs} implies that associated with every Heegaard triple  
%$H$ subordinate to the framed arcs $\Farc$ and the framed link $\Fsphere$ 
%we have a restriction map 
%\[r:\SpinC(\Cobb({\Farc},\Fsphere))\ra\SpinC(\Cobb_H).\] 
%Abusing the notation, let $\spinc_0\in\SpinC(M_{\beta\gamma})$ denote 
%the $\SpinC$ class obtained from $\spinc_0\in\SpinC(M_\sur)$ which satisfies
%\[\langle c_1(\spinc_0), S_i\rangle =0\ \ \ \forall\ \ i\in\{1,2,\ldots,\ell-n-m\}.\]

\begin{lem}
For every $\spinct\in\SpinC(\Cobb({\Farc},\Fsphere))$, the restriction 
$r(\spinct)|_{M_{\beta\gamma}}$ is the $\SpinC$ class  
$\spinc_0\in\SpinC(M_{\beta\gamma})$.
\end{lem}
\begin{proof}
This is straightforward. 
\end{proof}
Suppose $\la_\sur:\pi_0(\sur)\to\Ring$ be an $\Ring$-coloring for $\Cobb(\Farc,\Fsphere)$. Given a $\SpinC$ class $\spinct\in\SpinC(\Cobb(\Farc,\Fsphere))$
we abuse the notation and call the  Heegaard triple   
$(\Sig,\alphas,\betas,\gammas,\la)$ 
$\spinct$-admissible,  if it is 
$r(\spinct)$-admissible. Here $\la:\z\to\Ring$ denotes the map induced by $\la_\sur$. Note that
every  diagram $(\Sig,\alphas,\betas,\gammas,\la)$ subordinate to  $\Farc$ 
and $\Fsphere$ may be transformed to a $\spinct$-admissible Heegaard 
diagram by performing isotopies on the curves in $\alphas$, following a 
procedure similar to \cite[Section 4.2]{AE-1}. 
Let us assume that  the Heegaard 
diagram $(\Sig,\alphas,\betas,\gammas,\la)$ subordinate to $\Farc$ and $\Fsphere$
is $\spinct$-admissible. We then call $\HD=(\Sig,\alphas,\betas,\gammas,\la,\spinct)$ an
$\Ring$-diagram for 
\[\Cob=\Cob(\Farc,\Fsphere)=[W,\sur,\spinct,\la_\sur]:\Tangle=[M,T,\spinc=\spinct|_M,\la_T]
\leadsto  \Tangle'=[M',T',\spinc'=\spinct|_{M'},\la'_T],\]
subordinate to $\Farc$ and $\Fsphere$. Here, $\la_T$ and $\la'_{T}$ are the $\Ring$-colorings induced by $\la_\sur$ on $(M,T)$ and $(M',T')$, respectively.

For an appropriate generic family of almost complex structures $J$, the $\Ring$-diagram $\HD$ determines a holomorphic triangle map 
%The $\Ring$-diagram $H$  determines a holomorphic triangle map 
\begin{displaymath}
\begin{split}
&\fmap_{\alpha\beta\gamma}:\HFT_J(\Sig,\alphas,\betas,\la,\spinc)\otimes_\Ring
\HFT_J(\Sig,\betas,\gammas,\la,\spinc_0)
%\\&\ \ \ \ \ \ \ \ \ \ \ \ \ \ \ \ \ \ \ \ \ \ \ \ \ \ \ \ \ \ \ \ \ 
\lra \HFT_J(\Sig,\alphas,\gammas,\la,\spinc').
\end{split}
\end{displaymath}

From the Heegaard diagram 
$(\Sig,\betas,\gammas,\z)$ for $(M_{\beta\gamma},T_{\beta\gamma})$ we may 
construct a Heegaard diagram $H'=(\Sig',\betas',\gammas',\z)$ for 
$(M_{\sur},T_{\sur})$ where $\Sig'$ is obtained by surgery on$\Sig$ along 
$\beta_{n+m+1},\ldots,\beta_{\ell}$, while 
\[\betas'=\{\beta_{1},\ldots,\beta_{n+m}\}\ \  \text{and}\ \  
\gammas'=\{\gamma_{1},\ldots,\gamma_{n+m}\}.\]
%Similarly, we may construct another Heegaard diagram $H''=(\Sig'',\betas'',\gammas'',\z)$ 
%for $(M_{\sur},T_{\sur})$ where $\Sig''$ is obtained from $\Sig'$ by surgering out 
%$\beta_{n+1},\ldots,\beta_{n+m}$, while 
%\[\betas''=\{\beta_{1},\ldots,\beta_{n}\}\ \  \text{and}\ \  
%\gammas''=\{\gamma_{1},\ldots,\gamma_{n}\}.\]
%Note that $H'$ is obtained from $H''$ by $m$ stabilizations.
Following the construction in Section \ref{sec:1-handlemap}, we get a homomorphism  

$$\fmap_{\beta\gamma}:\HFT(\Sig',\betas',\gammas',\la,\spinc_0)\ra
\HFT(\Sig,\betas,\gammas,\la,\spinc_0).$$
Then, we set
$\Theta_{\beta\gamma}:=\fmap_{\beta\gamma}(\Theta_{\sur})$
and define
\begin{displaymath}
\begin{split}
&\fmap_{\HD,J}:
\HFT_J(\Sig,\alphas,\betas,\la,\spinc)\lra\HFT_J(\Sig,\alphas,\gammas,\la,\spinc')\\
&\fmap_{\HD,J}(\x):=\fmap_{\alpha\beta\gamma}(\x\otimes \Theta_{\beta\gamma})
\ \ \ \ \ \ \text{for~any}\ \ \ \x\in \HFT_J(\Sig,\alphas,\betas,\la,\spinc).
\end{split}
\end{displaymath}

\begin{thm}\label{thm:inv-arcs-2-handles}
For any $\Ring$-module $\M$, the homomorphism $\fmap_{\HD,J}$ induces a homomorphism 
\[\fmap_{\Cob,\Farc,\Fsphere}^{\M}:\HFT^{\M}(\Tangle)\to\HFT^{\M}(\Tangle').\] 
\end{thm}

\subsection{Proof of Theorem~\ref{thm:inv-arcs-2-handles}}
It follows from standard arguments in Floer theory that for appropriate families of almost complex structures $J$ and $J'$, we have 
\[\fmap_{\HD,J'}\circ\Phi_{J\to J'}=\Phi_{J\to J'}\circ\fmap_{\HD,J}.\]
So we denote the induced map by $\fmap_{\HD}$. Let
$$\HD=(\Sig,\alphas,\betas,\gammas,\la,\spinct)\ \ \text{and}\ \ 
\HD'=(\Sig',\alphas',\betas',\gammas',\la',\spinct)$$ 
be $\Ring$-diagrams for $\Cob$ subordinate to $\Farc$ and $\Fsphere$. Assume that $\HD'$ is obtained from $\HD$ by one Heegaard move $e$ of the type specified in Lemma~\ref{lem:HDs-for-framed-arcs}. Associated with $\hmove$, let  
\begin{displaymath}
\begin{split}
& \Phi_{\hmove}:
\HFT^{\M}\left(\Sig,\alphas,\betas,\la,\spinc\right)\ra
\HFT^{\M}\left(\Sig',\alphas',\betas',\la',\spinc\right),\\
& \Psi_{\hmove}:
\HFT^{\M}\left(\Sig,\alphas,\gammas,\la,\spinc'\right)\ra
\HFT^{\M}\left(\Sig',\alphas',\gammas',\la',\spinc'\right).
\end{split}
\end{displaymath}
be the isomorphisms defined in Section ~\ref{sec:HFisom}. We need to prove that the diagram  
\begin{diagram}
\HFT^{\M}(\Sig,\alphas,\betas,\la,\spinc)&\rTo{\fmap_{\HD}}
&\HFT^{\M}(\Sig,\alphas,\gammas,\la,\spinc')\\
\dTo{\Phi_{\hmove}}&&\dTo{\Psi_{\hmove}}\\
\HFT^{\M}(\Sig',\alphas',\betas',\la',\spinc)&\rTo{\fmap_{\HD'}}
&\HFT^{\M}(\Sig',\alphas',\gammas',\la',\spinc')
\end{diagram}
commutes. Let us first consider the Heegaard move $\hmove$
which changes $\alphas$ to $\alphas'$ and keeps $\Sig,\betas,\gammas$ and $\z$
unchanged. The $4$-manifold $W_{\alpha\beta\gamma}$ is obtained by attaching 
$3$-handles to $W_{\alpha'\alpha\beta\gamma}$. Abusing the notation, denote 
the restriction of $\spinct$ to $W_{\alpha'\alpha\beta\gamma}$ by $\spinct$.  Let 
$\Sqmap_{\hmove}$ be the homomorphism associated with the $\Ring$-diagram 
$(\Sig,\alphas',\alphas,\betas,\gammas,\la,\spinct)$ and the distinguished 
generators $\Theta_{\alpha'\alpha}$ and $\Theta_{\beta\gamma}$, defined by counting 
holomorphic squares. Considering different possible degenerations of a square either to 
a bigon and a square or to a pair of triangles  gives the relation
\begin{displaymath}
\Psi_{\hmove}\circ \fmap_{\HD}-
\fmap_{\HD'}\circ \Phi_{\hmove}
=\Sqmap_{\hmove}\circ d+d\circ \Sqmap_{\hmove},
\end{displaymath}
following the standard arguments in Heegaard Floer theory. With a similar argument, if $e$ is a Heegaard move that changes $\betas$ to $\betas'$ (respectively $\gammas$ to $\gammas'$) and keeps $\Sig$, $\alphas$, $\gammas$ (respectively $\betas$) and $\z$ fixed,  the diagram commutes. 

Suppose $\hmove$ changes both $\betas$ and $\gammas$ simultaneously to $\betas'$ and $\gammas'$.  We obtain a pair of maps 
$\Sqmap_i=\Sqmap_{\hmove,i},\ i=1,2$ which correspond to the Heegaard
quadruples 
\begin{displaymath}
(\Sig,\alphas,\betas,\betas',\gammas',\la,\spinct)\ \ \text{and}\ \ 
(\Sig,\alphas,\betas,\gammas,\gammas',\la,\spinct),
\end{displaymath}
respectively. 
Denote the holomorphic triangle maps corresponding to the Heegaard subdiagrams 
\begin{displaymath}
(\Sig,\alphas,\betas,\gammas',\la,\spinct),\ (\Sig,\betas,\gammas,\gammas',\la,\spinct)
\  \  \text{and}\ \ (\Sig,\betas,\betas',\gammas',\la,\spinct)
\end{displaymath} 
by $\fmap_{\alpha\beta\gamma'},
\fmap_{\beta\gamma\gamma'}$ and $\fmap_{\beta\beta'\gamma'}$, respectively.\\

The images of the distinguished generator 
$\Theta_{\beta\gamma}\otimes\Theta_{\gamma\gamma'}$
under $\fmap_{\beta\gamma\gamma'}$ and the distinguished generator 
$\Theta_{\beta\beta'}\otimes\Theta_{\beta'\gamma'}$ under 
$\fmap_{\beta\beta'\gamma'}$ is the distinguished generator 
$\Theta_{\beta\gamma'}$ which corresponds to $(\Sig,\betas,\gammas',\la,\spinct)$.
This should be done independently for each one of the Heegaard moves. 
The proofs follow from the
standard arguments in Heegaard Floer theory 
since the Heegaard triples $(\Sig,\betas,\gammas,\gammas',\z)$ and 
$(\Sig,\betas,\betas',\gammas')$ have  standard forms.\\

Let us abuse the notation and denote 
$\fmap_{\alpha\beta\gamma'}(-\otimes \Theta_{\beta\gamma'})$
by $\fmap_{\alpha\beta\gamma'}$.
Setting $\Sqmap_{\hmove}=\Sqmap_1+\Sqmap_2$, 
the study of different possible degenerations of a square to a bigon  
and a square or to two triangles gives 
\begin{displaymath}
\begin{split}
&\fmap_{\alpha\beta\gamma'}-\fmap_{\HD'}\circ \Phi_{\hmove}=
\Sqmap_1\circ d+d\circ \Sqmap_1\ \ \text{and}\\
&\Psi_{\hmove}\circ\fmap_{\HD}-\fmap_{\alpha\beta\gamma'}
= \Sqmap_2\circ d+d\circ \Sqmap_2\\
\Rightarrow \ & \Psi_{\hmove}\circ\fmap_{\HD}
-\fmap_{\HD'}\circ \Phi_{\hmove}=
\Sqmap_{\hmove}\circ d+d\circ \Sqmap_{\hmove}.
\end{split}
\end{displaymath}

Note that the map $\Sqmap_{\hmove}$ is trivial 
when $\hmove$ is a stabilization or destabilization, provided that the complex 
structure is sufficiently stretched along the neck.

\newpage

\section{The cobordism map and its invariance}\label{sec:inv}
%\subsection{The map associated with a cobordism}
Let $\Cob=[W,\sur,\spinct,\la_\sur]$ be a stable 
$\Ring$-cobordism from 
$\Tangle=[M,T,\spinc,\la]$ to  $\Tangle'=[M',T',\spinc',\la']$. 
Consider an indexed parametrized decomposition 
$$\ti{\Cerf}:\Cobb=\Cobb_1\cup_{(M_1,T_1)}
\Cobb_2\cup_{(M_{2},T_2)}\Cobb_{3}.$$
Recall that for $i=1,3$, the cobordism $\Cobb_i$ is 
parametrized by a set 
$\Fsphere_i\subset M_{i-1}\setminus T_{i-1}$ of pairwise 
disjoint framed $(i-1)$-spheres, and a diffeomorphism 
\[d_i:(M_{i-1}(\Fsphere_i),T_{i-1})\to (M_i,T_i).\] 
Further, $\Cobb_2$ is parametrized by a framed link 
$\Fsphere_2$ and a set of pairwise disjoint framed arcs 
$\Farc$ in $(M_1,T_1)$ along with a diffeomorphism 
\[d_2:(M_{1}(\Fsphere_2),T_{1}(\Farc))\to (M_2,T_2).\]

For each $\Cobb_i=(W_i,\sur_i)$, let $\spinct_i$ and 
$\la_i$ be the $\SpinC$ structure and the $\Ring$-coloring 
induced by $\spinct$ and $\la_{\sur}$, respectively. Note that 
$\spinct$ determines $\spinct_i$ for $i=1,2,3$, while 
$\spinct_2$ determines $\spinct$. Let 
$\Cob_i=[W_i,\sur_i,\spinc_ti,\la_i]$.  In addition, $\spinct$ 
and $\la_{\sur}$ induce an $\SpinC$ structure on $(M_i,T_i)$, 
denoted by $\spinc_i$, and an $\Ring$-coloring denoted by 
$\la_i$, abusing the notation. For 
$0\le i\le 3$, let $\Tangle_i=[M_i,T_i,\spinc_i,\la_i]$, 
where $\Tangle_0=\Tangle$, and $\Tangle_3=\Tangle'$.

For every $\Ring$-module $\Rin$,
the constructions of Section~\ref{sec:connected-sum} and 
Section~\ref{sec:map} associate naturally defined 
$\Ring$-homomorphisms 
\begin{displaymath}
\fmap_i=\HFT^\M(d_i)\circ \fmap_{\Fsphere_i}^\M:
\HFT^{\Rin}(\Tangle_{i-1})
\lra \HFT^{\Rin}(\Tangle_i)
\end{displaymath}
to the parametrized $\Ring$-cobordism $\Cob_i$ for $i=1,3$ and 
\begin{displaymath}
\fmap_2=\HFT^\M(d_2)\circ \fmap_{\Cob_2,\Farc,\Fsphere_2}^\M:
\HFT^{\Rin}(\Tangle_{1})
\lra \HFT^{\Rin}(\Tangle_{2})
\end{displaymath}
to the parametrized $\Ring$-cobordism $\Cob_2$.
Subsequently, we may define
\begin{equation}\label{eq:cob-defention}
\begin{split}
&\fmap_{\Cob,\ti{\Cerf}}^\M:\HFT^{\M}(\Tangle)\lra 
\HFT^\M(\Tangle'),\quad\quad
\fmap_{\Cob,\ti{\Cerf}}^\M=\fmap_3\circ
\fmap_2\circ\fmap_1.
\end{split}
\end{equation}
The homomorphism $\fmap_{\Cob,\ti{\Cerf}}^\M$ is 
well-defined and natural, while {\emph{a priori}} it depends 
on the indexed parametrized decomposition $\ti{\Cerf}$.

\begin{thm}\label{thm:map-definition}
Let $\Cob$ be an $\Ring$-cobordism from the $\Ring$-tangle 
$\Tangle$ to the $\Ring$-tangle
$\Tangle'$. For every 
$\Ring$-module $\Rin$, the $\Ring$-homomorphism
$$\fmap_{\Cob,\ti{\Cerf}}^\M:
\HFT^{\Rin}(\Tangle)\ra\HFT^{\Rin}(\Tangle')$$
is an invariant of $\Cob$. 
More precisely, this $\Ring$-homomorphism
does not depend on the choice of the indexed parametrized 
decomposition $\ti{\Cerf}$ for $\Cobb$, which was used in its 
definition.
\end{thm}
\begin{proof}
With the above notation fixed, Theorem~\ref{thm:inddecom} 
reduces the proof to showing the 
invariance of the homomorphism $\fmap_{\Cob,\ti{\Cerf}}^{\M}$ 
under the following changes: 
\begin{enumerate}
\item Sliding a component of $\Fsphere_i$ on another component 
of $\Fsphere_i$ for $i=1,2,3$,
\item Sliding a component of $\Farc$ on another component of 
$\Farc$ or a component of $\Fsphere_2$,
\item Sliding a component of $\Fsphere_2$ on a component of 
$\Farc$,
\item Creation and cancellation of index one/two or 
two/three critical points,
\item Diffeomorphism equivalences.
\end{enumerate}
Invariance under the move $(1)$ follows from the arguments in 
\cite[Subsection 4.4]{OS-4mfld}, and invariance under the moves 
$(4)$ and $(5)$ is straightforward.  
%\todo{AA I am not sure about this.\\
%EE: We also need to address move (5).\\
%EE: I thought again about (5) and I think it is actually 
%trivial, having proved the other parts.} 
We will prove invariance under the moves $(2)$ and $(3)$ in Sections 
\ref{sec:arcslides} and \ref{sec:weakcomposition}. 
\end{proof}

\subsection{Invariance under arc slides.}\label{sec:arcslides}
Let $\Cob=[W,\sur,\spinct,\la_\sur]$ be an $\Ring$-cobordism 
defined by an acceptable set of framed arcs $\Farc$ and a 
framed link $\Fsphere$ in an $\Ring$-tangle 
$\Tangle=[M,T,\spinc,\la]$. Thus, the $\Ring$-cobordism 
$\Cob$ is from $\Tangle$ to 
$\Tangle'=[M(\Fsphere),T(\Farc),\spinc',\la']$ where 
$\spinc'=\spinct|_{M(\Fsphere)}$ and $\la'$ is the 
$\Ring$-coloring induced by $\la_{\sur}$ on 
$(M(\Fsphere),T(\Farc))$. Suppose $\ti{\Farc}$ is an 
acceptable set of framed arcs obtained from $\Farc$ by arc 
slides. Corresponding to $(\ti{\Farc},\Fsphere)$ we get a 
diffeomorphism 
\[D:(W,\sur)\to (\ti{W},\ti{\sur})\]
where $(\ti{W},\ti{\sur})$ is the cobordism from $(M,T)$ to 
$(M(\Fsphere),T(\ti{\Farc}))$ defined by 
$(\ti{\Farc},\Fsphere)$. 
Note that $d^0=D|_{(M,T)}$ is isotopic to identity and let 
$d=D|_{(M(\Fsphere),T(\Farc))}$. So, $(\Farc,\Fsphere,d)$ and 
$(\ti{\Farc},\Fsphere,\mathrm{id})$ are two parameterizations 
of $(\ti{W},\ti{\sur})$. The $\SpinC$ structure $\spinct$ and 
the $\Ring$-coloring $\la_{\sur}$ induce a $\SpinC$ structure 
and and $\Ring$-coloring on $(\ti{W},\ti{\sur})$, denoted by 
$\ti{\la}_{\sur}$ and $\ti{\spinct}$. Set 
$\ti{\Cob}=[\ti{W},\ti{\sur},\ti{\spinct},\ti{\la}_{\sur}]$ 
and let $\Tangle''=[M(\Fsphere),T(\ti{\Farc}),\spinc'',\la'']$ 
where $\spinc''$ and $\la''$ are induced by $\ti{\la}_{\sur}$ 
and $\ti{\spinct}$, respectively. Thus, $d$ is a diffeomorphism 
between the $\Ring$-tangles $\Tangle'$ and $\Tangle''$.

\begin{thm}\label{thm:Handle-Slide-Invariance}
With the above notation fixed, for any $\Ring$-module $\M$
\[\HFT^{\M}(d)\circ\fmap^{\M}_{\Cob,\Farc,\Fsphere}=\fmap^{\M}_{\ti{\Cob},\ti{\Farc},\Fsphere}:
\HFT^\M(\Tangle)\lra 
\HFT^\M(\Tangle'').
\]
\end{thm}
\begin{proof}
Assume $\ti{\Farc}$ is obtained from $\Farc$ by a single arc 
slide of $\Farc_1$ over $\Farc_2$. Let
$$\HD=(\Sig,\alphas=\{\alpha_1,\ldots,\alpha_{\ell}\},\betas
=\{\beta_1,\ldots,\beta_{\ell}\},\gammas=
\{\gamma_1,\ldots,\gamma_{\ell}\},\la:\z\ra \Ring,\spinct)$$
be an $\Ring$-diagram subordinate to $\Farc$ and $\Fsphere$, as 
in Definition~\ref{def:subordinate-HD}. Then, we obtain an 
$\Ring$-diagram subordinate to $\ti{\Farc}$ and $\Fsphere$ by 
handle sliding $\beta_2$ over $\beta_1$ and $\gamma_1$ over 
$\gamma_2$ as in Figure~\ref{fig:arc-slide}.  Here, $\beta_i$ 
and $\gamma_i$ are the closed curves corresponding to $\Farc_i$ 
for $i=1,2$. Let 
\[\ti{\HD}=(\Sig,\alphas,\ti{\betas},
\ti{\gammas},\la:\z\to\Ring,\ti{\spinct})\] 
be the resulting 
Heegaard triple. Following an argument analogous to the proof 
of Theorem \ref{thm:inv-arcs-2-handles}, we show that the 
diagram
\begin{diagram}
\HFT^{\M}(\Sig,\alphas,\betas,\la,\spinc)&\rTo{\fmap^{\M}_{\HD}}
&\HFT^{\M}(\Sig,\alphas,\gammas,\la,\spinc')\\
\dTo{\Phi}&&\dTo{\Phi'}\\
\HFT^{\M}(\Sig,\alphas,\ti{\betas},\la,\spinc)&\rTo{\fmap^{\M}_{\ti{\HD}}}
&\HFT^{\M}(\Sig,\alphas,\ti{\gammas},\la,\spinc'')
\end{diagram}
commutes. Here, $\Phi$ and $\Phi'$ are isomorphisms 
corresponding to the aforementioned handle slides. 

On the other hand, let $h:\Sig\to\Sig$ be the diffeomorphism, 
which maps $\beta_2$ to $\beta_1$, $\gamma_1$ to $\gamma_2$ 
and preserves the rest of $\beta$ and $\gamma$ curves. Then, 
$h$ induces the diffeomorphism $D$ from $(W,\sur)$ to 
$(\ti{W},\ti{\sur})$ as well as the diffeomorphism $d$, and 
\[\ti{\HD}'=(\Sig,h(\alphas),\ti{\betas},\ti{\gammas},
\la:\z\to\Ring,\ti{\spinct})\]
is the corresponding Heegaard triple subordinate to 
$(\ti{\Farc},\Fsphere)$ for $(\ti{W},\ti{\sur})$. Note that 
$\ti{\HD}'$ is obtained from $\ti{\HD}$ by a sequence $e$ of 
isotopy and handle slide on $\alpha$ curves i.e. moves of 
type (1) in Lemma \ref{lem:HDs-for-framed-arcs}. 
Theorem  \ref{thm:inv-arcs-2-handles} implies that 
$\Psi_e\circ \fmap_{\ti{\HD}}=\fmap_{\ti{\HD}'}\circ\Phi_e$, 
where 
\[\begin{split}
&\Phi_e:\HFT^{\M}(\Sig,\alphas,\ti{\betas},\la,\spinc)\to
\HFT^{\M}(\Sig,h(\alphas),\ti{\betas},\la,\spinc)\ \ 
\text{and}\\ 
&\Psi_e:\HFT^{\M}(\Sig,\alphas,\ti{\gammas},\la,\spinc')\to
\HFT^{\M}(\Sig,h(\alphas),\ti{\gammas},\la,\spinc'')
\end{split}\]  
are the isomormphisms associated with the Heegaard moves $e$. 
Since $d$ is induced by $h$, it follows that, 
\[d_\star\circ \fmap^{\M}_{\HD}=\Psi_e\circ \Phi'\circ
\fmap^{\M}_{\HD}=\fmap^{\M}_{\ti{\HD}'}\circ\Phi_e\circ \Phi=
\fmap^{\M}_{\ti{\HD}'}\circ d^0_\star.\]
Therefore, $\HFT^{\M}(d)\circ\fmap^{\M}_{\Cob,\Farc,\Fsphere}
=\fmap^{\M}_{\ti{\Cob},\ti{\Farc},\Fsphere}.$
\end{proof}

\begin{figure}%[ht]
\def\svgwidth{11cm}
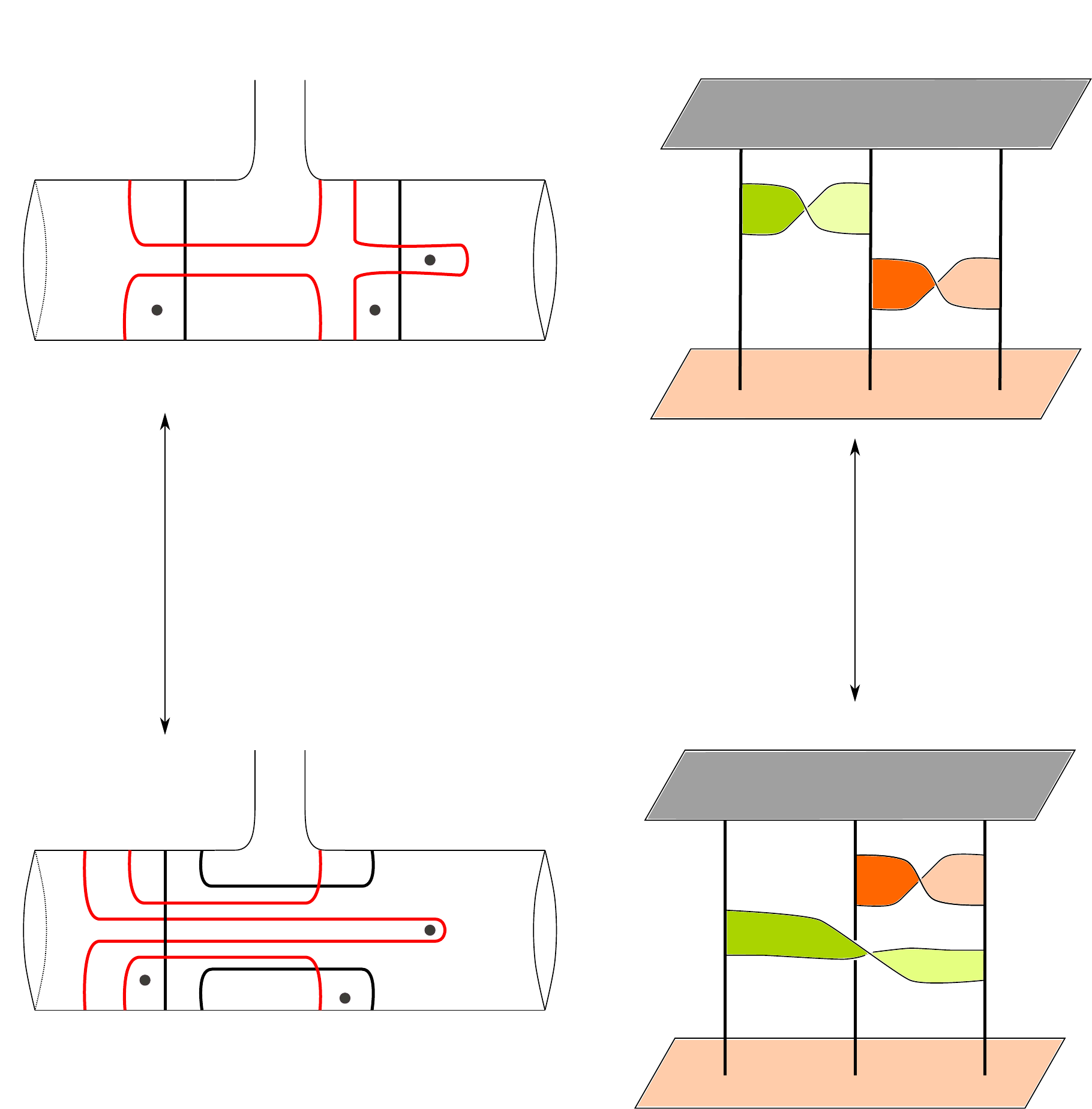
\caption{Sliding a framed arc $\Farc_1$ over 
another framed arc $\Farc_2$.
}\label{fig:arc-slide}
\end{figure}

We may thus restrict our attention to ordered sets of framed 
arcs. This is particularly useful when we study the 
composition law in the following subsections.

\subsection{A composition law for framed arcs and links}
\label{sec:weakcomposition}
Let $\Cob=[W,\sur,\spinct,\la_\sur]$ be an $\Ring$-cobordism 
from the 
$\Ring$-tangle $\Tangle=[M,T,\spinc,\la]$ to the $\Ring$-tangle 
$\Tangle'=[M',T',\spinc',\la']$. Further, assume that 
$(W,\sur)$ is the cobordism corresponding to an acceptable set 
of framed arcs $\Farc$, and a framed link $\Fsphere$. Given a 
decomposition $\Fsphere=\Fsphere^1\amalg \Fsphere^2$ and 
$\Farc=\Farc^1\amalg\Farc^2$ where $\Farc^1$ is acceptable, 
we obtain a decomposition 
\[(W,\sur)=(W_1,\sur_1)\cup_{(M'',T'')}(W_2,\sur_2).\]
Here, $(W_1,\sur_1)$ is the cobordism associated with 
$(\Farc^1,\Fsphere^1)$ and $(W_2,\sur_2)$ is the cobordism 
associated with $(\Farc^2,\Fsphere^2)$ in 
$(M'',T'')=(M(\Fsphere^1),T(\Farc^1))$. For $i=1,2$, let 
$\la_i$ denote the $\Ring$-coloring induced by $\la_\sur$ on 
$(W_i,\sur_i)$, and $\spinct_i=\spinct|_{W_i}$. Consider the 
corresponding $\Ring$-cobordisms 
\begin{align*}
&\Cob^1=[W_1,\sur_1,\spinct_1,\la_1]
:\Tangle\leadsto\Tangle''\quad\text{and}\quad 
\Cob^2=[W_2,\sur_2,\spinct_2,\la_2]
:\Tangle''\leadsto\Tangle',
\end{align*}
where $\Tangle''$ denotes the $\Ring$-tangle obtained by 
equipping $(M'',T'')$ with the induced $\Ring$-coloring and 
$\SpinC$ structure from $\la_\sur$ and $\spinct$, respectively. 

The restrictions of $\spinct$ to $W_1$ and $W_2$ remain 
unchanged, under modifying $\spinct$ by adding an element of 
$\delta H^1(M'',\Z)$, where $\delta:H^1(M'',\Z)\to H^2(W,\Z)$
is the connecting homomorphism in the Mayer-Vietoris 
sequence for $(W_1,W_2)$. So, let $\spincT$ be the set of all 
$\SpinC$ classes on $W$ so that their restrictions to $W_1$ 
and $W_2$ are equal to $\spinct_1$ and $\spinct_2$, 
respectively.

\begin{thm}\label{thm:composition}
With the above notation fixed, for  every $\Ring$-module 
$\M$ we  have
\[\fmap^{\M}_{\Cob^2,\Farc^2,\Fsphere^2}\circ
\fmap^{\M}_{\Cob^1,\Farc^1,\Fsphere^1}=\sum_{\bar{\spinct}
\in\spincT}\fmap_{\Cob(\bar{\spinct}),\Farc,\Fsphere}^\M,\]
where $\Cob(\bar{\spinct})$ is the $\Ring$-cobordism obtained 
from $\Cob$ by replacing $\spinct$ with $\bar{\spinct}$.
\end{thm}

\begin{proof}
Any acceptable set of framed arcs can be turned into an ordered 
set by arc slides. So Theorem \ref{thm:Handle-Slide-Invariance} 
implies that without loss of generality, we may assume $\Farc$ 
is ordered such that 
\[\Farc^1=\left(\Farc_1,\ldots,\Farc_{n_1}\right)\ \ \ \ 
\text{and}\ \ \ \ 
\Farc^2=\left(\Farc_{n_1+1},\ldots,\Farc_{n_1+n_2}\right), 
\ \ \ \text{where}\ n_1+n_2=n.\]
Furthermore, suppose that
\[ \Fsphere^1=\left(\Fsphere_1,\ldots,\Fsphere_{m_1}\right)
\ \ \ \ \text{and}\ \ \ \ 
\Fsphere^2=\left(\Fsphere_{m_1+1},\ldots,\Fsphere_{m=m_1+m_2}
\right).\] 

Let us fix an $\Ring$-diagram 
$\HD=(\Sig,\alphas,\betas,\gammas,\la,\spinct)$ subordinate
to $(\Farc,\Fsphere)$.  Let
$\deltas=\{\delta_1,\ldots,\delta_\ell\}$ denote a set of curves 
obtained as follows:
\begin{itemize}
\item Let $\delta_i$ be a Hamiltonian isotope of $\gamma_i$ 
for every $i$ in $$A=\{1,\ldots,n_1\}\cup\{n+1,\ldots,n+m_1\}.$$
\item For every $i\in\{1,\ldots,\ell\}-A$, let $\delta_i$ be a 
Hamiltonian isotope of $\beta_i$.
\end{itemize}
The $\SpinC$ structures $\spinct_1$ and $\spinct_2$ induce 
$\SpinC$ structures on $W_{\alpha\beta\delta}$ and 
$W_{\alpha\delta\gamma}$, which will also be denoted by 
$\spinct_1$ and $\spinct_2$, by slight abuse of notation.
Then the $\Ring$-diagrams
\[\HD_1=(\Sig,\alphas,\betas,\deltas,\la,\spinct_1)
\ \ \ \text{and}\ \ \ 
\HD_2=(\Sig,\alphas,\deltas,\gammas,\la,\spinct_2)\]
are subordinate to $(\Farc^1,\Fsphere^1)$ and 
$(\Farc^2,\Fsphere^2)$ and correspond to $\Cob_1$ and 
$\Cob_2$, respectively.
These two Heegaard triples determine the maps
$\fmap_1=\fmap_{\Cob^1,\Farc^1,\Fsphere^1}$ and
$\fmap_2=\fmap_{\Cob^2,\Farc^2,\Fsphere^2}$. 
Furthermore, the $\Ring$-diagram
\[(\Sig,\alphas,\betas,\deltas,\gammas,\la,\spincT)\]
and the distinguished generators 
\[\Theta_{\beta\delta}\in 
\HFT(\Sig,\betas,\deltas,\la,\spinct_1|_{M_{\beta\delta}}
=\spinc_0)\ \ \ \text{and}\ \ \
\Theta_{\delta\gamma}\in 
\HFT(\Sig,\deltas,\gammas,\la,\spinct_2|_{M_{\delta\gamma}}
=\spinc_0)\]
determine a holomorphic square map 
\[\Sqmap:\HFT(\Sig,\alphas,\betas,\la,\spinc)=\HFT(\Tangle)
\lra \HFT(\Sig,\alphas,\gammas,\la,\spinc')=\HFT(\Tangle').\]
Considering different possible degenerations of a square 
class of index $0$ and applying a mild generalization of   
\cite[Theorem 8.16]{OS-3m1}, we obtain the relation
\[\fmap_2\circ \fmap_1-
\sum_{\bar{\spinct}\in\spincT}
\fmap_{\alpha\beta\gamma,\bar{\spinct}}
(-\otimes \fmap_{\beta\delta\gamma}(\Theta_{\beta\delta}
\otimes\Theta_{\delta\gamma}))=\Sqmap\circ d+d\circ \Sqmap.\]
Here $\fmap_{\alpha\beta\gamma,\bar{\spinct}}$ and 
$\fmap_{\beta\delta\gamma}$ are the 
holomorphic triangle maps associated with the Heegaard triples
\[(\Sig,\alphas,\betas,\gammas,\la,\bar{\spinct})
\ \ \ \text{and}\ \ \ 
(\Sig,\betas,\deltas,\gammas,\la,\spinct_{0}),\]
respectively, and  $\spinct_{0}$ denotes a canonically 
determined 
$\SpinC$ structure on the $4$-manifold $W_{\beta\delta\gamma}$.
In order to complete the proof, it thus suffices to show that 
\[\fmap_{\beta\delta\gamma}(\Theta_{\beta\delta}
\otimes\Theta_{\delta\gamma})=\Theta_{\beta\gamma}.\]

Since $(\Sig,\alphas,\betas,\gammas,\z)$ is a Heegaard triple 
subordinate to $(\Farc,\Fsphere)$, the proof of 
Proposition~\ref{prop:generator} implies the existence of a 
labeling for the intersection points of $\beta_i$ and 
$\gamma_i$ by $d_i^+$ and $d_i^-$, for 
$i\in\{1,\ldots,\ell\}\setminus\{n+1,\ldots,n+m\}$, 
such that $\Theta_{\beta\gamma}$ is represented by 
\[\theta^+_{\beta\gamma}=\{d_1^+,\ldots,d_{\ell}^+\}.\]
Here, $d_i^+$ denotes the only intersection point of 
$\beta_i$ and $\gamma_i$ for $n+1\le i\le n+m$. Similarly, it 
follows from the definition of $\deltas$ that we may label the 
intersection points of $\deltas$ with $\betas$ and $\gammas$ 
such that $\Theta_{\beta\delta}$ and $\Theta_{\delta\gamma}$ 
are respectively represented by:
\[\theta^+_{\beta\delta}=\{c_1^+,\ldots,c_{\ell}^+\}\ \ \ \ \ \ 
\text{and}\ \ \ \ \ \theta^+_{\delta\gamma}
=\{b_1^+,\ldots,b_{\ell}^+\}.\] 

Let us now assume that 
$\theta\in\mathbb{T}_{\beta}\cap\mathbb{T}_\gamma$  
contributes to 
$\fmap_{\beta\gamma\delta}(\Theta_{\beta\delta}
\otimes\Theta_{\delta\gamma})$ through a triangle class 
\[\Delta=\Delta^\ell\in
\pi_2\left(\theta^+_{\beta\delta},\theta^+_{\delta\gamma},
\theta\right)\]
of Maslov index $0$. For any $1\le i\le \ell$, consider the 
Heegaard triple 
\[H^i=(\Sig^i=\Sig[\betas_{i+1},\ldots,\betas_{\ell}],\betas^i
=\{\beta_1,\ldots,\beta_i\},\deltas^i=\{\delta_1,\ldots,\delta_i\},
\gammas^i=\{\gamma_1,\ldots,\gamma_i\},\z)\]
where $\Sig[\betas_{i+1},\ldots,\betas_{\ell}]$ denotes the 
surface obtained from $\Sig$ by performing surgery along the 
$\beta$-curves $\{\beta_{i+1},\beta_{i+2},\ldots,\beta_{\ell}\}$. 
Therefore,  $\Sig^\ell$ is obtained from $\Sig^{\ell-1}$ by 
attaching a $1$-handle, and corresponding to $\Delta^{\ell}$, 
when the necks are sufficiently stretched, we obtain a 
triangles class $\Delta^{\ell-1}$ on $\Sig^{\ell-1}$ and a 
class $\Delta'$ on the attached one-handle which connects 
\[c^+_\ell\in \beta_{\ell}\cap \delta_\ell, \ \ \ \ \ 
b^+_\ell\in \delta_\ell\cap \gamma_\ell\ \ \ \ \ 
\text{and}\ \ \ \ \ 
d_\ell^\star\in\beta_\ell\cap\gamma_\ell,\]
where $\star$ is either $+$ or $-$. It also follows that 
\[\mu(\Delta^\ell)=\mu(\Delta^{\ell-1})
-\epsilon(d_\ell^\star),\]
where $\epsilon(d_\ell^+)=0$ and $\epsilon(d_\ell^-)=1$. 
We are thus forced to have $\star=+$.
It also follows from the argument of 
Proposition~\ref{prop:limit-moduli-space} and the second 
part of Lemma~\ref{lem:three-curves-general} that, if the 
the necks are sufficiently stretched,
$\Mod(\Delta^\ell)$ may be identified with 
$\Mod(\Delta^{\ell-1})$. This 
argument may in fact be repeated again and again to show that 
the generator $\theta$ uses the intersection points 
\[d^+_i\in\beta_i\cap\gamma_i,\ \ \ \ \text{for}\ \ 
i=n+m+1,\ldots,\ell,\]
and that $\Mod(\Delta^\ell)$ may be identified with 
$\Mod(\Delta^{n+m})$, for a triangle class over $\Sig^{n+m}$.

Now $\Sig^{n+m}=\Sig^{n+m-1}\# E$ where $E$ is a surface of 
genus $1$. If we stretch the connected sum neck, the complex 
structure on $\Sig^{n+m}$ converges to the join of complex 
structures on the $\Sig^{n+m-1}-\{w\}$ and $E-\{w'\}$.  
We also obtain a decomposition of $\Delta^{n+m}$ to the 
triangle classes $\Delta^{n+m-1}$ (on $\Sig^{n+m-1}$) and 
$\Delta'$ (on $E$). The generator $\theta$ is forced to use 
the unique intersection point 
$d_{n+m}\in\beta_{n+m}\cap\gamma_{n+m}$.
From the choice of the intersection points in 
$\theta_{\beta\delta}^+$ and $\theta_{\delta\gamma}^+$ it also 
follows that $\mu(\Delta^{n+m-1})=\mu(\Delta^{n+m})=0$. 
Furthermore, as the weak limit of a sequence of holomorphic 
curves in $\Mod(\Delta^{n+m})$ as the neck is stretched, 
we obtain a degenerate  holomorphic curve $u^{n+m-1}$ in 
the $0$-dimensional moduli space $\Mod(\Delta^{n+m-1})$, 
which has coefficient $k$ at $w$. The 
holomorphic curve $u^{n+m-1}$  determines a point 
$\rho(u^{n+m-1})$ in $\Sym^{k}(\D)$. Let $\Delta_k$ denote 
the union of all triangle classes $\Delta'$ over the genus-one 
surface $E$ with coefficient $k$ over the marked point $w'$. 
The arguments of Section 12 of \cite{Robert-cylindrical}, and 
in particular Lemma 12.2, Proposition 12.4 (in fact, 
Proposition A.3), imply that in the aforementioned weak limit,
$u^{n+m-1}$ is paired with a degenerate curve $v$ on $E$, 
which is the union of 
\[E\times \rho(u^{n+m-1})\subset E\times \D\]
and the unique holomorphic representative of $\Delta_0$, to 
produce the only possible curve in the weak limit. Every such 
weak limit may be perturbed to a holomorphic curve representing 
$\Delta^{n+m}$, giving an identification of 
$\Mod(\Delta^{n+m})$ with $\Mod(\Delta^{n+m-1})$, if the 
connected sum neck is sufficiently long.   Again, we may repeat 
the above argument to find an identification of 
$\Mod(\Delta^{\ell})$ with $\Mod(\Delta^n)$, provided that 
attaching the $1$-handles and taking connected sum with 
surfaces of genus $1$ is done using sufficiently stretched 
necks.

\begin{figure}
\def\svgwidth{14cm}
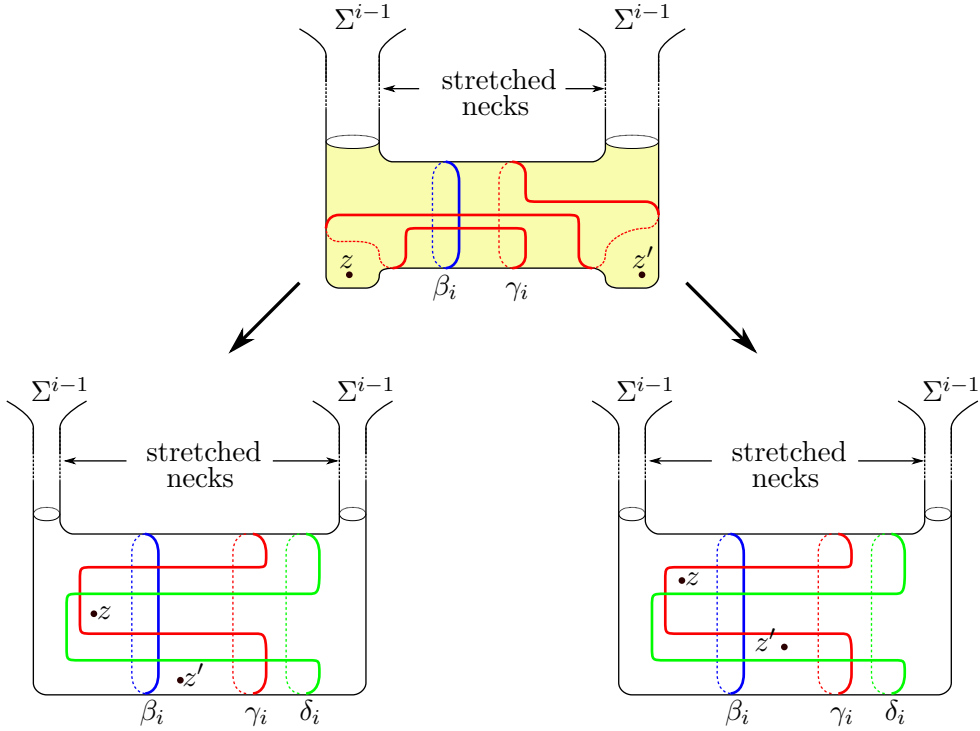
\caption{As we stretch the neck, $\Sig^i$ degenerates to 
a subsurface $\Sig^{i-1}$ and a sphere, which are attached 
by two long necks. When $i\in\{1,\ldots,n_1\}$ a pair of markings 
$z,z'\in\z$ will land on the sphere as illustrated on the 
left-hand-side, while for $i\in\{n_1+1,\ldots,n\}$ their location 
follows the pattern illustrated on the right-hand-side.
}\label{fig:Three-curves}
\end{figure}

The surface $\Sig^n$ is obtained from $\Sig^{n-1}$ by 
attaching a $1$-handle and moving  two of the markings (which 
we denote by $z,z'\in \z$) over the attached $1$-handle. 
The curves $\beta_n,\delta_n$ and $\gamma_n$ are all isotopic 
to the belt circle of the attached $1$-handle, while the 
positions of $z$ and $z'$ in the cylinder representing the neck 
is the position illustrated on the right-hand-side of  
Figure~\ref{fig:Three-curves}. Correspondingly, when the necks 
are sufficiently stretched, we obtain a triangles class 
$\Delta^{n-1}$ on $\Sig^{n-1}$ and a class $\Delta'$ on the 
attached one-handle which connects the intersection points
\[c^+_n\in \beta_{n}\cap \delta_n, \ \ \ \ \ 
b^+_n\in \delta_n\cap \gamma_n\ \ \ \ \ \text{and}\ \ \ \ \ 
d_n^\star\in\beta_n\cap\gamma_n,\]
where $\star$ is either $+$ or $-$.  It also follows that 
\[\mu(\Delta^n)=\mu(\Delta^{n-1})-\epsilon(d_n^\star),\]
where $\epsilon(d_n^+)=0$ and $\epsilon(d_n^-)=1$. 
We are thus forced to have $d_n=d^+_n$. It also follows from 
the argument of Proposition~\ref{prop:limit-moduli-space} 
and the second part of Lemma~\ref{lem:three-curves-general} 
that, if the the necks are sufficiently stretched,
$\Mod(\Delta^n)$ may be identified with $\Mod(\Delta^{n-1})$. 
This argument may  be repeated  to show that the 
generator $\theta$ uses the intersection points 
\[d^+_i\in\beta_i\cap\gamma_i,\ \ \ \ \text{for}
\ \ i=1,\ldots,n,\]
and that $\Mod(\Delta^\ell)$ may be identified with 
$\Mod(\Delta^{1})$, which consists of a single point.

The above argument shows that $\theta$ is forced to 
represent the generator $\Theta_{\beta\gamma}$, and that 
the total contribution of holomorphic triangles 
to the coefficient of $\theta$ in 
$\fmap_{\beta\gamma\delta}(\Theta_{\beta\delta}
\otimes\Theta_{\delta\gamma})$
is $1$. This completes the proof of the theorem.
\end{proof}

\begin{cor}
Assume $\Cob=[W,\sur,\spinct,\la_{\sur}]$ is an 
$\Ring$-cobordism from $\Tangle$ to $\Tangle'$, determined 
by an acceptable set of framed arcs $\Farc$ and a framed link 
$\Fsphere$ i.e. $(\Farc,\Fsphere,\mathrm{id})$ gives a 
parametrization of $(W,\sur)$. Let $\Farc'$ be an acceptable 
set of framed arcs obtained from $\Farc$ by sliding some of its 
components over some components of $\Fsphere$, and similarly 
$\Fsphere'$ be the framed link obtained from $\Fsphere$ by 
sliding some of its components over some of the framed arcs in 
$\Farc$. Denote the induced diffeomorphisms by
\[d_{\Farc}:\Tangle(\Farc',\Fsphere)\to \Tangle' \ \ \ 
\text{and}\ \ \ d_{\Fsphere}:\Tangle(\Farc,\Fsphere')\to
\Tangle'.\]
Then for any $\Ring$-module $\M$ we have
\[\fmap_{\Cob,\Farc,\Fsphere}^{\M}=\HFT^{\M}(d_{\Farc})\circ
\fmap^{\M}_{\Cob,\Farc',\Fsphere}\ \ \ \text{and}\ \ \ 
\fmap_{\Cob,\Farc,\Fsphere}^{\M}=\HFT^{\M}(d_{\Fsphere})\circ
\fmap^{\M}_{\Cob,\Farc,\Fsphere'}.\]
\end{cor}
\begin{proof}
Note that $(\Farc',\Fsphere,d_{\Farc})$ and 
$(\Farc,\Fsphere',d_{\Fsphere})$ are parametrizations of 
$\Cob$, or $(W,\sur)$). The claim is then a straightforward 
result of Theorems \ref{thm:Handle-Slide-Invariance} and 
\ref{thm:composition}
\end{proof}

This completes the proof of Theorem~\ref{thm:map-definition}, 
and we may now denote the $\Ring$-homomorphism
associated with the $\Ring$-cobordism $\Cob$
and the $\Ring$-module $\M$ by $\fmap_{\Cob}^{\M}$.

\subsection{The composition law}\label{subsec:composition}
In this subsection, we prove a generalization of 
Theorem~\ref{thm:composition}.

\begin{thm}\label{thm:composition-law}
Suppose that $\Cob_i=[W_i,\sur_i,\spinct_i,\la_i]:
\Tangle_{i-1}\ra \Tangle_i$ are $\Ring$-cobordisms for 
$i=1,\ldots,m$. Let $\Cobb=(W,\sur)$ be the stable cobordism 
obtained by putting $\Cobb_i=(W_i,\sur_i)$ together and 
$\la:\pi_0(\sur)\ra \Ring$ denote the $\Ring$-coloring 
induced by $\la_i$. For every $\spinct\in\SpinC(W)$ with 
$\spinct|_{W_i}=\spinct_i$, let
 $\Cob(\spinct)=[W,\sur,\spinct,\la]$. Then 
\begin{equation}\label{eq:composition-law}
\sum_{\substack{
\spinct\in\SpinC(W)\\ \spinct|_{W_i}
=\spinct_i}} \fmap_{\Cob(\spinct)}^{\M}
=\fmap_{\Cob_m}^\M\circ \cdots  \circ \fmap^{\M}_{\Cob_1}.
\end{equation}
\end{thm}
\begin{proof}
The definition of the cobordism invariants and 
Theorem~\ref{thm:composition} reduce the proof of 
Theorem~\ref{thm:composition-law} to the case where 
$m=2$, both $\Cob_1$ and $\Cob_2$ are parametrized elementary 
$\Ring$-cobordisms, and one of the following happens.
\begin{enumerate}
\item $\Cobb_1$ corresponds to a framed $2$-sphere 
$\Fsphere_1$, and $\Cobb_2$ corresponds to a 
framed $0$-sphere $\Fsphere_2$. 
\item $\Cobb_1$ corresponds to a framed $2$-sphere 
$\Fsphere_1$, and $\Cobb_{2}$ corresponds to a 
framed knot $\Fsphere_2$. 
\item $\Cobb_1$ corresponds to a framed $2$-sphere $\Fsphere$, 
and $\Cobb_{2}$ corresponds to a framed arc $\Farc$.
\item  $\Cobb_1$ corresponds to a framed knot $\Fsphere_1$, and 
$\Cobb_2$ corresponds to a framed $0$-sphere $\Fsphere_2$. 
\item $\Cobb_1$ corresponds to a framed arc $\Farc$, and 
$\Cobb_2$ corresponds to a framed $0$-sphere $\Fsphere$. 
\end{enumerate}
For all cases, there is a unique $\SpinC$ class $\spinct$ on 
$W$ such that $\spinct|_{W_i}=\spinct_i$ for $i=1,2$. We
may thus set  $\Cob=\Cob(\spinct)$. Let 
$\Tangle=\Tangle_0=[M,T,\spinc,\la]$ 
and $\Tangle'=\Tangle_2=[M',T',\spinc',\la']$. 

In the first case, $W_1$ is determined by the framed $2$-sphere 
$\Fsphere_1$, so the underlying tangle of $\Tangle_1$ is 
$(M_1,T_1)=(M(\Fsphere_1),T)$. Considering this identification, 
we may assume that $\Fsphere_2$ lies in 
$M\setminus \Fsphere_1=M\cap M_1$. Thus, $\Fsphere_2$ specifies 
a cobordism $\Cobb_1^r$ from $(M,T)$ to $(M(\Fsphere_2),T)$. 
Similarly, $\Fsphere_1$ specifies a cobordism $\Cobb_2^r$ from 
$(M(\Fsphere_2),T)$ to $(M',T')$. As a result we get an indexed 
parametrized decomposition 
$\Cobb_1^r\cup_{(M(\Fsphere_2),T)}\Cobb_2^r$ for $(W,\sur)$. 

The $\Ring$-coloring $\la$ and the $\SpinC$ class $\spinct$ 
make $\Cobb_1^r$ and $\Cobb_2^r$ into $\Ring$-cobordisms, 
denoted by $\Cob_1^r$ and $\Cob_2^r$, respectively. Let   
$\Tangle^r=[M(\Fsphere_2),T,\spinc^r
=\spinct|_{M(\Fsphere_2)},\la]$. 
Consider the following diagram:
\begin{displaymath}
\begin{diagram}
\HFT^\M(\Tangle)&\rTo{\fmap^{\M}_{\Cob_1}}&
\HFT^\M(\Tangle_1)\\
\dTo{\fmap^{\M}_{\Cob_1^r}}&
&\dTo{\fmap^{\M}_{\Cob_2}}\\
\HFT^\M(\Tangle^r)&\rTo{\fmap^{\M}_{\Cob_2^r}}&
\HFT^\M(\Tangle').
\end{diagram}
\end{displaymath}  

Every generator of $\CFT^\M(\Tangle)$ is of the form 
$\x\times\theta$, where $\x$ is a generator of the chain 
complex $\CFT^\M(\Tangle_1)$ and $\theta$ is one of the two  
intersection points $\theta_{\alpha\beta}^1$ and 
$\theta_{\beta\alpha}^1$ corresponding to the framed 
$2$-sphere $\Fsphere_1$, (see Section~\ref{sec:1-handlemap}). 
If the necks in the corresponding Heegaard diagrams for 
$\Fsphere_1$ and $\Fsphere_2$ are sufficiently stretched, this 
generator is mapped to 
$\x\times \theta\times \theta^2_{\alpha\beta}$ under 
$\fmap^{\M}_{\Cob_1^r}$. Here $\theta^2_{\alpha\beta}$ and 
$\theta^2_{\beta\alpha}$ are the two intersection points which 
correspond to the $1$-handle, attached to $\Fsphere_2$. 
The generator $\x\times\theta\times \theta^2_{\alpha\beta}$ 
goes to zero under $\fmap^{\M}_{\Cob_2^r}$ unless 
$\theta=\theta^1_{\beta\alpha}$, when it is mapped to 
$\x\times \theta^2_{\alpha\beta}$. On the other hand,
the image of $\x\times\theta$ under $\fmap^{\M}_{\Cob_1}$
is zero unless $\theta=\theta^1_{\beta\alpha}$, when 
\[\fmap^{\M}_{\Cob_1}(\x\times\theta^1_{\beta\alpha})=\x\ \ \ 
\text{and}\ \ \ \fmap^{\M}_{\Cob_2}\left(
\fmap^{\M}_{\Cob_1}(\x\times\theta^1_{\beta\alpha})\right)
=\x\times\theta^2_{\alpha\beta}.\]
This implies that the above diagram is commutative and 
\begin{align*}
\fmap_{\Cob_2}^\M\circ\fmap_{\Cob_1}^\M=\fmap_{\Cob_2^r}^{\M}
\circ\fmap_{\Cob_1^r}^{\M}=\fmap_{\Cob}^{\M}.
%\sum_{\substack{\spinct\in\SpinC(W)\\ \spinct|M=\spinc,
%\spinct|_{M'}=\spinc'}}
\end{align*}
A similar modification to the proofs of  
\cite[Proposition 4.19 and Proposition 4.18]{OS-4mfld}, 
imply the heorem in the second and fourth cases, respectively. 
The only remaining cases are thus the third and the fifth 
cases.

In the third case, $\Cob_{1}$ corresponds to attaching a 
$3$-handle along $\Fsphere$  and $\Cob_{2}$
corresponds to a framed arc  $\Farc$. This framed arc  
may clearly be isotoped to $M\setminus\Fsphere$. 
Therefore, we get an indexed parametrized decomposition of 
$(W,F)$ by switching the order of attached handles. More 
precisely, let $\Cobb_1^r$ to be the cobordism determined by 
$\Farc$ from $(M,T)$ to $(M,T(\Farc))$ and $\Cob_2^r$ be the 
cobordism determined by $\Fsphere$ from $(M,T(\Farc))$ to 
$(M',T')$. For $i=1,2$, the $\Ring$-cobordism obtained by 
equipping $\Cobb_i^r$ with the $\Ring$-coloring induced by 
$\la$ and $\SpinC$ structure induced by $\spinct$ is denoted 
by $\Cob_i^r$. 
Correspondingly, we may choose a Heegaard surface $\Sig$ in 
$M$ with the following properties:
\begin{itemize}
\item The intersection of $\Sig$ with $\Fsphere$ is an annulus 
$A$, such that the surface $\Sig^\circ$ in $M_1$ obtained by 
cutting off $A$ and gluing in a pair of disks to the two 
boundary components, is a Heegaard surface for $M_1$.  
Here, $M_1$ is the underlying $3$-manifold of $\Tangle_1$. 
\item Each component of $T$ intersects $\Sig$ transversely 
in exactly one point. We let $\z=T\cap \Sig$.
\item There are two collections $\alphas$ and $\betas$ of 
attaching circles on $\Sig$ such that $\alphas\cap A$ is a 
single circle $\alpha$ which is a small Hamiltonian isotope 
of the circle $\{\beta\}=\betas\cap A$. Moreover, the diagram 
$(\Sig,\alphas,\betas,\la:\z\ra \Ring,\spinc)$ is a Heegaard 
diagram for $\Tangle$ and 
\[(\Sig^\circ,\alphas\setminus\{\alpha\},\betas
\setminus\{\beta\},\la:\z\ra\Ring,\spinc_1)\]
is a Heegaard diagram from $\Tangle_1$, 
where $\spinc_1=\spinct|_{M_1}$.
\item There is a third collection $\gammas$ of pairwise 
disjoint circles on $\Sig$ so that $\gammas\cap A$ is a 
single circle $\gamma$ which is a small Hamiltonian isotope 
of both $\alpha$ and $\beta$ and
$$(\Sig^\circ,\alphas\setminus\{\alpha\},
\betas\setminus\{\beta\},\gammas\setminus\{\gamma\},\z)$$
is subordinate to the framed arc $\Farc$ in $\Tangle_1$. 
In particular, the diagram
$(\Sig^\circ,\alphas\setminus\{\alpha\},\gammas
\setminus\{\gamma\},\z)$ 
is a Heegaard diagram for $(M',T')$.
\end{itemize}
Note that $(\Sig,\alphas,\betas,\gammas,\z)$ is subordinate 
to $\Farc$, where $\Farc$ is considered as a framed arc in 
$\Tangle$. Consider such a Heegaard diagram and assume that the 
almost complex structure is sufficiently stretched along the 
boundary circles of $A$. Every generator of 
$\CFT^{\M}(\Tangle)$ is of the form $\x\times\theta$, where 
$\x$ is a generator of $\CFT^{\M}(\Tangle_1)$ and $\theta$ is 
one of the two intersection points $\theta_{\alpha\beta}$ and 
$\theta_{\beta\alpha}$ between $\alpha$ and $\beta$ which 
correspond to $\Fsphere$. The image of such a generator under 
$\fmap_{\Cob_1}^{\M}$ is trivial unless 
$\theta=\theta_{\beta\alpha}$, when we have 
$\fmap_{\Cob_1}^{\M}(\x\times \theta_{\beta\alpha})=\x$.
On the other hand, we may use the argument of 
Theorem~\ref{thm:composition} and show that 
\[\fmap_{\Cob_1^r}^{\M}(\x\times\theta)=\fmap_{\Cob_2}^{\M}
(\x)\times\theta'\] 
where $\theta'\in\alpha\cap\gamma$ is the intersection point 
which corresponds to $\theta$. Every such generator is mapped 
to $0$ by $\fmap_{\Cob_2^r}^{\M}$ unless 
$\theta'=\theta_{\gamma\alpha}$, or equivalently, unless 
$\theta=\theta_{\beta\alpha}$. If this condition is satisfied, 
then 
\[\fmap_{\Cob_2^r}^{\M}\left(\fmap_{\Cob_1^r}^\M
(\x\times\theta_{\beta\alpha})\right)=\fmap_{\Cob_2}^\M(\x),\]
which completes the proof of the third case. The proof of the 
fifth case is similar to the proof of the third case.
\end{proof}

The composition law of Theorem~\ref{thm:composition-law} 
implies, in particular, that the left-hand-side expression is 
well-defined. Let us define 
\[W^0=\bigcup_{i\equiv 0\ \mathrm{mod}\ 2 } W_i
\ \ \ \text{and}\ \ \ 
W^1=\bigcup_{i\equiv 1\ \mathrm{mod}\ 2 } W_i.\] 
Then $W^0\cup W^1=W$ and $W^0\cap W^1=\cup_{i=1}^{m-1} M_i$, 
where $M_i=M_{\Tangle_i}$, and we obtain the following 
cohomology long exact  sequence
\begin{diagram}
\cdots  &\rTo& \bigoplus_{i=1}^{m-1}H^1(M_i,\Z)&\rTo{\delta}
& H^2(W)&\rTo{\pi}&\bigoplus_{i=1}^mH^2(W_i,\Z)&\rTo&\cdots  
\end{diagram}
If $\spinct,\spinct'\in\SpinC(W)$ restrict to $\spinct_i$ 
on $W_i$, then $\spinct-\spinct'\in H^2(M,\Z)$ is a class in 
the kernel of $\pi$, and is thus in the image of $\delta$. 
In particular, the subset $\spincT\subset \SpinC(W)$ which 
appears in the summation of the left-hand-side of 
Equation~\ref{eq:composition-law} is the orbit of a fixed 
$\SpinC$ structure $\spinct$ under the action of the 
$\Z$-module $\Image(\delta)$ of $H^2(W,\Z)$. On the other hand, 
using an appropriate Morse datum $\Modi$, we may represent 
every $\Z$-submodule of $H^2(W,\Z)$ as $\Image(\delta)$ for 
some decomposition of the stable cobordism $(W,\sur)$. 
In particular, for every affine set $\spincT$ of $\SpinC$ 
structures over a $\Z$-submodule of $H^2(W,\Z)$ which 
restrict to $\spinc$ and $\spinc'$ on the two ends, the sum 
\[\sum_{\spinct\in\spincT}\fmap^\M_{\Cob_\spinct}:
\HFT^\M(\Tangle)\ra \HFT^\M(\Tangle')\]
is well-defined. 

\begin{defn}\label{def:cobordism-map}
Let $\Cob=[W,\sur,\spincT,\la_\sur]:\Tangle\leadsto \Tangle'$ 
be an arbitrary $\Ring$-cobordism from the $\Ring$-tangle 
$\Tangle$ to the $\Ring$-tangle $\Tangle'$, where $\spincT$ is 
a subset of $\SpinC(W)$ which is affine over a $\Z$-submodule 
of $H^2(W,\Z)$. For every $\spinct\in\spincT$ define 
$\Cob_\spinct=[W,\sur,\spinct,\la]$. We then define the 
cobordism map associated with $\Cob$ by 
\[\fmap_\Cob^\M:
=\sum_{\spinct\in\spincT}\fmap_{\Cob_\spinct}^\M.\]
\end{defn}  

With the above definition in place, we may then re-state 
Theorem~\ref{thm:map-definition} and 
Theorem~\ref{thm:composition} as the following theorem.

\begin{thm}\label{thm:functor}
Fix an algebra $\Ring$ over $\F$ and an $\Ring$-module 
$\M$. Assigning the $\Ring$-module $\HFT^\M(\Tangle)$ to every
$\Ring$-tangle $\Tangle\in \mathrm{Obj}(\ACobCat)$ and the 
$\Ring$-homomorphism $\fmap_\Cob^\M:\HFT^\M(\Tangle)\ra 
\HFT^\M(\Tangle')$ to every $\Ring$-cobordism
\[(\Cob:\Tangle\leadsto\Tangle')\in 
\mathrm{Mor}(\Tangle,\Tangle')\in \mathrm{Mor}(\ACobCat)\] 
gives a well-defined functor
\[\HFT^\M:\ACobCat\lra \AModuleCat.\]
\end{thm}

\subsection{Action of $\Lambda^*(H_1(W,\Z)/\mathrm{Tors})$}
Let us assume that $\Cob=[W,\sur,\spinct,\la_\sur]$ is an 
$\Ring$-cobordism from $\Tangle=[M,T,\spinc,\la]$ to 
$\Tangle'=[M',T',\spinc',\la']$. Consider a decomposition of 
$\Cob$ as
\[\Cob=\Cob_1\cup_{\Tangle_1}\Cob_2\cup_{\Tangle_2} \Cob_3\]
where $\Cob_1$ corresponds to the addition of $1$-handles, 
$\Cob_2$ corresponds to the addition of $2$-handles along 
some framed link and band surgeries along framed arcs,
and $\Cob_3$ corresponds to the addition of $3$-handles. Let 
$(M_i,T_i)$ be the underlying tangle of $\Tangle_i$ and 
$\Cob'=\Cob_2=[W',\sur',\spinct',\la_{\sur'}]$.
It is clear that $H_1(W',\Z)=H_1(W,\Z)$.

Assume $\Cob'$ is parametrized by a pair $(\Farc,\Fsphere)$ 
of an acceptable set of framed arcs $\Farc$, and a framed 
link $\Fsphere$ such that  $M_2=M_1(\Fsphere)$ and 
$T_2=T_1(\Farc)$. The homomorphism
$\fmap^\M_{\Cob'}=\fmap^{\M}_{\Cob',\Farc,\Fsphere}$ is 
defined  using an $\Ring$-diagram
\[\HD=(\Sig,\alphas,\betas,\gammas,\la,\spinct)\]
subordinate to $(\Farc,\Fsphere)$. Let us denote by 
$W_{\alpha\beta\gamma}$ the $4$-manifold obtained from the 
Heegaard triple $\HD$. There is an epimorphism 
\[\pi: H_1\left(M_{\alpha\beta}\amalg M_{\beta\gamma}
\amalg M_{\alpha\gamma},\Z\right)/\mathrm{Tors}
\lra H_1(W_{\alpha\beta\gamma},\Z)/\mathrm{Tors}=
H_1(W,\Z)/\mathrm{Tors}.\]
Every element $\zeta\in H_1(W,\Z)/\mathrm{Tors}$ may be 
represented as $\pi(\zeta_{\alpha\beta},\zeta_{\beta\gamma},
\zeta_{\alpha\gamma})$ with 
\begin{align*}
(\zeta_{\alpha\beta},\zeta_{\beta\gamma},\zeta_{\alpha\gamma})
\in H_1(M_{\alpha\beta}\amalg M_{\beta\gamma}\amalg 
M_{\alpha\gamma},\Z)/\mathrm{Tors}. 
\end{align*}
We may then define 
\begin{align*}
&\fmap_{\HD}^\zeta:
\HFT(\Sig,\alphas,\betas,\la,\spinct|_{M_{\alpha\beta}})\lra 
\HFT(\Sig,\alphas,\gamma,\la,\spinct|_{M_{\alpha\gamma}})\\
&\fmap^\zeta_{\HD}(\x):=
\fmap_{\alpha\beta\gamma}\left((\zeta_{\alpha\beta}\cdot\x)
\otimes \Theta_{\beta\gamma}+\x\otimes (\zeta_{\beta\gamma}
\cdot\Theta_{\beta\gamma})\right)-\zeta_{\alpha\gamma}\cdot
\fmap_{\alpha\beta\gamma}(\x\otimes\Theta_{\beta\gamma}).
\end{align*}
Correspondingly, we may define 
\begin{align*}
&\bar{\fmap}_{\HD}^\M:\HFT^\M(\Tangle_1)\otimes 
\Lambda^*(H_1(W,\Z)/\mathrm{Tors})\lra \HFT^\M(\Tangle_2),
\quad\quad\bar{\fmap}_{\HD}^\M(\x\otimes\zeta):=
\fmap_{\HD}^\zeta(\x).
\end{align*}
It is then implied by \cite[Lemma 2.6]{OS-4mfld}  that the 
above map is in fact well-defined, and does not depend on 
the representation of $\zeta$ as $\pi( \zeta_{\alpha\beta},
\zeta_{\beta\gamma},\zeta_{\alpha\gamma})$.
After composing with the maps $\fmap^\M_{\Cob_1}$ and 
$\fmap^\M_{\Cob_3}$ we obtain an induced map, which may 
be denoted by 
\begin{align*}
&\ovl\fmap_{\Cob}^\M:\HFT^\M(\Tangle)\otimes 
\Lambda^*(H_1(W,\Z)/\mathrm{Tors})\lra \HFT^\M(\Tangle').
\end{align*}
We may then follow the steps taken in \cite{OS-4mfld} to show 
the invariance of the map $\fmap_{\Cob}^\M$
and show that $\ovl\fmap_{\Cob}^\M$ is also well-defined.

\subsection{Relative $\SpinC$ structures and the cobordism map}
\label{subsec:Filtration}
Suppose $\Cobb=(W,\sur)$ is a stable cobordism from a balanced 
tangle $(M^0,T^0)$ to a balanced tangle $(M^1,T^1)$. Let 
$(X^i,\tau^i)$ be the sutured manifold associated with 
$(M^i,T^i)$ for $i=0,1$. Specifically, 
$X^i=M^i\setminus\nd(T^i)$ for a tubular neighborhood 
$\nd(T^i)$ around $T^i$. Moreover, if 
$T^i=\amalg_{j=1}^{\el}T^i_j$, then 
$\tau^i=\amalg_{j=1}^{\el}\tau^i_j$ 
where $\tau^i_j\subset \del \nd(T^i)$ is the meridian of 
$T^i_j$  for each $1\le j\le \el$. 

Denote the Poincar\'e dual of the homology class 
$[\tau^i_j]\in H_1(X^i,\Z)$ by $\chi^i_j\in 
\Ht^2(X^i,\del X^i,\Z)$. Recall that the set of relative 
$\SpinC$ classes on $(M^i,T^i)$,  denoted by 
$\SpinC(M^i,T^i)$, %(or $\SpinC(X^i,\tau^i)$) 
is an affine space over $\Ht^2(X^i,\del X^i,\Z)$ and sits 
in the exact sequence
\begin{displaymath}
\begin{diagram}
0&\rTo &\big\langle \chi_j^i%\in\Ht^2(X,\partial X;\Z)
\ |\ 1\le j\le \el\big\rangle_\Z
&\rTo &\SpinC(M^i,T^i)&\rTo&\SpinC(M^i)&\rTo&0.
\end{diagram}
\end{displaymath}
The surface $\sur$ induces an equivalence relation on the 
components of $T^i$ by setting $T^i_j\sim T^i_k$ if they are 
subsets of the boundary of the same component of $\sur$. Let
\begin{displaymath}
\mathcal{I}_{i}=\left\langle \chi^i_j-\chi^i_k~|~ T^i_j
\sim T^i_k\right\rangle\subset \Ht^2(X^i,\del X^i,\Z)
\end{displaymath}
and define
\begin{equation}\label{quotient:hom}
\begin{split}
&\SpinC_{\sur}(M^i,T^i):=\frac{\SpinC(M^i,T^i)}{\mathcal{I}_i}
\quad\text{and}\quad
\Hbb^i_{\sur}:= \frac{\Ht^2(X^i,\del X^i,\Z)}{\mathcal{I}_i}.\\
\end{split}
\end{equation}
Consequently, if $\sur=\amalg_{j=1}^m\sur_j$, each component 
$\sur_j$ of $\sur$ determines a equivalence class in 
$\Hbb^i_{\sur}$, denoted by $\eta_j^i$.

Correspondingly, this equivalence relation specifies an 
$\mathbb{F}$-algebra ($\mathbb{F}=\Z/2\Z$), denoted by 
$\Ring_{\sur}$, which is isomorphic to a quotient of 
$\Ring_i=\Ring_{(M_i,T_i)}\otimes_{\mathbb{Z}}\mathbb{F}$ 
by the ideal determined by $\sim$. More precisely, consider 
the polynomial ring 
$\F[\la_1,\ldots,\la_m]$.  If $\del^+M^0=\amalg_{i=1}^k S_i^-$ and 
$\del^-M^0=\amalg_{j=1}^lS_{j}^+$, then
$$\del_h^-W=\amalg_{i=1}^kS_i^-\times [0,1]\ \ \ \ 
\text{and}\ \ \ \ \del_h^+W=\amalg_{j=1}^lS_j^+\times [0,1].$$
Associated with any connected component of $\del_h W$, 
we define 
\[\la_i^-:=\prod_{\sur_j\cap S_i^-\neq\emptyset}\la_j\ \ \ 
\text{for}\ \ 1\le i\le k \ \ ,\ \ \la^+_i:=
\prod_{\sur_j\cap S_i^+\neq
\emptyset}\la_j\ \ \ \text{for}\ \ \ 1\le i\le l.\]
If we set $\la^-=\sum_{i=1}^k\la_i^-$ and 
$\la^+=\sum_{i=1}^l\la_i^+$, it follows that
\begin{equation}\label{Cob-algebra}
\Ring_{\sur}:=\frac{\F[\la_1,\ldots,\la_{m}]}
{\langle \la_i^-~|~g(S_i^-)>0\rangle 
+\langle \la_i^+~|~g(S_i^+)>0\rangle
+\langle\la^+-\la^-\rangle}.
\end{equation}

The map $\la_{\sur}:\pi_0(\sur)\ra \Ring_{\sur}$ mapping 
$\sur_i$ to $\la_i$ for $i=1,\ldots,m$, is an 
$\Ring_{\sur}$-coloring on $(W,\sur)$. Thus, 
$\Cob=[W,\sur,\spinct,\la_{\sur}]$ is an 
$\Ring_{\sur}$-cobordism for every $\spinct\in\SpinC(W)$. 
Let $\la^i:\pi_0(T^i)\ra\Ring_{\sur}$ denote the  
$\Ring_{\sur}$-coloring induced by $\la_{\sur}$ on 
$(M_i,T_i)$. Associated with this coloring, there is a natural 
filtration on $\Ring_{\sur}$ by $\Hbb_\sur^i$ defined by 
\begin{displaymath}
\begin{split}
&\chi^i:G(\Ring_{\sur})\ra\Hbb_{\sur}^i=\Ht^2(X^i,\del X^i,\Z)
\quad\quad
\chi^i(\prod_{j=1}^m\la_j^{a_j}):=\sum_{j=1}^m a_j\eta_j^i,
\end{split}
\end{displaymath}
see \cite[Section 3.2]{AE-1}.
Thus, for any $\SpinC$ class $\spinc^i\in\SpinC(M^i)$ we have
$$\HFT(M^i,T^i,\spinc^i,\la^i)=\bigoplus_{\relspinc^i
\in\spinc^i\subset\SpinC_{\sur}(M^i,T^i)}\HFT
(M^i,T^i,\relspinc^i,\la^i)$$ 

\begin{lem}\label{Filter-preserv}
Consider the $\Ring_\sur$-cobordism 
$\Cob=[W,\sur,\spinct,\la_\sur]$ as above and
let $\spinc^0=\tfrak|_{M^0}$ and $\spinc^1=\tfrak|_{M^1}$. 
Suppose that for an element $\x$ in 
$\HFT(M^0,T^0,\relspinc^0,\la^0)$ with 
$\relspinc^0\in\spinc^0$ we have 
$\fmap_{\Cob}(\x)\in\HFT(M^1,T^1,\relspinc^1,\la^1)$ where 
$\relspinc^1\in\spinc^1$. Then, $\fmap_{\Cob}$ induces the maps
\[\fmap_{\Cob}:\HFT(M^0,T^0,\relspinc^0
+\sum_{j=1}^m a_j\eta_j^0,\la^0)\to 
\HFT(M^1,T^1,\relspinc^1+\sum_{j=1}^ma_j\eta^1_j,\la^1),\]
for every choice of $a_1,\ldots,a_m\in\Z$.
\end{lem}
\begin{proof}
It is enough to prove the lemma in the case where $(W,\sur)$ is 
determined by $(\Farc,\Fsphere)$, for an acceptable set of 
framed arcs $\Farc$ and a framed link $\Fsphere$. Let
 $$\HD=(\Sig,\alphas,\betas,\gammas,\spinct,\la_{\sur})$$
be an $\Ring_{\sur}$-diagram for $\Cob$ whose underlying 
Heegaard triple is subordinate to $(\Farc,\Fsphere)$.  Further, 
assume that the distinguished generator $\Theta_{\beta\gamma}$ 
is represented by an intersection point 
$\theta_{\beta\gamma}\in\Tb\cap \Tc$ and the marked points in 
$\z$ are labelled such that $z_i$ corresponds to $T_i^0$.  
Consider the intersection points $\x,\y\in\Ta\cap\Tb$ and let
\[\relspinc(\y)-\relspinc(\x)=\sum_{j=1}^ma_j\eta_j^0,\]
where $a_j\in \Z$ for all $j$.  Here, abusing the notation 
$\relspinc(\cdot)$ denotes the equivalence class of the 
relative $\SpinC$ structure represented by the corresponding 
generator.  Let $\Delta_x\in\pi_2(\x,\Theta_{\beta\gamma},\x')$ 
and $\Delta_y\in\pi_2(\y,\Theta_{\beta\gamma},\y')$ be 
triangle classes representing the $\SpinC$ structure $\tfrak$ 
for some $\x',\y'\in\Ta\cap\Tc$.   Then, there are disks 
$\phi\in\pi_2(\y,\x)$ and $\psi\in\pi_2(\x',\y')$ such that 
$\Delta_y=\phi\star\Delta_x\star\psi$. Since 
$a_j=n_j(\phi)=\sum_{T_i^0\subset \del\sur_j}n_{z_i}(\phi)$, 
\[\begin{split}
\relspinc(\prod_{j=1}^m\la_j^{n_j(\Delta_y)}\cdot\y')-
\relspinc(\prod_{j=1}^m\la_j^{n_j(\Delta_x)}\cdot\x')
&=\sum_{j=1}^m(n_j(\Delta_y)-n_j(\Delta_x))\eta_j^1+
\relspinc(\y')-\relspinc(\x')\\
&=\sum_{j=1}^m(n_j(\Delta_y)-n_j(\Delta_x))\eta_j^1
-\sum_{j=1}^mn_j(\psi)\eta_j^1\\
&=\sum_{j=1}^mn_j(\phi)\eta_j^1=\sum_{j=1}^ma_j\eta_j^1
\end{split}
\]
where $n_j(\Delta_{\bullet})
=\sum_{T_i^0\subset\del\sur_j}n_{z_i}(\Delta_{\bullet})$ 
for $\bullet=x,y$ and 
$n_j(\psi)=\sum_{T_i^0\subset\del\sur_j}n_{z_i}(\psi)$.
This completes the proof of lemma.
\end{proof}

\newpage
\section{Applications and special cases}
\label{sec:applications}
\subsection{Cobordisms between closed $3$-manifolds}
Let $\sY=(Y,p)$ be an oriented, closed $3$-manifold $Y$ with 
a based point $p\in Y$.  Associated with $\sY$, there is a 
balanced tangle $(Y_p,T_p)$ defined as follows.  Let 
$p_-,p_+\in Y$ be points close to $p$ and $T\subset Y$ be 
an embedded oriented arc passing through $p$ such that 
$\del^-T=p_-$ and  $\del^+T=p_+$. Then, $Y_p$ is constructed 
by removing small disjoint balls (also disjoint from $p$) 
around $p_-$ and $p_+$ and $T_p=T\cap Y_p$. Note that 
$\del^-Y_p$ and $\del^+Y_p$ are the boundary of spheres around 
$p_-$ and $p_+$, respectively, and $T_{p}$ has one connected 
component. So, there is an obvious $\mathbb{F}[\la]$-coloring 
on $(Y_p,T_p)$ labeling $T_p$ by $\la$. Here, as before 
$\mathbb{F}=\Z/2\Z$. Let $\Tangle_{\sY,\spinc}=
(Y_p,T_p,\spinc,\la_p)$ for every $\spinc\in\SpinC(Y)$.

\begin{defn}
Let $\sY=(Y,p)$ and $\sY'=(Y',p')$ be oriented, closed, based 
$3$-manifolds. A \emph{decorated} cobordism $\sX=(X,\sig)$ 
from $\sY$ to $\sY'$ is a smooth, oriented $4$-manifold
$X$ with $\del X=-Y\amalg Y'$ and a properly embedded arc 
$\sig\subset X$ such that $\del \sig=p\amalg p'$.
\end{defn}

Associated with any decorated cobordism $\sX=(X,\sig)$ from 
$\sY=(Y,p)$ to $\sY'=(Y',p')$ one may construct a stable 
cobordism $(X_{\sigma},\sur_{\sigma})$ from $(Y_p,T_p)$ to 
$(Y'_p,T'_p)$, as follows.  Let $T\subset Y$ and $T'\subset Y'$ 
be embedded, oriented arcs containing $p$ and 
$p'$ respectively, and let $\del^\bullet T=p_\bullet$ and 
$\del^\bullet T'=p'_{\bullet}$ for $\bullet=+,-$. Consider 
parallel disjoint copies of $\sig$ in $X$, denoted by $\sig^-$ 
and $\sig^+$, such that $\del\sig^-=p_-\amalg p_-'$, 
$\del \sig^+=p_+\amalg p_+'$ and 
$\sig^-\cup T'\cup\sig^+\cup T$ bounds an embedded disk $D$ in 
$X$. Then, $X_\sigma$ is obtained from $X$ by removing small, 
disjoint tubular neighborhoods around $\sig^-$ and $\sig^+$, 
while $\sur_{\sigma}=D\cap X_{\sigma}$. Let $\la_{\sigma}$ 
denote the $\mathbb{F}[\la]$-coloring on 
$(X_\sigma,\sur_{\sigma})$ which labels $F_{\sigma}$ by $\la$, 
and set $\Cob_{\sX,\spinct}=(X_{\sigma},\sur_{\sigma},\spinct,
\la_{\sigma})$ for every $\spinct\in\SpinC(X)$.

For every closed, oriented, based $3$-manifold $\sY=(Y,p)$ and 
every $\spinc\in\SpinC(Y)$, the homology groups 
$\HFT^{\M}(\Tangle_{\sY,\spinc})$, for $\M$ equal to $\F$, 
$\F[\la]$, $\F[\frac{1}{\la}]$ or $\F[\la,\frac{1}{\la}]$, 
are equal to $\ov\HFT(Y,\spinc;\F)$, $\HFT^-(Y,\spinc;\F)$, 
$\HFT^+(Y,\spinc;\F)$ and $\HFT^\infty(Y,\spinc;\F)$, 
respectively. Moreover, for every decorated cobordism 
$\sX=(X,\sig)$ from $(Y,p)$ to $(Y',p')$ and every $\SpinC$ 
structure $\spinct\in\SpinC(X)$, the cobordism map 
$\fmap^{\Rin}_{\Cob_{\sX,\spinct}}$ is the cobordism map of 
Ozsv\'ath and Szab\'o in any of the aforementioned cases.

\subsection{Functoriality of link Floer homology} 
Another important example of $\Ring$-tangles is given by 
multi-pointed links.
\begin{defn}
A \emph{multi-pointed link} is a triple $\Lin=(Y,L,\p)$ 
where $L$ is an oriented link in a closed, connected, 
oriented $3$-manifold $Y$ together with a finite set 
$\p\subset L$ of based points such that every component of 
$L$ contains at least one base point.
\end{defn}
Associated with any multi-pointed link $\Lin=(Y,L,\p)$ we 
define a balanced tangle $(Y_{\p},L_{\p})$ as follows. Assume 
$\p=\{p_1,\ldots,p_n\}$ and consider $n$ pairwise of disjoint arc 
segments $I=\amalg_{i=1}^nI_i$ in $L$ such that $p_i\in I_i$. 
Using the orientation induced from $L$ on $I$, let 
$\p_-=\del^-I$ and $\p_+=\del^+I$. Then, $Y_{\p}$ is obtained 
from $Y$ by removing small disjoint ball neighborhoods around 
the points in $\p_-$ and $\p_+$ and 
$\del^\bullet Y_{\p}\subset \del Y_{\p}$ is the union of sphere 
boundary components around $\p_\bullet$ for $\bullet =+, -$.   
Furthermore, \[L_{\p}=-\left((L\backslash I)\cap 
Y_{\p}\right)\amalg\left(I\cap Y_{\p}\right).\]

\begin{defn}
A \emph{decorated cobordism} from $\Lin=(Y,L,\p)$ to 
$\Lin'=(Y',L',\p')$ is a triple  $\mathcal{F}=(X,\sur,\sigma)$ 
as follows.
\begin{enumerate}
\item $X$ is a smooth, oriented 4-manifold with 
$\del X=-Y\amalg Y'$.
\item $\sur\subset X$ is a smoothly embedded, oriented surface 
such that $\del \sur=-L\amalg L'$.
\item $\sig\subset \sur$ is a union of embedded, pairwise 
disjoint, oriented arcs such that $\del^-\sig=\p$, 
$\del^+\sig=\p'$ and every component of $\sur\setminus\sig$  
with positive genus intersects more than one component of 
$L\setminus\p$ and $L'\setminus\p'$.
\end{enumerate}
\end{defn}

To any decorated cobordism $\mathcal{F}=(X,F,\sigma)$ from 
$\Lin=(Y,L,\p)$ to $\Lin'=(Y',L',\p')$, we assign a cobordism 
$(X_{\sigma},\sur_{\sigma})$ from the tangle $(Y_{\p},L_{\p})$ 
to $(Y'_{\p'},L'_{\p'})$.  We choose the labeling for  
$\p=\amalg_{i=1}^np_i, \p'=\amalg_{i=1}^np_i'$ and 
$\sig=\amalg_{i=1}^n\sig_i$ such that 
$\del\sig_i=p_i\amalg p_i'$ for all $1\le i\le n$. Let 
$I=\amalg_{i=1}^nI_i$ and $I'=\amalg_{i=1}^nI'_i$ where 
$I_i\subset L$ and $I_i'\subset L'$ are connected segments 
containing $p_i$ and $p_i'$, respectively. Let 
$$\p_\bullet=\del^\bullet 
I=\amalg_{i=1}^np_{i\bullet}\quad\text{and}\quad\p'_\bullet
=\del^\bullet I'=\amalg_{i=1}^np'_{i\bullet}$$
where $p_{i\bullet}=\del^\bullet I_i$ and 
$p'_{i\bullet}=\del^\bullet I'_i$ for $\bullet=-,+$. Consider 
parallel copies $\sig_i^-,\sig_i^+\subset \sur$ of each 
$\sig_i$ such that $\del \sig_i^\bullet
=p_{i\bullet}\amalg p'_{i\bullet}$ for $\bullet=-,+$. 
Moreover, $\sig_i^-\cup\sig_i^+\cup I_i\cup I_i'$ bounds a 
disk $D_i\subset \sur$ containing $\sig_i$ such that 
$D_1,\ldots,D_n$ are pairwise disjoint. Let 
$\sur^\circ=\sur\setminus(\amalg_{i=1}^nD_i)=\amalg_{j=1}^m F_j^\circ$. For each $1\le i\le n$, $m^{\bullet}(i)$ is defined  such that $\sig_i^\bullet\subset\del F_{m^{\bullet}(i)}^{\circ}$ where $\bullet=+,-$.

Then, $X_{\sigma}$ is constructed from $X$ by removing disjoint 
small tubular neighborhood of the arcs 
$\amalg_{i=1}^n(\sig_i^-\amalg\sig_i^+)$ while 
$$\sur_{\sigma}=-(\sur^\circ\cap X_{\sigma})
\amalg(\amalg_{i=1}^nD_i\cap X_{\sigma}).$$

Abusing the notation, we denote $D_i\cap X_{\sigma}$ by $D_i$ 
and $\sur^\circ_j\cap X_{\sigma}$ by $\sur^\circ_j$ where 
$\sur^\circ=\amalg_{j=1}^m\sur^\circ_j$. The algebra associated 
to the cobordism $(X_{\sigma},F_{\sigma})$, defined as in 
Equation~\ref{Cob-algebra}, is equal to
\[\Ring=\frac{\F[\zet_1,\ldots,\zet_m,\la_1,\ldots,\la_n]}{\sum_{i=1}^n \la_i\zet_{m^+(i)}-\sum_{i=1}^n\la_i\zet_{m^-(i)}},\]
and the map $\la_{\sigma}:\pi_0(\sur_{\sigma})\ra \Ring$,
defined by
\begin{displaymath}
\begin{cases}
\la(\sur^\circ_j)=\zet_j\ \ \ \ 1\le j\le m\\
\la(D_i)=\la_i\ \ \ \ \ 1\le i\le n,
\end{cases}
\end{displaymath}
is an $\Ring$-coloring on $(X_{\sigma},F_{\sigma})$. So, 
$\Cob_{\mathcal{F},\spinct}=[W_{\sigma},\sur_{\sigma},
\spinct,\la_{\sigma}]$ is an $\Ring$-cobordism for every 
$\spinct\in\SpinC(X)$.  Let $\Tangle_{\Lin,\spinc}$ 
(respectively $\Tangle_{\Lin',\spinc'}$) denote the 
$\Ring$-tangle obtained from equipping $(Y_{\p},L_{\p})$ 
(respectively $(Y'_{\p'},L'_{\p'})$) with the coloring induced 
from $\la_{\sigma}$ and the $\SpinC$ structure 
$\spinc=\spinct|_{Y}$ (respectively $\spinc'=\spinct|_{Y'}$).

Then for every $\Ring$-module $\Rin$ we have an 
$\Ring$-homomorphism 
$$\fmap_{\Cob_{\mathcal{F},\spinct}}^{\Rin}:
\HFT^{\Rin}(\Tangle_{\Lin,\spinc})
\ra\HFT^{\Rin}(\Tangle_{\Lin',\spinc'}).$$
For $\Rin=\Ring$, following the discussions of 
Section~\ref{subsec:Filtration}, 
the chain complexes $\CFT(\Tangle_{\Lin,\spinc})$ 
and  $\CFT(\Tangle_{\Lin',\spinc'})$ are 
$(\Ring,\Hbb_{\sur})$ and $(\Ring,\Hbb_{\sur}')$ filtered chain 
complexes where $\Hbb_{\sur}$ and $\Hbb_{\sur}'$ are defined 
as in Equation~\ref{quotient:hom}. Moreover, the cobordism map 
$\fmap_{\Cob_{\mathcal{F},\spinct}}$ preserves the relative 
filtration in the sense of Lemma \ref{Filter-preserv}. 
In particular, for $\Rin=\F[\la_1,\ldots,\la_n]$ which has 
the structure of an $\Ring$-module via  the homomorphism 
$\phi:\Ring\ra\Rin$ which maps all $\zet_i$ to zero, 
we obtain an invariant homomorphism 
$$\fmap_{\mathcal{F},\spinct}=\fmap_{\Cob_{\mathcal{F},
\spinct}}^{\Field[\la_1,\ldots,\la_n]}:\mathrm{HFL}^-
(Y,L,\p,\spinc)\ra\mathrm{HFL}^-(Y',L',\p',\spinc').$$

\newpage
%----------------------------------------------
\bibliographystyle{hamsalpha}
\bibliography{HFBibliography}
\end{document}